\documentclass[reqno,10pt]{amsart}
\usepackage{amsthm,amsfonts,amssymb,euscript,mathrsfs,graphics,color,amsmath,latexsym,marginnote,hyperref,todonotes,multicol,yhmath,imakeidx}
\usepackage[latin1]{inputenc}

\theoremstyle{plain}

\usepackage{times}

\numberwithin{equation}{section}

\def\cD{{\mathcal D}}
\def\cP{{\mathcal P}}

\def\Z{{\mathbb Z}}
\def\N{{\mathcal N}}
\def\R{{\mathbb R}}
\def\T{{\mathbb T}}
\def\e{{\varepsilon}}
\def\g{{\gamma}}

\newcommand{\bt}{\beta}
\def\s{{\sigma}}

\def\d{{\delta}}

\def\og{{\lambda}}
\def\lal{{\mathfrak l}}
\newcommand{\tq}{{\mathtt q}}

\newcommand{\op}{{\rm Op}}

\newcommand{\ol}{\overline}

\def\B{{\mathcal{B}}}

\def\+R{+_{_{ \!\! \R}}}

\def\diag{\,\mbox{diag}\,}

\def\pa{\partial}

\def\hat{\widehat}

\def\bar{\overline}

\DeclareMathAlphabet{\mathpzc}{OT1}{pzc}{m}{it}

\theoremstyle{plain}%definition
\newtheorem{Teo}{Theorem}
\newtheorem{teor}{Theorem}[section]
\newtheorem{ese}[teor]{Example}
\newtheorem{prop}[teor]{Proposition}
\newtheorem{lem}[teor]{Lemma}
\newtheorem{cor}[teor]{Corollary}
\newtheorem{defi}[teor]{Definition}
\newtheorem{remark}[teor]{Remark}

\newcommand{\bdm}{\begin{displaymath}}
\newcommand{\edm}{\end{displaymath}}
\newcommand{\bpb}{\begin{prob}}
\newcommand{\epb}{\end{prob}}

\newcommand{\beq}{\begin{equation}}
\newcommand{\eeq}{\end{equation}}
\newcommand{\bem}{\begin{multline}}
\newcommand{\eem}{\end{multline}}
\newcommand{\bes}{\begin{ese}}
\newcommand{\ees}{\end{ese}}
\newcommand{\bde}{\begin{defi}}
\newcommand{\ede}{\end{defi}}
\newcommand{\bpr}{\begin{prop}}
\newcommand{\epr}{\end{prop}}
\newcommand{\ble}{\begin{lem}}
\newcommand{\ele}{\end{lem}}
\newcommand{\bte}{\begin{teor}}
\newcommand{\ete}{\end{teor}}
\newcommand{\bco}{\begin{cor}}
\newcommand{\eco}{\end{cor}}

\newcommand{\gotp}{{\mathfrak p}}

\newcommand{\gotC}{{\mathfrak C}}

\newcommand{\gotL}{{\mathfrak L}}
\newcommand{\gotM}{{\mathfrak M}}
\newcommand{\gotN}{{\mathfrak N}}

\newcommand{\al}{\alpha}
\newcommand{\be}{\beta}

\newcommand{\x}{\xi}

\newcommand{\f}{\varphi}
\renewcommand{\s}{\sigma}

\newcommand{\del}{\partial}
\newcommand{\td}{{\mathtt D}}

\newcommand{\oo}{\omega}

\newcommand{\tb}{\mathtt{b}}

\newcommand{{\resonance}}{relevant self-energy cluster }

\newcommand{\ii}{{\rm i}}

\numberwithin{equation}{section}
\def\f{{\varphi}}

\def\g{{\gamma}}

\def\pa{\partial}

\def\cO{\mathcal{O}}

\def\e{\varepsilon}
\def\N{{\mathbb N}}
\newcommand{\calO}{{\mathcal O}}
\renewcommand{\to}{\rightarrow}
\newcommand{\tk}{{\mathtt k}}
\newcommand{\su}{\s_1}
\newcommand{\lD}{{\langle D_x\rangle}}

\renewcommand{\R}{\mathbb{R}}

\renewcommand{\T}{\mathbb{T}}

\newcommand{\calF}{{\mathcal F}}
\newcommand{\calG}{{\mathcal G}}
\newcommand{\calH}{{\mathcal H}}

\newcommand{\calL}{{\mathcal L}}

\renewcommand{\calO}{{\mathcal O}}

\newcommand{\calQ}{{\mathcal Q}}
\newcommand{\RR}{{\mathcal R}}

\theoremstyle{plain}
\makeatletter
\def\l@subsection{\@tocline{2}{0pt}{2.5pc}{5pc}{}}
\def\l@subsubsection{\@tocline{2}{0pt}{4.5pc}{5pc}{}}
\renewcommand\tocchapter[3]{%
  \indentlabel{\@ifnotempty{#2}{\ignorespaces#2.\quad}}#3%
}
\newcommand\@dotsep{4.5}
\def\@tocline#1#2#3#4#5#6#7{\relax
  \ifnum #1>\c@tocdepth % then omit
  \else
    \par \addpenalty\@secpenalty\addvspace{#2}%
    \begingroup \hyphenpenalty\@M
    \@ifempty{#4}{%
      \@tempdima\csname r@tocindent\number#1\endcsname\relax
    }{%
      \@tempdima#4\relax
    }%
    \parindent\z@ \leftskip#3\relax \advance\leftskip\@tempdima\relax
    \rightskip\@pnumwidth plus1em \parfillskip-\@pnumwidth
    #5\leavevmode\hskip-\@tempdima{#6}\nobreak
    \leaders\hbox{$\m@th\mkern \@dotsep mu\hbox{.}\mkern \@dotsep mu$}\hfill
    \nobreak
    \hbox to\@pnumwidth{\@tocpagenum{#7}}\par
    \nobreak
    \endgroup
  \fi}
\makeatother
\AtBeginDocument{%
\makeatletter
\expandafter\renewcommand\csname r@tocindent0\endcsname{0pt}
\makeatother
}
\def\l@subsection{\@tocline{2}{0pt}{2.5pc}{5pc}{}}

\begin{document}

\def\red{\color{red}\tiny}
\def\EX+{\marginnote{\red Explain more?}}
\def\EX{\marginnote{\red Explain better}}
\def\TD{\marginnote{\red To do/check}}
\def\TA{\marginnote{\red To adjust}}

\title{\bf Reducible KAM tori for Degasperis-Procesi equation}
%\date{}

\author{Roberto Feola}
\address{UnivNantes, Nantes}
\email{roberto.feola@univ-nantes.fr}

\author{Filippo Giuliani}
\address{UPC, Barcelona}
\email{filippo.giuliani@upc.edu}

\author{Michela Procesi}
\address{RomaTRE}
\email{procesi@mat.uniroma3.it}

\thanks{
This research was supported by Miur-PRIN 2015 ``Variational methods, with applications to problems in 
mathematical physics and geometry'' n. 2015KB9WPT-008, by
the ERC grant ``Hamiltonian PDEs and small divisor problems: a dynamical systems approach'' n. 306414 under FP7 and
 by the ERC project ``FAnFArE''
n. 637510.\\
This project has received funding from the European Research Council (ERC) under the European Union's Horizon 2020 research and innovation programme under grant agreement No 757802.
}

%\maketitle

%%%%%%%%%%%%%%%%%%%%%%%%%%%%%%%%%%%%%%%%%%%%%%%%%%%%%%%%%%%%%%%%%%%%%%%%%
%%%%%%%%%%%%%%%%%%%%%%%%%%%%%%%%%%%%%%%%%%%%%%%%%%%%%%%%%%%%%%%%%%%%%%%%%
\begin{abstract}
We develop KAM theory close to an elliptic fixed point for 
quasi-linear Hamiltonian perturbations of the dispersive Degasperis-Procesi equation on the circle.  
The overall strategy in KAM theory for quasi-linear PDEs is based on Nash-Moser nonlinear iteration, pseudo differential calculus and normal form techniques.
In the present case the complicated {\it symplectic structure}, the {\it weak dispersive} effects of the linear flow and the presence of {\it strong resonant interactions}  require a novel set of ideas. The main points are to exploit the integrability of the unperturbed equation, to look for special {\it wave packet} solutions and to perform a very careful algebraic analysis of the resonances.\\
Our approach is quite general and can be applied also to other $1$d integrable PDEs. We are confident for instance that the same strategy should work for the Camassa-Holm equation.  
\end{abstract} 
%%%%%%%%%%%%%%%%%%%%%%%%%%%%%%%%%%%%%%%%%%%%%%%%%%%%%%%%%%%%%%%%%%%%%%%%%
%%%%%%%%%%%%%%%%%%%%%%%%%%%%%%%%%%%%%%%%%%%%%%%%%%%%%%%%%%%%%%%%%%%%%%%%%

\maketitle

\tableofcontents

\section{Introduction and Main result}

In this paper we prove 
existence and  stability of Cantor families of quasi-periodic, 
small amplitude,  
solutions for  quasi-linear Hamiltonian 
perturbations of the Degasperis-Procesi (DP) equation
\begin{equation}\label{DP}
u_t-u_{x x t}+u_{xxx}-4 u_x-u u_{xxx}-3 u_x u_{xx}+4 u u_x
+\mathcal{N}_8( u, u_x, u_{xx}, u_{xxx})=0
\end{equation}
under periodic boundary conditions 
$x\in \mathbb{T}:=\mathbb{R}\setminus2\pi\mathbb{Z}$, 
where
\begin{equation}\label{N6}
\mathcal{N}_8( u, u_x, u_{xx}, u_{xxx})
:=-(4-\partial_{xx}) \partial_x[(\partial_u f)( u)]\, ,
\end{equation} 
 the ``Hamiltonian density'' $f$ belongs to $C^{\infty}(\mathbb{R}, \mathbb{R})$ 
and is such that 
\begin{equation}\label{HamiltonianDensity}
f(u)=O(u^9),
\end{equation}
where $O(u^9)$ denotes a function with a zero of order at least nine at the origin.
The equation \eqref{DP} is a  Hamiltonian PDE
of the form $u_t=J\,\nabla H(u)$ where 
$\nabla H$ is the $L^2(\mathbb{T},\mathbb{R})$ gradient and the function
\begin{equation}\label{DPHamiltonian}
H(u)=\int \frac{u^2}{2}-\frac{u^3}{6}+f(u)\, dx\,,
\quad J=(1-\partial_{xx})^{-1}(4-\partial_{xx})\partial_x\,
\end{equation}
is defined on the  phase space $H_0^1(\mathbb{T}):=\left\{ u\in H^1(\mathbb{T}, \mathbb{R}) : 
\int_{\mathbb{T}} u\,dx=0\right\}$.
%\begin{equation}\label{H01}
%H_0^1(\mathbb{T}_x):=\left\{ u\in H^1(\mathbb{T}, \mathbb{R}) : 
%\int_{\mathbb{T}} u\,dx=0\right\}\,.
%%\quad \Omega(u, v):=\int_{\mathbb{T}} (J^{-1} u)\,v\,dx, 
%%\quad \forall u, v\in H_0^1(\mathbb{T}_x).
%\end{equation}
The equation \eqref{DP} for $f=0$ is the 
 DP equation which was first proposed in \cite{DegPro} in the form  
\begin{equation}\label{DPvera}
u_t+c_0 u_x+\gamma u_{xxx}-\alpha^2 u_{xx t}=
\left(-\frac{2 c_1}{\alpha^2} u^2+c_2 (u_x^2+ u u_{xx})\right)_x\,,
\end{equation}
where $c_0,c_1,c_2,\gamma,\al\in \mathbb{R}$, $\al\neq0$.
%which is equivalent to \eqref{DP} 
By applying 
Galilean boosts, translations 
and time rescaling to \eqref{DPvera} one obtains equation \eqref{DP} 
with $f=0$.

The DP equation can be 
regarded as a model for nonlinear shallow water 
dynamics and its asymptotic accuracy is the same 
as for the Camassa-Holm equation and a degree 
more than the KdV equation \cite{CL}.
There is a rather large literature on this equation
starting form the paper \cite{Deg} in which the \emph{complete integrability}
is proved.
The local and global well-posedness, for instance, have been extensively studied
as well as existence of wave breaking phenomena 
(peakons, N-peakons solutions).  Without trying to be exhaustive 
we quote   \cite{ConE}, \cite{Const1}, \cite{ConstIv}, \cite{CamHolm}, \cite{LiuYin}, \cite{Yin1}
%, \cite{Yin2} %, \cite{Wu} 
and we refer to 
\cite{FGPa} and references therein for more literature  about Degasperis-Procesi equation.

Actually many of these results (notably the wave breaking) 
are studied in the {\it dispersionless} case, 
which corresponds to \eqref{DP} with $f=0$ and 
$u\rightsquigarrow u+1$. % (controllare segno).
In the present paper 
the presence of the dispersive terms $-4 u_{x}+ u_{xxx}$ 
is fundamental. 
Our main purpose is to prove existence of quasi-periodic solutions 
in high Sobolev regularity by following a KAM approach. In this setting
 a \emph{quasi}-periodic solution 
 with $\nu\in \mathbb{N}$ 
 frequencies  is defined by an embedding  
\begin{equation}\label{6.6}
\T^{\nu}\ni\f\mapsto U(\f,x) \in  H_0^1(\mathbb{T})
\end{equation}
and a frequency vector
$\oo \in \R^{\nu}$, with rationally independent entries,
such that $u(t,x)=U(\oo t,x)$ is a solution  of \eqref{DP} and 
$U(\f,x)\in H^p(\T^{\nu+1},\R)$ for some $p$ sufficently large .

Notice that, in a neighbourhood of $u=0$, \eqref{DP}
 can be seen 
% either as a perturbation of the integrable 
% Degasperis-Procesi equation (the case $f=0$) 
% or 
as a perturbation  of the linear PDE
\begin{equation}\label{lineare}
v_t-v_{x x t}+v_{xxx}-4 v_x=0\,,
\end{equation}
whose bounded solutions have the form 
%\begin{equation}\label{linearwave}
%v(t, x)=\sum_{j\in\mathbb{Z}} v_j\,e^{\mathrm{i}(\og(j) t+j x)},
%\end{equation}
%where
\begin{equation}\label{dispersionLaw}
v(t, x)=\sum_{j\in\mathbb{Z}} v_j\,e^{\mathrm{i}(\og(j) t+j x)}\,,\qquad 
\og(j):=j\,\frac{4+j^2}{1+j^2}=j+\frac{3 j}{1+j^2}, \qquad j\in\mathbb{Z}\,,
\end{equation}
 where $j\mapsto \lambda(j)$
is the {\it linear} dispersion law.
It is easily seen that all solutions of \eqref{lineare} with compact Fourier support are periodic,
but with period depending on the support.
%This difference arises from the fact that, for any $j\in\mathbb{Z}$, $\og(j)\in\mathbb{Z}$ in the KdV case, while $\og(j)\in\mathbb{Q}$ %in the DP case.
%
%\smallskip
%
%In the KdV case the dispersion law is superlinear as $j\to \infty$, in the DP case is \textit{asymptotically linear}. This fact makes a significant difference in the study of quasi-periodic motions for these equations. In particular the case with less dispersion, namely the DP case, is much harder and
%we underline that Theorem \ref{MainResult}, at the best of our knowledge, is the first KAM result for quasi-linear PDE's with this kind of dispersion, which also implies the first existence result for quasi-periodic solutions of the Degasperis-Procesi equation. 
%%This is part of a joint work with Roberto Feola and Michela Procesi. 
In this context it is natural to investigate whether equation \eqref{DP} 
has periodic or quasi-periodic solutions \emph{close to} to 
small amplitude 
linear solutions  \eqref{dispersionLaw}.
We remark that, since the solutions of \eqref{dispersionLaw}
are all periodic, the existence of 
quasi-periodic solutions, if any,  
strongly 
rely on the presence of the quadratic nonlinearity in \eqref{DP}. 

In the present paper 
we construct 
quasi-periodic solutions mainly supported  
 in Fourier space at 
$\nu\ge 2$ distinct  \textit{tangential sites} %. To this purpose we define
\begin{equation}\label{TangentialSitesDP}
S^+:=\{\overline{\jmath}_1, \dots, \overline{\jmath}_{\nu} \}, \quad S:=S^{+}\cup (-S^+), \quad %S:=S^+\cup(-S^+)=\{ \pm j : j\in S^+\}, \quad 
\overline{\jmath}_i\in \mathbb{N}\setminus\{ 0\}, \quad \forall i=1,\dots, \nu\,,
\end{equation}
where, without loss of generality, we shall always assume that  
$\overline{\jmath}_1=\max_{i=1,\ldots, \nu}\overline{\jmath}_i$.
We denote by
\begin{equation}\label{LinearFreqDP}
\overline{\omega}:=\left(\frac{\overline{\jmath}_1(4+\overline{\jmath}_1^2)}{1+\overline{\jmath}_1^2},\dots, \frac{\overline{\jmath}_{\nu}(4+\overline{\jmath}_{\nu}^2)}{1+\overline{\jmath}_{\nu}^2}\right)\in\mathbb{Q}^{\nu}
\end{equation}
the linear frequencies of oscillations related to the tangential sites. 
More precisely our solutions will have the form
\begin{equation}\label{SoluzioneEsplicitaDP}
u(t, x; \xi)=2\sum_{i=1}^\nu\sqrt{\xi_i}\,\cos(\omega_i t+\bar\jmath_i x)
+o(\sqrt{\lvert \xi \rvert}), \quad \omega=\bar\omega +O(\lvert \xi \rvert),
\end{equation}
where $o(\sqrt{\lvert \xi \rvert})$
is meant  in the 
 $H^{s}$-topology with $s$ large.  
 It is well know that in looking for quasi-periodic solutions 
 ``small divisors'' problems arise.
To overcome such problems
 we shall require
 that $S^+$ satisfies a {\it wave packet} condition and that the \emph{unperturbed amplitudes} 
 $\xi$ belong to 
 an appropriate Cantor-like set of positive measure. 

\noindent
The following definition quantifies  the \emph{wave packet} condition.
\begin{defi}\label{Def:cono}
For $\mathtt{r}\in (0, 1)$, we say that a set of natural numbers
$S^+=\{\overline{\jmath}_1, \dots, \overline{\jmath}_{\nu}\}$ 
%as in \eqref{TangentialSitesDP} 
is in $\mathcal{V}(\mathtt{r})$ if
\begin{equation}\label{caso2}
\min_{i=1, \dots, \nu}  \overline{\jmath}_i >\frac{1}{\mathtt{r}}
\qquad \mbox{and} 
\qquad 
\left\lvert \frac{\overline{\jmath}_i}{\overline{\jmath}_1}
-1 \right\rvert\le\mathtt{r}\,;
\end{equation}
\begin{equation}\label{GenericAssumption1}
\begin{aligned}
&\sum_{i=1}^{\nu} 
\frac{\overline{\jmath}_i}{1+\overline{\jmath}_i^2}\,\ell_i\ne  0\,, 
\quad  \forall \ell\in \Z^{\nu}:\quad \lvert \ell \rvert = 4\,.
\end{aligned}
\end{equation}
\end{defi}
\noindent
Denoting by $B(0,\varrho)$
the 
ball centred at the origin of $\mathbb{R}^{\nu}$ of radius $\varrho>0$,  
our result can be stated as follows.
\begin{Teo}\label{MainResult}
Let $\nu\in \mathbb{N}$, $\nu\geq 2$, and consider  
$f\in C^{\infty}(\mathbb{R},\mathbb{R})$ 
satisfying \eqref{HamiltonianDensity}. 
There exists a constant $\mathtt{r}_0>0$ 
such that, for any 
choice of $S^+$ in $\mathcal{V}(\mathtt{r})$, with $0<\mathtt{r}\leq \mathtt{r}_0$,
there exist $s\gg1$,  $0<\varrho\ll1$  %such that for any $0<\e\leq \varrho$ there exists   
and a positive measure  
 Cantor-like set $\mathfrak{A}\subseteq B(0,\varrho)$
 such that
the following holds. 
For any $\x\in \mathfrak{A}$,
the equation \eqref{DP} possesses a small amplitude 
quasi-periodic solution $u(t,x; \xi)=U(\omega t,x; \xi)$ of the form \eqref{SoluzioneEsplicitaDP} 
where 
$U(\f,x)\in H^{s}(\mathbb{T}^{\nu+1},\mathbb{R})$ and 
$\omega:=\omega(\xi)\in\mathbb{R}^{\nu}$ is a diophantine frequency vector. Moreover for  $0<\e\le \sqrt{\varrho}$, the set  $\mathfrak A$ has asymptotically full relative measure in $[\e^2,2\e^2]^\nu$.% as $\e \to 0$.
%Moreover the set $\mathfrak{A}$ has density $1$ as $\e\to0$.
%at $\x=0$ 
%and
%these quasi-periodic solutions are \textbf{linearly stable}.
\end{Teo}
\noindent
Moreover we have the following stability result.
\begin{Teo}{\bf (Linear stability)}\label{MainResult2}
The quasi-periodic solutions \eqref{SoluzioneEsplicitaDP} $u(t, x)=U(\omega t, x)$ of equation \eqref{DP}
are \emph{linearly stable} and \emph{reducible} in the following sense. Consider equation \eqref{DP} linearized at the embedded torus $U(\f,x)$, then the corresponding operator has purely imaginary spectrum and there exists a change of variables  $H^s(\T, \mathbb{R}) \to H^s(\T, \mathbb{R})$, quasi periodic in time with frequency $\omega$, which diagonalizes it in the  directions normal to the torus.
As a consequence the  Cauchy problem of the linearized equation is stable, i.e. the Sobolev norms are uniformly bounded in $t$. 
\end{Teo}

%Theorem \ref{MainResult} is the first KAM result for a resonant, quasi-linear PDE with \textbf{asymptotically linear} dispersion.\\
%For the proof we use the overall strategy of \cite{KdVAut}, but we need to exploit a set of new ideas.
%
%We exploit the fact that we are close to an elliptic fixed point and using the overall strategy of \cite{KdVAut};
Theorems \ref{MainResult}, \ref{MainResult2} are formulated in the typical style of results on reducible KAM tori for PDEs. 
For the proof we use the overall strategy of \cite{KdVAut}, which 
however has to be substantially  developed to deal with \eqref{DP}. 
Let us briefly explain the main new issues.
 \begin{itemize}
\item The dispersion law is \emph{asymptotically linear}
as for the Klein-Gordon equation, studied for instance in \cite{BBiP1}, \cite{BBiP2}.
As explained in those papers, the fact that the dispersive effects are very weak
(essentially time and space play the same role)
creates a number of difficulties even in the study of KAM theory for semi-linear PDEs.
Of course, since \eqref{DP} is quasi-linear, there are additional serious difficulties 
coming from the strong perturbative effects of the nonlinearity.

\item 
The DP equation is \emph{resonant} at zero and does 
not depend on any \emph{external} parameters. 
This is a fundamental difference w.r.t. the Klein-Gordon equation, where 
one modulates the mass in order to avoid resonances. 
Moreover the DP has non-trivial resonances already at order four (see section \ref{schema}), differently 
from the previous KAM  results for quasi-linear PDEs.
As a further difficulty the algebraic structure of the resonances is quite complicated. %much more complicated  w.r.t. the case of the KdV. 
In order to avoid the inherent problems   we rely on the presence of 
``many'' (precisely $eight$) approximate constants of motion of \eqref{DP}
coming from the integrable structure of the DP equation. Dealing with the problems related to resonances is the core of this paper and
 requires a set of new ideas and a careful analysis.
%(which give information on the Birkhoff normal form). 
\item The very strong restriction of the 
tangential sites $S^{+}$ is exploited several times to simplify the problems arising from the rational and asymptotically linear dispersion law. Physically we are looking for solutions mainly supported in Fourier space on modes which are \emph{relatively close} to each other.\\
It seems reasonable that such condition could be weakened, but it is not clear to us how to deal with the technical difficulties which would arise.
\item As in other resonant cases, 
the diophantine constant $\gamma$ is related to the size of the 
solution one is looking for (see \eqref{SoluzioneEsplicitaDP}).
Moreover, due to the linear dispersion law, we are forced to impose 
very ``weak'' non-degeneracy conditions on the linear frequencies of oscillations.
As a consequence we need a refined bifurcation analysis in order to 
 find a very good first  approximate solution and fulfil the smallness conditions required
 for the Nash-Moser scheme.
\end{itemize}

%
%The main  difference in our statement  is the very strong restriction on $S^+$ which is used in order to simplify the many problems arising from the rational and asymptotically linear dispersion law.  
%\\
%Some comments on the results are in order. 
%\noindent
%The proof of Theorem \ref{MainResult}
% is based on two key ideas:
% \begin{itemize}
%\item[(i)] We work in the context of  KAM for quasi-linear  PDEs, exploiting the fact that we are close to an elliptic fixed point and using the overall strategy of \cite{KdVAut};
%
%\item[(ii)] The DP equation is resonant at zero, in order to avoid the inherent problems   we rely on the presence of 
%``many'' (precisely $eight$) approximate constants of motion of \eqref{DP}
%coming from the integrable structure of the DP equation. Dealing with the problems related to resonances is the core of this paper; we wish to stress that  the structure of the resonances in the DP equation is much more complicated  w.r.t. the case of the KdV. 
%Dealing with them requires a set of new ideas and a very careful analysis.
%%(which give information on the Birkhoff normal form). 
%\end{itemize}

\noindent
Some comments on equation \eqref{DP} and on  Theorems \ref{MainResult}, \ref{MainResult2}
are in order.
\subsubsection*{The unperturbed DP equation}
We look at \eqref{DP} as a perturbation of 
the linear equation \eqref{lineare},
in order to fit the typical perturbative setting of KAM for PDEs
%which have been developed also to study non integrable PDEs
, we refer to subsection \ref{liter} for more details.

Actually,  since the Degasperis-Procesi equation is completely integrable (see \cite{Deg}) it would be very natural to try to construct solutions of \eqref{DP} 
which bifurcate from quasi-periodic solutions 
of the unperturbed DP equation
\begin{equation}\label{DPVera}
u_t-u_{x x t}+u_{xxx}-4 u_x-u u_{xxx}-3 u_x u_{xx}+4 u u_x=0,
\end{equation}
which corresponds to \eqref{DP} with $f=0$. Indeed, near zero, the \eqref{DP} can be seen \emph{also}
as a perturbation of \eqref{DPVera}.
Unfortunately even though 
algebro-geometric {\it finite-gap} solutions have been already
constructed in literature for the DP equation (see  \cite{sC})
it is not clear to us whether they are real quasi-periodic 
solutions in the sense of \eqref{6.6}. 
%Even in that case, 
%this would be probably very difficult. 
%This partially justify  
%why we constructed the quasi-periodic solutions in 
%Theorem \ref{MainResult}
%as perturbation of linear solutions of \eqref{lineare} .
%%and we ended up
%%with \emph{small} amplitude solutions of the form \eqref{SoluzioneEsplicitaDP}.
%
Of course
 if one were able to bifurcate from {\it finite-gap} solutions of \eqref{DPVera} 
 %this scheme 
%through 
then it would be possible  to prove existence of 
{\it large} quasi-periodic solutions, by requiring that $f$ is small. 
Such a strategy has been followed successfully for the KdV and cubic NLS equation on the circle. 
Actually for those equations it has been proved  the existence of global Birkhoff coordinates \cite{KdVeKAM}, 
\cite{GreKap} (the cartesian version of action-angle variables), which trivialize the dynamics (in the sense that the solutions turn out to be all periodic, quasi-periodic or almost periodic) and provide a fundamental tool for investigating the dynamical consequences of small perturbative effects, also far from the origin.\\ For $1$d integrable PDEs one would expect it to be the typical scenario, but, up to now, such results are not available for the DP equation. Then 
Theorem \ref{MainResult}  provides, as far as we know, 
the first existence result of 
quasi-periodic solutions, in the sense of \eqref{6.6},
for \eqref{DPVera}, at least near the origin.
It would be interesting to apply our KAM approach to the Camassa-Holm equation, which is a well-known integrable PDE with an asymptotically linear  dispersion law, but with a different symplectic structure.
Even though we have not performed the computations, we expect to be able to prove the equivalent of Theorems \ref{MainResult}, \ref{MainResult2} also for this equation. We remark that in this case, the finite gap solutions are known to be quasi-periodic tori, see \cite{Const1}.\\
 One could start by comparing them with the solutions predicted by our method and then possibly develop KAM theory close to large finite gap solutions.

%This lack of information regarding the dynamics of \eqref{DPVera} 
 
%
%Secondly, one can note that equation \eqref{DP} can be seen, close to $u=0$,
%as a perturbation of 
%the completely integrable DP equation.
%
%This also (partially) justify our strategy of looking for quasi-periodic solutions of \eqref{DP}
%close to linear solutions of \eqref{lineare}. 
%
%\noindent
%Indeed it would be very natural to try to construct solutions 
%of \eqref{DP} which bifurcate from quasi-periodic solutions \eqref{DPVera}.
%%of the Degasperis-Procesi equation. 
%Unfortunately 

\subsubsection*{Approximate constants of motion of \eqref{DP}}
Even though 
we do not fully exploit the integrability of \eqref{DPVera}
%, 
%by the smallness of  $f$ 
%(see  \eqref{HamiltonianDensity}), 
it is fundamental for us that (the non integrable) \eqref{DP}  has at least {\it eight} 
approximate constants of motion (up to an error of order $O(u^9)$). 
%in our work we strongly rely on the fact that such equation has at least {\it eight} 
%approximate constants of motion (up to an error of order $O(u^9)$). 
%Of course the DP equation has infinitely many such constants of motion; h
It is interesting to notice that, as shown in \cite{DegPro}, 
no other  equation with 
the same dispersion law, and the same symplectic structure, 
 has $eight$ 
approximate conserved quantities. This means that in \eqref{DP} we cannot  
consider any quadratic nonlinearity, but we really need the DP structure.

The request of the presence of such approximate conserved 
quantities it is not only a
technical matter. In order to implement a Nash/Moser-KAM algorithm one looks for
 a family of approximately invariant tori
of \eqref{DP} (with a sufficiently good approximation) such that the dynamics on the tori is integrable and non-degenerate, while the dynamics normal to the torus is non-degenerate at the linear level and satisfies the Melnikov conditions. 
If there are external parameters modulating the linear frequencies, 
then we can consider as approximate solutions the linear ones. Otherwise the modulation must come from the initial data and, hopefully, this
 can be achieved 
by means of Birkhoff normal form (BNF), see for instance \cite{KdVAut},\cite{Giuliani}. 
In this case, where the  the dispersion law in \eqref{dispersionLaw}
is a rational number and  is asymptotically linear, 
such procedure  is very difficult.
One has to explicitly compute some potentially 
dangerous resonant terms in the Hamiltonian
and show  that they vanish. 
This is the same type of computations
which  have been done by  Craig-Worfolk, Craig-Sulem in  \cite{CraigBF}, %\cite{CraigSulem},
where the authors verify the vanishing of the coefficients at
some fourth order resonant interactions, 
the so called \emph{Benjamin-Feir resonances}. 
In our case 
we have to deal with higher order resonances (up to $eight$),
so this would be computationally heavy to do.
Our approach is to use the approximate constant of motions. This
will be explained more in details in subsection \ref{schema}.
Once we have constructed the approximate invariant tori we have 
to impose the non-degeneracy and Melnikov conditions.
Differently form the KdV case, this will not be possible for any choice of the tangential set, and it is where 
we will use the condition $S^+\in \mathcal{V}(\mathtt{r})$, 
see Definition \ref{Def:cono}.

\subsubsection*{Linear stability}
The linear stability result of Theorem \ref{MainResult2} is of course 
a relevant dynamical information in the study of evolutionary PDEs, but it is also
the consequence of a fundamental ingredient of our proof: the 
 reducibility  of the linearized 
equation at any quasi-periodic approximate solution.  
Reducibility for the 
 Degasperis-Procesi equation linearized at a quasi-periodic function has been obtained in \cite{FGP1}, under some appropriate diophantine  conditions on the frequencies.
Unfortunately, due to the resonances, our case does not fit such hypotheses,
and a major point will be to overcome this difficulty.
Here we shall use such result (appropriately adapted) 
inside a nonlinear algorithm to prove the existence 
of quasi-periodic solutions. This is a classical feature of the literature of KAM theory. 
%The existence, without stability, 
%of quasi-periodic solutions can be 
%proved by also 
%using the Lyapunov-Schmidt decomposition combined with Nash-Moser implicit function theorems. {\red non sono convinta che stia bene qui}% In this paper we follow the approach developed in \cite{KdVAut}, \cite{BM1} which combines ideas of both formulations. This will be discussed in detail in subsection \ref{schema}. 

\subsection{Some literature}\label{liter}
Proving existence and stability  for quasi-periodic solutions for PDEs close to an elliptic fixed point is  a natural  extension of  the classical KAM theory for lower dimensional tori \cite{Po2}.
The first results in this direction were for model PDEs on an interval with  no derivatives in the nonlinearity 
and with either Dirichlet, \cite{ImaginarySpectrum,Wayne,KP,Po2} 
 or periodic, \cite{CW,JeanB,CY}, %\cite{Bou},
 boundary conditions.
For extension of KAM theory to higher spatial dimension we mention \cite{B3,EK,PP,NLSberBol,BCP,CM}.
While KAM methods for constructing quasi-periodic solutions for PDEs on the circle with no derivatives in the nonlinearity are by now well established, generalizing to cases with derivatives is in general not at all trivial, even in the semi-linear cases (where the derivatives in the nonlinearity are of lower order w.r.t. the linear terms). We mention  \cite{Ku2} for the KdV, \cite{LY} for the derivative NLS, 
and   \cite{BBiP1,BBiP2} for the derivative NLW. Recently an innovative strategy was proposed, \cite{Airy,KdVAut}  
to deal with  quasi-linear and fully nonlinear PDEs on the circle. This approach was first developed for the KdV equation but can be applied to many equations of interest in hydrodynamics, such as 
 NLS, \cite{FP1,FP2} 
 Kirchoff \cite{Mon1} 
 %\cite{Mon,Mon1} 
 or directly  the water wave equation
\cite{BM1,BBHM}. While these methods were first thought for PDEs on the circle, of course a very interesting point is  the generalization to higher dimensions.

Equation \eqref{DP} is a quasi-linear PDE on the circle and in our study we shall follow the general strategy of \cite{KdVAut},  extended and adapted to our case.
Let us briefly explain the point of view of \cite{KdVAut}, referring 
also to \cite{BBHM} for more details.

\subsection{The general strategy.} We describe the strategy to prove existence and linear stability for  small, reducible quasi-periodic solutions of completely resonant quasi-linear PDEs\label{fimo}.

\medskip

\noindent$(i)$ The starting point is a Nash-Moser theorem of hypothetical 
conjugation following \cite{BertiBolle}.
The strategy is to construct quadratically convergent sequence of families
of 
approximately  invariant (isotropic) tori.
Such construction is based on \emph{tame} estimates on the inverse 
of the operator associated to the equation
\eqref{DP} linearized at an approximate torus and restricted to the normal direction.
This is proved by exploiting the Hamiltonian structure 
and exhibiting symplectic variables adapted to each approximate invariant torus,
which essentially decouple the linearized dynamics.
Then the bounds on the inverse are achieved 
by removing all the ``bad'' values of the parameters.
We mention also
\cite{KAMtotale} for a parallel strategy 
which does not rely on the Hamiltonian structure. 

\vspace{0.5em}
\noindent
$(ii)$ To construct the sequence of item $(i)$ we need a good starting point, i.e.
 a first  family of approximately invariant tori parametrized by real vectors
$\x\in \mathbb{R}^{\nu}$.

\noindent
As explained before this is achieved by BNF techniques.
In particular, in the quasi-linear context, it is convenient to 
perform a {\it Weak} BNF, i.e. to exhibit
a change of variables, close to the identity up to a finite rank operator,
such that the following holds.
The Hamiltonian $H$ transforms to 
$H_{\rm Birk}+ R$ where $R$ is a small remainder, and 
\begin{itemize}
\item[1.]  the finite dimensional subspace
$U_S:=\{u_j=0,\;\forall j\notin S\}$ is  invariant for $H_{\rm Birk}$; 
\item[2.] the Hamiltonian restricted to $U_S$ is integrable and non-degenerate
 in the sense that
the ``frequency-to-amplitude'' map is invertible.\
\end{itemize}
In order to describe in a simpler way the dynamics in a neighborhood of $U_S$ it is convenient to define \emph{action-angle} variables. This allows to distinguish the tangential and normal dynamics to the approximately invariant tori.\\
%In this step  the  dynamics normal to the torus  is not addressed.
We remark that, for semi-linear PDEs, 
typically one performs a stronger BNF 
preliminary step, in order to ``normalize'' also the linearized dynamics normal to the torus, i.e. the terms in the Hamiltonian which are
quadratic in the normal directions.
In this case the Birkhoff map is close to the identity up to a bounded operator (at most 
one-smoothing), see for instance \cite{Po2}, \cite{KP}.
Compared to the latter approach, the weak procedure has the disadvantage that the normal form depends on the angles; on the other hand we do not have to address well-posedness issues in order to construct the Birkhoff maps, since these changes of coordinates are time-one flow maps of Banach space ODEs. 

\vspace{0.5em}
\noindent
$(iii)$ The third key point is to study the invertibility 
of the linearized operator restricted to the normal directions. 
Thanks to the very ``mild'' conjugation procedure of item $(ii)$
(with a map = identity+finite rank) it turns out that such linear operator
is
\emph{pseudo differential} (with non constant coefficients) up to a  finite rank remainder. This is the most important reason for adopting the weak procedure described in $(ii)$.   \\
The invertibility of the linearized operator, with appropriate tame estimates, is  based on a 
\emph{reducibility} argument
which is  divided into two parts:
\begin{itemize}
\item[(a)]\label{stepREGU} a reduction in decreasing order procedure 
which conjugates the linearized operator to 
a pseudo differential  one
with constant coefficients
up to a remainder which is a  bounded/regularizing term 
i.e. maps $H^{s}(\mathbb{T}, \mathbb{R}) $ to $H^{s+\rho}(\mathbb{T}, \mathbb{R})$, 
$\rho\geq0$. The choice of $\rho$ depends of course on the problem one is studying.

\item[(b)] a quadratic KAM scheme (for \emph{bounded} operators) 
 which completely diagonalizes
 the bounded/smoothing remainder of the previous step.

 \end{itemize}
 We want to point out the following:
 \begin{itemize}
 \item the step (a) strongly relies on the pseudo differential structure
of the operator; 
\item the normal form contains angle-dependent terms and some of them turn out to be not perturbative for the KAM scheme $(b)$. The conjugation at constant coefficients of such terms relies on purely algebraic arguments. We refer to this procedure as \emph{linear Birkhoff normal form};
\item as a consequence of having applied the weak and the linear Birkhoff procedure, the normal form around the approximately invariant tori has constant coefficients also in the normal directions.
\end{itemize}
In order to perform the diagonalization procedure of step (b) one needs the {\it second Melnikov conditions}, which essentially amount to requiring that the operator  has simple eigenvalues with a lower bound on the differences.
Once one has diagonalized the operator, the bounds on the inverse follow trivially from lower bounds on the eigenvalues, i.e. {\it first Melnikov conditions}. 

\vspace{0.5em}
\noindent
$(iv)$ In the scheme above, at each step we have removed some bad values of the parameters $\xi$ where the Melnikov conditions do not hold. Hence the last (but not least) step is to prove that at the end of the procedure one has still a positive measure set of parameters.

\subsection{Main novelties and scheme of the proof.}\label{schema}
We describe the structure of the paper following subsection \ref{fimo}, and with particular attention to the main novelties.

\medskip

In \textbf{section} \ref{sezione functional setting} we introduce the Hamiltonian formalism for the DP equation and the functional spaces on which we shall work. 

\medskip

In \textbf{section} \ref{SezioneWBNF} we perform the weak Birkhoff normal form explained in item (ii) of the previous section. The result  is stated in  Proposition \ref{WBNFdp}. In order to reach a sufficiently good first approximate solution we need to perform $6$-BNF steps. As is well-known, at the $n$-th step of this procedure one has to take into account the denominators (recall \eqref{dispersionLaw})
\begin{equation}\label{riso}
\lambda(j_1)+\dots+\lambda(j_{n+2}).
\end{equation}
We say that a $(n+2)$-uple of integer indices $(j_1, \dots, j_{n+2})$ is a resonance, and hence may appear in $H_{Birk}$, if \eqref{riso}$=0$ and the momentum condition holds, namely $\sum_{i=1}^{n+2} j_i=0$. We say that a resonance is \emph{trivial} if it has the form $(i, -i, j, -j, \dots)$ so that the corresponding monomial is integrable.\\
As mentioned before a major difficulty comes from the fact that the DP equation has many non-trivial resonances (already at order four) and in principle there is no reason why the Birkhoff Hamiltonian restricted to $U_S$ should be integrable.
By the fact that the Hamiltonian density $f$ is of order $O(u^9)$ the perturbation does not affect the leading terms of the Birkhoff Hamiltonian and we can exploit the integrability of the DP equation. Indeed the same Birkhoff transformation should normalize simultaneously all the commuting Hamiltonians. This means that a resonant monomial contributes to $H_{Birk}$ if and only if it is resonant for all the constants of motion. 
 This was proved in detail in \cite{FGPa} at the level of formal power series. Here we adapt this result to the equation \eqref{DP} which is only approximately integrable (close to the origin) and we reformulate that in a way better suited to the weak Birkhoff normal form context, see Proposition \ref{LemmaBelloWeak}.

%
%
% To prove that the normal form is action-preserving (which is a fundamental ingredient in the KAM scheme of \cite{KdVAut}) we use the fact that equation \eqref{DP} is a perturbation of order $O(u^9)$ of the integrable equation \eqref{DPVera} so  the Birkhoff Hamiltonian does not see the perturbation $f$, and hence we can exploit the integrability of the DP equation. 
%More precisely we show that the coefficients of the Hamiltonian corresponding to the non-trivial resonances vanish, see Proposition \ref{vvv}.\\

\smallskip

\noindent Once we have shown that the $H_{Birk}$-dynamics restricted to $U_S$ is integrable, in \textbf{section} \ref{sezioneActionAngle}, we prove that it is non-degenerate, i.e. that the frequency to amplitude map is a diffeomorphism. We have a very explicit description of this map and hence this step amounts to proving that the matrix $\mathbb{A}$ in \eqref{TwistMatrixDP} (which depends only on $S^+$) has  determinant bounded away from zero (the so-called twist condition), see Lemma \ref{Twist1}.  A big difference  with \cite{KdVAut} is that, in our case, the determinant of $\mathbb{A}$ is a rational function of several variables $\overline\jmath_i$ that could accumulate to zero as $\lvert  \overline\jmath_i \rvert\to \infty$. By imposing the wave packet condition we restrict the study of its asymptotic behaviour to regions in which it behaves like a one variable function. Then we use continuity arguments to guarantee the invertibility of $\mathbb{A}$ for every choice of $S^+\in \mathcal{V}({\mathtt{r}})$ (see Definition \ref{Def:cono}) for $\mathtt{r}$ small enough. Outside $\mathcal{V}({\mathtt{r}})$ the proof of lower bounds for $\det\mathbb{A}$ should rely on purely algebraic arguments and not on perturbative ones.

 \smallskip
 
In \textbf{section} \ref{sezioneNonlinearFunct} we introduce  the Nash Moser hypothetical conjugation theorem (see Theorem \ref{IlTeoremaDP}) and in \textbf{section} \ref{sezione6DP} we explain how to prove the invertibility of the linearized operator at an approximate solution by only studying it in the normal direction. Since there is no difference with \cite{KdVAut} we only give a synopsis.

\smallskip

In \textbf{sections} \ref{regularization} and \ref{SezioneDiagonalization} we prove the Theorems \ref{risultatosez8} and \ref{ReducibilityDP} which provide the reducibility of the linearized operator following item (iii) of subsection \ref{fimo}. 
As we already mentioned, in \cite{FGP1} we provide a reducibility result for the DP equation \eqref{DP} linearized at sufficiently small quasi-periodic functions under appropriate diophantine conditions on the frequencies . Unfortunately in our case the diophantine constant $\gamma$ is related to the size of the approximate solutions (see \eqref{gammaDP}) and then the smallness and diophantine conditions above cannot be met.
%
%These results are related to the Propositions ... in \cite{FGP}, but unfortunately, in this case, we cannot apply them directly because of the presence of non perturbative terms. \\

In \cite{KdVAut} this issue appears only in the step $(b)$ of the strategy, where it is solved by the linear Birkhoff normal form method. A first difficulty in our case is that this problem appears also in step $(a)$. So that we first need to perform some preliminary steps (see section \ref{preliminare}), more precisely we need changes of coordinates, preserving the pseudo differential structure, that conjugate the leading order of the linearized operator to a diagonal one plus a correction, which is perturbative in the sense of \cite{FGP1}. In such steps the provided changes of coordinates are similar in structure to those of step $(a)$ but they are proved to be well-defined not by using perturbative arguments, but by algebraic computations involving the Birkhoff resonances (see Lemma \ref{ecologia}). 
 These difficulties appear also for the quasi-linear generalized KdV \cite{Giuliani}, but here we have several further problems  due to the complexity of the symplectic structure of the DP equation.
 The first step, removing terms of order $\e$, is straightforward. Already at the second step we encounter the difficulties arising form the presence of non-trivial resonances of order $4$, and a priori there is no reason why the normal form should be integrable. Here it does not appear simple to apply the strategy of the weak BNF, using the constants of motion.  On the other hand, computing the normal form explicitly by hand, as done in \cite{Giuliani}, is unmanageable. To bypass this problem we take a different point of view, based on an \emph{a posteriori identification argument} of normal forms. 
 More precisely in Theorem \ref{PartialWeak} we prove that the normal form obtained after the weak BNF, the preliminary steps and the linear BNF coincides with the one that we would obtain by performing the full {\it formal} BNF and then  projecting on the quadratic terms in the normal variables. 
 This result strongly relies on the fact that all the resonances contributing to the normal form are trivial.
 A similar identification argument has been used, for instance, in  
 \cite{BFP}.
 
  A further point is that, due to the rational dispersion law $\lambda(j)$, it is possible that a {\it denominator} in the linear BNF is not zero but is still  uncontrollably small. In the third step, in order to deal with this problem we need to take into account in the unperturbed Hamiltonian also the integrable terms of order $\e^2$ coming from the previous steps of linear BNF.
  For this reason it is important to know the exact expression of the main order of the correction at the eigenvalues given by the perturbation, see for instance \eqref{divisoriLBNF3}.
This is also needed in the KAM scheme $(b)$, in order to impose the second Melnikov conditions.
Computing these corrections by hand would be a very difficult task, 
but this comes for free from Theorem \ref{PartialWeak}.

\smallskip
In the first part of \textbf{Section} \ref{sezione9DP} we show the convergence of the Nash-Moser algorithm (see Theorem \ref{NashMoserDP}), which requires the ratio between the size of $R=H-H_{Birk}$ and $\gamma^{7/2}$ to be small (see the smallness condition \eqref{SmallnessConditionNMdp}); in the second part we prove that the set of "bad" parameters, i.e. the amplitudes of the approximate solutions whose frequency do not meet the first and second Melnikov conditions, has small measure (see \eqref{UnionDP}). As it is well known the most difficult conditions to impose are the second Melnikov conditions, so we decide to focus on them.\\
 The key difficulty is that
the spectral gap $\lambda(j)-\lambda(k)$ is asymptotically constant, hence there is a bad separation property of the eigenvalues. The same occurs for the wave equation \cite{BBiP1}, \cite{BBiP2}. \\
In Lemma \ref{singolo} we provide the measure of the single bad set. Here we use the algebraic arguments provided by Lemma \ref{ecologia}, which guarantees the non-degeneracy of the leading terms of the small divisors. In section \ref{siSomma} we deal with the summability of the bad sets in $j, k$. Due to the asymptotically constant spectral gap, these sets are infinitely many.
Then the key ingredient is to show that for $j, k$ sufficiently large the second Melnikov conditions are implied by the first ones. This is possible provided that we consider two different diophantine constants. More precisely we have to impose second order Melnikov conditions with $\gamma^{3/2}$ (see \eqref{GnDP}), which is clearly much smaller than $\gamma$. This is why we have to perform many steps of Birkhoff normal form in order to obtain a very good first approximate solution.\\
We point out that, differently from \cite{BBHM}, our Melnikov conditions do not imply a loss of regularity in space. In \cite{BBHM} this loss is acceptable, since in the regularization step ($(a)$ page \pageref{stepREGU}) the diagonalization is perfomed up to a very smoothing remainder. In this procedure it is fundamental that the diophantine constant $\gamma$ is independent of the size of the solution.
Of course in our case this is not true and thus in the regularization step
we end up with a remainder of order $-1$, and then 
in the measure estimates
we put some extra efforts to prove second Melnikov conditions without loss of regularity.
%\comment{bisogna dire che questo è il motivo per cui bisogna fare tanta biforcazione}
%dispersione lineare allora riducibilità presenta small div.
%identificazione forma normale, correzioni agli autoval vanno conosciute per stime di misura

%For instance the assumptions of Proposition $3.6$, which provides the change of coordinates that conjugates the linearized operator to constant coefficients at the leading order, are not immediately satisfied (see \eqref{BoraMaledetta}, ). The same holds true for the KAM scheme.
%
%
% So before performing $(a)$ we need to perform some preliminary steps (see ...) where the provided changes of coordinates are proved to be well-defined not by perturbative arguments, but by algebraic computations which again involve the Birkhoff resonances (see Lemma \ref{ecologia}). 

\subsection*{Acknowledgements} The authors wish to thank Andy Hone, Luca Biasco, Livia Corsi and Marcel Guardia for useful discussions and comments.

\section{Functional Setting}\label{sezione functional setting}

{\bf Hamiltonian formalism of the Degasperis-Procesi equation.}
For any $u, v$ in the space
\[
H_0^1(\mathbb{T}):=\left\{ u\in L^2(\mathbb{T}, \mathbb{R}) : 
\int_{\mathbb{T}} u\,dx=0\right\}
\]
we define
 the non-degenerate symplectic form
\begin{equation}\label{simpleDP}
\Omega(u, v):=\int_{\mathbb{T}} (J^{-1} u)\,v\,dx=(J^{-1}u,v)_{L^{2}} 
%\quad \forall u, v\in L_0^2(\mathbb{T}_x)
\end{equation}
where  $J$ is defined in \eqref{DPHamiltonian} and $(\cdot,\cdot)_{L^{2}}$
is the $L^{2}(\mathbb{T}, \mathbb{R})$ scalar product.
To any $C^{1}$ function  $H : H_0^{1}(\mathbb{T})\to \mathbb{R}$
we associate a vector field $X_{H}$ by requiring 
\[
dH(u)[h]=(\nabla H(u),h)_{L^{2}}=\Omega(X_{H}(u),h), \quad 
\forall\, u,h\in H_0^{1}(\mathbb{T})\,.
\]
The Hamiltonian vector field $X_{H}$ is uniquely determined since 
the symplectic form $\Omega$ in \eqref{simpleDP} is non-degenerate, in particular $X_{H}(u)=J\nabla H(u)$.
The \emph{Poisson bracket} between two $C^{1}$ functions
 $F,G : H_0^{1}(\mathbb{T})\to \mathbb{R}$ is
 \begin{equation}\label{PoissonBracketDP}
\{F,G\}:=\Omega(X_{F},X_{G})=(\nabla F, J\nabla G)_{L^2}\,.
\end{equation}
In this way 
\begin{equation}\label{LieBracketDP}
X_{\{F,G\}}=[X_{F},X_{G}]\,, \quad {\rm where}\quad [X,Y]:=dX [Y]-dY [X]\,.
\end{equation}
Finally, given a Hamiltonian $H$ we define its \emph{adjoint action} 
as the operator
\begin{equation}\label{adjActHAM}
\mathrm{ad}_{H}[\cdot]:=\{H, \cdot\}\,.
\end{equation}
Consider now two Hamiltonians $H,G$ and let $\Phi_{G}$ be the time-$1$ flow map
of the vector field $X_{G}$. Then we have (formally)
\begin{equation}\label{formLIE}
H\circ\Phi_{G}=\sum_{k\geq0}\frac{(-1)^{k}}{k!}\mathrm{ad}_{G}^{k}[H]\,,
\qquad H\circ\Phi^{-1}_{G}=\sum_{k\geq0}\frac{1}{k!}\mathrm{ad}_{G}^{k}[H]\,,\qquad
\mathrm{ad}_{G}^{k}[H]:=\mathrm{ad}_{G}\big[\mathrm{ad}_{G}^{k-1}[H]\big]\,,
\end{equation}
where $\mathrm{ad}_{G}^{0}:=\mathrm{I}$ is the identity map.

\noindent
{\bf Functional space.} We consider functions $u(\varphi, x)$ defined on $\T^{\nu}\times \T$.
Passing to the Fourier representation
\begin{equation}\label{realfunctions}
u(\varphi, x)=\sum_{j\in\mathbb{Z}} u_{j}(\varphi)\,e^{\mathrm{i} j x}=\sum_{\ell\in\mathbb{Z}^{\nu}, j \in \mathbb{Z}} u_{\ell j} \,e^{\mathrm{i}(\ell \cdot \varphi+j x)}\,,\quad 
\overline{u}_j(\varphi)=u_{-j}(\varphi)\,, \quad \overline{u}_{\ell j}=u_{-\ell, -j}\,.
\end{equation}
We define the scale of Sobolev spaces %\red{ma a che ci serviva $\mathbb C$?, metto $\R$.}
\begin{equation}\label{space} 
H^{s}:=\Big\{ u(\varphi, x)\in L^{2}(\T^{\nu+1}, \mathbb{R}) : \lVert u \rVert_s^2:=\sum_{\ell\in\mathbb{Z}^{\nu}, j\in\mathbb{Z}} \lvert u_{\ell j} \rvert^2 \langle\ell, j \rangle^{2 s}<\infty \Big\}
\end{equation}
where $\langle \ell, j \rangle:=\max\{ 1, \lvert \ell \rvert, \lvert j \rvert\}$, $\lvert \ell \rvert:=\sum_{i=1}^{\nu} \lvert \ell_i \rvert$. Note that the spatial average of  a function $u\in H^s$ is zero. 
We denote by $\mathfrak{B}_{r}(0, X)$ the ball of radius $r$ centered at the origin of a Banach space $X$.

\smallskip

\noindent
{\bf Lipschitz norm.} Fix $\nu\in\mathbb{N}^{*}:=\mathbb{N}\setminus\{0\}$ 
and let $\calO$ be a compact subset of $\mathbb{R}^{\nu}$. 
For a function $u\colon \calO\to E$, where $(E, \lVert \cdot \rVert_E)$ is a Banach space, we define the sup-norm and the lip-seminorm of $u$ as
\begin{equation}\label{suplip}
\begin{aligned}
&\lVert u \rVert_E^{sup}:=\lVert u \rVert_{E}^{sup, \calO}
:=\sup_{\omega\in\calO} \lVert u(\omega) \rVert_E,\qquad \lVert u \rVert_{E}^{lip}
:=\lVert u \rVert_{E}^{lip, \calO}
:=\sup_{\substack{\omega_1, \omega_2\in \calO,\\ \omega_1\neq \omega_2}} 
\frac{\lVert u(\omega_1)-u(\omega_2)\rVert_E}{\lvert \omega_1-\omega_2\rvert}\,.
\end{aligned}
\end{equation}
If $E$ is finite dimensional,  for any $\gamma>0$ we introduce the 
weighted  Lipschitz norm
\begin{equation}
\label{tazzone}
\lVert u \rVert_E^{\g, \calO}
:=\lVert u \rVert_E^{sup, \calO}+\gamma \lVert u \rVert_{E}^{lip, \calO}\,.
\end{equation}
If $E$ is a scale of Banach spaces, say $E=  H^s$, for $\gamma>0$ we introduce the 
weighted  Lipschitz norm
\begin{equation}\label{tazza10}
\lVert u \rVert_s^{\g, \calO}:=\lVert u \rVert_s^{sup, \calO}
+\gamma \lVert u \rVert_{s-1}^{lip, \calO}\,, \quad \forall s\geq [\nu/2]+4
\end{equation}
where we denoted by $[ r ]$ the integer part of $r\in\R$.

\medskip

\noindent
{\bf Linear operators. } Let $A\colon \T^{\nu}\to \mathcal{L}(L^2(\T,\mathbb{R}))$, $\varphi\mapsto A(\varphi)$, be a $\varphi$-dependent family 
of linear operators acting on $L^2(\T,\mathbb{R})$. 
We consider $A$ as an operator acting on $H^{s}(\T^{\nu+1},\mathbb{R})$ by setting
\[
(A u)(\varphi, x)=(A(\varphi)u(\varphi, \cdot))(x)\,.
\]  
This action is represented in Fourier coordinates as
\begin{equation}\label{cervino}
A u(\varphi, x)=\sum_{j, j'\in\mathbb{Z}} A_j^{j'}(\varphi) \,u_{j'}(\varphi)\,e^{\mathrm{i} j x}=\sum_{\ell\in\mathbb{Z}^{\nu}, j\in\mathbb{Z}} \sum_{\ell'\in\mathbb{Z}^{\nu}, j'\in\mathbb{Z}} A_{j, \ell}^{j', \ell'}\,u_{j' \ell'}\,e^{\mathrm{i}(\ell\cdot \varphi+j x)}\,.
\end{equation}
Conversely, given a T\"opliz in time operator $A$, namely such that its matrix coefficients (with respect to the Fourier basis) satisfy
\begin{equation}\label{topliz}
A_{j, \ell}^{j', \ell'}=A_j^{j'}(\ell-\ell')\qquad \forall j, j'\in\mathbb{Z},\,\,\ell, \ell'\in\mathbb{Z}^{\nu}\,,
\end{equation}
we can associate it a time dependent family of operators acting on $H^s(\T)$ by setting
\[
A(\varphi) h=\sum_{j, j'\in\mathbb{Z}, \ell\in\mathbb{Z}^{\nu}} 
A_j^{j'}(\ell) h_{j'}\,e^{\mathrm{i} j x} e^{\mathrm{i}\ell\cdot \varphi}\,, 
\qquad \forall h\in H^s(\T,\mathbb{R})\,.
\]

\noindent
For $m=1, \dots, \nu$ we define the operators $\partial_{\varphi_m} A(\varphi)$ as
\begin{equation}\label{cervino2}
\begin{aligned}
&(\partial_{\varphi_m} A(\varphi)) u(\varphi, x)=
\sum_{\ell\in\mathbb{Z}^{\nu}, j\in\mathbb{Z}} \,
\sum_{\ell'\in\mathbb{Z}^{\nu}, j'\in\mathbb{Z}}\mathrm{i}(\ell_m-\ell_m')\, A_j^{j'}(\ell-\ell')\,u_{\ell' j'}\,e^{\mathrm{i}(\ell\cdot \varphi+j x)}\,.
\end{aligned}
\end{equation}
We say that $A$ is a \textit{real} operator if it maps real valued functions in real valued functions. For the matrix coefficients this means that
\[
\overline{A_j^{j'}(\ell)}=A_{-j}^{-j'}(-\ell)\,.
\]

\noindent 
{\bf Hamiltonian linear operators.} 
In the paper we shall deal with operators which are Hamiltonian according to the following Definition.
\begin{defi}\label{dynamicaldef}
 We  say that a linear map is symplectic if it preserves  the $2$-form $\Omega$ in \eqref{simpleDP};
similarly we say that a linear operator $M$  is Hamiltonian if $Mu $ is a linear Hamiltonian vector field w.r.t. $\Omega$ in \eqref{simpleDP}. This means that each $J^{-1} M$  is symmetric respect to the real $L^2$-scalar product.
Similarly, we call a family of maps $\f\to A(\f)$  symplectic if,  for each fixed $\f$, $A(\f)$ is symplectic, 
same for the Hamiltonians.
We shall say that an operator of the form $\oo\cdot\del_{\f}+M(\f)$ is Hamiltonian if $M(\f)$ is Hamiltonian.
\end{defi}

\paragraph{Notation.} We use the notation $A\lesssim B$ to denote $A\le C B$ where $C$ is a positive constant possibly depending on fixed parameters given by the problem. We use the notation $A\lesssim_y B$ to denote $A\le C(y) B$ if we wish to highlight the dependence on the variable $y$ of the constant $C(y)>0$.

\medskip

\noindent
{\bf Linear Tame operators.} Here we introduce rigorously the spaces and the classes of operators on which we work.

\begin{defi}[{\bf $\s$-Tame operators}]\label{TameConstants}
Given  $\s\geq 0$ we say that 
a linear operator $A$ is $\sigma$-\textit{tame} w.r.t. a non-decreasing sequence $\{\mathfrak M_A(\s,s)\}_{s=s_0}^\mathcal{S}$ (with possibly $\mathcal{S}=+\infty$) if
\begin{equation}\label{SigmaTame}
\lVert A u \rVert_{s}\le \mathfrak{M}_A(\s,s) 
\lVert u \rVert_{s_0+\sigma}+\mathfrak{M}_A(\s,s_0) \lVert u \rVert_{s+\sigma} 
\qquad u\in H^s\,,
\end{equation} 
for any $s_0\le s\le \mathcal{S}$. We call $\mathfrak{M}_A(\s,s)$  a {\sc tame constant} for the operator $A$. When the index $\s$ is not relevant 
we write $\gotM_{A}(\s,s)=\gotM_{A}(s)$. 
\end{defi}

\begin{defi}[{\bf Lip-$\s$-Tame operators}]\label{LipTameConstants}
Let $\s\geq 0$ and $A=A(\oo)$ be a linear operator defined for $\oo\in \calO\subset \mathbb{R}^{\nu}$.
Let us define
\begin{equation}\label{defDELTAomega}
\Delta_{\oo,\oo'}A:=\frac{A(\oo)-A(\oo')}{|\oo-\oo'|}\,, \quad \oo,\oo'\in \calO\,.
\end{equation}
Then $A$ is \emph{Lip-$\s$-tame} w.r.t. a 
non-decreasing sequence $\{\mathfrak M_A(\s,s)\}_{s=s_0}^\mathcal{S}$ if
 the following estimate holds
\begin{equation}\label{lipTAME}
\sup_{\oo\in \calO}\|Au\|_{s},\g\sup_{\oo\neq\oo'}\|(\Delta_{\oo,\oo'}A)\|_{s-1}
\leq _{s}\gotM^{\g}_{A}(\s,s)\|u\|_{s_0+\s}+
\gotM^{\g}_{A}(\s,s)\|u\|_{s+\s}, \quad u\in H^{s}\,.
\end{equation}
We call $\mathfrak{M}^{\gamma}_A(\s,s)$ a {\sc Lip-tame constant } of the operator $A$. 
When the index $\s$ is not relevant 
we write $\gotM^{\gamma}_{A}(\s,s)=\gotM^{\gamma}_{A}(s)$. 

\end{defi}

\noindent {\bf Modulo-tame operators and majorant norms.}
The modulo-tame operators are introduced in Section $2.2$ of \cite{BM1}. Note that we are interested only in the Lipschitz variation of the operators respect to the parameters of the problem, whereas in \cite{BM1} the authors need to control also higher order derivatives.
% The main difference with the work \cite{BM1} is that in the KAM reducibility procedure we involve modulo-tame operators which regularize in space. This fact holds also in  \cite{BM2}.

\begin{defi}
Let $u\in H^s$, $s\geq 0$, we define the majorant function
$
\underline{u} (\varphi, x):=\sum_{\ell\in\mathbb{Z}^{\nu}, j\in\mathbb{Z}} \lvert u_{\ell j} \rvert e^{\mathrm{i}(\ell\cdot \varphi+j x)}.
$
Note that $\lVert u \rVert_s=\lVert \underline{u} \rVert_s$.
\end{defi}
\begin{defi}[{\bf Majorant operator}]
Let $A\in\mathcal{L}(H^s)$ and recall its matrix representation \eqref{cervino}.
We define the majorant matrix $\underline{A}$ as the matrix with entries
\[
(\underline{A})_j^{j'}(\ell):=\lvert (A)_{j}^{j'}(\ell) \rvert 
\qquad j, j'\in\mathbb{Z},\,\,\ell\in\mathbb{Z}^{\nu}\,.
\]
\end{defi}

\noindent
We consider the majorant operatorial norms 
\begin{equation}\label{majorantnorm}
\|\underline{M}\|_{\mathcal L(H^s)}:= 
\sup_{\lVert u \rVert_s\le 1} \lVert\underline{M}u \rVert_{s}\,.
\end{equation}
We have a partial ordering relation in the set of the infinite dimensional matrices, i.e.
if
\begin{equation}\label{partialorder}
M \preceq N \Leftrightarrow |M_{j}^{j'}(\ell)|\leq |N_{j}^{j'}(\ell)|\;\;\forall j,j',\ell\; 
\Rightarrow \lVert \underline{M} \rVert_{\mathcal L(H^s)}\le 
\lVert \underline{N} \rVert_{\mathcal L(H^s)}\,, 
\quad \lVert  {M}u \rVert_s\le 
\lVert  \underline{M}\,\underline u \rVert_s \le \lVert\underline{N}\, 
\underline u \rVert_s\,.
\end{equation}
Since we are working on a majorant norm we have the continuity of the projections on monomial subspace, in particular we define the following functor acting on the matrices
\[
\Pi_K M:=
\begin{cases}
 M_{j}^{j'}(\ell) \qquad \qquad \text{if} \; |\ell|\le K\,, \\
0 \qquad \qquad \qquad \mbox{otherwise}
\end{cases}
\qquad \qquad  \Pi_K^\perp:= \mathrm{I}-\Pi_K\,.
\]
Finally we define for $\mathtt b_0\in \N$
\begin{equation}\label{funtore}
(\langle \pa_\f \rangle^{\mathtt b_0} M )_{j}^{j'}(\ell) :=  
\langle \ell  \rangle^{\mathtt b_0} M_j^{j'}(\ell)\,.
\end{equation}
In the sequel let $1>\g>\g^{3/2}>0$ be fixed constants.
\begin{defi}[{\bf Lip-$\sigma$-modulo tame}]\label{def:op-tame} 
Let $\s\geq 0$. A  linear operator $ A := A(\omega) $, $\omega\in \calO\subset\mathbb{R}^{\nu}$,  is  
Lip-$\s$-modulo-tame  w.r.t. a non-decreasing sequence $\{	{\mathfrak M}_{A}^{\sharp, \g^{3/2}} (\s, s) \}_{s=s_0}^{\mathcal{S}}$ if 
the majorant operators $  \underline{ A }, \underline{\Delta_{\omega,\omega'} A}$ are Lip-$\s$-tame w.r.t. these constants, i.e. they 
satisfy the following weighted tame estimates:  
for $\s\geq 0$,  for all %$\omega\neq \omega'\in \cO$,
 $ s \geq s_0 $ and  for any $u \in H^{s} $,  
\begin{equation}\label{CK0-tame}	
\sup_{\omega\in\calO}\| \underline{A} u\|_s\,,
 \sup_{\omega\neq \omega'\in \cO}{\g^{3/2}}  \|\underline{\Delta_{\omega,\omega'} A}  u\|_s 
 \leq  
{\mathfrak M}_{A}^{\sharp, \g^{3/2}} (\s,s_0) \| u \|_{s+\s} +
{\mathfrak M}_{A}^{\sharp, \g^{3/2}} (\s,s) \| u \|_{s_0+\s} \,.
\end{equation}
%where  the functions $ s \mapsto  {\mathfrak M}_{A}^{{\sharp, \g^{3/2}}} (\s, s)  \geq 0  $ 
%are non-decreasing in $ s $. 
The constant $ {\mathfrak M}_A^{{\sharp, \g^{3/2}}} (\s,s) $ 
is called the {\sc modulo-tame constant} of the operator $ A $. When the index $\s$ is not relevant 
we write $ {\mathfrak M}_{A}^{{\sharp, \g^{3/2}}} (\s, s) = {\mathfrak M}_{A}^{{\sharp, \g^{3/2}}} (s) $. 
\end{defi}
\begin{defi}\label{menounomodulotame}
We say that $A$ is Lip-$-1$-modulo tame if 	
$\langle D_x\rangle^{1/2}{A}  \langle D_x\rangle^{1/2}$ 
is Lip-$0$-modulo tame. We denote
\begin{equation}\label{anagrafe} 
\begin{aligned}
&\mathfrak M^{{\sharp, \g^{3/2}}}_{A}(-1,s):=  
\mathfrak M^{\sharp, \g^{3/2}}_{ \langle D_x \rangle^{1/2}A \langle D_x \rangle^{1/2}}(0,s), \quad \mathfrak M^{\sharp, \g^{3/2}}_{A}(-1,s,a):= 
\mathfrak M^{\sharp, \g^{3/2}}_{\langle \pa_\f \rangle^{a}  \langle D_x \rangle^{1/2}A \langle D_x \rangle^{1/2}}(0, s)\,, 
\quad a\ge 0\,.
\end{aligned}
\end{equation}
\end{defi}

In the following we shall systematically use $-1$ modulo-tame operators. 
We refer the reader to the Appendix of \cite{FGP1} for the properties of Tame and Modulo-tame operators.

\smallskip

\noindent {\bf Pseudo  differential operators.}
Following \cite{BM1} %and \cite{Taylor} 
we give the following definitions.
\begin{defi}\label{pseudoR}
Let $m\in \R$.
A linear operator $A$ is called pseudo differential 
of order $\le m$ if its action on any $H^s(\T, \mathbb{R})$ with 
$s\ge m$ is given by
\[
A\sum_{j\in\Z} u_j e^{\mathrm{i} j x} = 
\sum_{j\in\Z} a(x,j) u_j e^{\mathrm{i}j x} \,,
\]
where   $a(x, j)$, called the {\it symbol} of $A$,  is the restriction to 
$\mathbb{T}\times \mathbb{Z}$ of a complex valued function 
$a(x, y)$ which is $C^{\infty}$ smooth on $\mathbb{T}\times\mathbb{R}$, 
$2\pi$-periodic in $x$ and satisfies
\begin{equation}\label{space3}
|\del_{x}^{\al}\del_{y}^{\be}a(x,y)|\leq C_{\al,\be}\langle y\rangle^{m-\be}\,,
\;\;\forall \; \al,\be\in \mathbb{N}\,.
\end{equation} 
We denote by 
$A[\cdot]=\op(a)[\cdot]$
the pseudo operator with symbol $a:=a(x, j)$.
We call $OPS^m$ the class of the pseudo 
differential operator of order less or equal to $m$ and
$OPS^{-\infty}:=\bigcap_m OPS^m$.
We define the class $S^m$ as the set of symbols which satisfies \eqref{space3}. 
\end{defi}

We will consider mainly operators acting on $H^s(\T,\mathbb{R})$ with a quasi-periodic time dependence.  
In the case of pseudo differential operators this corresponds\footnote{since $\oo$ 
is diophantine we can replace the time variable with angles $\varphi\in\T^{\nu}$. 
The time dependence is recovered by setting $\varphi=\omega t$.} 
to consider symbols $a(\varphi, x, y)$ with $\varphi\in\T^{\nu}$. 
Clearly these operators can be thought as acting on functions 
$u(\f,x)=\sum_{j\in\Z}u_{j}(\f)e^{\ii jx}$ in 
$H^s(\T^{\nu+1},\mathbb{R})$ in the following sense:
\[
(Au)(\f,x)=\sum_{j\in\mathbb{Z}}a(\f,x,j) u_{j}(\f)e^{\mathrm{i} jx}\,, 
\quad a(\f,x,j)\in S^{m}\,.
\]
The symbol $a(\varphi, x, y)$ is $C^{\infty}$ smooth also in the variable $\varphi$. 
We still denote 
$A:=A(\varphi)=\op(a(\varphi, \cdot))=\op(a)$.

\begin{defi}
Let $a:=a(\f,x, y)\in S^{m}$ and set $A:=\op(a)\in OPS^{m}$,
\begin{equation}\label{norma}
|A|_{m,s,\al}:=\max_{0\leq \be\leq \al} 
\sup_{y\in\mathbb{R}}\|\del_{y}^{\be}a(\cdot,\cdot, y)\|_{s}
\langle y\rangle^{-m+\be}\,.
\end{equation}
We will use also the notation 
$\lvert a \rvert_{m, s, \alpha}:=|A|_{m,s,\al}$.
\end{defi}
Note that the norm $|\cdot|_{m,s,\al}$ is non-decreasing in $s$ and $\al$.
Moreover given a symbol $a(\f,x)$ independent of $y$, the norm of the associated multiplication operator $\op(a)$ is just the $H^{s}$ norm of the function $a$.
If on the contrary the symbol $a(y)$ depends only on $y$, then the norm of the corresponding 
Fourier multipliers $\op(a(y))$ is just controlled by a constant.

As in formula \eqref{tazza10}, if $A=\op(a(\oo,\f,x, y))\in OPS^{m}$ is a family of 
pseudo differential operators with symbols $a(\oo,\f,x, y)$ belonging to $S^{m}$ and 
depending in a Lipschitz way on some parameter $\oo\in \calO\subset \mathbb{R}^{\nu}$, 
we set
\begin{equation}\label{norma2}
|A|_{m,s,\al}^{\g,\calO}:=\sup_{\oo\in \calO}|A|_{m,s,\al}+
\g \sup_{\oo_1,\oo_2\in \calO}\frac{{|\rm Op}\big(a(\oo_1,\f,x, y)-a(\oo_2,\f,x, y)\big)|_{m,s-1,\al}}{|\oo_1-\oo_2|}\,.
\end{equation}
For the properties  of compositions, adjointness and quantitative estimates of the actions
on the Sobolev spaces $H^{s}$  of pseudo differential operators we 
refer to Appendix B of \cite{FGP1}.

%%%%%%%%%%%%%%%%%%%%%%%%%%%%%%%%%%%%%%%%%%%

\section{Weak Birkhoff Normal form}\label{SezioneWBNF}

%\comment{Ho cambiato questo paragrafo iniziale.}
The aim of this section is to construct a $\xi$-parameter family of approximately invariant, finite dimensional tori supporting quasi-periodic motions with frequency $\omega(\xi)$. We will impose the map $\xi\mapsto \omega(\xi)$ to be a diffeomorphism and we will consider such approximate solutions as the starting point for the Nash-Moser algorithm.

\noindent
 In order to state the main result of this section, we need some preliminary definitions.

\medskip
\noindent
We write the DP Hamiltonian in \eqref{DPHamiltonian}
 in the following way:
\begin{equation}\label{Rinco}
\begin{aligned}
H(u)&=H^{(2)}(u)+H^{(3)}(u)+H^{(\geq 9)}\, , \\
 H^{(2)}(u)&:=\frac{1}{2} \int_{\T} u^2\,dx\,, 
\quad H^{(3)}(u):=-\frac{1}{6}\int_{\mathbb{T}} u^3\,dx\,,
\quad H^{(\geq 9)}(u):=\int_{\T} f(u)\,dx\, .
\end{aligned}
\end{equation}
Recall $S$ in \eqref{TangentialSitesDP} and define 
$S^{c}:=\mathbb{Z}\setminus\big(S\cup\{0\}\big)$.
We decompose the phase space as 
\begin{equation}\label{decomposition}
H_0^1(\mathbb{T}):=H_S\oplus H_S^{\perp}\,, 
\quad H_S:=\mbox{span}\{ e^{\mathrm{i}\,j\,x} : j\in S \}\,, \quad 
H_S^{\perp}:=\mbox{span}\{ e^{\mathrm{i}\,j\,x} : j\in S^{c} \}\,,
\end{equation}
and we denote by $\Pi_S, \Pi_S^{\perp}$ the corresponding orthogonal projectors. The subspaces $H_S$ and $H_S^{\perp}$ are symplectic respect to the $2$-form $\Omega$ (see \eqref{simpleDP}). We write
\begin{equation*}
u=v+z\,, \quad v:=\Pi_S u:=\sum_{j\in S} u_j\,e^{\mathrm{i}\,j\,x}\,, 
\quad z=\Pi_S^{\perp} u:=\sum_{j\in S^c} u_j\,e^{\mathrm{i}\,j\,x}\,.
\end{equation*}
For a finite dimensional space 
\begin{equation}\label{FinitedimensionalSubspaceDP}
E:=E_C:=\mbox{span}\left\{ e^{\mathrm{i}\,j\,x} : 0<\lvert j \rvert\le C \right\}\,, 
\quad C>0\,,
\end{equation}
let $\Pi_E$ denote the corresponding $L^2$-projector on $E$.
The notation $R(v^{k-q} z^q)$ indicates a homogeneous polynomial of degree $k$ in $(v, z)$ of the form
\begin{equation*}\label{notaTRA}
R(v^{k-q} z^q)=M[\underbrace{v, \dots, v}_{(k-q)-times},\underbrace{z, \dots, z}_{q-times}\,]\,,
\quad M=k\mbox{-linear}\,.
\end{equation*}
We denote with $H^{(n, \geq k)}, H^{(n, k)}, H^{(n, \le k)}$ the terms of type $R(v^{n-s}\, z^s)$, where, respectively, $s\geq k, s=k, s\le k$, that appear in the homogeneous polynomial $H_n$ of degree $n$ in the variables $(v, z)$.
Given an $n$-uple $\{j_1,\dots , j_n\}\subset \mathbb Z\setminus\{0\}$  and a set $B\subset \mathbb Z\setminus\{0\}$ we define
\begin{equation*}\label{cardinale}
\sharp(\{j_1,\dots , j_n\},B) :=\mbox{number of $j_i$ belonging to $B$}\,. 
\end{equation*}
%In this way  $H^{(n, k)}$ is supported on the set 
%$\big\{j_1,\dots , j_n\in \mathbb{Z}\setminus\{0\}\;: \; \sum_{i=1}^n j_i=0\,, \quad \sharp(\{j_1,\dots , j_n\},S^c)=k\big\}$.

\medskip

Now we start the ``weak'' Birkhoff normal form procedure, i.e. we look for a change
of coordinates which normalizes the terms in \eqref{Rinco} independent and linear in the normal variable $z$.

As it is well known, one of the main problem of the Birkhoff normal form procedures is to deal with the resonances given by the equations \eqref{riso}$=0$
which arise from considering the kernel of the adjoint action $\mathrm{ad}_{H^{(2)}}$ (see \eqref{adjActHAM}). 
It turns out that when $n\geq 2$ there are many non-trivial solutions of \eqref{riso}=0. A way to deal with this problem is to exploit the integrability of the DP equation.

\noindent
In \cite{FGPa} the authors construct an infinite number of conserved quantities $K_n$ for the equation \eqref{DP} with $f=0$ starting from the ones given in \cite{Deg}. 
%They prove that in a small neighborhood of the origin these functions are analytic and control the Sobolev norms of the solutions. Moreover, 
By an explicit characterization of the quadratic part of each $K_n$, they deduce that, at a purely formal level, the Birkhoff normal form of the Degasperis-Procesi equation is action preserving (or integrable).
Here we rename these constants of motion in the following way, writing only the quadratic parts (which are fundamental for the study of the Birkhoff resonances at $u=0$)
\begin{equation}\label{CostantiMotoDP}
\begin{aligned}
& K_0(u):=H(u)\,, 
\qquad K_{1}(u):=\frac{1}{2}\int_{\mathbb{T}} (J^{-1} u_x) \,u\,dx\,, 
\qquad K_{n+2}:=\int_{\T} (\partial_x^n w)^2\,dx+O(u^3)\,, 
\quad n\geq 0\,,
\end{aligned}
\end{equation}
where we denoted by 
\begin{equation}\label{Helmotz}
w:=\Lambda^{-1} u:=u-u_{xx}\,, \qquad \Lambda:=(1-\partial_{xx})^{-1}\,.
\end{equation}
We remark that $K_1$ is the \textit{momentum} Hamiltonian arising from the translation invariance of the equation.

\begin{defi}
Given a quadratic diagonal Hamiltonian $Q(u)=\sum_j \lal(j) \lvert u_j \rvert^2$, we define $\Pi_{\mbox{Ker}(Q)}$ as the projection on the kernel of the adjoint action (recall \eqref{PoissonBracketDP} and $J=\mathrm{diag}_{j\in\mathbb{Z}}(\lambda(j))$)
\begin{equation}\label{defadj}
\mathrm{ad}_{Q}(K)=\sum_{j_1, \dots, j_n} 
\Big( \sum_{i=1}^n \lal(j_i) \og(j_i)\Big) K_{j_1 \dots j_n} u_{j_1}\dots u_{j_n}\,, \quad K(u):=\sum_{j_1, \dots, j_n} K_{j_1 \dots j_n} u_{j_1}\dots u_{j_n}.
\end{equation}
We define the projector on the range of the adjoint action as 
$\Pi_{\mbox{Rg}(Q)}:=\mathrm{I}-\Pi_{{\mbox{Ker}(Q)}}$.
\end{defi}
\noindent
We say that $K$, as in \eqref{defadj}, "preserves" momentum if and only if 
\begin{equation*}\label{momentino}
%\mathrm{ad}_{K_1}(K)=
\Big( \sum_{i=1}^n j_i \Big) 
K_{j_1 \dots j_n}=0 \qquad \forall j_1, \dots, j_n\in \mathbb{Z}\setminus\{0\}\,.
\end{equation*}

\noindent
The main result of this section is the following.

\begin{prop}\label{WBNFdp}
There exist $r>0$, depending on  $S$ (see \eqref{TangentialSitesDP}), and an analytic  symplectic change of coordinates 
\begin{equation}\label{FBI}
\Phi_B\colon \mathcal B_{r}(0,H_0^1(\T))\to H_0^1(\T)\,, \qquad \Phi_B= \mathrm{I} + \Psi\,,\quad  \Psi= \Pi_E \circ  \Psi \circ \Pi_E 
\end{equation}
where $E$ is a finite dimensional space 
as in \eqref{FinitedimensionalSubspaceDP}, 
such that the Hamiltonian $H$ in \eqref{Rinco} transforms into
\begin{equation}\label{piotta}
\begin{aligned}
\mathcal{H}:= H\circ \Phi_B 
&= H^{(2)} +\mathcal{H}^{(4,0)} +\mathcal{H}^{(6,0)} 
+\mathcal{H}^{(8, 0)}+\mathcal{H}^{(\geq 9, \le 1)}
+ \mathcal H^{(\geq 3,\geq 2)}\,,
\end{aligned}
\end{equation}
where 
\begin{equation}\label{FBI2}
\begin{aligned}
\mathcal{H}^{(3, \geq 2)}&:=-\frac{1}{2}\int_{\T} v\, z^2\,dx-\frac{1}{6} \int_{\T} z^3\,dx\,, 
\\[2mm]
\mathcal{H}^{(4,0)} &	:= 
\frac{1}{2}\sum_{j\in S^+} \,\frac{\og(2 j)}{2\og(j)-\og(2 j)}\,\lvert u_j \rvert^4
+\sum_{\substack{j_1, j_2\in S^+,\\ j_1- j_2\neq 0}} 
\frac{\og(j_1+j_2)}{\og(j_1)+\og(j_2)-\og(j_1+j_2)}
\lvert u_{j_1} \rvert^2\lvert u_{j_2}\rvert^2\\
&\,\quad
+\sum_{\substack{j_1, j_2\in S^+,\\ j_1 - j_2\neq 0}} 
\frac{\og(j_1-j_2)}{\og(j_1)-\og(j_2)-\og(j_1-j_2)}
\lvert u_{j_1} \rvert^2\lvert u_{j_2}\rvert^2
\end{aligned}
\end{equation}
and $\mathcal{H}^{(k, 0)}=\Pi_{\mbox{Ker}(H^{(2)})} \mathcal{H}^{(k, 0)}$ with $k=4, 6, 8$ depend only on $\lvert u_j \rvert^2$.
The same change of variables $\Phi_B$ puts all the Hamiltonians in \eqref{CostantiMotoDP} in weak Birkhoff normal form  up to order eight as in \eqref{piotta}. In particular we have
$K_1\circ\Phi_B=K_1$.
\end{prop}
%\comment{E' giusto quello in rosso sopra?}

\noindent
In order to prove the Proposition \ref{WBNFdp} above we need some preliminary results proved in detail in \cite{FGPa}.

\begin{defi}[{\bf $M$-resonances}]\label{nresonances}
Fix $M\in\mathbb{N}$, $M\geq 3$. We recall that the quadratic part of $H$ and $K_r$, $2\le r\le M$, in \eqref{CostantiMotoDP} are
\[
K^{(2)}_r(u):=\sum_{j\in\mathbb{Z}\setminus\{0\}} 
(1+j^2)^2\,j^{2 (r-2)}\,\lvert u_{j} \rvert^2\,, 
\quad H^{(2)}(u)=\sum_{j\geq 1} \lvert u_{j} \rvert^2\,.
\]
We say that an $n$-uple $\{j_1,\dots,j_{n}\}\subset \Z\setminus\{0\}$, with $n\le M$, is a $M$-{\em resonance}  of order $n$ for the DP hierachy if
\begin{equation}\label{sgomento}
\sum_{i=1}^n j_i=0\,, \quad\sum_{i=1}^n \og(j_i)=0\,, 
\quad \sum_{i=1}^n (1+j_i^2)^2\,j_i^{2 (r-2)}\og(j_i)=0\, \quad \forall r=2, \dots, M+1\,.
\end{equation} 
\end{defi}

\begin{prop}\label{vvv}
Fix $M\in\mathbb{N}$, $M\geq 3$. All the $M$-resonances of the DP equations \eqref{CostantiMotoDP} are trivial, namely there are no resonances of odd order and the even ones are, up to permutations, of the form
\begin{equation}\label{coppiette}
(i, -i, j, -j,k, -k, p, -p, \dots)\,.
\end{equation}
\end{prop}
\begin{proof} Since this Proposition is proved in \cite{FGPa} with different notations, for completeness  we restate here  a concise proof by induction on $M$. 
For $M=3$  the thesis  follows trivially: 
indeed direct computations show that 
\[
\sum_{i=1}^3 j_i=0, \quad\sum_{i=1}^3 \og(j_i)=0 \;
\Leftrightarrow\;  j_1=-j_2\,,\; j_3=0
\]
up to permutations, and this solution is  
incompatible with $j_i\in \Z\setminus\{0\}$.
	
Let us now suppose that the thesis is true up to $M-1\ge 3$ 
and prove it for $M$.
We start by noticing that if $n<M$ then  
\eqref{sgomento} with  $r\le M-1$ can hold only 
if $\{j_1,\dots,j_{n}\}$ is a $M-1$ resonance of order $n$. 
	The inductive hypothesis then says that $\{j_1,\dots,j_{n}\}$ is trivial. Similarly if  $\{j_1,\dots,j_{n}\}$ contains a trivial resonance, i.e. if $j_{i_1}+ j_{i_2}=0$ for $1\le i_1,i_2\le n$, then $j_{i_1}, j_{i_2}$ do not appear  in \eqref{sgomento} and hence $\{j_1,\dots,j_{n}\} $ is  an $M$-resonance of order $n$ if and only if 
\[
\{j_1,\dots,j_{n}\}\setminus \{j_{i_1}, j_{i_2}\}\,,
\quad \mbox{is an $M-2$ resonance of order $n-2$}.
\]
Without loss of generality  we assume that  $n=M$ and 
that $j_{i_1}+ j_{i_2}\ne 0$ for any $1\le i_1,i_2\le M$. 

\noindent
Up to a permutation we can assume that  for some $M\ge k\ge 1$ and $\alpha_1,\dots,\alpha_k\ge 1$  one has
\[
\{j_1,\dots,j_{n}\}= \{\underbrace{\hat\jmath_1,\dots, \hat\jmath_1}_{\alpha_1}, \dots, \underbrace{\hat\jmath_k,\dots, \hat\jmath_k}_{\alpha_k}\}\,.
\]
Consequently rewrite the third equation in \eqref{sgomento} as
$
 \sum_{i=1}^{k} \alpha_i (1+\hat\jmath_i{}^2)^2\,
 \hat\jmath_i{}^{2 (r-2)}\og(\hat\jmath_i)=0$,  $\forall r=2, \dots, M+1$.
Then we can extract $k$ equations from these ones and write them in the form
\begin{equation}\label{sgomento3}
\begin{pmatrix}
1 & \dots & 1\\
\hat\jmath_1{}^{2} & \dots & \hat\jmath_k{}^{2} \\
\vdots &   & \vdots\\
\hat\jmath_1{}^{2(k-1)} & \dots & \hat\jmath_k{}^{2(k-1)}
\end{pmatrix}
\begin{pmatrix}
\alpha_1(1+\hat\jmath_1{}^2)^2\,\og(\hat\jmath_1) \\
\vdots\\
\alpha_k(1+\hat\jmath_k{}^2)^2\,\og(\hat\jmath_k)
\end{pmatrix}
=
\begin{pmatrix}
0 \\
\vdots \\
0
\end{pmatrix}\,.
\end{equation}
The determinant of the  Vandermonde matrix in \eqref{sgomento3} is 
$
\prod_{i\neq h} (\hat\jmath_i{}^2-\hat\jmath_h{}^2)\ne 0,
$
since, by hypothesis, $\hat\jmath_i\ne \pm \hat\jmath_h$. Then the only possible solution corresponds to $\hat\jmath_i=0$ for all $i$, which is not compatible with
$\hat\jmath_i\in \Z\setminus\{0\}$.
\end{proof}

\begin{remark}\label{wellDefDP}
Notice that 
if $j_1, \dots , j_N\in\mathbb{Z}\setminus\{ 0\}$, 
$j_1+\dots+ j_N=0$ and 
$\#(\{ j_1, \dots , j_N \}, S^c)\le 1$,
then $\max_{i=1,\dots, N} \lvert j_i \rvert \le (N-1) \overline{\jmath}_1$. 
Thus, the vector field $X_{F^{(N,\leq1)}}$, generated by 
the finitely supported Hamiltonian 
\begin{equation}\label{formaGen}
F^{(N,\leq1)}=\sum_{\substack{j_1+\dots+j_N=0 \\ 
\#(\{ j_1, \dots , j_N \}, S^c)\le 1}} F^{(N,\leq1)}_{j_1 \dots j_N} u_{j_1}\dots u_{j_N}\,
\end{equation}
is finite rank, and, in particular, 
it vanishes outside the finite dimensional 
subspace $E:=E_{(N-1) \overline{\jmath}_1}$ 
(see \eqref{FinitedimensionalSubspaceDP} ) and it has the form 
\[
X_{F^{(N,\leq1)}}(u)=\Pi_E X_{F^{(N,\leq1)}} (\Pi_E u)\,.
\] 
Therefore its flow $\Phi^{(N)}$ is analytic and invertible on the phase space 
$H_0^1(\mathbb{T})$, provided that $| \Pi_E u|$ is appropriately small.
\end{remark}

In order to prove Proposition \ref{WBNFdp} we need the following result.
\begin{prop}\label{LemmaBelloWeak}
Fix $M\in\mathbb{N}$, $M\geq 2$ and consider $H$ in \eqref{DPHamiltonian} and $K_m$, $m=1, \dots, M$, in \eqref{CostantiMotoDP}.  Then, for any $N\le M-2$, there exists  $r>0$ and an analytic symplectic  change of coordinates $\Phi_N^{\pm1} \colon \mathcal B_r(0,H_0^1(\T))\to H_0^1(\T)$ of the form
\begin{equation}
\begin{aligned}
&\Phi_N= \mathrm{I} + \Psi_N\,, \quad \Phi_0=\mathrm{I},\quad  \Psi_N(u)= \Pi_E \circ  \Psi_N  \circ \Pi_E  ,
\end{aligned}
\end{equation}
where $E$ is a finite dimensional space as in \eqref{FinitedimensionalSubspaceDP}, such that
\begin{equation}\label{wnormalform}
\begin{aligned}
&H\circ \Phi_N^{-1}=H^{(2)}+Z_N^{(\le N+2, 0)}+R_N^{(\geq N+3, \le 1)}
+H_{N}^{(\geq 3, \geq 2)}\,,\qquad\;\; K_1\circ \Phi_N^{-1}=K_1\,,\\
&K_m\circ \Phi_N^{-1}=K_m^{(2)}+W_{m, N}^{(\le N+2, 0)}
+Q_{m, N}^{(\geq N+3, \le 1)}
+K_{m,N}^{(\geq 3, \geq 2)}\,, 
\quad m=2, \dots, M\,,
\end{aligned}
\end{equation}
where $Z_N^{(\le N+2, 0)}, W^{(\le N+2, 0)}_{m, N}\in \bigcap_{m=1}^M \mbox{Ker}(K_m^{(2)}) \cap \mbox{Ker}(H^{(2)})$.
\end{prop}
\begin{proof}
%In \eqref{wnormalform} the normalized terms are the ones which commute with the quadratic parts and that are Fourier supported only on indices belonging of $S$.  
The terms of degree at most $2$ in the variable $z$ are not affected by the procedure that we are going to describe. \\
We argue the result by induction on the number of steps $N$. For $N=0$ it is trivial since $\Phi_0$ is the identity map.\\
Suppose that we have performed $N$ steps. By the fact that $\{ H, K_m \}=0$ then $\{ H, K_m\}\circ \Phi_N^{-1}=0$. For the latter, we are interested in the corresponding equations for the terms of homogeneity at most $N+3$ and degree in the variable $z$ less or equal than one. So we consider the projection $\Pi^{(\le N+3, \le 1)}\Big(\{ H, K_m\}\circ \Phi_N^{-1}\Big)=0$ and
 we get, for any $m=1, \dots, M$, the following system of equations
 $\{ H^{(2)}, K_m^{(2)}\}=0\,$ and 
\begin{align*}
%&\{ H^{(2)}, K_m^{(2)}\}=0\,,\\
&\{ H^{(2)}, W_{N, m}^{(N+2, 0)}  \}
+\{  Z^{(N+2, 0)}_N, K_m^{(2)}   \}
+\Pi^{(\le N+2)}\{  Z_N^{(N+2, 0)}, W_{N, m}^{(N+2, 0)}\}=0\,,\\
&\Pi^{(N+3)}\{  Z_N^{(N+2, 0)}, W_{N, m}^{(N+2, 0)}\}
+\{ H^{(2)}, Q_{m,N}^{(N+3, \le 1)}\}
+\{ R_N^{(N+3, \le 1)}, K_m^{(2)} \}=0\,.
\end{align*}
By the inductive hypothesis 
$W_{N, m}^{(N+2, 0)}, Z_N^{(N+2, 0)}\in
\bigcap_{m=1}^M \mbox{Ker}(K_m^{(2)})
\cap \mbox{Ker}(H^{(2)})$, 
hence 
\[
\{ H^{(2)}, W_{N, m}^{(N+2, 0)}\}=\{ Z_N^{(N+2, 0)}, K_m^{(2)}\}=0
\]
 and 
\begin{equation}\label{identita}
\{ H^{(2)}, Q_{m,N}^{(N+3, \le 1)}\}+\{ R_N^{(N+3, \le 1)}, K_m^{(2)} \}=0\,, 
\quad m=1, \dots, M\,,
\end{equation}
since $\{ H^{(2)}, Q_{m,N}^{(N+3, \le 1)}\}\in \mbox{Rg}(H^{(2)})$ 
and $\{ R_N^{(N+3, \le 1)}, K_m^{(2)} \}\in \mbox{Rg}(K_m^{(2)})$.

\smallskip
\noindent
We note the following fact, which derives from the Jacobi identity: 
if $f\in \mbox{Ker}(H^{(2)})$ then $\{ f, K_m^{(2)}\}\in \mbox{Ker}(H^{(2)})$.

\noindent
Then we have that $\{ \Pi_{\mbox{Ker}(H^{(2)})} R_N^{(N+3, \le 1)}, K_m^{(2)}  \}
\in \mbox{Ker}(H^{(2)})$ and by \eqref{identita}
\[
\{ \Pi_{\mbox{Ker}(H^{(2)})} R_N^{(N+3, \le 1)}, K_m^{(2)}  \}=
-\{ \Pi_{\mbox{Rg}(H^{(2)})} R_N^{(N+3, \le 1)}, K_m^{(2)}  \}
+\{ H^{(2)}, Q_{N, m}^{(N+3, \le 1)}\}\in \mbox{Rg}(H^{(2)})\,.
\]
Thus $\{ \Pi_{\mbox{Ker}(H^{(2)})} R_N^{(N+3, \le 1)}, K_m^{(2)}  \}=0$ and 
\[
\Pi_{\mbox{Ker}(H^{(2)})} R_N^{(N+3, \le 1)}
=\Pi_{\mbox{Ker}(H^{(2)})} \Pi_{\mbox{Ker}(K_m^{(2)})} R_N^{(N+3, \le 1)}\,.
\]
By symmetry 
$\Pi_{\mbox{Ker}(K_m^{(2)})} Q_{m, N}^{(N+3, \le 1)}
=\Pi_{\mbox{Ker}(H^{(2)})} \Pi_{\mbox{Ker}(K_m^{(2)})} Q_{m, N}^{(N+3, \le 1)}$. 
Hence
\begin{equation}\label{prisoner}
\Pi_{\mbox{Rg}(H^{(2)})} \Pi_{\mbox{Ker}(K_m^{(2)})} Q_{m, N}^{(N+3, \le 1)}
=\Pi_{\mbox{Rg}(K_m^{(2)})} \Pi_{\mbox{Ker}(H^{(2)})} R_N^{(N+3, \le 1)}=0\,, 
\quad m=1,\dots, M\,.
\end{equation}
In order to obtain the Birkhoff normal form at 
order $N+3$ we consider a Birkhoff transformation 
$\Phi_{F^{(N+3,\leq1)}}$ with generator 
$F^{(N+3),\leq1}$ of the form \eqref{formaGen} 
( with $N\rightsquigarrow N+3$)
%\[
%F^{(N+3,\leq1)}(u)=\sum_{\substack{j_1+\dots+j_{N+3}=0,\\ \#(j_i, S^c)\le 1}} 
%F_{ j_1 \dots j_{N+3}}^{(N+3,\leq1)}\,u_{j_1}\dots u_{j_{N+3}}\,,
%\]
 and we define $\Phi_{N+1}:= \Phi_{F^{(N+3,\leq1)}} \circ \Phi_N$. By Remark \ref{wellDefDP} the flow $\Phi_{F^{(N+3,\leq1)}}$ is well defined in an appropriately small ball and it has the form Identity plus a finite rank operator. Note that, since $F^{(N+3,\leq1)}$ is Fourier supported on $(j_1, \dots, j_{N+3})$ such that $j_1+\dots+j_{N+3}=0$, the Hamiltonian $K_1$ commutes with $F^{(N+3,\leq1)}$ and, by the inductive hypothesis, $K_1\circ \Phi_{N+1}^{-1}=K_1$.
The function $F^{(N+3,\leq1)}$ is chosen in order to solve the homological equation
\[
\{ H^{(2)}, F^{(N+3,\leq1)}\}=\Pi_{\mbox{Rg}(H^{(2)})} R_N^{(N+3, \le 1)} 
\stackrel{(\ref{prisoner})}{=}
\Pi_{\mbox{Rg}(K_m^{(2)})}\Pi_{\mbox{Rg}(H^{(2)})} R_N^{(N+3, \le 1)}\,.
\]
We now show that $F^{(N+3,\leq1)}$ solves also the homological equation 
for the commuting Hamiltonians 
$K_m\circ \Phi_N^{-1}$, $m=1, \dots, M$. 
Indeed, by the fact that $\mathrm{ad}_{H^{(2)}}^{-1}$ 
commutes with $\mathrm{ad}_{K_m^{(2)}}$ 
on the intersection $\mbox{Rg}(H^{(2)})\cap \mbox{Rg}(K_m^{(2)})$, we have
\[
\{ K_m^{(2)}, F^{(N+3,\leq1)}\}=
\mathrm{ad}_{H^{(2)}}^{-1}\{ K_m^{(2)}, 
\Pi_{Rg(K_m^{(2)})}\Pi_{Rg(H^{(2)})} R_N^{(N+3, \le 1)} \}\,,
\]
and by \eqref{identita}, \eqref{prisoner} we get
\[
\{ K_m^{(2)}, \Pi_{\mbox{Rg}(K_m^{(2)})}
\Pi_{\mbox{Rg}(H^{(2)})} R_N^{(N+3, \le 1)} \}
=\{H^{(2)}, \Pi_{\mbox{Rg}(K_m^{(2)})}\Pi_{\mbox{Rg}(H^{(2)})} 
Q_{m, N}^{(N+3, \le 1)}  \}\,.
\]
By \eqref{prisoner} we have that the resonant term 
$Z^{(N+3, \le 1)}_{N+1}:=\Pi_{\mbox{Ker}(H^{(2)})} R_N^{(N+3, \le 1)}$ 
belongs to the intersection of the kernels 
and by Proposition \ref{vvv} these terms are supported only 
on $n$-ples of indices of the form $(i, -i, j, -j, k, -k, \dots)$. 
By the symmetry of the tangential set $S$ 
this is possible for a set of indices with at most one 
outside $S$ if and only if all the indices belong to $S$. 
Hence $Z^{(N+3, 1)}_{N+1}=0$ and we define $Z_{N+1}^{(\le N+3, 0)}:=Z_{N+1}^{(N+3, 0)}+Z_N^{(\le N+2, 0)}$.
We do not compute explicitly the radius $r$ 
of the ball in which we can perform the Birkhoff change of variables, 
however one can easiliy check that 
$r\to 0$ as $N\to \infty$ or as 
$\mathtt r\to 0$ in Definition \ref{TangentialSitesDP}.
\end{proof}

\begin{proof}[{\bf Proof of Proposition \ref{WBNFdp}}]
We apply Proposition \ref{LemmaBelloWeak} with $N=6$ and $M=8$ 
and we obtain \eqref{FBI}, \eqref{piotta} by setting $\Phi_B:=\Phi_N^{-1}$. To prove \eqref{FBI2} 
we have to show explicitly 
the computations of the first step of Birkhoff normal form.

\smallskip
\noindent
First we remove the cubic terms independent of $z$ and linear in $z$ from the Hamiltonian 
\begin{equation}\label{terzogrado}
H^{(3)}= -\frac16 \int_{\T} u^3\,dx=
-\frac{1}{6} \int_{\T}v^3\,dx-\frac{1}{2}\int_{\T} v^2 z\,dx
-\frac{1}{2}\int_{\T} v z^2\,dx-\frac{1}{6}\int_{\T} z^3\,dx\,. 
\end{equation}
We consider $\Phi_1:=(\Phi^t_{F^{(3,\leq 1)}})_{|_{t=1}}$ as the time-$1$ flow map generated by the Hamiltonian vector field $X_{F^{(3,\leq1)}}$, with an auxiliary Hamiltonian  $F^{(3,\leq 1)}$ of the form \eqref{formaGen} 
with $N=3$.
%\begin{equation}\label{AuxliaryHam3}
%F^{(3,\leq1)}(u):=\sum_{j_1+j_2+j_3=0} F^{(3,\leq1)}_{j_1\,j_2\,j_3}\,u_{j_1}\,u_{j_2}\,u_{j_3}\,.
%\end{equation}
The transformed Hamiltonian is
$H_1:=H\circ \Phi_1^{-1}=H^{(2)}+H_1^{(3)}+H_1^{(4)}+H_1^{(\geq 5)}$ with
\begin{align}
&H_1^{(3)}:=\{ F^{(3,\leq1)}, H^{(2)}\}+H^{(3)}\,, 
\quad H_1^{(4)}:=
\frac{1}{2}\{ F^{(3,\leq1)} , \{F^{(3,\leq1)}, H^{(2)}\}\}+\{ F^{(3,\leq1)}, H^{(3)}\}\,,\label{NewHam3}
\end{align}
and where $H_1^{(\geq 5)}$ collects all the terms of order at least five in $(v, z)$.
We choose $F^{(3,\leq1)}$ 
%in \eqref{AuxliaryHam3} 
such that the following homological equation holds
\begin{equation}\label{HomologicalEquation3}
\{ F^{(3,\leq1)}, H^{(2)}\}+H^{(3)}=H^{(3, \geq 2)} \quad \Leftrightarrow \quad \{H^{(2)}, F^{(3,\leq1)}\}=\Pi_{\mbox{Rg}(H^{(2)})} H^{(3, \leq 1)}\,.
\end{equation}
%We define the set
%\[
%\mathcal{A}_3:=\{ \{ j_1, j_2, j_3\}\subset \mathbb{Z}\setminus \{0\} :  j_1+j_2+j_3=0\,\,\mbox{and at least two indices among}\,\, j_1, j_2, j_3 \,\,\mbox{belong to}\,\,A\}.
%\]
Recalling  \eqref{PoissonBracketDP} and \eqref{terzogrado}, the solution of the equation \eqref{HomologicalEquation3} 
is given by $F^{(3,\leq1)}$ as in \eqref{formaGen} with $N=3$ 
with coefficients defined as 
%\begin{equation}
%\sum_{\substack{j_1, j_2, j_3\in \mathbb{Z}\setminus\{0\},\\ j_1+j_2+j_3=0}} \mathrm{i}\,(\og(j_1)+\og(j_2)+\og(j_3))\,F^{(3,\leq1)}_{j_1 j_2 j_3}\,u_{j_1} u_{j_2} u_{j_3}=\frac{1}{6} \sum_{\substack{j_1, j_2, j_3\in\mathbb{Z}\setminus\{ 0\},\\\sharp(\{j_1,j_2,j_3\}, S^c)\le 1\\ \og(j_1)+\og(j_2)+\og(j_3)\neq 0\\ j_1+j_2+j_3=0}} u_{j_1} u_{j_2} u_{j_3}\,,
%\end{equation}
%and we have to determine $F^{(3,\leq1)}_{j_1 j_2 j_3}$.
%By Proposition \ref{vvv} there are no non-trivial $3$-resonances of order $3$.
%Hence we  set
\begin{equation}\label{DefF3}
F_{j_1 j_2 j_3}^{(3,\leq1)}:=\begin{cases}
\dfrac{1}{6\,\mathrm{i}\, (\og(j_1)+\og(j_2)+\og(j_3))} \qquad \mbox{if}\,\,\sharp (\{ j_1, j_2, j_3 \}, S^c)\le 1,\, j_1+j_2+j_3=0\,,\\
0 \qquad\qquad \qquad \qquad\qquad\qquad\qquad\qquad \mbox{otherwise\,.}
\end{cases}
\end{equation}
The Hamiltonian $F^{(3,\leq1)}$ is well defined since,
by Proposition \ref{vvv},
 there are no non-trivial $3$-resonances of order $3$.
Since $\Pi_{\mbox{Rg}(H^{(2)})} H^{(3, \leq 1)}=H^{(3, \leq 1)}$ we get (see \eqref{NewHam3}, \eqref{HomologicalEquation3})
\begin{equation}\label{mastite}
H^{(3)}_1= H^{(3, \geq 2)}\,, 
\qquad
H_1^{(4)}=\frac{1}{2}\{F^{(3,\leq1)}, H^{(3, \leq 1)}\}+\{ F^{(3,\leq1)}, H^{(3, \geq 2)}\}\,.
\end{equation}

%\smallskip
\noindent
In the second step we normalize the terms of total degree $4$ and $\le 1$ in the variable $z$. The term $\Pi_{\mbox{Ker}(H^{(2)})}H_1^{(4,\leq 1)}$ is Fourier supported on the set of $4$-resonances  of order $4$, which are trivial 
by Proposition \ref{vvv}.
By Proposition \ref{LemmaBelloWeak} $\Pi_{\mbox{Ker}(H^{(2)})} H_1^{(4, 1)}=0$.
Thus we have to compute only $\Pi_{\mbox{Ker}(H^{(2)})} H_1^{(4, 0)}$. 
We have
\begin{equation}\label{posta2}
Z_2^{(4,0)}= \Pi_{\mbox{Ker}(H^{(2)})}H^{(4, 0)}_1=\frac{1}{8} \sum_{\substack{j_1, j_2, j_3, j_4\in S,\\j_1+j_2+j_3+j_4=0\\ j_1+j_2\neq 0,\, j_3+j_4\neq 0,\\\sum_{k=1}^4 \og(j_k)=0} }\frac{\og(j_1+j_2)}{\og(j_1)+\og(j_2)-\og(j_1+j_2)}\,u_{j_1} u_{j_2} u_{j_3} u_{j_4}.
\end{equation}
The remaining steps of this procedure do not affect the terms with degree of homogeneity less or equal than $4$. Hence by \eqref{posta2}, the fact that $\og(-j)=-\og(j)$ (see \eqref{dispersionLaw}) and the symmetry of $S$ we obtain \eqref{FBI2}.
\end{proof}

\section{Action-angle variables}\label{sezioneActionAngle}

On the submanifold $\{ z=0\}$ we put the following action-angle variables
\begin{equation}\label{aa0dp}
\begin{aligned}
\mathbb{T}^{\nu}&\times [0, \infty)^{\nu} \longrightarrow \{ z=0\}\,,
%\\[1.5mm]
%& 
\qquad
(\theta, I) \longmapsto v=\sum_{j\in S} \sqrt{I_j} \,e^{\mathrm{i} \theta_j}\,e^{\mathrm{i} j x} .
\end{aligned}
\end{equation}
Note that this change of 
coordinates is real 
%(according to Definition \eqref{dynamicaldef}-$(4)$)
 if and only if $I_{-j}=I_j$ and $\theta_{-j}=-\theta_j$.
The symplectic form in \eqref{simpleDP} 
restricted to the subspace $H_S$ transforms into the $2$-form
$
\sum_{j\in S^+}\frac{1}{\og(j)}  d\theta_j\wedge \, d I_j\,.
$
We have that the Hamiltonian
$\mathcal{H}^{(\leq 8)}(\theta, I, 0)=\sum_{j\in S^+} I_j
+\mathcal{H}^{(4, 0)}(I)+\mathcal{H}^{(6, 0)}(I)+\mathcal{H}^{(8, 0)}(I)$
depends only by the actions $I$ and its equations of motion read as
\begin{equation}\label{HamiltonianSisteminoDP}
\begin{cases}
\dot{\theta}_j=\og(j)\,\,\partial_{I_j} \mathcal{H}^{(\leq 8)}(\theta, I, 0)=\omega_j(I)\,, 
\qquad j\in S^+\,,\\[3mm]
\dot{I}_j=-\lambda(j)\partial_{\theta_j} \mathcal{H}^{(\leq 8)}(\theta, I, 0)=0\,, 
\qquad \quad\,\,\,\, j\in S^+\,,
\end{cases}
\end{equation}
where, by \eqref{FBI2},
\begin{align}
\omega_j(I)&=\og(j)+\frac{1}{2}\frac{\og(j) \og(2 j)}{2 \og(j)-\og(2 j)}\, I_j+ \og(j) \sum_{k\in S^+, k\neq  j}\mathtt{b}_{j k}\,I_k+O(I^2)\, , \quad j\in S^+\, ,\label{omeghinoDP}\\
\qquad\mathtt{b}_{j k}:\!
&=\frac{2}{3} \frac{(1+k^2)(1+j^2)(2+k^2+j^2)}{(3+k^2+j^2+k j)(3+k^2+j^2-k j)}\,.\label{parco1}
\end{align}
The submanifold $\{z=0\}$ is foliated by tori supporting 
small amplitude quasi-periodic solutions for the
 truncated system with Hamiltonian $\mathcal{H}^{(\le 8)}$. 
 We shall select some of them as starting point of the 
 Nash-Moser scheme. This selection shall be dictated 
 by the imposition of non-resonance conditions 
 on the frequencies $\omega(I)$ of such approximate solutions. 
 The \emph{unperturbed actions} $I$ will be 
 used as parameters to modulate 
 the vectors $\omega(I)$. 
%
%\textcolor{red}{Hence, in a small neighbourhood of the origin of the phase space $H_0^1(\mathbb{T})$, 
%the submanifold 
%$\{ z=0 \}$ is foliated 
%by invariant tori of fixed 
%small amplitude $I$ and frequency vector 
%$\omega(I):=(\omega_j(I))_{j\in S^+}$ as in \eqref{omeghinoDP}.\\
%We shall select from this set of tori the approximately invariant quasi-periodic solutions to be continued and we will use their \textit{unperturbed actions} $I$ as parameters to this purpose we require that the frequency-amplitude map  \eqref{omeghinoDP} is a diffeomorphism. 
% \\}
% \comment{PROPOSTA: \\
% The submanifold $\{z=0\}$ is foliated by tori supporting small amplitude quasi-periodic solutions for the truncated system with Hamiltonian $\mathcal{H}^{(\le 8)}$. We shall select some of them as starting point of the Nash-Moser scheme. This selection shall be dictated by the imposition of non-resonance conditions on the frequencies of such approximate solutions. The \emph{unperturbed actions} $I$ will be used as parameters to modulate the vectors $\omega(I)$. }

\noindent
In order to highlight the fact that the amplitudes are small we introduce a small parameter $\varepsilon>0$ and we rescale $\xi\mapsto \varepsilon^2 \xi$, so that the \emph{frequency-amplitude map} can be written as
\begin{equation}\label{FreqAmplMapDP}
\alpha(\xi)=\overline{\omega}+\e^2\mathbb{A}\,\xi+O(\e^4)\,,
\end{equation}
where $\overline{\omega}$ is the vector of the linear frequencies (see \eqref{LinearFreqDP}),
\begin{equation}\label{TwistMatrixDP}
\begin{aligned}
\mathbb{A}:=\frac{1}{2}\mathbb{D}\,\,\,\mathrm{diag} 
\left(\frac{\og(2 j)}{2 \og(j)-\og(2 j)}\right)_{j\in S^+}+\mathbb{D}\,\,\,\mathbb{B}\,, 
\quad \mathbb{D}:=\mathrm{diag} \big(\og(j)\big)_{j\in S^+}\,,\quad 
\mathbb{B}_j^k:=\begin{cases}
\mathtt{b}_{j k}\quad 	\,\,\mbox{if}\,\,\,\,j\neq k\,,\\
0 \qquad \,\, \mbox{if}\,\,\,j=k\,.
\end{cases}
\end{aligned}
\end{equation}
In order to work in a small neighbourhood of the unperturbed torus $\{ I\equiv \varepsilon^2\xi \}$ it is advantageous to introduce a set of coordinates $(\theta, y, z)\in\mathbb{T}^{\nu}\times \mathbb{R}^{\nu}\times H_S^{\perp}$ adapted to it, defined by
\begin{equation}\label{AepsilonDP}
u=A_{\varepsilon}(\theta, y, z)=
\e v_\e(\theta,y)+\e^{b}z \qquad \Longleftrightarrow \qquad
\begin{cases}
u_j:=\varepsilon\sqrt{ \xi_j+\varepsilon^{2b-2}\lvert \og(j) \rvert y_j }\,
e^{\mathrm{i}\theta_j}\,e^{\mathrm{i}\,j\,x}\,, \,\,\quad j\in S\,,\\[2mm]
u_j:=\varepsilon^b z_j\,, \qquad \qquad \qquad \qquad \qquad\qquad j\in S^c\,,
\end{cases}
\end{equation}
with $b>1$ and
where  (recall $\overline{u}_j=u_{-j}$)
\begin{equation*}
\xi_{-j}=\xi_j\,, \quad \xi_j>0\,, \quad y_{-j}=y_{j}\,, \quad \theta_{-j}=-\theta_j\,, \quad \theta_j\in\mathbb{T}\,,\,\, y_j\in\mathbb{R}\,, \quad \forall j\in S\,.
\end{equation*}
The parameter $b$ will be chosen close to one, to this purpose we shall set
\begin{equation}\label{donnapia}
a := 2b -2\,,
\end{equation}
and fix $a>0$ appropriately small.
For the tangential sites 
$S^+:=\{ \overline{\jmath}_1,\dots, \overline{\jmath}_{\nu}\}$ 
we will also denote 
$\theta_{\overline{\jmath}_i}:=\theta_i$, $y_{\overline{\jmath}_i}:=y_{i}$, 
 $\xi_{\overline{\jmath}_i}:=\xi_i$, 
$ i=1,\dots, \nu$.
%\[
%\theta_{\overline{\jmath}_i}:=\theta_i\,,\quad y_{\overline{\jmath}_i}:=y_{i}\,, 
%\quad \xi_{\overline{\jmath}_i}:=\xi_i\,, 
%%\quad \og({\overline{\jmath}_i}):=\bar{\omega}_i\,, 
%\quad i=1,\dots, \nu\,.
%\]
The symplectic $2$-form $\Omega$ in \eqref{simpleDP}, up to rescaling of time, becomes
\begin{equation}\label{NewSimplFormDP}
\mathcal{W}:=\sum_{i=1}^{\nu} d\theta_i\wedge d y_i+\frac{1}{2} \sum_{j\in S^c} \frac{1}{\mathrm{i} \og(j)}\,d z_j\wedge d z_{-j}=\Big( \sum_{i=1}^{\nu} d \theta_i\wedge d y_i \Big) \oplus \Omega_{S^{\perp}}\,,
\end{equation}
where $\Omega_{S^{\perp}}$ is the symplectic form $\Omega$ in \eqref{simpleDP}
restricted to the subspace $H_{S}^{\perp}$ in \eqref{decomposition}.
The Hamiltonian system generated by $\mathcal{H}$ in \eqref{piotta} becomes
\begin{equation}\label{HamiltonianaRiscalataDP}
H_{\varepsilon}:=\varepsilon^{-2 b}\,\mathcal{H}\circ A_{\varepsilon}\,.
\end{equation}

In the following lemma we prove that, under an appropriate choice of the tangential set \eqref{TangentialSitesDP}, the function \eqref{FreqAmplMapDP} is a diffeomorphism for $\varepsilon$ small enough (recall that we rescaled $\xi\mapsto \varepsilon^2 \xi$) and then the system \eqref{HamiltonianSisteminoDP} is integrable and non-isochronous.

\begin{lem}[{\bf Twist condition}]\label{Twist1}
There exist $\mathtt{r}_0, \mathtt{c}_{*}>0$ 
such that, for  any choice of the tangential sites 
$S^+\in \mathcal{V}(\mathtt{r})$ with $0<\mathtt{r}\leq \mathtt{r}_0 $ (see Definition \ref{Def:cono}),
 one has
$\lvert \det \mathbb{A} \rvert\geq \mathtt{c}_* \,\,
  \overline{\jmath}_1^{3 \nu}\,. $
%\begin{equation}\label{GenericAssumption2}
%\lvert \det \mathbb{A} \rvert\geq \mathtt{c}_* \,\,
%  \overline{\jmath}_1^{3 \nu}\,, 
%\end{equation}

\end{lem}

\begin{proof}
The proof is postponed in Appendix \ref{PrelimEst}.
\end{proof}

\noindent
As a consequence of the
 non-degeneracy condition 
  in Lemma \ref{Twist1}
the map in \eqref{FreqAmplMapDP} is invertible and we denote
\begin{equation}\label{xiomega}
\xi:=\xi(\omega) := \alpha^{(-1)}(\omega) = \e^{-2}\mathbb A^{-1}(\omega-\bar\omega) +O( \e^2)\,.
\end{equation}

\section{The nonlinear functional setting}\label{sezioneNonlinearFunct}

We write the Hamiltonian in \eqref{HamiltonianaRiscalataDP}
(possibly eliminating  constant terms depending only on $\x$ 
which are irrelevant for the dynamics) as
\begin{equation}\label{HepsilonDP}
\begin{aligned}
H_{\varepsilon}&=\mathcal{N}+P\,, \\
%\qquad
\mathcal{N}(\theta, y, z)
&=\omega\cdot y+\frac{1}{2} (N(\theta) z, z)_{L^2}\,,
\qquad 
\frac{1}{2}(N(\theta) z, z)_{L^2}:=
\frac{1}{2}((\td_z \nabla H_{\varepsilon}) (\theta, 0, 0)[z], z)_{L^2}\,,
\end{aligned}
\end{equation}
where $\mathcal{N}$ describes the 
linear dynamics normal to the torus, 
and $P:=H_{\varepsilon}-\mathcal{N}$ collects the nonlinear perturbative effects.
Note that both $N$ and $P$ depend on $\omega$ through the map $\omega\mapsto \xi(\omega)$.\\
We consider $H_{\varepsilon}$ as a $(\omega, \varepsilon)$-parameter family of Hamiltonians and we note that, for $P=0$, $H_{\varepsilon}$ possess an invariant torus at the origin with frequency $\omega$, which we want to continue to an invariant torus for the full system.

\noindent
We will select the frequency parameters from the 
following set (recall \eqref{xiomega})
\begin{equation}\label{OmegaEpsilonDP}
\Omega_{\varepsilon}:=\{\omega\in \R^\nu\,: \,\, \xi(\omega)\in [1, 2]^{\nu}\}\,.
\end{equation}
Setting (see \eqref{AepsilonDP})
\begin{equation}\label{gammaDP}
\gamma=\varepsilon^{2b}\,, \quad \tau:=2\nu+6, %\quad \mbox{for\,\,some}\,\,a>0\,
\end{equation}
 we define the non-resonant sets
 \begin{align}
\mathcal{G}^{(0)}_0&:=
\big\{ \omega \in \Omega_{\varepsilon} : 
\lvert \omega\cdot \ell \rvert\geq \gamma\,\langle \ell \rangle^{-\tau}, 
\,\, \forall \ell \in\mathbb{Z}^{\nu}\setminus\{0\}\big\}\,,\label{ZEROMEL}\\
\mathcal{G}_0^{(1)}&:=
\Big\{ \omega \in \Omega_{\varepsilon} : 
\lvert \overline{\omega}\cdot \ell +\varepsilon^2 \mathbb{A}\xi(\omega)\cdot \ell + \og(j')-\og(j)+\varepsilon^2 (\og(j')\lal_{j'}-\og(j)\lal_j)\rvert> C\gamma,\label{divisoriLBNF3}\\
&\qquad\qquad\sum_{i=1}^{\nu} \overline{\jmath}_i \ell_i+j'-j=0,\,\, \forall \lvert \ell \rvert\le 3,\,\, \ell\in\mathbb{Z}^{\nu}\setminus\{0\}\,\,j, j'\in S^c, \,\,\, (\ell, j, j')\neq (0, j, j)\Big\}\,,\nonumber
\end{align}
for some constant $C$ depending on $S$,
where $\mathbb{A}$ is defined in \eqref{TwistMatrixDP}
and
\begin{equation}\label{lambdaJ0}
\lal_j:=\frac{2}{3} \sum_{j_2\in S^+}  \,
\frac{(1+j_2^2)(1+j^2)(2+j_2^2+j^2)}{(3+j_2^2-j_2 j+j^2)
(3+j_2^2+j_2 j +j^2)}\xi_{j_2}(\omega)\,.
\end{equation}
%For compactness, in the following,  we shall denote by $\xi$ the function $\xi(\oo)$ defined in \eqref{xiomega}.
We require that
\begin{equation}\label{0di0Melnikov}
\omega\in\mathcal{G}_0:=\mathcal{G}^{(0)}_0\cap \mathcal{G}^{(1)}_0\,.
\end{equation} 
\begin{lem}\label{measG0}
We have that $\lvert \Omega_{\varepsilon}\setminus\mathcal{G}_0 \rvert\leq C_{*}\varepsilon^{2(\nu-1)} \gamma\,$ for some $C_*=C_*(S)>0\,$.
\end{lem}
\begin{proof}
The proof is postponed in Appendix \ref{quadro}.
%\ref{PrelimEst}.
\end{proof}

\begin{remark}
The diophantine condition $\omega\in \mathcal{G}_0^{(0)}$ is typical of KAM scheme.
The lower bound in $\mathcal{G}_0^{(1)}$
involves resonances of order five with two normal modes. As explained in the introduction,
in order to impose such lower bounds we need to take into account also the
corrections of order $\e^{2}$. The matrix $\mathbb{A}$ comes from the weak BNF
of section \ref{SezioneWBNF}. The terms $\lal_{j}$ come from the \emph{linear} BNF
procedure
of subsection \ref{LinearBNF}. In particular they are evaluated explicitly using the
identification argument of  Theorem \ref{PartialWeak}.
\end{remark}
%In Section \ref{sezioneStimeMisura}, by using \eqref{gammaDP}, 
%
%we shall show that $\calG_{0}$ has positive measure.
\begin{remark}
Note that the definition of $\gamma$ in \eqref{gammaDP} 
is slightly stronger than the minimal condition
for which is possible to prove that $\mathcal{G}_0^{(0)}$ has large measure, 
namely $\gamma\le c\,\varepsilon^2$, with $c>0$ small enough. 
Our choice turns out to be useful for proving that the 
Cantor set of frequencies of the expected 
quasi-periodic solutions has asymptotically 
full measure (as $\varepsilon\to 0$).
\end{remark}
We look for an embedded invariant torus
\begin{equation}\label{iDP}
i \colon \mathbb{T}^{\nu}\to 
\mathbb{T}^{\nu}\times \mathbb{R}^{\nu}\times H_S^{\perp}\,, 
\quad \varphi \mapsto i(\varphi):=(\theta(\varphi), y(\varphi), z(\varphi))
\end{equation}
of the Hamiltonian vector field $X_{H_{\varepsilon}}$ (see \eqref{HepsilonDP}) supporting quasi-periodic solutions with diophantine frequency $\omega\in \mathcal{G}_0$.

For technical reason, it is useful to consider the modified Hamiltonian
\begin{equation}\label{HepsilonZetaDP}
H_{\varepsilon, \zeta}(\theta, y, z):=H_{\varepsilon}(\theta, y, z)
+\zeta\cdot\theta, \quad \zeta\in\mathbb{R}^{\nu}\,.
\end{equation}
More precisely, we introduce $\zeta$ in order to control the average in the $y$-component in our Nash Moser scheme.
The vector $\zeta$ has no dynamical consequences since an invariant torus for the Hamiltonian vector field $X_{H_{\varepsilon, \zeta}}$ is actually invariant for $X_{H_{\varepsilon}}$ itself.

Thus, we look for zeros of the nonlinear operator 
$\mathcal{F}(i, \zeta)\equiv \mathcal{F}(i, \zeta, \omega, \varepsilon):=
\omega\cdot\partial_{\varphi} i(\varphi)-X_{\mathcal{N}}(i(\varphi))-X_P(i(\varphi))+(0, \zeta, 0)$ 
defined as
\begin{align}\label{NonlinearFunctionalDP}
\mathcal{F}(i, \zeta)
%&:=\mathcal{F}(i, \zeta, \omega, \varepsilon):=\omega\cdot\partial_{\varphi} i(\varphi)-X_{\mathcal{N}}(i(\varphi))-X_P(i(\varphi))+(0, \zeta, 0)\\[3mm]
%&:=\begin{pmatrix}
%\omega\cdot\partial_{\varphi} \theta(\varphi)-\partial_y H_{\varepsilon}(i(\varphi))\\
%\omega\cdot\partial_{\varphi} y(\varphi)+\partial_{\theta} H_{\varepsilon}(i(\varphi))+\zeta\\
%\omega\cdot\partial_{\varphi} z(\varphi)-J \nabla_z H_{\varepsilon}(i(\varphi))
%\end{pmatrix} 
=\begin{pmatrix}
\omega\cdot\partial_{\varphi} \Theta(\varphi)-\partial_y P(i(\varphi))\\
\omega\cdot\partial_{\varphi} y(\varphi)+\frac{1}{2} \partial_{\theta} (N(\theta(\varphi))z(\varphi))_{L^2(\mathbb{T})}+\partial_{\theta} P(i(\varphi))+\zeta\\
\omega\cdot\partial_{\varphi} z(\varphi)-J N(\theta(\varphi))\,z(\varphi)-J\nabla_z P(i(\varphi))
\end{pmatrix}  
\end{align}
where $\Theta(\varphi):=\theta(\varphi)-\varphi$ is $(2\pi)^{\nu}$-periodic. 
We define the Sobolev norm of the periodic component of the embedded torus
\begin{equation}\label{frakkiIDP}
\mathfrak{I}(\varphi):=i(\varphi)-(\varphi, 0, 0):=(\Theta(\varphi), y(\varphi), z(\varphi)) \;\quad \lVert \mathfrak{I} \rVert_{s}:=\lVert \Theta \rVert_{s}+\lVert y \rVert_{s}+\lVert z \rVert_s,
\end{equation}
where $ z \in H^s_{S^{\perp}}:= H^s \cap H_{S}^\perp$ (recall \eqref{decomposition}) with norm defined in \eqref{space} and  with abuse of notation, we are denoting by $\|\cdot\|_{s}$ the Sobolev norms of functions in 
$H^{s}(\T^{\nu},\R^{\nu})$. 
From now on we fix $s_0:=[\nu/2]+4$.

Notice that in the coordinates \eqref{AepsilonDP}, 
 a quasi-periodic solution corresponds to an embedded invariant torus \eqref{iDP}. Therefore we can reformulate the main Theorem \ref{MainResult} as follows.
\begin{teor}\label{IlTeoremaDP}
There exists a small constant $\mathtt{r}>0$ such that, for 
any $S^{+}\in \mathcal{V}(\mathtt{r})$
(see \eqref{TangentialSitesDP} and Definition \ref{Def:cono}), 
there exists $\varepsilon_0>0$, small enough, such that the following holds.
For all $\varepsilon\in (0, \varepsilon_0)$ there exist positive constants $C=C(\nu)$, $\mu=\mu(\nu)$ and a Cantor-like set $\mathcal{C}_{\varepsilon}\subseteq\Omega_{\varepsilon}$ (see \eqref{OmegaEpsilonDP}), with asymptotically full measure as $\varepsilon\to 0$, namely
\begin{equation}\label{frazionemisureDP}
\lim_{\varepsilon \to 0}
\dfrac{\lvert \mathcal{C}_{\varepsilon} \rvert}{\lvert \Omega_{\varepsilon} \rvert}=1\,,
\end{equation}
such that, for all $\omega\in\mathcal{C}_{\varepsilon}$, there exists a solution 
$i_{\infty}(\varphi):=i_{\infty}(\omega, \varepsilon)(\varphi)$ 
of the equation $\mathcal{F}(i_{\infty}, 0, \omega, \varepsilon)=0$ 
(see \eqref{NonlinearFunctionalDP}). 
Hence the embedded torus $\varphi\mapsto i_{\infty}(\varphi)$ 
is invariant for the Hamiltonian vector field 
$X_{H_{\varepsilon}}$, and it is filled by quasi-periodic 
solutions with frequency $\omega$. The torus $i_{\infty}$ satisfies
\begin{equation*}
\lVert i_{\infty}(\varphi)-(\varphi, 0, 0) 
\rVert_{s_0+\mu}^{\gamma, \mathcal{C}_{\varepsilon}}\le 
C\,\varepsilon^{9-2 b}\,\gamma^{-1}\,.
\end{equation*}
Moreover the torus $i_{\infty}$ is linearly stable.
\end{teor}
\noindent
We can deduce Theorem \ref{MainResult} from Theorem \ref{IlTeoremaDP}, indeed
the quasi-periodic solution $u$ in \eqref{SoluzioneEsplicitaDP} is 
\[
u(t, x)=\Big(\Phi_B\circ A_{\varepsilon}\Big)  i_{\infty} (\omega t)
\]
for $\omega=\omega(\xi)\in \mathcal{C}_{\e}$, where $\omega(\xi)$ is the frequency amplitude map \eqref{FreqAmplMapDP}.\\
The rest of the paper is devoted to the proof of Theorem \ref{IlTeoremaDP}.

\subsection{Tame estimates of the nonlinear vector field}

We give tame estimates for the composition operator induced by the Hamiltonian vector fields $X_{\mathcal{N}}$ and $X_{P}$ in \eqref{NonlinearFunctionalDP}.
Since the functions $y\to \sqrt{\xi+\varepsilon^{2(b-1)} y},
 \theta\to e^{\mathrm{i}\,\theta}$ 
 are analytic for $\varepsilon$ small enough and $\lvert y \rvert\le C$, 
 classical  composition results 
 (see for instance Lemma $6.2$ in \cite{Airy}) 
 imply that, 
 for all $\lVert \mathfrak{I} \rVert_{s_0}^{\gamma, \calO}\le 1$,
\begin{equation*}
\lVert A_{\varepsilon}(\theta(\varphi), y(\varphi), z(\varphi)) 
\rVert_s^{\gamma, \calO}\lesssim_s 
\varepsilon (1+\lVert \mathfrak{I} 
\rVert_s^{\gamma, \calO})\,.
\end{equation*}
In the following lemma we collect tame 
estimates for the Hamiltonian vector fields $X_{\mathcal{N}}, X_{P}, X_{H_{\varepsilon}}$, see \eqref{HepsilonDP}. These bounds rely on tame estimates for composition operators 
and their proof is  completely analogous to the one in Section $5$ of \cite{KdVAut}.
\begin{lem}\label{TameEstimatesforVectorfieldsDP}
Let $\mathfrak{I}(\varphi)$ in \eqref{frakkiIDP} satisfy $\lVert \mathfrak{I} \rVert_{s_0+1}^{\gamma, \calO}\lesssim\,\varepsilon^{9-2 b}\gamma^{-1}$. Then we have
\begin{equation}\label{EstimatesVecFieldDP}
\begin{aligned}
\lVert \partial_y P(i) \rVert_s^{\gamma, \calO}&\lesssim_s \varepsilon^7+\varepsilon^{2 b} \lVert \mathfrak{I} \rVert_{s+1}^{\gamma, \calO}\,, \qquad  \qquad \quad
\lVert \partial_{\theta} P(i) \rVert_s^{\gamma, \calO}\lesssim_s 
\varepsilon^{9-2 b}(1+\lVert \mathfrak{I} 
\rVert_{s+1}^{\gamma, \calO})\,,
\\
\lVert \nabla_z P(i) \rVert_s^{\gamma, \calO}&\lesssim_s 
\varepsilon^{8-b}+\varepsilon^{9-b}\gamma^{-1} 
\lVert \mathfrak{I} \rVert_{s+1}^{\gamma, \calO}\,, 
\qquad
\lVert X_P (i) \rVert_s^{\gamma, \calO}\lesssim_s \varepsilon^{9-2 b}+\varepsilon^{2 b} 
\lVert \mathfrak{I} \rVert_{s+1}^{\gamma, \calO}\,, \\
\lVert \partial_{\theta}\partial_y P(i) \rVert_s^{\gamma, \calO}
&\lesssim_s \varepsilon^7+\varepsilon^8\gamma^{-1} 
\lVert \mathfrak{I} \rVert_{s+1}^{\gamma, \calO}\,, 
\qquad \; \lVert \partial_y\nabla_z P(i) \rVert_s^{\gamma, \calO}
\lesssim_s \varepsilon^{6+b}+\varepsilon^{2 b-1}
\lVert \mathfrak{I} \rVert_{s+1}^{\gamma, \calO}\,,\\
\lVert \partial_{y y} P(i)-\frac{\varepsilon^{2 b}}{2} \mathbb{A}\Omega \rVert_s^{\gamma, \calO}
&\lesssim_s \varepsilon^{5+2 b}+\varepsilon^{6+2 b}\gamma^{-1} \lVert \mathfrak{I} \rVert_{s+1}^{\gamma, \calO}\,,
\end{aligned}
\end{equation}
and for all $\hat{\imath}:=(\hat{\Theta}, \hat{y}, \hat{z})$,
\begin{align*}
\lVert \partial_y \td_i X_P(i)[\hat{\imath}] \rVert_s^{\gamma, \calO}
&\lesssim_s \varepsilon^{2 b-1} (\lVert \hat{\imath} \rVert_{s+1}^{\gamma, \calO}
+\lVert \mathfrak{I} \rVert_{s+1}^{\gamma, \calO}\lVert \hat{\imath} \rVert^{\g, \calO}_{s_0+1})\,,\\
\lVert \td_i X_{H_{\varepsilon}}(i)[\hat{\imath}]+(0, 0, J\, \hat{z})\rVert_s^{\gamma, \calO}
&\lesssim_s \varepsilon (\lVert \hat{\imath} \rVert_{s+1}^{\gamma, \calO}
+\lVert \mathfrak{I} \rVert_{s+1}^{\gamma, \calO}\lVert \hat{\imath} \rVert^{\g, \calO}_{s_0+1})\,,\\
\lVert \td_i^2 X_{H_{\varepsilon}} (i) [\hat{\imath}, \hat{\imath}]\rVert_s^{\gamma, \calO}
&\lesssim_s \varepsilon (\lVert \hat{\imath} \rVert_{s+1}^{\gamma, \calO}\lVert \hat{\imath} \rVert_{s_0+1}^{\gamma, \calO}
+\lVert \mathfrak{I} \rVert_{s+1}^{\gamma, \calO}(\lVert \hat{\imath} \rVert^{\g, \calO}_{s_0+1})^2)\,.
\end{align*}
\end{lem}

\noindent
In the sequel we will use that, 
by the diophantine condition \eqref{0di0Melnikov}, 
the operator $(\omega\cdot\del_{\f})^{-1}$ 
 is defined for all functions $u$ with zero $\varphi$-average, 
 and satisfies
\begin{equation*}\label{stimaDomegaDP}
\lVert (\omega\cdot\partial_{\varphi})^{-1} u \rVert_s
\lesssim_s \gamma^{-1}\,\lVert u \rVert_{s+\tau}, 
\quad \lVert (\omega\cdot\partial_{\varphi})^{-1} u \rVert_s^{\gamma, \calO}\lesssim_s \gamma^{-1} \lVert u \rVert^{\gamma, \calO}_{s+2\tau+1}\,.
\end{equation*}

\section{Approximate inverse}\label{sezione6DP}

We want to solve the nonlinear functional equation (see \eqref{NonlinearFunctionalDP})
\begin{equation}\label{EquazioneFunzionaleDP}
\mathcal{F}(i, \zeta)=0
\end{equation}
by applying a Nash-Moser scheme. It is well known that the main issue in implementing this algorithm concerns the approximate inversion of the linearized operator of $\mathcal{F}$ at any approximate solution $(i_n, \zeta_n)$, namely $\td \mathcal{F}(i_n, \zeta_n)$. Note that $\td \mathcal{F}(i_n, \zeta_n)$ is independent  of $\zeta_n$.
%Zehnder noted in \cite{Z} that actually it is sufficient to find an approximate inverse of this operator.\\
One of the main problem is that the $(\theta, y, z)$-components of $\td\mathcal{F}(i_n, \zeta_n)$ are coupled and then the linear system 
\begin{equation}\label{coupled}
\td \mathcal{F}(i_n, \zeta_n)[\hat\imath, \hat{\zeta}]=\omega\cdot \partial_{\varphi} \hat\imath-\td_{i} X_{H_\varepsilon} (i_n)[\hat\imath]-(0, \hat{\zeta}, 0)=g=(g^{(\theta)}, g^{(y)}, g^{(z)})
\end{equation}
is quite involved.
In order to approximately solve \eqref{coupled} we follow the scheme developed by Berti-Bolle in \cite{BertiBolle} which describe a way to approximately triangularize \eqref{coupled}.  This method has been applied in \cite{KdVAut}, \cite{Giuliani}. Since the strategy is identical to \cite{Giuliani} we only summarize it and underline the differences which mainly come from the symplectic structure. 
For a fully detailed expository presentation see \cite{GiulianiPhD}.
\medskip

We now study the solvability of equation \eqref{coupled} at an approximate solution, which we denote by $(i_0, \zeta_0)$, $i_0(\varphi)=(\theta_0(\varphi), y_0(\varphi), z_0(\varphi))$ in order to keep the notations of \cite{KdVAut},  \cite{Giuliani} . Assume the following hypothesis, which we shall verify at any step of the Nash-Moser iteration,
\begin{itemize}
\item \textbf{Assumption}. The map $\omega \mapsto i_0(\omega)$ is a Lipschitz function defined on some subset $\calO_0\subseteq \mathcal{G}_0 \subseteq\Omega_{\varepsilon}$ (recall \eqref{0di0Melnikov},\eqref{OmegaEpsilonDP}) and, for some $\mathfrak{p}_0:=\mathfrak{p}_0( \nu)>0$,
\begin{equation}\label{AssumptionDP}
\lVert \mathfrak{I}_0 \rVert_{s_0+\mathfrak{p}_0}^{\gamma, \calO_0}\le \varepsilon^{9-2b} \gamma^{-1}, \quad \lVert Z \rVert_{s_0+\mathfrak{p}_0}^{\gamma, \calO_0}\le \varepsilon^{9-2 b}, \quad \gamma=\varepsilon^{2b}, %\quad a\ll 1,
\end{equation}
where $\mathfrak{I}_0(\varphi):=i_0(\varphi)-(\varphi, 0, 0)$ and $Z$ is the error function
\begin{equation}
Z(\varphi):=(Z_1, Z_2, Z_3)(\varphi):=\mathcal{F}(i_0, \zeta_0)(\varphi)=\omega\cdot \partial_{\varphi} i_0 (\varphi)-X_{H_{\varepsilon, \zeta_0}}(i_0(\varphi)).
\end{equation}
\end{itemize}
By estimating the Sobolev norm of the function $Z$ we can measure how the embedding $i_0$ is close to being invariant for $X_{H_{\varepsilon, \zeta_0}}$. If $Z=0$ then $i_0$ is a  solution. In general we say that $i_0$ is "approximately invariant" up to order $O(Z)$.
We observe that by Lemma $6.1$ in \cite{KdVAut} we have that if $i_0$ is a solution, then the parameter $\zeta_0$ has to be naught, hence the embedded torus $i_0$ supports a quasi-periodic solution of the ``original'' system with Hamiltonian $H_{\varepsilon}$ (see \eqref{HepsilonDP}). 

\smallskip

By \cite{BertiBolle} we know that it is possible to construct an embedded torus $i_{\delta}(\varphi)=(\theta_0(\varphi), y_{\delta}(\varphi), z_0(\varphi))$, which differs from $i_0$ only for a small modification of the $y$-component, such that the $2$-form $\mathcal{W}$ (recall  \eqref{NewSimplFormDP}) vanishes on the torus $i_{\delta}(\T^{\nu})$, namely $i_{\delta}$ is isotropic. In particular $i_{\delta}(\varphi)$ is approximately invariant up to order $O(Z)$ (see Lemma $7$ in \cite{BertiBolle}) and, more precisely, there exists $\tilde{\mathfrak{p}}:=\tilde{\mathfrak{p}}(\nu)>0$ such that 
\begin{equation}\label{chegioia2}
\lVert i_{\delta}-i_0 \rVert^{\g, \calO_0}_s\lesssim_s \lVert \mathfrak{I}_0 \rVert^{\g, \calO_0}_{s+\tilde{\mathfrak{p}}}.
\end{equation}
 The strategy is to construct an approximate inverse for $\td \mathcal{F}(i_0, \zeta_0)$ by starting from an approximate inverse for the linear operator $\td \mathcal{F}(i_\delta, \zeta_0)$. The advantage of analyzing the linearized problem at $i_{\delta}$ is that it is possible to construct a \emph{symplectic} change of variable which approximately triangularizes the linear system thanks to the isotropicity of $i_{\delta}$. 
 For   the details we refer to \cite{BertiBolle} and \cite{KdVAut}, here we  only give the relevant definitions and state the main result.
We define the change of coordinates
\begin{equation}\label{GdeltaDP}
\begin{pmatrix}
\theta \\ y \\ z
\end{pmatrix}
:=G_{\delta}\begin{pmatrix}
\varphi\\ \eta \\ w
\end{pmatrix}
:=\begin{pmatrix}
\theta_0(\varphi) \\
y_{\delta}(\varphi)+[\partial_{\varphi}\theta_0(\varphi)]^{-T}\eta+[(\partial_{\theta} \tilde{z}_0)(\theta_0(\varphi))]^T\,J^{-1}w\\
z_0(\varphi)+w
\end{pmatrix}
\end{equation}
where $\tilde{z}_0:=z_0 (\theta_0^{-1} (\theta))$.  % then $i_{\delta}(\varphi)$ is the trivial embedded torus $(\varphi, 0, 0)$.
We denote the transformed Hamiltonian by $K:=K(\varphi, \eta, w, \zeta_0)$.  
We then define
\begin{equation}\label{eq3}
\mathcal{L}_{\omega}:=\omega\cdot\partial_{\varphi}-J K_{02}(\varphi)\,,
\end{equation}
where $K_{02}$ is the linear operator representing the terms quadratic in $w$ of $K$, i.e.
\begin{equation}\label{K020}
 \frac12 (K_{02}(\varphi)[w],w):= \Pi^{d_w=2}K = \Pi^{d_w=2} H_\e \circ G_\delta\,. 
 %=  
%\frac12((\td_{z} \nabla \mathcal{H})(T_\delta)[w], w)_{L^2(\mathbb{T})}+(R(\varphi) w, w)_{L^2(\mathbb{T})}
\end{equation}
$\calL_{\omega}$  corresponds to the $w$-component of the linearized operator after the change of variable $G_\delta$.
%where 
%$R(\varphi)$ has the ``finite dimensional'' form
%\begin{equation}\label{FiniteDimFormDP}
%R(\varphi) w=\sum_{\lvert j \rvert\le C} \int_0^1 (w, g_j(\tau, \varphi))_{L^2(\mathbb{T})}\,\chi_j(\tau, \varphi)\,d\tau.
%\end{equation}

\medskip

%In Sections \ref{regularization} and \ref{SezioneDiagonalization} we will study \eqref{eq3} and we will prove the following claim (see Theorem \ref{InversionLomegaDP}).

%
%By the above result we are able to prove the following. 
%\begin{prop}
%Assume \eqref{AssumptionDP} and \eqref{InversionAssumptionDP}. Then, for all $\omega\in \Omega_{\infty}$, for all $g:=(g_1, g_2, g_3)$, the system \eqref{Ddp} has a solution $\mathbb{D}^{-1} g:=(\hat{\psi}, \hat{\eta}, \hat{w}, \hat{\zeta})$ where $(\hat{\psi}, \hat{\eta}, \hat{w}, \hat{\zeta})$ are defined in \eqref{psiDP}, \eqref{etaDP}, \eqref{wDP}, \eqref{ValoreperZetaDP}. Moreover, we have
%\begin{equation}\label{EstimateonDdp}
%\lVert \mathbb{D}^{-1} g \rVert_s^{\gamma, \Omega_{\infty}}\lesssim_s \gamma^{-1} (\lVert g \rVert_{s+\s}^{\gamma, \Omega_{\infty}}+\varepsilon \gamma^{-5/2}\{ \lVert \mathfrak{I}_0 \rVert_{s+\s}^{\gamma, \calO_0}+\gamma^{-1}\lVert \mathfrak{I}_0 \rVert_{s_0+\s}^{\gamma, \calO_0}\lVert \mathcal{F}(i_0, \zeta_0)\rVert_{s+\s}^{\gamma, \calO_0}\}\lVert g \rVert_{s_0+\s}^{\gamma, \Omega_{\infty}})
%\end{equation}
%for some $\s:=\s(\tau, \nu)>0$.
%\end{prop}

%\begin{equation}\label{ApproxInverseDP}
%\mathbf{T}_0:=(\td \tilde{G}_{\delta})(\varphi, 0, 0)\circ\mathbb{D}^{-1}\circ (\td G_{\delta}(\varphi, 0, 0))^{-1}
%\end{equation}
%is an approximate right inverse of $\td \mathcal{F}(i_0)$.

 In \cite{BertiBolle} (see also \cite{KdVAut},\cite{Giuliani}) the following result is proved.
 
%  Denote $\lVert (\psi, \eta, w, \zeta)\rVert_s^{\gamma, \calO}:=\max \{ \lVert (\psi, \eta, w) \rVert_s^{\g, \calO}, \lvert \zeta \rvert^{\gamma, \calO} \}$.
\begin{teor}\label{TeoApproxInvDP}
Assume \eqref{AssumptionDP} and the following %inversion assumption \eqref{InversionAssumptionDP}.
%\begin{itemize}
\\
\textbf{Inversion Assumption:} 
There exist $\gotp_1:=\gotp_1(\nu)>0$ 
and a set 
$\Omega_{\infty}\subset \mathcal{G}_0\subseteq \Omega_{\varepsilon}$ 
such that for all $\omega\in \Omega_{\infty}$ 
and every function $g\in H^{s+2\tau+1} \cap H_{S}^\perp$, 
there exists a solution $h:=\mathcal{L}_{\omega}^{-1} g$ 
of the linear equation $\mathcal{L}_{\omega} h=g$ 
which satisfies
\begin{equation}\label{InversionAssumptionDP}
\lVert \mathcal{L}_{\omega}^{-1} g \rVert_s^{\gamma, \Omega_{\infty}}
\lesssim_s
\gamma^{-1} (\lVert g \rVert_{s+2\tau+1}^{\gamma, \Omega_{\infty}}
+\varepsilon \gamma^{-5/2}
\lVert \mathfrak{I}_\delta \rVert_{s+\gotp_1}^{\gamma, \calO_0}
% \{ \lVert \mathfrak{I}_0 \rVert_{s+\gotp_2}^{\gamma, \calO_0}+\gamma^{-1}
%\lVert \mathfrak{I}_0 \rVert_{s_0+\gotp_2}^{\gamma, %\calO_0} 
%\lVert Z \rVert_{s+\gotp_2}^{\gamma, \calO_0}\}
\lVert g \rVert_{s_0}^{\gamma, \Omega_{\infty}})\,.
\end{equation}	
Then there exists $\mu:=\mu( \nu)$  such that, 
for all $\omega\in \Omega_{\infty}$ there exists 
a linear  operator $\mathbf{T}_0$ such that:

\noindent
1.  for all $g:=(g^{(\theta)}, g^{(y)}, g^{(z)})$, one has
\begin{equation}\label{TameEstimateApproxInvDP}
\lVert \mathbf{T}_0 g \rVert_s^{\gamma, \Omega_{\infty}}\lesssim_s\gamma^{-1}(\lVert g \rVert_{s+\mu}^{\gamma, \Omega_{\infty}}
+\varepsilon\gamma^{-5/2} 
\lVert \mathfrak{I}_0 \rVert_{s+\mu}^{\gamma, \calO_0}
 \lVert g \rVert_{s_0+\mu}^{\gamma, \Omega_{\infty}})\,.
\end{equation}
\noindent
2. $\mathbf{T}_0$  is an approximate inverse of $\td \mathcal{F}(i_0)$, 
namely
\begin{equation}\label{6.41DP}
\begin{aligned}
&\lVert (\td \mathcal{F}(i_0)\circ \mathbf{T}_0-\mathrm{I}) 
g \rVert_s^{\gamma, \Omega_{\infty}}
\lesssim_s \varepsilon^{2 b-1}\g^{-2} 
\Big( \lVert \mathcal{F}(i_0, \zeta_0) \rVert_{s_0+\mu}^{\gamma, \calO_0}
\lVert g \rVert_{s+\mu}^{\gamma, \Omega_{\infty}}\\
&\qquad\qquad\qquad 
+\{\lVert \mathcal{F}(i_0, \zeta_0)\rVert_{s+\mu}^{\gamma, \calO_0}
+\varepsilon \gamma^{-5/2}
\lVert \mathcal{F}(i_0, \zeta_0) \rVert_{s_0+\mu}^{\gamma, \calO_0}
\lVert \mathfrak{I}_0 \rVert_{s+\mu}^{\gamma, \calO_0}\} 
\lVert g \rVert_{s_0+\mu}^{\gamma, \Omega_{\infty}} \Big)\,.
\end{aligned}
\end{equation}
\end{teor}

\subsection{The linearized operator in the normal directions}

Recalling the assumption \eqref{AssumptionDP}, 
in the sequel we assume that 
$\mathfrak{I}_{\delta}:=\mathfrak{I}_{\delta}(\varphi; \omega)
=i_{\delta}(\varphi;\, \omega)-(\varphi,\, 0,\, 0)$ 
satisfies, for some $\gotp_1>0$,
\begin{equation}\label{IpotesiPiccolezzaIdeltaDP}
\lVert \mathfrak{I}_{\delta} \rVert_{s_0+\gotp_1}^{\g,\calO_0}
\lesssim\,\varepsilon^{9-2 b}\gamma^{-1}\,.
\end{equation}
We note moreover that $G_\delta$ in \eqref{GdeltaDP} 
is the identity plus a translation plus a finite rank linear 
operator; moreover, assuming 
\eqref{AssumptionDP}, one has that $G_{\delta}$ is $O(\varepsilon^{9-2 b}\gamma^{-1})$-close to the identity in low norm.
Returning to the initial variables 
we set  (see \eqref{AepsilonDP},\eqref{GdeltaDP})
\begin{equation}\label{TdeltaDP}
T_{\delta}:=A_{\varepsilon}(G_{\delta}(\varphi, 0, 0))
= \e v_\delta +\e^b z_0\,,
\quad v_\delta = \sum_{j\in S} \sqrt{\xi_j + \e^{2b-2}|\lambda(j)|y_{\delta j}(\varphi)} 
e^{\mathrm{i} (j  x + \theta_{0 j}(\varphi))  }
\end{equation}
and we have, for some $\sigma:=\sigma(\nu)>0$,
\begin{equation}\label{balaban}
\lVert \Phi_B(T_{\delta}) \rVert_s^{\gamma, \calO_0}
\lesssim_s \varepsilon\, (1+\lVert \mathfrak{I}_{\delta} 
\rVert^{\gamma, \calO_0}_{s+\sigma})\,, 
\qquad
\lVert \td_i \Phi_B(T_{\delta}) [\hat{\imath}] \rVert_s
\lesssim_s \varepsilon(\lVert i\rVert_{s+\sigma}
+ \lVert \mathfrak{I}_{\delta} \rVert_{s+\sigma}\lVert i \rVert_{s_0+\sigma})\,.
\end{equation}
By following Section $7$ in \cite{KdVAut} (see Lemma $7.1$),  
$K_{02}$ in \eqref{K020} has rather explicit estimates.

\begin{prop}\label{LinearizzatoIniziale}
Assume \eqref{IpotesiPiccolezzaIdeltaDP}. Then there exists $\sigma_0=\sigma_0(\nu)>0$ such that the following holds. The Hamiltonian operator $\mathcal{L}_{\omega}$ in \eqref{eq3}
has the form
\begin{equation}\label{BoraMaledetta}
\mathcal{L}_{\omega}=\Pi_S^{\perp}\big(\omega\cdot\partial_{\varphi}-J\circ(1+a_0(\varphi, x))+\mathcal{Q}_0\big), \qquad 
a_0(\varphi, x):=-(\Phi_B(T_{\delta})+\partial_{u}^{2}f(\Phi_{B}(T_{\delta})))\,.
\end{equation}
Recall that $T_{\delta}$ is defined in \eqref{TdeltaDP}, $\Phi_B$ is the Birkhoff map given in Proposition \ref{WBNFdp},  $f$ is the Hamiltonian density in \eqref{HamiltonianDensity}.
The operator $\mathcal{Q}_0$ is finite rank and has the form
\begin{equation}\label{FiniteDimFormDP}
\calQ_0(\varphi) w=\sum_{\lvert j \rvert\le C} \int_0^1 (w, g_j(\tau, \varphi))_{L^2(\mathbb{T})}\,\chi_j(\tau, \varphi)\,d\tau.
\end{equation}
In particular we divide $\calQ_0= \sum_{i=1}^5 \e^i\RR_i +\RR_{>5}$,
%\begin{equation}\label{LomegaDP}
%\calQ_0= \sum_{i=1}^5 \e^i\RR_i +\RR_{>5},
%\end{equation}
 where the $\mathcal{R}_i$, $\mathcal{R}_{>5}$ are finite rank operators.
 Moreover
 we have 
 \begin{equation}\label{unasuaziaricca}
\lVert a_0 \rVert_{s}^{\gamma, \calO_0}\lesssim_s
\varepsilon(1+\lVert \mathfrak{I}_{\delta}\rVert^{\gamma, \calO_0}_{s+\sigma_0})\,, 
\qquad \lVert \td_i a_0 [\hat{\imath}] \rVert_s\lesssim_s 
\varepsilon (\lVert \hat{\imath} \rVert_{s+\sigma_0}
+\lVert \mathfrak{I}_{\delta}\rVert_{s+\sigma_0}\lVert \hat{\imath}\rVert_{s_0})\,.
\end{equation}
The remainders $\mathcal{R}_i$ do not depend on $\mathfrak{I}_\delta$ and satisfy
\begin{equation}\label{DecayR2dp}
\lVert g^{(i)}_j \rVert_s^{\gamma, \calO_0}
+\lVert \chi_j^{(i)} \rVert_s^{\gamma, \calO_0}\lesssim_s 1\,,
\end{equation}
while $\mathcal{R}_{>5}$ satisfies
\begin{align}
 \lVert g^{>5}_j \rVert_s^{\gamma, \calO_0}
 \lVert \chi_j^{>5} \rVert_{s_0}^{\gamma, \calO_0}&
 +\lVert g^{>5}_j \rVert_{s_0}^{\gamma, \calO_0}
 \lVert \chi_j^{>5} \rVert_{s}^{\gamma, \calO_0}\lesssim_s 
\varepsilon^6+\varepsilon^{2}\lVert \mathfrak{I}_{\delta} 
\rVert_{s+\sigma}^{\gamma, \calO_0}\,, \label{DecayR*dp}\\[2mm]
\lVert \td_i g^{>5}_j [\hat{\imath}] \rVert_s\lVert \chi_j^{>5} \rVert_{s_0}
+\lVert \td_i g^{>5}_j [\hat{\imath}] \rVert_{s_0}\lVert \chi_j^{>5} \rVert_{s}
&+\lVert g^{>5}_j \rVert_{s_0} \lVert \td_i \chi^*_j \rVert_{s}
+\lVert g^{>5}_j \rVert_{s} \lVert \td_i \chi^{>5}_j \rVert_{s_0} \label{DecayR*dp2}\\ 
&\qquad \qquad \qquad \qquad
\lesssim_s \varepsilon^{2} \lVert \hat{\imath} \rVert_{s+\sigma}
+\varepsilon^{b}\lVert \mathfrak{I}_{\delta} \rVert_{s+\sigma}
 \lVert \hat{\imath} \rVert_{s_0+\sigma}\,.\nonumber
\end{align}
Finally, recalling the Definition \ref{LipTameConstants},
 we have
\begin{align}
\mathfrak{M}^{\gamma}_{\mathcal{Q}_0}(0,s)
&\lesssim_s \varepsilon^2 
(1+\lVert \mathfrak{I}_{\delta}
\rVert^{\gamma, \calO_0}_{s+\sigma_0})\,,\label{dude}\\
\mathfrak{M}_{ \td_i \calQ_0 [\hat{\imath}] }(0,s)&\lesssim_s 
\varepsilon^2 \lVert \hat{\imath} \rVert_{s+\sigma_0}+\varepsilon^{2b-1} 
\lVert \mathfrak{I}_{\delta}\rVert_{s+\sigma_0}
\lVert \hat{\imath} \rVert_{s_0+\sigma_0}\,.\label{dude2}
\end{align}
\end{prop}

\begin{proof} 
The expression \eqref{BoraMaledetta} follows from the 
definition \eqref{K020}  by remarking that $G_\delta$ 
and the weak BNF transformation $\Phi_B$ is the identity plus a 
finite rank operator, while the action angle change 
of coordinates is a rescaling plus 
a finite rank operator (acting only on the $v$).  
Then, in applying the chain rule, we get  
\begin{equation}
\begin{aligned}
\td_{w} \nabla_w (H_\e\circ G_\delta ) &\stackrel{\eqref{HamiltonianaRiscalataDP}}{=}
 \e^{-2b} (\td_{z} \nabla_z (\calH\circ A_\e))\circ G_\delta + R_1 
=   (\td_{z} \nabla_z \calH)\circ A_\e\circ G_\delta + R_1   \\ 
&\stackrel{\eqref{piotta}}{=}
(\td_{z} \nabla_z (H\circ\Phi_{B}))\circ A_\e\circ G_\delta + R_1 
=  (\td_{z} \nabla_z H)\circ\Phi_{B}\circ A_\e\circ G_\delta
 + R_1 + R_2
\end{aligned}
\end{equation} where the finite rank part contain 
all the terms where a derivative falls 
on the change of variables. 
Then \eqref{BoraMaledetta} follows from 
the definition of $H$ in \eqref{DPHamiltonian}.
Regarding the estimates, \eqref{unasuaziaricca} 
follows from \eqref{balaban}; regarding the 
bounds \eqref{DecayR2dp}, \eqref{DecayR*dp}, 
we split the finite rank part $R_1+R_2$ as follows. 
The operator $R_1$ contains all terms arising form 
derivatives of $G_\delta$. By tame estimates 
on the map $G_{\delta}$ 
(see for instance Lemma $6.7$ in \cite{KdVAut}), 
it satisfies the bounds \eqref{DecayR*dp} and 
 we put it in $\RR_{>5}$. The finite rank term
$R_2$   comes from the Birkhoff map.  
This is an analytic map so we consider
 the Taylor expansion 
\[
\Phi_B(u)=u+\sum_{i=2}^5\Psi_i(u)+\Psi_{\geq 6}(u)\,,
\] 
where each  $\Psi_i(u)$
is homogeneous of degree $i$ in $u$, while   
$\Psi_{\geq 6}=O(u^6)$ and all map $H^1_0(\mathbb{T})$ in itself. 
We have to evaluate $\Phi_B$ and 
its derivatives (up to order two) at  
$u=T_{\delta}=\varepsilon v_{\delta}+\varepsilon^b z_0$.
We denote by $\overline{v}$ the function
\begin{equation}\label{vsegnatoDP}
\overline{v}(\varphi, x):=\sum_{j\in S} \sqrt{\xi_j} 
e^{\mathrm{i} (j x+\mathtt{l}(j)\cdot \varphi)}= A_\e(\f,0,0)
\end{equation}
where $\mathtt{l}(\bar\jmath_i)$ is the $i$-th vector of the canonical 
basis of $\mathbb{Z}^{\nu}$ and is such that 
$\mathtt{l}(-\bar{\jmath}_i)=-\mathtt{l}(\bar{\jmath}_i)$.
\noindent
We observe that\footnote{The function $\e\bar v$ 
represents a torus supporting a 
quasi-periodic motion which is invariant for the system 
\eqref{HepsilonDP} with $P=0$, namely it is 
the approximate solution from which we bifurcate. }
\begin{equation*}\label{differenzaV}
\lVert v_{\delta}-\overline{v}\rVert^{\gamma, \calO_0}_s
\lesssim \lVert \mathfrak{I}_{\delta}\rVert^{\gamma, \calO_0}_{s}\,,
\end{equation*}
and hence we can expand 
\begin{equation}\label{PhiBdp}
\Phi_B(T_{\delta})=\e \bar v+\sum_{i=2}^5\e^i\Psi_i( \bar v)
+\tilde{q}= \Phi_B^{\le 5} + \tilde q \,,
\end{equation}
where $\tilde{q}$ is a remainder which satisfies
\begin{equation}\label{QtildaDP}
\lVert \tilde{q} \rVert_s^{\gamma, \calO_0}\lesssim_s 
\varepsilon^6+\varepsilon\lVert \mathfrak{I}_{\delta} \rVert_s^{\gamma, \calO_0}\,, 
\quad \lVert \td_i \tilde{q} [\hat{\imath}] \rVert_s
\lesssim_s \varepsilon (\lVert \hat{\imath} 
\rVert_s+\lVert \mathfrak{I}_{\delta} \rVert_s 
\lVert \hat{\imath} \rVert_{s_0})\,. 
\end{equation}
Then in $\RR_i$ we include all the terms homogeneous of degree $i$  coming from derivatives of $\Phi_B-\mathrm{I}$ , evaluated at $\tilde q=0$; we put in $\RR_{>5}$ all the rest.
The \eqref{dude}, \eqref{dude2} 
follows by \eqref{DecayR2dp}, 
\eqref{DecayR*dp} and \eqref{DecayR*dp2}.
\end{proof}

\begin{remark}\label{ordinivari}
The motivation for separating the $\RR_i$ and $\RR_{>5}$ 
is the following. Consider the Hamiltonian $H_\e$ 
as a function of $\xi$ instead of $\omega$. 
Then in all our expressions we can, and shall, 
evidence a {\it purely polinomial} term 
$\sum_{i=0}^5 \e^i f_i$ (where the $f_i$ are 
$\e$ independent)  plus a {\it remainder}, which is not analytic in $\e$, of size 
$\varepsilon^6+\varepsilon\lVert \mathfrak{I}_{\delta} 
\rVert_s^{\gamma, \calO_0}$. 
By the assumption \eqref{AssumptionDP}, 
this means that in {\it low norm} $s=s_0+\gotp_1$ 
all these remainders are negligible 
w.r.t. terms of order $\e^5$.
This distinction 
is needed because, 
due to the resonant nature of the DP equation, 
we need to perform (see subsections \ref{preliminare}
and \ref{LinearBNF})
five steps of 
the order reduction 
and of the linear 
BNF by hand, before entering 
in a perturbative regime. 

\noindent
In this framework $\RR_{>5}$ is purely 
a remainder, while the $\RR_i$  are 
homogeneous polynomial terms. 
One could apply the same division to the 
non finite rank terms, one would get
\begin{equation}
\label{fosco}
\Pi_S^{\perp} \big(\omega(\xi)\cdot\partial_{\varphi} h
-J\,[(1-\Phi_B(T_{\delta})-
\partial_{u }^2 f(\Phi_B(T_{\delta})))=
\Pi_S^{\perp} \big(\bar\omega\cdot\partial_{\varphi} h 
+ \e^2 \mathbb A\xi \cdot\partial_{\varphi} h
-J\,(1-\Phi^{\le 5}_B(\e\bar v) )h + g h
\end{equation}
where $g$ satisfies the same estimates as \eqref{DecayR*dp}.
\end{remark}

\subsubsection{Hamiltonian of the linearized operator}\label{HamLin}

%\paragraph{\bf The Hamiltonian of the linearized operator.}
%\textcolor{red}{As should be expected $\calL_{\omega}$ is a Hamiltonian operator, moreover it is simpler to express the {\it purely polynomial terms} in terms of the Hamiltonian function (indeed such terms arise from the weak BNF, which we performed on the Hamiltonian).}
%\comment{quello in rosso non lo capisco.}
Following Remark \ref{ordinivari}, we  evidence the terms homogeneous in the Hamiltonian of $\mathcal{L}_{\omega}$, let us call it $\mathsf{H}$, whose Hamiltonian vector fields have  degree $\le 5$, since they are NOT perturbative. 
As explained in \eqref{PhiBdp} this entails expanding the map $\Phi_B(T_{\delta})$ in powers of $\e$ up to order five plus a {\it small remainder $\tilde q$}.
\\ 
We consider the  symplectic form in the extended phase space $\mathbb{R}^{\nu}\times\mathbb{R}^{\nu}\times H_S^{\perp}$
\begin{equation}\label{formasimpletticaestesa}
\Omega_{e}(\varphi, \eta, z):=d\varphi \wedge d\eta+\sum_{j\in S^c} \frac{1}{\mathrm{i} \lambda(j)} d z_j\wedge d z_{-j}
\end{equation}
with the  Poisson brackets (recalling  $\{ \cdot, \cdot \}$  defined in \eqref{PoissonBracketDP})
\begin{equation}\label{PoissonEstesa}
\{  F, G \}_{e}:=\partial_{\varphi} F \partial_{\eta} G
-\partial_{\eta} F \partial_{\varphi} G+\{ F, G\}\,.
\end{equation}

%\smallskip

\noindent
The Hamiltonian of the operator \eqref{BoraMaledetta} respect to the symplectic form \eqref{formasimpletticaestesa} is (see \eqref{fosco})
\begin{equation}\label{hamiltonianalinearizzata}
 \mathsf{H}:=\mathsf{H}_0+\sum_{i=1}^5\varepsilon^i\mathsf{H}_i+ H_{>5}+\sum_{i=2}^5\varepsilon^i\mathsf{H}_{\mathcal{R}_i}+\mathsf{H}_{\mathcal{R}_{>5}}
\end{equation}
with
\begin{equation}\label{leHamiltonianeLin}
\begin{aligned}
&\mathsf{H}_0=\overline{\omega}\cdot \eta
+\frac{1}{2}\int_{\T} z^2\,dx\,, 
\quad\mathsf{H}_1=-\frac{1}{2}\int_{\T} \overline{v}\,z^2\,dx\,,
 \quad
\mathsf{H}_2=\mathbb{A}\xi\cdot \eta
-\frac{1}{2}\int_{\T} \Psi_2(\overline{v})\,z^2\,dx\,,\\
& \mathsf{H}_i=-\frac{1}{2} \int_{\T} \Psi_i (\overline{v})\,dx\,,
\quad 3\le i\le 5\,,\;\;
{\rm and} \;\;
\|X_{\mathsf{H}_{>5}}
 \rVert^{\gamma, \calO_0}_s,
  \|X_{\mathsf{H}_{\mathcal{R}_{>5}}}
 \rVert^{\gamma, \calO_0}_s
 \lesssim_s 
 \varepsilon^6+\varepsilon\lVert \mathfrak{I}_{\delta}
 \rVert^{\gamma, \calO_0}_{s+\sigma}
\end{aligned}
\end{equation}
for some $\sigma>0$.  The functions 
$\mathsf{H}_{\mathcal R_i}$, $\mathsf{H}_{\mathcal{R}_{>5}}$ 
are the quadratic forms associated 
to the corresponding linear operators, 
thus the estimates on the Hamiltonian 
vector fields can be deduced from 
\eqref{DecayR*dp}, \eqref{DecayR2dp}. 
It is easily seen that
\[
\mathsf{H}_0+\sum_{i=1}^5\varepsilon^i(\mathsf{H}_i+ \RR_i)
= \Pi^{2d_y + d_z=2} (\calH^{(\le 7)}\circ A_\e)
\vert_{\substack{y=\eta \\ \theta=\varphi} }\,.
\]
Of course one can be even more explicit and write 	everything in terms of the original Hamiltonian \eqref{DPHamiltonian} and of the generating functions of the weak BNF, for example one has 
\begin{equation}\label{pioveNantes}
\begin{aligned}
&\mathsf{H}_0 = 
(H^{(2)}\circ A_1)\vert_{\substack{y=\eta \\ \theta=\varphi } }\,,
\qquad  \mathsf{H}_1 \stackrel{\eqref{mastite}}{=} 
(H^{(3,2)}\circ A_1)\vert_{\substack{y=\eta \\ \theta=\varphi} }\,,\\
&\mathsf{H}_2 +\mathsf{H}_{\RR_2}= 
(\Pi^{d_y=1}(Z^{(4,0)}_2\circ A_1) 
+  (\Pi^{d_z=2} \{F^{(3,\le 1)}, H^{(3)}\})\circ A_1)
\vert_{\substack{y=\eta \\ \theta=\varphi } }\,.
\end{aligned}
\end{equation}
The terms $\mathsf{H}_i$ can be computed explicitly, however this is not
necessary for our scopes. We only need to prove
that they fit the following definitions.

%\comment{La definizione di almost diagonal era solo sulle matrici, ma noi parliamo di operatori e Hamiltoniane almost diagonbal.}
\begin{defi}\label{almostDIAG}
We say that a matrix $\mathtt{B}:=\Big((\mathtt{B})_{j}^{j'}(l-l')\Big)_{j, j'\in\mathbb{Z}, l, l'\in\mathbb{Z}^{\nu}}$ is
\emph{almost diagonal} if there exists 
a constant $\mathtt{C}>0$
such that, if $(\mathtt{B})_{j}^{j'}(l-l')\neq0$, then 
$\langle j-j', l-l'\rangle\le \mathtt{C}\,$  
for all $j, j'\in S^c$, $l, l'\in\mathbb{Z}^{\nu}$. \\
Let $B(\varphi)\colon H^s(\T)\to H^s(\T)$ be a T\"opliz in time operator 
(recall \eqref{topliz}). 
We say that $B(\varphi)$ is almost diagonal if 
its associated matrix is almost diagonal.\\
Let $H:=H(\varphi)$ be a quadratic Hamiltonian of the form 
$H=( A(\varphi) z, z)_{L^2}$, where $A(\varphi)$ is a 
T\"opliz in time operator. We say that 
$H$ and its vector field are almost 
diagonal if $A(\varphi)$ is almost diagonal. 
\end{defi}
\begin{remark}\label{remarkalmost}
It is easy to verify that if $X$ and $Y$ are almost diagonal vector fields then $X+Y$, $X \circ Y$ are almost diagonal.
\end{remark}

\begin{defi}\label{PseudoSIMho}
Let $p\in \mathbb{N}$ and $m\in \mathbb{R}$. 
We say that a pseudo differential operator $\mathfrak{B}=\op(b(\f,x,j))$
(recall Definition \ref{pseudoR}) is \emph{homogenenous} of degree $p$  in the 
function $\bar{v}$ in \eqref{vsegnatoDP} if its symbol $b(\f,x,j)\in S^{m}$ 
has the form
\begin{equation}\label{simboOMO}
b(\f,x,j):=\sum_{j_1,\ldots,j_{p}\in S}C_{j_1,\ldots,j_{p}}(j)
\sqrt{\x_{j_1}\cdots\x_{j_p}}e^{\ii(j_1+\ldots+j_p)x}
e^{\ii(\mathtt{l}(j_1)+\ldots+\mathtt{l}(j_p))\cdot\f}.
\end{equation}
\end{defi}

\begin{defi}\label{HAMOMO}
Let $p\in \mathbb{N}$. Let $f_{p}$ be an homogeneous real valued 
function of $\bar{v}$ (of degree $p$) 
of the form
\begin{equation}\label{simboOMO2}
f_p(\bar{v}):=\sum_{j_1,\ldots,j_{p}\in S}(f_p)_{j_1,\ldots,j_{p}}
\sqrt{\x_{j_1}\cdots\x_{j_p}}e^{\ii(j_1+\ldots+j_p)x}
e^{\ii(\mathtt{l}(j_1)+\ldots+\mathtt{l}(j_p))\cdot\f},
\end{equation}
and let $\mathfrak{B}_p\in OPS^{-2}$ be a p-homogeneous 
pseudo differential  operator according to Definition \ref{PseudoSIMho}
which is self-adjoint w.r.t. to $(\cdot,\cdot)_{L^2}$.
We say that a Hamiltonian is 
pseudo differential and $p$-homogeneous if it has the form
\begin{equation}\label{HAMpseudo}
H_p(z)=\frac{1}{2}\int_{\mathbb{T}}f_p(\bar{v})z\cdot zdx
+\frac{1}{2}\int_{\mathbb{T}}\mathfrak{B}_p z\cdot z
+\frac{1}{2}\int_{\mathbb{T}}\mathcal{R}_pz\cdot z\,, 
\end{equation}
with some \emph{finite dimensional} 
operator $\mathcal{R}_{p}$ of the form \eqref{FiniteDimFormDP}
with $g_j, \chi_j$ $p$-homogeneous functions of $\bar{v}$.
\end{defi}

\begin{lem}
The Hamiltonians $\mathsf{H}_i + \mathsf{H}_{\RR_i}$  
in \eqref{hamiltonianalinearizzata} 
are almost diagonal according to Definition 
\ref{almostDIAG} and  pseudo differential
homogeneous Hamiltonians 
according to Definition \ref{HAMOMO}.
\end{lem}
\begin{proof}
It follows by Proposition \ref{LinearizzatoIniziale}
and Remark \ref{ordinivari}.
\end{proof}

\begin{lem}\label{pseudoHAM1}
Let $p,q\in \mathbb{N}$
and consider $H_p$, $G_q$ two pseudo differential and homogenous
Hamiltonians of degree respectively $p$
and $q$. Then there is 
a pseudo differential and $(p+q)$-homogeneous Hamiltonian $\widetilde{H}$ 
such that $X_{\widetilde{H}}=X_{\{H_p,G_{q}\}_{e}}$,
where $\{\cdot,\cdot\}_e$
are defined in \eqref{PoissonEstesa} and $X_{H}$ 
denotes the Hamiltonian vector field generated by $H$.
\end{lem}

\begin{proof}
By assumption $H_p$ and $G_{q}$ have the form \eqref{HAMpseudo}
for some $f_{p},f_q$ real valued and some self-adjoint 
pseudo differential homogenenous operators 
$\mathfrak{B}_p$ and $\mathfrak{B}_q$.
Then we have (recalling \eqref{PoissonBracketDP},  
\eqref{DPHamiltonian} and  \eqref{Helmotz})
\[
\begin{aligned}
\{H_p,G_q\}_{e}&=\int_{\mathbb{T}}
A_1 z\cdot zdx +\int_{\mathbb{T}}
A_2 z\cdot zdx, \qquad A_1:=f_{p}\circ \partial_x\circ f_{q}\,,\\
A_2&:=f_p\circ J\circ \mathfrak{B}_q+\mathfrak{B}_p\circ J\circ f_q+
\mathfrak{B}_p\circ J\circ \mathfrak{B}_q+3f_p\circ\Lambda\partial_x\circ f_q\,.
\end{aligned}
\]
One has that the Hamiltonian 
\[
\widetilde{H}:=\frac{1}{2}\int_{\mathbb{T}}(A_1+A_1^{*})z\cdot zdx
+\frac{1}{2}\int_{\mathbb{T}}(A_2+A_2^{*})z\cdot zdx
\]
is equivalent to $\{H_p,G_q\}_e$ 
in the sense that they generate the same vector field. 
%\eqref{760H} holds.
Here $A_i^{*}$, $i=1,2$, denotes
the adjoint of $A_i$ w.r.t. the $L^{2}$ scalar product.
Notice that
\[
A_1+A_1^{*}=f_{p} f_{q}\partial_x+f_p(f_q)_x-f_q\circ\partial_x\circ f_p=
f_{p} f_{q}\partial_x+f_p(f_q)_x-f_qf_p\partial_x-f_q(f_p)_x=
f_p(f_q)_x-f_q(f_p)_x,
\]
which is an homogeneous function of $\bar{v}$ of degree $p+q$.
Using the results on compositions of pseudo differential operators in Section $2$ of \cite{FGP1},
the fact that $J$ is skew-self-adjoint, $\mathfrak{B}_i$, $i=p,q$,
are self-adjoint, and $f_q,f_p$ are real valued, we deduce that
the operator $A_2$ is a skew-self-adjoint operator in $OPS^{-1}$.
Hence, using the formula $(2.13)$  in \cite{FGP1} for the adjoint,
we have that ${A}_2+{A}_2^*$ is pseudo differential homogeneous
operator (according to Definition \ref{PseudoSIMho}) 
in $OPS^{-2}$. 
%This concludes the proof.
\end{proof}

\section{Reduction and Inversion of the linearized operator}\label{regularization}

%In this section we conjugate, by symplectic, tame and bounded changes of coordinates, the linearized operator \eqref{LomegaDP} 
%to a diagonal one, up to smoothing remainders.

The aim of the section is to prove
the claim in \eqref{InversionAssumptionDP}.
As explained in the introduction, first one should reduce the unbounded parts of $\calL_{\oo}$ and then  use classical KAM reducibility 
results to diagonalize. The difficulties arise from the fact that a few steps of this procedure must be done by hand, since they do not fit the typical smallness condition. \\
The key result of this section is the following.
\begin{teor}\label{risultatosez8}
Consider $\calL_{\oo}=\calL_{\oo}(\mathfrak{I}_{\delta})$ in \eqref{BoraMaledetta} and fix 
\begin{equation}\label{parametriKAM}
\tau= 2\nu+6,\quad \tb_0:=6\tau+6, \quad \tb=\tb_0+s_0.
\end{equation}
There exists $\mu_1=\mu_1(\nu)>0$ such that, if condition \eqref{IpotesiPiccolezzaIdeltaDP}
is satisfied  with $\gotp_1=\mu_1$,  then the following holds.
There exists a constant $m(\oo)$ defined for  $\oo\in\Omega_{\e}$ with 
\begin{equation}\label{clinica100}
\lvert m-1-\e^{2}c(\oo) \rvert^{\gamma, \Omega_{\varepsilon}}\lesssim  \varepsilon^4\,, 
\quad \lvert m \rvert^{lip}\lesssim 1\,,
\quad c(\oo):=\vec{v}\cdot\xi,\quad \vec{v}_k=\frac{2}{3} (1+\ol{\jmath}_{k}^2)\, , 
\;\; k=1,\ldots, \nu\,,
\footnote{Notice that $\x=\x(\oo)$, recall \eqref{xiomega}.}
\end{equation}
such that
for all $\oo$ in the set $\cO_{\infty}^{2\gamma}$, where (recall that $\calO_0\subseteq \mathcal{G}_0$, see \eqref{0di0Melnikov})
\begin{equation}\label{condizMel1}
\cO_{\infty}^{2\gamma}=\cO_{\infty}^{2\gamma}(i):= \{\omega\in \cO_0: 
\; |\omega\cdot \ell - m(\omega) j |>\frac{2\g }{\langle \ell\rangle^{\tau}}\,,
\;\forall \ell \in \Z^\nu, \,\,\forall j\in S^{c}\}\,,
\end{equation}
there exists a real, bounded linear operator  
$\Upsilon=\Upsilon(\oo) : H^{s}_{S^{\perp}}\to H^{s}_{S^{\perp}}$, 
for all $s_0\le s\le \mathcal{S}$,
such that
\begin{equation}\label{operatorefinale}
\calL:=\Upsilon \calL_{\oo}\Upsilon^{-1}=\Pi_S^{\perp}\big(\omega\cdot\del_{\f}-m J-\varepsilon^2 \mathfrak{D}(\omega) +\mathcal{P}_0 \big)
\end{equation}
where $\mathfrak{D}(\omega)$ is  the diagonal operator of order $-1$
defined as 
$\mathfrak{D}:=\mathfrak{D}(\omega)=
\mathrm{diag}(\mathrm{i}\kappa_j)_{j\in S^c} $, 
with
\begin{equation}\label{diagonalopFinale}
  \kappa_{j}:=\kappa_j(\omega):= \lambda(
 j)\big(\lal_j(\oo)-c(\oo)\big)\in \R\,, 
 \quad |\kappa_{j}|^{ sup}\lesssim |j|^{-1}\, ,
\end{equation}
where  $\lal_{j}$ is defined in \eqref{lambdaJ0}.
The constant $m$ depends on $i$ and for $\oo\in 
\cO_{\infty}^{2\gamma}(i_1)\cap \cO_{\infty}^{2\gamma}(i_2)$ 
one has
\begin{equation}\label{clinica1000}
\lvert \Delta_{12}m \rvert\lesssim \varepsilon\lVert i_1-i_2 \rVert_{s_0+\mu_1}\,,
\end{equation}
where $\Delta_{12}m:=m(i_1)-m(i_2)$.
The remainder $\mathcal{P}_0$ in \eqref{operatorefinale}
is defined and Lipschitz in $\omega$ belonging to the set $\calO^{2\gamma}_\infty$
and is Lip-$-1$-modulo tame (see Definition \ref{menounomodulotame}) with
\begin{align}\label{pasqua200}
  &\mathfrak{M}^{\sharp, \gamma^{3/2}}_{\cP_0}(-1,s,\tb_0)  \lesssim_s \varepsilon^{4-3 a}
+\varepsilon \gamma^{-1}\lVert \mathfrak{I}_{\delta}\rVert_{s+{\mu_1}}^{\gamma, \cO_0}\,,
 \\
 &\|\lD^{1/2}\underline{\Delta_{12}\cP_0 }\lD^{1/2}\|_{\mathcal L(H^{s_0})}, \,\, \|\lD^{1/2}\underline{\Delta_{12} \langle \partial_{\f}\rangle^{{\mathtt b_0}} 
 \cP_0  }\lD^{1/2}\|_{\mathcal L(H^{s_0})} 
 \lesssim  \e \g^{-1}  \|i_1-i_2\|_{s_0+\mu_1}\,,
 \label{nomidimerda2bis}
 \end{align}
 for all $\omega\in \calO_{\infty}^{2\g}(i_1)\cap \calO_{\infty}^{2\g}(i_2)$. 
Moreover if $u=u(\omega)$ depends on the parameter 
$\omega\in \calO_{\infty}^{2 \gamma}$ in a Lipschitz way then
\begin{equation}\label{grecia}
\lVert \Upsilon^{\pm 1} u \rVert_s^{\g, \calO_{\infty}^{2 \gamma}}
\lesssim_s \lVert u \rVert_s^{\g,\calO_{\infty}^{2\gamma}}
+\varepsilon \gamma^{-1}
\lVert \mathfrak{I}_{\delta} \rVert^{\g, \calO_0}_{s+\mu_1}
\lVert u \rVert^{\gamma, \calO_{\infty}^{2\gamma}}_{s_0}\,,
\;\;\; s_0\le s\le \mathcal{S}\,.
\end{equation}
\end{teor}
The result above 
has two relevant consequences. Firstly it shows
that the operator $\mathcal{L}_{\omega}$
in \eqref{BoraMaledetta} can be conjugated 
to an operator (see \eqref{operatorefinale})
which is ``diagonal'', at the highest order of derivatives,
plus a remainder which is $-1$-smoothing.
In addition to this, thanks to a \emph{linear BNF} procedure 
(performed in subsection \ref{LinearBNF}),
the non-diagonal term $\mathcal{P}_0$ in 
\eqref{operatorefinale} has a size much smaller than $\e$
(see estimates \eqref{pasqua200}, \eqref{nomidimerda2bis}).
In particular  it is ``perturbative'' w.r.t. the constant $\gamma$ in \eqref{gammaDP}.
This allows us to apply the reducibility scheme of \cite{FGP1}
in order to complete the diagonalization of the operator $\mathcal{L}$ (see Theorem \ref{ReducibilityDP}).
%This is the content of 
%Theorem \ref{ReducibilityDP} which is proved in .
%Thank to this reducibility result on $\mathcal{L}_\omega$ it is easy to deduce
%Theorem \ref{InversionLomegaDP})  which implies 
 Then the inversion assumption  \eqref{InversionAssumptionDP} follows directly from Proposition \ref{InversionLomegaDP}).

\bigskip

\noindent
{\it Strategy of the Proof of Theorem \ref{risultatosez8}.}

%\comment{Vogliamo aggiungere più riferimenti in questa sezione di spiegazione della strategia?}

\smallskip

$\bullet${ \bf Reduction at the highest order.} \; The first step is to 
exploit the pseudo differential structure of the operator $\calL_{\oo}$ 
in order to	conjugate it to an operator 
which has constant coefficients 
up to a smoothing remainder of order $-1$. 
To this purpose we use  changes of 
variables  generated as the time-one 
flow map $\Phi^{\tau}\vert_{\tau=1}$ of the Hamiltonian
\begin{equation}\label{pseudo}
S(\tau,\f,z)= \int b(\tau,\f,x) z^2 dx\,,
\qquad
b(\tau,\f,x):=\frac{\be(\f,x)}{1+\tau \be_{x}(\f,x)},
\end{equation}	
\begin{equation}\label{rosso}
\partial_{\tau} \Phi^{\tau} u=
\Pi_S^{\perp}[(J\circ b) \Pi_S^{\perp}[\Phi^{\tau}u]]\,,\qquad
\Phi^0 u=u\,,
\end{equation}
where $\beta$ is some smooth function.	
In Proposition \ref{ConjugationLemma} 
we show that $\Phi^{\tau}$ is well 
defined as symplectic map on $H^{s}_{S^{\perp}}$
(see Lemma \ref{differenzaFlussi}) and  
study  the structure of 
$\Phi^{\tau}\calL_\oo (\Phi^{\tau})^{-1}$.
Proposition \ref{ConjugationLemma}  gives an explicit formula for the new coefficient at the 
highest order (see \eqref{Round}). 
Then  Corollary $3.6$ of \cite{FGMP} 
(see also Proposition $3.6$ in \cite{FGP1})
%Proposition \ref{moser} 
provides the solution for the equation 
\eqref{Round}=const provided 
that some smallness condition is satisfied.
This smallness condition has the form
\begin{equation}\label{picci}
C(s_1)\gamma^{-1}\|a_0\|_{s_1}^{\gamma,\calO}\ll1
\end{equation}
for some $s_0+\gotp_1>s_1>s_0$ and some 
constant $C(s_1)>0$. As shown in \cite{FGP1}, 
due to the Hamiltonian structure,  
this reduces $\calL_{\omega}$ to constant 
coefficients up to a correction of order $-1$.

Unfortunately, since here $\gamma= \e^{2+a}$, $a>0$, 
by \eqref{unasuaziaricca}, the coefficient $a_0(\f,x)$ 
in $\mathcal{L}_{\omega}$ does not satisfy \eqref{picci}.
This is why we have to perform some 
preliminary steps in order to enter in a perturbative regime 
for the scheme described in the proof 
of Corollary $3.6$ in \cite{FGMP}.

We first "regularize" the purely polynomial 
terms $\mathsf{H}_i$ (see \eqref{leHamiltonianeLin}) by hand, by exploiting their 
homoegeneity according to Definition \ref{HAMOMO}. 
After that we are left with 
only unbounded terms which satisfy 
the smallness conditions of \cite{FGP1}. 
We "regularize" them by applying the results of \cite{FGP1} adapted to our slightly more 
general setting, see Proposition 
\ref{ConjugationLemma}.

\smallskip
In order to determine the correct change of variables 
in the regularization of $\mathsf{H}_i$, it will 
be convenient to use the { \it Formal Lie expansions}.
We recall  that $H\circ (\Phi^{\tau}_{S(\tau)})^{-1}$ 
satisfies, for $\tau\in[0,1]$,
\begin{equation}\label{deritau}
\partial_{\tau}(H\circ (\Phi^{\tau}_{S(\tau)})^{-1})
=\{S(\tau),H\circ(\Phi^{\tau}_{S(\tau)})^{-1}\}.
\end{equation}
By setting  $S:=S(0)$ and 
$\Phi:=(\Phi^{\tau}_{S(\tau)})_{|\tau=1}$, 
the Lie expansion of the conjugated Hamiltonian
$H\circ\Phi^{-1}$ is the following:
\begin{equation}\label{lietau}
\begin{aligned}
H\circ(\Phi^\tau)^{-1}&=H+\tau \{S,H\}_{e}
+\frac{\tau^2}{2}\Big(\{S,\{S,H\}_{e}\}_{e}
+\{(\partial_{\tau}S)(0),H\}_{e}\Big)+\dots\,,
\end{aligned}
\end{equation}
where %$O(\tau^{3})$ contains all the terms at least cubic in the Hamiltonian $S$
 the Poisson brackets $\{\cdot,\cdot\}_{e}$ 
 are  in \eqref{PoissonEstesa}.
Recall  that $\Phi$ is a $C^k$ map from $H^s $ 
to $H^{s-k}$. Therefore the Taylor expansion 
of the  conjugated Hamiltonian coincides 
with the Lie series of the generator up to any order $\tau^k$.

\medskip
$\bullet${ \bf Linear BNF.} \;The second step is to diagonalize the bounded terms. 
Here we diagonalize "by hand" the terms up to order $\e^3$, 
by exploiting the fact that they are almost diagonal 
according to Definition \ref{almostDIAG} and applying a linear BNF. 
Once this is done, the full diagonalization follows 
by a standard KAM reducibility theorem (see Theorem \ref{ReducibilityDP}).

\subsection{Reduction at the highest order}\label{preliminare}

In the following we shall assume that the \eqref{IpotesiPiccolezzaIdeltaDP} holds with some 
$\gotp_1\gg1$. The loss of regularity $\gotp_1$ will be determined explicitly 
at the end of the section.
In order to perform the {\em non-perturbative} steps, we construct changes of coordinates $\B_{i}$, $i=1,2,3,4,5$,
as the time-one flow maps generated by  \eqref{pseudo}.
In order to have a quantitative control on the remainders 
appearing in our procedure we shall use the class of smoothing operators
$\mathfrak{L}_{\rho,p}$ (introduced in \cite{FGP1}) 
given
in Definition \ref{ellerho} with (recall that $s_0:=[\nu/2]+4$)
\begin{equation}\label{tornado}
\rho\geq s_0+6\tau+9\,, %\qquad s_0+\gotp_1-\hat{\s}\geq p\geq s_0.
\end{equation}
and $s_0\leq p$, which will be fixed later on.
The changes of coordinates $\B_{i}$, $i=1,\ldots,5$,
are constructed in such a way,
setting
\begin{equation}\label{iesimo}
\mathcal{L}_i:=\B_i \mathcal{L}_{i-1} \B_i^{-1}=\Pi_S^{\perp} 
\Big( \omega\cdot\partial_{\varphi}-J\circ \big(1+\varepsilon^2 c_i(\omega)+a_i(\varphi, x) \big)+\op(\mathtt{q}_i)+\widehat{\mathcal{Q}}_i\Big)
\end{equation}
with $\mathcal{L}_0=\mathcal{L}_{\omega}$ and some constant $c_i(\oo)$. We have 
\begin{equation}\label{uranioi}
\lVert a_i \rVert_s^{\gamma, \calO_0}\lesssim_s \varepsilon^{i+1}
+\varepsilon\lVert \mathfrak{I}_{\delta}\rVert_{s+\sigma_0+\sigma_{i+3}}^{\gamma, \calO_0}, \qquad 
\lVert \Delta_{12} a_i   \rVert_{s} \lesssim_{p} 
\varepsilon (1+\lVert \mathfrak{I}_{\delta} \rVert_{p+\sigma_0+\sigma_{i+3}})\|i_1-i_2\|_{p+\s_0+\s_{i+3}}.
\end{equation}
Note that the size of $a_{i}$ (in the low norm) is decreasing in $i$. Regarding the remainders, the  $\op(\mathtt{q}_i)$ are of order $-1$  while $\widehat{\mathcal{Q}}_i$ is in $\mathfrak{L}_{\rho,p}$. The numbers
$\mathbb{M}^{\gamma}_{\widehat{\mathcal{Q}}_i}(s, \tb)$ control the norm of the corresponding operator, see Definition  \ref{ellerho}. Note that by Lemma \ref{lrofiniterank} $\mathcal{Q}_0=:\widehat{\mathcal{Q}}_0$ belongs to $\mathcal{L}_{\rho, p}$ with bounds on $\mathbb{M}^{\gamma}_{\widehat{\mathcal{Q}}_0}(s, \tb)$ given in the same lemma. We have
\begin{equation}\label{Napalmi}
\begin{aligned}
\lvert \mathtt{q}_i \rvert^{\gamma, \calO_0}_{-1, s, \alpha} & \lesssim_{s, \alpha, \rho} 
\varepsilon (1+\lVert \mathfrak{I}_{\delta}\rVert^{\gamma, \calO_0}_{s+\sigma_0+\sigma_{i+3}}), \quad  s\geq s_0\\
\mathbb{M}^{\gamma}_{\widehat{\mathcal{Q}}_i}(s, \tb) & \lesssim_{s, \rho} \varepsilon 
(1+\lVert \mathfrak{I}_{\delta}\rVert^{\gamma, \calO_0}_{s+\sigma_0+\sigma_{i+3}}),  \quad s_0\le s\le \mathcal{S} ,,\qquad 0\le \tb\le \rho-2,
\end{aligned}
\end{equation}
\begin{equation}\label{Napalmii}
\begin{aligned}
\lvert \Delta_{12} \mathtt{q}_i    \rvert_{-1, p, \alpha} & \lesssim_{p, \alpha, \rho} 
\varepsilon (1+\lVert \mathfrak{I}_{\delta} \rVert_{p+\sigma_0+\sigma_{i+3}})\|i_1-i_2\|_{p+\s_0+\s_{i+3}},\\
\mathbb{M}_{\Delta_{12} \widehat{\mathcal{Q}}_i  }(p, \tb) & \lesssim_{p, \alpha} 
\varepsilon (1+\lVert \mathfrak{I}_{\delta} \rVert_{p+\sigma_0+\sigma_{i+3}})\|i_1-i_2\|_{p+\s_0+\s_{i+3}},\qquad 0\le \tb\le \rho-3,
\end{aligned}
\end{equation}
with $\sigma_0$ defined in Proposition \ref{LinearizzatoIniziale} and $\sigma_{i+3}>0$ $i=1,\ldots,5$, depending only on $\nu$
(essentially $\s_{i+3}$ are the losses coming from the application of  Proposition \ref{ConjugationLemma}). Note that if we want \eqref{Napalmi} to hold for some given $p$, we have to assume a smallness condition \eqref{IpotesiPiccolezzaIdeltaDP} with $p+\sigma_0+\sigma_{i+3}<\gotp_1$.

%As explained before, in order to choose the generator 
%of each flow we perform Lie expansion of the Hamiltonian.
%We will make a deep use of the conjugation result of Proposition 
%\ref{ConjugationLemma}. 

%This is possible since it will turns out that the maps $\mathcal{B}_i$ will be
%, as functions of $\e$, 
%$C^{3}$ with values in $\mathcal{L}(H^{s}, H^{s-3})$ .
\medskip

\noindent\textbf{Step ($\varepsilon$)}. Consider the Hamiltonian
\begin{equation}\label{Stau}
\begin{aligned}
S(\tau):=\frac{1}{2} \int_{\mathbb{T}} b_1(\tau,\varphi, x)\,z^2\,dx &=\varepsilon S_1+\varepsilon^2\tau S_2+\varepsilon^3\tau^{2} S_3+S_4(\tau), \qquad  b_1:=\frac{\varepsilon\beta_1}{1+\tau \varepsilon  (\beta_1)_x},
\end{aligned}
\end{equation}
\begin{equation}\label{esse1}
S_1:=\frac{1}{2} \int_{\T} \beta_1\,z^2\,dx, \quad S_2:=-\frac{1}{4}\int_{\T}\partial_x (\beta_1^2)\,z^2\,dx, \quad S_3:=\frac{1}{2} \int_{\T}\beta
_1 (\beta_1)_x^2z^{2}\,dx, 
\end{equation}
with $S_4(\tau)\sim O(\tau^3\e^{4})$ and 
for some function $\beta_1$ of the form \eqref{simboOMO2}  with $p=1$ and some 
coefficients $(\beta_1)_{j}$, $j\in S$, 
 to be determined.
The Hamiltonian system associated to $S(\tau)$ is
of the form \eqref{rosso} with $b\rightsquigarrow b_1$.
%\begin{equation}\label{systemS1}
%u_{\tau}=\Pi_S^{\perp}[\big(J\circ b_1(\tau)\big)\,u].
%\end{equation}
We call $\B_1$ the flow at time-one generated by $S(\tau)$,
%of the system \eqref{systemS1}, 
then the Hamiltonian of the conjugated linearized operator $\B_1 \mathcal{L}_{\omega} \B_1^{-1}$ is (recall \eqref{hamiltonianalinearizzata}, \eqref{leHamiltonianeLin})
\begin{align}
\!\!\!\! \mathsf{H}\circ \B_1^{-1} &=
\mathsf{H}_0+\e \mathsf{H}^{(1)}_1+\e^{2}\mathsf{H}^{(1)}_2+\e^{3}\mathsf{H}_3^{(1)}
+o(\e^{3})\,, 
\qquad \mathsf{H}_1^{(1)}:= \{  S_1, \mathsf{H}_0\}_{e}+\mathsf{H}_1\,,
\label{espansioLIE}\\
\mathsf{H}_2^{(1)}&:=\frac{1}{2}\{ S_1,\{ S_1,  \mathsf{H}_0\}_e\}_{e}
+ \{S_1,\mathsf{H}_1\}_{e}+\frac{1}{2}\{S_2,\mathsf{H}_0\}_{e}
+\mathsf{H}_2+\mathsf{H}_{\mathcal{R}_2}\,,\label{espansioLIE22}
\end{align}
where, by Lemma \ref{pseudoHAM1}, $\mathsf{H}_3^{(1)}$
is some pseudo differential $3$-homogeneous Hamiltonian of the form \eqref{HAMpseudo}.
Notice that also $H_1^{(1)}$ and $\mathsf{H}_2^{(1)}$ (for $\eta=0$)
are pseudo differential and $1-$homogeneous, resp. 2-homogenenous, Hamiltonians according to  Definition \ref{HAMOMO}.
We want to solve the following equation 
\begin{equation}\label{trieste}
\mathsf{H}_1^{(1)}=\mathsf{H}_1+\{ S_1, \mathsf{H}_0\}_{e}=\mathsf{H}_1+\{ S_1, \mathsf{H}_0\}+\overline{\omega}\cdot\partial_{\varphi} S_1= \int  \mathfrak{B}_1( z)\,z\,dx\,,
\end{equation}
where $\mathfrak{B}_1$ is some pseudo differential operator of order $-2$. 
Recalling \eqref{leHamiltonianeLin}, 
and expanding $\{S_1,\mathsf{H}_0\}$ as in the proof of Lemma \ref{pseudoHAM1},
we note that
the equation \eqref{trieste} is equivalent to the following one
\begin{equation}\label{MisterX}
\overline{\omega}\cdot\partial_{\varphi} \beta_1-(\beta_1)_x-\overline{v}=0.
\end{equation}
Hence we choose 
$
\beta_1=\frac{1}{3} (\Lambda\partial_x)^{-1} \overline{v}
$
and we note that
\begin{equation}\label{Drugo}
\lVert \varepsilon\beta_1 \rVert^{\gamma, \calO_0}_s
\lesssim_s \varepsilon\,, \quad \forall s\geq s_0\,.
\end{equation}
With the choice in  \eqref{MisterX} we have 
\begin{equation}\label{Buno}
\mathfrak{B}_1:=[3 \Lambda\partial_x, \beta_1].
\end{equation}
In this way the Hamiltonians in \eqref{espansioLIE}, \eqref{espansioLIE22} 
become
\begin{equation}\label{espansioLIE100}
\mathsf{H}^{(1)}_1:=\int  \mathfrak{B}_1( z)\,z\,dx\,,\qquad
\mathsf{H}^{(1)}_2:=\mathsf{H}_2+\frac{1}{2}\{S_{2},\mathsf{H}_0\}_{e}
+\mathsf{H}_{\mathcal{R}_2}
+\frac{1}{2}\{ S_1, \mathsf{H}_1\}+\frac{1}{2}\{S_1, \int_{\mathbb{T}} \mathfrak{B}_1(z)\,z\,dx\} \,.
\end{equation}
By \eqref{Drugo} the smallness assumption of Lemma \ref{differenzaFlussi} is satisfied.
By \eqref{unasuaziaricca}, \eqref{dude}, \eqref{dude2} and using the assumption 
\eqref{IpotesiPiccolezzaIdeltaDP} with $\gotp_1$ sufficiently large
%given in \eqref{ipotesimu} 
the condition \eqref{filini} holds.
In this case $\mathtt{q}\rightsquigarrow 0$ (see \eqref{docq}).\\
Then Proposition \ref{ConjugationLemma} applies and the new linearized operator 
$\mathcal{L}_1:=\B_1 \mathcal{L}_{\omega} \B_1^{-1}$ has the form in \eqref{iesimo} with $i=1$ and $c_1=0$. 
By \eqref{Drugo}, \eqref{dude}, \eqref{JeTame}, \eqref{Round}, \eqref{BoraMaledetta}, we have
that $\widehat{\mathcal{Q}}_1\in \gotL_{\rho, p}(\calO)$ (see Definition  
\ref{ellerho}) (with $\rho$, $p$ as in \eqref{tornado}) and 
\eqref{Napalmi}, \eqref{Napalmii} hold for $i=1$.

\noindent
The only estimates that are not given by Lemma \ref{ConjugationLemma} are \eqref{uranioi}. The coefficient $a_1$ is given by \eqref{Round} with $m\rightsquigarrow 1$, $a\rightsquigarrow a_0$, $a_+\rightsquigarrow a_1$ and $\tilde{\beta}$ such that $x\mapsto x+\tilde{\beta}$ is the inverse of $x\mapsto x+\beta_1$. By the choice of $\beta_1$ in \eqref{MisterX} we have eliminated the $\varepsilon$-terms from $a_1$. Hence by \eqref{Drugo} and \eqref{unasuaziaricca} we get \eqref{uranioi} for $i=1$.

\medskip

\noindent\textbf{Step ($\varepsilon^2$)}. Now we deal with the terms of order $\varepsilon^2$ of the Hamiltonian \eqref{hamiltonianalinearizzata}. We consider the auxiliary Hamiltonian
\begin{equation}
\begin{aligned}
\tilde{S}(\tau)&=\frac{1}{2} \int_{\mathbb{T}} b_2( x,\varphi)\,z^2\,dx=\varepsilon^2\tilde{S}_2+\tilde{S}_4(\tau), \quad b_2 :=\frac{\varepsilon^2\beta_2}{1+\tau\varepsilon^2 (\beta_2)_x},
\quad
\tilde{S}_2:=\frac{1}{2} \int_{\T} \beta_2\,z^2\,dx, 
%\\&\tilde{S}_4(\tau):=\tilde{S}(\tau)-\varepsilon^2\tilde{S}_2,
\end{aligned}
\end{equation}
where  
$\tilde{S}_4(\tau):=\tilde{S}(\tau)-\varepsilon^2\tilde{S}_2
\sim O(\tau\e^{4})$ and
 $\beta_2$ is some function of the form \eqref{simboOMO}, with $p=2$, 
 to be determined. 
 Notice that $(\partial_{\tau}\tilde{S})(0)\sim O(\e^{4})$. 
 The Hamiltonian system associated to $\tilde{S}(\tau)$ is
of the form \eqref{rosso} with $b\rightsquigarrow b_2$.
If $\B_2$ is the flow at time-one generated by $\tilde{S}(\tau)$,
then the Hamiltonian of the conjugated 
linearized operator $\mathcal{L}_2:=\B_2\mathcal{L}_1 \B_2^{-1}$ is
(recall \eqref{espansioLIE},\eqref{espansioLIE100},\eqref{lietau})
\begin{equation}\label{espansioLIE2}
\begin{aligned}
&\mathsf{H}\circ \B_1^{-1}\circ \B_2^{-1}=
\mathsf{H}_0+\e \mathsf{H}^{(1)}_1+\e^{2}\mathsf{H}^{(2)}_2
+\e^{3}\mathsf{H}^{(2)}_3+o(\e^{3})\\
&\mathsf{H}^{(2)}_2:=\mathsf{H}^{(1)}_2+\{  \tilde{S}_2, \mathsf{H}_0\}_{e}\,,
\qquad \mathsf{H}^{(2)}_3:=\mathsf{H}_3^{(1)}+\{\tilde{S}_2,\mathsf{H}_1^{(1)}\}_{e}\,.  
\end{aligned}
\end{equation}
We want to solve the equation 
\begin{equation}\label{omopato199}
\begin{aligned}
\mathsf{H}^{(2)}_2&=\mathsf{H}^{(1)}_2
+\overline{\omega}\cdot\partial_{\varphi} \tilde{S}_2+\{\tilde{S}_2 ,\mathsf{H}_0\}
= c+\int_{\mathbb{T}}  \mathfrak{B}_2( z)\,z\,dx\,+\mathsf{H}_{\mathcal{R}_2},
\end{aligned}
\end{equation}
where $\mathfrak{B}_2$ is some  pseudo differential and $2$-homogeneous
operator of order $-2$ (see Definition \ref{PseudoSIMho}), 
$c$ is some constant to be determined and $\mathsf{H}_{\mathcal{R}_2}$ (possibly different from the one in \eqref{hamiltonianalinearizzata}) is
a Hamiltonian with the form \eqref{HAMpseudo}. 
By Lemma \ref{pseudoHAM1} we have that 
$\mathsf{H}^{(1)}_2
+\{\tilde{S}_2 ,\mathsf{H}_0\}$ can be written in the form
\eqref{HAMpseudo} with, in particular,
\begin{equation}\label{tappo}
f_2(\bar{v}):=-\Psi_2(\overline{v})
+\frac{1}{4}\partial_{xx}(\beta_1^2)-\frac{1}{2}\beta_1 \overline{v}_x+\frac{1}{2}\overline{v} (\beta_1)_x,
\end{equation}
and some $\mathfrak{B}_2\in OPS^{-2}$, as in Definition \ref{PseudoSIMho},
up to a finite rank remainder.
%
%First, recalling \eqref{espansioLIE}, \eqref{espansioLIE100}, we need to note that
%the Hamiltonian $1/2\{S_1,\mathsf{H}_1\}+1/2\{S_1, \int \mathfrak{B}_1(z)\,z\,dx \}$
%is equivalent to
%\begin{equation}\label{tappo}
%\frac{1}{2}\int_{\mathbb{T}}\Big[ \bar{v}(\beta_1)_x-\beta_1 \bar{v}_x+\frac{\tilde{A}+\tilde{A}^*}{2}\Big]z\cdot zdx, \qquad 
%\tilde{A}:=-\beta_1\circ3\Lambda \partial_x\circ \bar{v}+\beta_1\circ J\circ\mathfrak{B}_1.
%\end{equation}
%Since $\tilde{A}$ is pseudo-differential (of order $-1$) one can show, using the results on compositions of pseudo-differential operators in \cite{FGP1},
%that $\tilde{A}+\tilde{A}^*$ is pseudo-differential of order $-2$.
%Similarly $\{S_2,\mathsf{H}_0\}$ is equivalent to the Hamiltonian 
%\begin{equation}\label{tappo2}
%\frac{1}{4}\int_{\mathbb{T}}\Big[(\beta_1^{2})_{xx}+[3\Lambda \partial_{x},(\beta_1^{2})_x]\Big]
%z\cdot zdx.
%\end{equation}
Hence the equation \eqref{omopato199} is equivalent to
%
%\begin{equation}\label{omopatologica}
%\mathsf{H}_2+\{S_{2},\mathsf{H}_0\}_{e}
%+\overline{\omega}\cdot\partial_{\varphi} \tilde{S}_2
%+\{ \tilde{S}_2, \mathsf{H}_0\}
%+\frac{1}{2} \int_{\T} \overline{v} z\,\partial_x (\beta_1 z)\,dx
%= \int  \mathfrak{B}_2( z)\,z\,dx\,.
%\end{equation}
%In particular
%\begin{align*}
%&\frac{1}{2}\{S_1, \mathsf{H}_1\}=\frac{1}{2} \int_{\T} \overline{v} z\,J (\beta_1 z)\,dx-\frac{1}{2}\int_{\T} \overline{v} z \,J \Pi_S[\beta_1 z]\,dx, \quad\frac{1}{2}\{S_1, \int \mathfrak{B}_1(z)\,z\,dx \}=- \int_{\T} \mathfrak{B}_1(z)\,J \Pi_S^{\perp}[\beta_1 z]\,dx.
%\end{align*}
%Since $\mathfrak{B}_1$ is a pseudo differential operator of order $-2$, the Hamiltonian above generates a vector field of order $-1$. We write
%\begin{equation}\label{HtildeR2}
%\tilde{\mathsf{H}}_{\mathcal{R}_2}:=\mathsf{H}_{\mathcal{R}_2}-\frac{1}{2}\int_{\T} \overline{v} z \,J \Pi_S[\beta_1 z]\,dx.
%\end{equation}
%Note that $\frac{1}{2}\int_{\T} v z \,J \Pi_S[\beta_1 z]\,dx$ generates a finite-rank vector field and $\B_2$ is $\mathrm{I}+O(\varepsilon^2)$, hence the terms of size $O(\varepsilon)$ 
%in $\B_1 \mathcal{L}_{\omega} \B_1^{-1}$ are not changed.
\begin{equation}\label{MisterY}
\overline{\omega}\cdot\partial_{\varphi} \beta_2 - (\beta_2)_x
+f_2(\bar{v})=c\,.
\end{equation}
Since $f_2$ in \eqref{tappo} has the form \eqref{simboOMO2} with $p=2$,
we look for a function $\beta_2$ of the same form in \eqref{simboOMO2}
with some coefficients $(\beta_2)_{j_1,j_2}\in\mathbb{C}$.
Hence equation \eqref{MisterY} reads
\begin{equation}\label{MisterY30}
\begin{aligned}
&\big[\og(j_1)+\og(j_2)-(j_1+j_2)\big](\beta_2)_{j_1,j_2}+(f_2)_{j_1,j_2}=0,
\quad {\rm for}\quad\og(j_1)+\og(j_2)-(j_1+j_2)\neq0\,,\\
&(f_2)_{j_1,j_2}=c,\quad {\rm for}\quad \og(j_1)+\og(j_2)-(j_1+j_2)=0\,.
\end{aligned}
\end{equation}
We have that, for $j_1, j_2\in S$, 
$
\og(j_1)+\og(j_2)-(j_1+j_2)=0
$ if and only if $j_1+j_2=0$, since $j_1j_2\neq -1$. 
The terms with $j_1=-j_2$ corresponds to the average in $x$ of the function 
$f_2(\bar{v})$. 
Hence we set
\begin{equation}\label{Bdue}
%\mathfrak{B}_2\stackrel{\eqref{tappo},\eqref{tappo2}}{:=}
%\frac{1}{4}[3\Lambda \partial_{x},(\beta_1^{2})_x]+\frac{1}{4}
%(\tilde{A}+\tilde{A}^{*})\,, \qquad 
c:=\frac{1}{2\pi}\int_{\mathbb{T}}f_{2}(\bar{v})dx.
\end{equation}
and we  evaluate explicitly it.
The functions $\Psi_2(\overline{v})$ and 
$\partial_{xx}(\beta_1^2)$ do not contribute since they have spatial zero average.
Recalling 
that $\beta_1=\frac{1}{3} (\Lambda\partial_x)^{-1} \overline{v}$ we have 
\[
\int_{\mathbb{T}}f_{2}(\bar{v})dx=
\frac{1}{6}\int_{\mathbb{T}}\Big((\Lambda^{-1}\partial_{x}^{-1}\partial_x\bar{v})\cdot\bar{v}
-(\Lambda^{-1}\partial_x^{-1}\bar{v})\cdot \bar{v}_x\Big)
dx=\frac{1}{3}\int_{\mathbb{T}}(\Lambda^{-1}\bar{v})\cdot\bar{v}dx
\stackrel{\eqref{Helmotz}}{=}\frac{1}{3}\int_{\mathbb{T}}(\bar{v}^2+\bar{v}^{2}_x)\,.
\]
Then the constant $c=c(\omega)$ (recall the \eqref{xiomega}) in \eqref{Bdue} is given by 
\begin{equation}\label{Cxi}
c(\oo)=\frac{1}{3} \sum_{j\in S} (1+j^2)\,\xi_j=\frac{2}{3} \sum_{j\in S^+}(1+j^2)\,\xi_j.
\end{equation}
By noting that
\begin{equation}\label{Drugo2}
\lVert \varepsilon^2\beta_2 \rVert^{\gamma, \calO}_s\lesssim_s \varepsilon^2 \quad \forall s\geq s_0,
\end{equation}
by \eqref{Napalmi}-\eqref{uranioi} with $i=1$ 
and using the assumption \eqref{IpotesiPiccolezzaIdeltaDP} 
with $\gotp_1$ sufficiently large
%given in \eqref{ipotesimu} 
the smallness assumption of Lemma \ref{differenzaFlussi} and the condition \eqref{filini} are satisfied.
In this case $\mathtt{q}\rightsquigarrow \mathtt{q}_1$, hence by \eqref{Napalmi}, \eqref{Napalmii} 
the bounds \eqref{docq}, \eqref{docqi} hold with 
$\mathtt{k}_1\rightsquigarrow\varepsilon$, $\mathtt{k}_2\rightsquigarrow \varepsilon$, $\mathtt{k}_3\rightsquigarrow \varepsilon$, $f\rightsquigarrow \mathfrak{I}_{\delta}$. Then Proposition \ref{ConjugationLemma} applies and $\mathcal{L}_2:=\Phi_2 \mathcal{L}_1 \Phi_2^{-1}$ with $i=2$ and $c_2(\omega):=c(\omega)$ given in \eqref{Cxi}. 
By \eqref{Drugo2}, \eqref{dude}, \eqref{JeTame}, \eqref{Round} we have
that $\widehat{\mathcal{Q}}_2\in \gotL_{\rho, p}(\calO)$ (with $\rho$, $p$ as in \eqref{tornado}) and 
\eqref{Napalmi}, \eqref{Napalmii} hold for $i=2$.
By \eqref{Drugo2}, \eqref{Napalmi}-\eqref{uranioi} for $i=1$, \eqref{ServelloniMazzantiVienDalMare}, we have 
that \eqref{Napalmi}-\eqref{uranioi} holds for $i=2$.
\medskip

\textbf{Steps $(\e^{3})$-$(\varepsilon^4)$-$(\varepsilon^5)$}. Consider $i=3,4, 5$. 
We proceed exactly as in the previous steps. 
We consider a change of coordinates $\B_i$ as the time-one flow map of
\begin{equation}\label{systemS45}
u_{\tau}=\Pi_S^{\perp}[\big(J\circ b_i(\tau)\big)\,u], \quad
b_i:=\frac{\varepsilon^i\beta_i}{1+\varepsilon^i\tau (\beta_i)_x}
\end{equation}
for some smooth function $\beta_i$ of the form \eqref{simboOMO2} (with $p=i$)
to be determined. 
Using Lemma \ref{pseudoHAM1} for the Hamiltonians of order $\e^{i}$, $i=3,4,5$,
we can choose 
$\beta_i$ in order to solve an equation like the following
\begin{equation}\label{homostep45}
\overline{\omega}\cdot\partial_{\varphi} \beta_i-(\beta_i)_x=f_i(\overline{v}),
\end{equation}
where $f_i$ is a homogeneous  function as in
 \eqref{simboOMO2} (with $p=i$).
The condition \eqref{GenericAssumption1} 
implies that the equation \eqref{homostep45} for 
$i=4$ is solved up to remainders of the form
\begin{equation}\label{dXi}
d(\oo):=d(\xi(\oo))=\sum_{j_1, j_2\in S} 
\mathtt{d}(j_1, j_2)\xi_{j_1}\xi_{j_2}\,.
\end{equation}
By \eqref{GenericAssumptionbis} there are no small divisors for \eqref{homostep45} 
if $i=3$ or $i=5$.
By \eqref{Napalmi}, \eqref{Napalmii} and by noting that
\begin{equation}\label{Drugo45}
\lVert \varepsilon^i\beta_i \rVert^{\gamma, \calO}_s
\lesssim_s \varepsilon^i \quad i=3,4\,, 5\,,\,\, \forall s\geq s_0\,,
\end{equation}
the smallness assumption of Lemma \ref{differenzaFlussi} and the condition \eqref{filini} are satisfied for the system \eqref{systemS45}. Arguing as in the previous steps we obtain that 
$\mathcal{L}_5:=\B_5\B_4 \mathcal{L}_3 \B_4^{-1}\B_5^{-1}$ 
has the form \eqref{iesimo} with $i=5$ and $c_5=c_2+\varepsilon^2 d(\xi)$. 
Moreover the bounds \eqref{Napalmi}-\eqref{uranioi} hold for $i=5$.

\medskip

Now we apply Proposition $3.6$ in \cite{FGP1} (or Corollary $3.6$ in \cite{FGMP})
in order to make constant the coefficient $a_5$ of the linearized operator $\mathcal{L}_5$, namely we find $\beta$ such that
\begin{equation}\label{borina}
\omega\cdot\partial_{\varphi} \beta-(1+\varepsilon^2 c(\omega)+\varepsilon^4 d(\omega)+a_5(\varphi, x))(1+\beta_x)=\mbox{constant}.
\end{equation}
Note that, by \eqref{uranioi} with $i=5$ and \eqref{IpotesiPiccolezzaIdeltaDP}, the smallness condition  \eqref{picci} is satisfied by the function $a_{5}$. We have the following.

\begin{prop}\label{corollarioStraight}
There exists $\beta^{(\infty)}(\varphi, x)$ such that 
$(\varphi, x)\mapsto (\varphi, x+\beta^{(\infty)}(\varphi, x))$ 
is a diffeomorphism of the torus $\T^{\nu+1}$ with the following estimates (recall \eqref{condizMel1}),
\begin{equation}\label{saintetienne}
 \lVert \beta^{(\infty)}\rVert_s^{\g, \calO_0}
 \lesssim_s \g^{-1}\lVert a_5 \rVert_{s+2\tau+4}^{\g, \calO_{\infty}^{2 \g}}\,, 
 \quad \forall s\geq s_0
\end{equation}
and the following holds.
%\comment{dove abbiamo definito $\calO_{\infty}^{2 \g}$?}
If $\beta$ is the function such that $(\varphi, x)\mapsto (\varphi, x+\beta(\varphi, x))$ is the inverse of the above diffeomorphism and $\B_6$ is the flow of the Hamiltonian PDE
\begin{equation}\label{saintetienne10}
u_{\tau}=\Pi_S^{\perp}[\big(J \circ b(\tau)\big)\,u)]\,, 
\qquad b(\tau):=b(\tau, \varphi, x)=\frac{\beta}{1+\tau \beta_x}\,,
\end{equation}
then the conjugated of the operator $\mathcal{L}_5$ 
in \eqref{iesimo} with $i=5$ is
\begin{equation}\label{BrigitteBardot}
\mathcal{L}_6:=\B_{6}\,\mathcal{L}_5\,\B^{-1}_{6}
= \Pi_S^{\perp}\Big(\omega\cdot\partial_{\varphi}-m J+\mathcal{Q}_6  \Big)\,,
\end{equation}
where $\mathcal{Q}_6=\op(\mathtt{q}_6)+\widehat{\mathcal{Q}}_6$ is of order $-1$, as in Proposition \ref{ConjugationLemma}, and $m$ is a constant such that
\begin{equation}\label{clinica}
\lvert m-1-\e^2 c(\oo)-\e^4 d(\oo) \rvert^{\gamma}\lesssim \e^{10}\g^{-2}\,, 
\quad |m|^{lip}\lesssim 1\,,
\qquad \lvert \Delta_{12} m   \rvert\lesssim 
\varepsilon\lVert i_1-i_2 \rVert_{s_0+2}\,,
\quad \forall \omega\in \calO^{2\gamma}_{\infty}\,.
\end{equation}
Moreover, for any $s\geq s_0$,
\begin{equation}\label{todd}
\lvert \mathtt{q}_6 \rvert^{\gamma, \calO^{2\gamma}_{\infty}}_{-1, s, \alpha}
\lesssim_s \varepsilon(1+\lVert \mathfrak{I}_{\delta} 
\rVert^{\gamma, \calO_0}_{s+\hat{\s}})\,,
\qquad
\lvert \Delta_{12} \mathtt{q}_6   \rvert_{-1, p, \alpha}\lesssim_{p} 
\varepsilon \gamma^{-1} (1+\lVert \mathfrak{I}_{\delta} \rVert_{p+\hat{\s}})
\lVert i_1-i_2 \rVert_{p+\hat{\s}}\,,
\end{equation}
and $\widehat{\calQ}_{6}\in \gotL_{\rho, p}$,
for $s_0\le s\le \mathcal{S}$, satisfies
\begin{align}
&\mathbb{M}^{\gamma}_{\widehat{\mathcal{Q}}_6}(s, \tb) 
\lesssim_s 
\varepsilon (1+\lVert \mathfrak{I}_{\delta}
\rVert_{s+\hat{\s}}^{\gamma, \calO_0})\,, 
\quad 0\le \tb\le \rho-2\,,\label{capoinb}\\
&\mathbb{M}_{ \Delta_{12} \widehat{\mathcal{Q}}_6  }(p, \tb)
\lesssim_{p} 
\varepsilon \gamma^{-1} (1+\lVert \mathfrak{I}_{\delta} 
\rVert_{p+\hat{\s}})
\lVert i_1-i_2 \rVert_{p+\hat{\s}}\,,\qquad 0\le \tb\le \rho-3\label{capoinb2}
\end{align}
with $\hat{\s}=\s_0+\s_9+\rho+s_1-s_0$ for some $\s_9$, possibly larger than $\s_8$ (recall \eqref{iesimo} with $i=5$ and $s_1$ given in Proposition $3.6$ in \cite{FGP1}).
\end{prop}

\begin{proof}
The first order linear differential operator (recall \eqref{Cxi}, \eqref{dXi})
\begin{equation}\label{neve}
\oo\cdot\del_{\f}-\big(1+\varepsilon^2 c(\oo)+\varepsilon^4 d(\xi)+a_5(\varphi, x) \big)\partial_x
\end{equation}
defined on $H_{S^{\perp}}^s(\T^{\nu+1})$ is associated to the vector field on $\T^{\nu+1}$
\begin{equation}\label{Lubiana}
X_0:=\omega\cdot\frac{\partial}{\partial \varphi} - 
\big(1+\varepsilon^2 c(\oo)+\varepsilon^4 d(\xi)
+a_5(\varphi, x) \big)\frac{\partial}{\partial x}\,.
\end{equation}
For  $\gotp_{3}$ in \eqref{IpotesiPiccolezzaIdeltaDP} large enough, i.e. if 
$\gotp_{3}\gg \s_0+s_1+\s_8$, 
%By assumption \eqref{ipotesimu} 
and by \eqref{uranioi} with $i=5$,  we have
\begin{equation*}
C(s_1)\,\gamma^{-1} \lVert a_5 \rVert^{\gamma, \calO_{\infty}^{2 \gamma}}_{s_1}\le C(s_1) \varepsilon^{4-3 a}=\delta^* \ll 1,
\end{equation*}
provided that $\varepsilon$ is small enough. 
This is the condition \eqref{picci}, hence Proposition $3.6$ in \cite{FGP1}  applies to the vector field \eqref{Lubiana}. Thus there exist $\beta^{(\infty)}$ and $m$ such that the bounds \eqref{saintetienne}, \eqref{clinica} hold.
Moreover the operator \eqref{neve} conjugated by the transformation
\[
\mathcal{T}_{\beta^{(\infty)}}\colon  u(\varphi, x)\mapsto 
u(\varphi, x+\beta^{(\infty)}(\varphi, x))\,,
\]
%where $(\varphi, x)\mapsto (\varphi, x+\beta^{(\infty)}(\varphi, x))
%$ is a diffeomorphism of $\T^{\nu+1}$,
is associated to the vector field
\[
(\mathcal{T}_{\beta^{(\infty)}})_* X_0=\omega\cdot\frac{\partial}{\partial \varphi}+
(\mathcal{T}_{\beta^{(\infty)}})^{-1}\Big(\omega\cdot\partial_{\varphi} \beta^{(\infty)}
-(1+\varepsilon^2 c(\oo)+\varepsilon^4 d(\xi)+a_5(\varphi, x)(1+\beta_x^{(\infty)})\Big)\frac{\partial}{\partial x}\,,
\]
and by Proposition $3.6$ in \cite{FGP1} we have that
\begin{equation}
\omega\cdot\partial_{\varphi} \beta^{(\infty)}-(1+\e^2 c(\omega)+\e^4 d(\oo)+a_5(\varphi, x))(1+\beta_x^{(\infty)})=-m\,.
\end{equation}
By Lemma $11.4$ in \cite{Ono} the function $\beta$ satisfies the bound \eqref{saintetienne}.
By \eqref{saintetienne} and $\gotp_1$ large enough,
%\eqref{ipotesimu}, 
for $\varepsilon$ small enough, 
the function $\beta$ satisfies 
the smallness  condition of Lemma \ref{differenzaFlussi},
 indeed
\begin{equation}\label{eppinger}
\lVert \beta \rVert^{\g, \calO_0}_{s_0+\s_1}
\le C(s_1) \gamma^{-1} \lVert a_5 
\rVert^{\g, \calO_{\infty}^{2 \gamma}}_{s_0+\s_1+2\tau+4}
\stackrel{(\ref{uranioi})}{\le} 
C(s_1)\,\gamma^{-1} \big( \varepsilon^{6}
+\varepsilon\lVert \mathfrak{I}_{\delta}\rVert^{\g, \calO_0}_{s_0+\gotp_1} \big)\stackrel{(\ref{IpotesiPiccolezzaIdeltaDP})}{\le} 
C(s_1)\varepsilon^{4- 3 a}\,.
\end{equation}
Hence $\B_6$ is well defined.
By \eqref{uranioi}, \eqref{Napalmi}, \eqref{Napalmii}, $i=5$,  the bounds \eqref{docq}, \eqref{docqi} hold with $\tk_1\rightsquigarrow \varepsilon^6$, $\tk_2\rightsquigarrow \varepsilon$, $\tk_3\rightsquigarrow\varepsilon \gamma^{-1}$ and Proposition \ref{ConjugationLemma} applies
and the thesis follows.
\end{proof}

\noindent
Let us define
\begin{equation}\label{fi}
\B:=\B_{6}\circ \B_5\circ\B_4\circ\B_3\circ \B_2\circ \B_1\,.
\end{equation}
Then the Hamiltonian of the operator $\mathcal{L}_6$ is (recall \eqref{hamiltonianalinearizzata} and \eqref{BrigitteBardot})
\begin{align}\label{Divac}
\mathcal{K}&:=\mathsf{H}\circ \B^{-1}=\mathsf{H}_0+\varepsilon \,\mathcal{K}_1
+\varepsilon^2\Big(\mathtt{Z}_0+\,\mathcal{K}_2\Big)
+\varepsilon^3 \,\mathcal{K}_3+o(\varepsilon^3)  \,,
%\\
%&\mathcal{K}_1:=\mathsf{H}_1^{(1)},\qquad \mathcal{K}_3:=\mathsf{H}_3^{(3)}\,,
%\label{Divac2}
\end{align}
where, recalling \eqref{espansioLIE}, \eqref{espansioLIE2}, \eqref{espansioLIE22},
\eqref{Cxi}, 
\begin{equation}\label{georgeHill10}
\mathcal{K}_1:=\mathsf{H}_1^{(1)},\qquad 
\mathtt{Z}_0+ \mathcal{K}_2:=\mathsf{H}_2^{(2)} \,,\qquad 
\mathtt{Z}_0:=\mathbb{A} \xi\cdot \eta+\frac{c(\oo)}{2}\int_{\T} z^2\,dx\,,
%\qquad
% \mathcal{K}_3:=\mathsf{H}_3^{(2)}+\{\hat{S}_3,\mathsf{H}_0\}_{e},
\end{equation}
and $\mathcal{K}_3$ is some pseudo differential $3$-homogeneous Hamiltonian as in \eqref{HAMpseudo}
with the corresponding function $f_{3}(\bar{v})=0$.
Notice also that $\{\mathtt{H}_0, \mathtt{Z}_0\}_e=0$.
%
%where $\mathsf{H}_1^{(1)}$, 
%$\mathtt{Z}_0$, $\mathcal{K}_2$ , $\mathsf{H}_3^{(3)}$
%collect all the terms of order $\varepsilon$, $\varepsilon^2$ and $\varepsilon^3$ 
%respectively (recall 
%\eqref{espansioLIE}, \eqref{georgeHill10}).
%%\eqref{Buno}, \eqref{Bdue}).
%Note that $\{\mathtt{H}_0, \mathtt{Z}_0\}_e=0$.

\subsection{Linear Birkhoff Normal Form}\label{LinearBNF}
 The aim of this section is to eliminate $\mathcal{K}_1$, $\mathcal{K}_3$ and normalize the Hamiltonian $\mathcal{K}_2$ from \eqref{Divac}. 
 Our first point is that the $-1$ smoothing  remainder 
 $o(\varepsilon^3)$ belongs to a special class of 
 operators defined in Definition \ref{Cuno} 
 and denoted by $\gotC_{-1}$ . It turns out that this 
 class is preserved under the changes of variables used
 in  the linear Birkhoff normal form procedure (see Lemmata \ref{IncluecompoCLASSI}, \ref{commutatoC1}).
\begin{remark}\label{tazzadiBatman}
In the following steps of linear Birkhoff normal 
form we shall use the relation
\begin{equation}\label{disperato}
\sum_{i=1}^{\nu} \overline{\jmath}_i\ell_i+j'-j=0 \,,
\qquad \mbox{if}\,\,\,\,\lvert \ell \rvert\le 3,\,\,\,\,\, \forall j, \,j'\in S^c\,,
\end{equation}
which holds by the conservation of momentum.
\end{remark}

\smallskip
\noindent

\begin{lem}\label{glenn}
Recall \eqref{BrigitteBardot}.  We have
\begin{equation}\label{elle3}
\mathcal{L}_6 =\Pi_S^{\perp}\big(\omega\cdot\partial_{\varphi}-m J-\varepsilon X_{\mathcal{K}_1}-\varepsilon^2 X_{\mathcal{K}_2}-\varepsilon^3 X_{\mathcal{K}_3}+\mathfrak{R}\big)
\end{equation}
where $X_{\mathcal{K}_i}:=J \nabla \mathcal{K}_i$, $i=1, 2, 3, $
are almost diagonal 
and in $\mathfrak{C}_{-1}(\calO^{2\gamma}_{\infty})$ (recall  Definitions \ref{almostDIAG} and \ref{Cuno})
satisfying
\begin{equation}\label{alieni4}
\mathbb{B}^{\g}_{\varepsilon^{k} J \nabla \mathcal{K}_k}(s) 
\le \varepsilon^{k} C(s)\,, \quad k=1,2,3\,.
\end{equation}
The remainder $\mathfrak{R}$ belongs to $\gotC_{-1}(\calO^{2\gamma}_{\infty})$ and satisfies
\begin{equation}\label{fracchia2}
\begin{aligned}
\mathfrak{R}&:=\mathcal{Q}_6
+\varepsilon J \nabla \mathcal{K}_1
+\varepsilon^2 J \nabla \mathcal{K}_2
+\varepsilon^3 J \nabla \mathcal{K}_3\,.
\end{aligned}
\end{equation}
%belongs to $\gotC_{-1}(\calO^{2\gamma}_{\infty})$ with
\begin{equation}\label{contadini}
\begin{aligned}
\mathbb{B}^{\gamma}_{ \mathfrak{R}}(s)\lesssim_s  
\varepsilon^{4-3 a}+\varepsilon \gamma^{-1}\lVert \mathfrak{I}_{\delta}
\rVert^{\gamma, \calO_0}_{s+\hat{\s}}\,, \qquad 
 \mathbb{B}_{ \Delta_{12} \mathfrak{R}   }(s_0)\lesssim 
 \varepsilon \gamma^{-1} 
\lVert i_1-i_2 \rVert_{s_0+\hat{\s}}\,,
\end{aligned}
\end{equation}
for $\hat{\s}$ given in Proposition \ref{corollarioStraight}.
\end{lem}

\begin{proof}
By the discussion of subsection \ref{preliminare}
%\eqref{espansioLIE100}, \eqref{omopato199}, \eqref{georgeHill10}, \eqref{omopatologica100} 
$\mathcal{K}_i$, $i=1,2,3$, are of the form \eqref{HAMpseudo}
with $f_{i}=0$ for $i=1,2,3$. Hence the vector field 
 $X_{\mathcal{K}_i}$  are pseudo differential of order $-1$ 
 up to a finite rank term. In addition, they are almost diagonal by \eqref{simboOMO} and the momentum condition
 \eqref{disperato}.
% s
% ince $H_{\mathcal{R}_*, 3}$ generates a finite rank vector field and 
% $S_1$, $S_2$, $\hat{S}_3$ in \eqref{esse1}, \eqref{s3cappuccio} 
% generate a pseudo differential vector field. 
By Lemma \ref{IncluecompoCLASSI}-(ii) $\varepsilon X_{\mathcal{K}_1}, \varepsilon^2 X_{\mathcal{K}_2}, \varepsilon^3 X_{\mathcal{K}_3}$ belong to $\gotC_{-1}$ and, by \eqref{Buno}, \eqref{Drugo}, \eqref{Bdue}, \eqref{Drugo2}, \eqref{Drugo45}, satisfy \eqref{alieni4}.
By Proposition \ref{corollarioStraight}, the choice of $\rho$ as in \eqref{tornado} and by Lemma \ref{IncluecompoCLASSI}-$(i)$ taking $p=s_0$ and $\gotp_1$ large enough,
%as in \eqref{ipotesimu} 
$\mathcal{Q}_6\in\gotC_{-1}$. Thus $\mathfrak{R}\in\gotC_{-1}$.\\
Note that only $\calQ_6$ in \eqref{fracchia2} depend on the torus embedding $i_{\delta}$, then the second bound in \eqref{contadini} follows by Lemma \ref{IncluecompoCLASSI}-$(i)$, \eqref{todd} and \eqref{capoinb2}. The first  bound in \eqref{contadini} follows by the Proposition \ref{ConjugationLemma} , the bound \eqref{saintetienne} and \eqref{uranioi} for $i=5$.
\end{proof}

In order to normalize the vector fields $\varepsilon^i J \nabla \mathcal{K}_i$ we will look for changes of coordinates $\Upsilon_i$ of the form \eqref{Fi1dp} generated as one-time flow of Hamiltonians $H_{\mathtt{A}_i}$ of the form \eqref{HAMa1}, where $\mathtt{A}_i$ are almost diagonal and have the form \eqref{us}, \eqref{them}. 
We remark that the Hamiltonian $\varepsilon^2\mathtt{Z}_0$ will be left invariant by these change of coordinates, since $\{ \mathtt{H}_0, \mathtt{Z}_0\}_e=0$.
At any step of the procedure we shall verify that  $J \mathtt{B}_i$ (see \eqref{us}, \eqref{them}) are almost diagonal and belong to $\mathfrak{C}_{-1}$ in order to apply Lemma \ref{delfino}, which guarantees well-posedness and tame estimates of $\Upsilon_i$.

%\comment{Cade un po' dal nulla, non si capisce. Dobbiamo spiegare, forse legare alla mini spiegazione di strategia ll'inizio della sezione 7.\\
%Chi sono i B e A? bisogna riferirci all'appendice}

\medskip

\noindent
{\bf Step one (order $\varepsilon$).} At this step we want to eliminate $\varepsilon X_{\mathcal{K}_1}$ from \eqref{elle3}. We have
\begin{align}
\mathcal{K}^{(1)}&:=\mathcal{K}\circ \Upsilon_1^{-1}= \,\mathsf{H}_0+\varepsilon \mathcal{K}^{(1)}_1 +\varepsilon^2\Big(\mathtt{Z}_0+ \mathcal{K}^{(1)}_2\Big)+\varepsilon^3 \mathcal{K}_3^{(1)}+o(\varepsilon^3)\,, \label{Kone10}\\ 
\mathcal{K}^{(1)}_1&:= \mathcal{K}_1+\{H_{\mathtt{A}_1},\mathsf{H}_0\}_e=
\mathcal{K}_1+\overline{\omega}\cdot\partial_{\varphi} H_{\mathtt{A}_1}
+\{ H_{\mathtt{A}_1}, \mathsf{H}_0 \}\,,\nonumber\\
\mathcal{K}^{(1)}_2&:= 
\mathcal{K}_2+\frac{1}{2}\{ H_{\mathtt{A}_1}, \{H_{\mathtt{A}_1},\mathsf{H}_0\}_e\}_e
+\{H_{\mathtt{A}_1}, \mathcal{K}_1\}_e\label{Kone}
\\ 
&=\mathcal{K}_2+\frac{1}{2}\{ H_{\mathtt{A}_1}, \mathcal{K}_1^{(1)}\}
+\frac{1}{2}\{H_{\mathtt{A}_1}, \mathcal{K}_1\}\,.\nonumber
\end{align}
We choose $\mathtt{A}_1$ such that
\begin{equation}\label{equazioneEps}
\mathcal{K}_1^{(1)}=\overline{\omega}\cdot\partial_{\varphi} H_{\mathtt{A}_1}
+\{ H_{\mathtt{A}_1}, \mathsf{H}_0 \}+\mathcal{K}_1=0.
\end{equation}
%or, equivalently, at the vector field level
%\begin{equation}
%-\overline{\omega}\cdot\partial_{\varphi} X_{\mathtt{A}_1}+m [J, X_{\mathtt{A}_1}]+\varepsilon J \mathfrak{B}_1=0.
%\end{equation}
Recalling that $\mathcal{K}_1:=\mathsf{H}_1^{(1)}$,
we have (see  \eqref{Buno}, \eqref{espansioLIE100})
\begin{equation}\label{lupus}
\begin{aligned}
&\mathcal{K}_1(u)=\sum_{j, j'\in S^c} (\mathfrak{B}_1)_j^{j'}(\varphi)\,u_{j'}\,u_{-j}. 
%& (\mathfrak{B}_1)_j^{j'}(\ell)=\frac{(1+(j-j')^2)(1-j j')\sqrt{\xi_{j-j'}}}{(1+j^2)(1+j'^2)}, \quad j, j'\in S^c,\, \,j-j'\in S, \,\,\ell=\mathtt{l}(j-j'),
\end{aligned}
\end{equation}
Then we choose $\mathtt{B}_1=\mathfrak{B}_1$ in \eqref{us}. By recalling the definition of $\mathfrak{B}_1$ in \eqref{Buno} it is easy to see that $J \mathtt{B}_1\in\gotC_{-1}$, since it is a pseudo differential operator of order $-1$. Moreover it is almost diagonal because $J$, $3\Lambda\partial_x$ are diagonal operators and $\beta_1$ is a function supported on the finite set $S$.

\noindent
Given $X\in \mathfrak{C}_{-1}$ 
to shorten the notation in the following lemma we write (recall \eqref{LieBracketDP})
%and a linear operator $Y : H^{s}\to H^{s-m}$, for some $m$,
\begin{equation}\label{def:AD}
\mathrm{ad}_{X}[\cdot]:=[X,\cdot]\,.
\end{equation}
Under   this notation we have the following lemma.

\begin{lem}\label{LBNF1}
The transformed operator is (recall \eqref{elle3}, \eqref{Kone10})
\begin{equation}\label{L4dp}
\mathcal{L}_7 h:=\Upsilon_1 \mathcal{L}_6 \Upsilon_1^{-1} =\Pi_S^{\perp} \Big(\omega\cdot\partial_{\varphi}-m J-\varepsilon^2 X_{\mathcal{K}_2^{(1)}}-\varepsilon^3 X_{\mathcal{K}_3^{(1)}}+\mathcal{R}_7\Big)
\end{equation}
where 
\begin{align}
X_{\mathcal{K}_2^{(1)}}&:=J\nabla \mathcal{K}_2^{(1)} =X_{\mathcal{K}_2}+\mathrm{ad}_{X_{\mathtt{A}_1}}[X_{\mathcal{K}_1}]+\frac{1}{2}\mathrm{ad}^2_{X_{\mathtt{A}_1}}[\omega\cdot\partial_{\varphi}-m J],\label{K12}\\
X_{\mathcal{K}_3^{(1)}}&:=J\nabla \mathcal{K}_3^{(1)} = X_{\mathcal{K}_3}+\mathrm{ad}_{X_{\mathtt{A}_1}}[X_{\mathcal{K}_2}]+\frac{1}{2}\mathrm{ad}^2_{X_{\mathtt{A}_1}}[X_{\mathcal{K}_1}]+\frac{1}{6}\mathrm{ad}^3_{X_{\mathtt{A}_1}}[\omega\cdot\partial_{\varphi}-m J], \label{K13}
\end{align}
the operator $\RR_{7}\in \gotC_{-1}$ 
with
\begin{align}\label{caffelatte}
\mathbb{B}^{\g}_{\mathcal{R}_7}(s)  \lesssim_s 
\varepsilon^{4-3 a}+\varepsilon\g^{-1}
\lVert \mathfrak{I}_{\delta}\rVert_{s+\tilde{\sigma}}^{\gamma, \calO_0}, \qquad 
\mathbb{B}_{\Delta_{12} \mathcal{R}_7  }(s_0)  \lesssim
\varepsilon \gamma^{-1}(1+\lVert \mathfrak{I}_{\delta}\rVert_{s_0+\tilde{\sigma}})
\lVert i_1-i_2\rVert_{s_0+\tilde{\sigma}},
\end{align}
 for some $\tilde{\sigma}\geq\hat{\s}$
 (recall the loss of regularity in \eqref{todd}, \eqref{capoinb}, \eqref{capoinb2}).
\end{lem}

\begin{proof}
By using \eqref{charlie} we have that
formul\ae \,\eqref{K12}, \eqref{K13} hold and that
\begin{equation}\label{R7}
\begin{aligned}
\mathcal{R}_7:&=\mathfrak{R}+\varepsilon\,\mathrm{ad}_{X_{\mathtt{A}_1}}[\varepsilon^3 X_{\mathcal{K}_3}+\mathfrak{R}]+\frac{\varepsilon^2}{2}\mathrm{ad}^2_{X_{\mathtt{A}_1}}[-\varepsilon^2 X_{\mathcal{K}_2}-\varepsilon^3 X_{\mathcal{K}_3}+\mathfrak{R}]\\
&+\frac{\varepsilon^3}{6}\mathrm{ad}^3_{X_{\mathtt{A}_1}}[-\varepsilon X_{\mathcal{K}_1}
-\varepsilon^2 X_{\mathcal{K}_2}-\varepsilon^3 X_{\mathcal{K}_3}
+\mathfrak{R}]+\sum_{k\geq 4} \frac{\varepsilon^k}{k!}\mathrm{ad}^k_{X_{\mathtt{A}_1}}[\mathcal{L}_6].
\end{aligned}
\end{equation}
For $Y=Y(i)\in \mathfrak{C}_{-1}$ 
define $\mathcal{Z}_n:=\sum_{k\geq n} \frac{\varepsilon^k}{k!}\mathrm{ad}_{X_{\mathtt{A}_1}}^k[Y]$ for any $n\geq 1$.
By 
Lemma \ref{commutatoC1}, and using  
  \eqref{bojack}, we deduce that
\begin{equation}\label{alieni2}
\begin{aligned}
&\mathbb{B}^{\g}_{\mathcal{Z}_n}(s)\lesssim_s 
%\mathbb{B}_{X_{\mathtt{A}_1}}^{\g}(s, \tb)\mathbb{B}^{\g}_Y(s_0, \tb)  \sum_{k\geq n} \frac{\varepsilon^k}{k!}%\big(\mathbb{B}_{X_{\mathtt{A}_1}}^{\g}(s_0, \tb)\big)^{k-1}+\mathbb{B}^{\g}_Y(s, \tb)\sum_{k\geq n} \frac{\varepsilon^k}{k!}\big(\mathbb{B}_{X_{\mathtt{A}_1}}^{\g}%(s_0, \tb)\big)^k \\
%& 
%\le_{s} 
C(s, n)\mathbb{B}^{\g}_Y(s_0)+ C(s_0, n) \mathbb{B}^{\g}_Y(s)\\
&\mathbb{B}_{\Delta_{12} \mathcal{Z}_n}(s_0)\lesssim C(s_0, n)\,\Big(\mathbb{B}_{Y}(s_0)+\mathbb{B}_{\Delta_{12} Y}(s_0)  \Big)
\end{aligned}
\end{equation}
for $\varepsilon$ small enough.
In \eqref{R7} there are terms of the form $\mathrm{ad}^k_{X_{\mathtt{A}_1}}[Y]$, for some $k\geq 1$, with $Y=X_{\mathcal{K}_1}, X_{\mathcal{K}_2}, X_{\mathcal{K}_3}, \mathfrak{R}$ which belong to $\gotC_{-1}$ by Lemmata \ref{IncluecompoCLASSI} and \ref{glenn}.
We note that by \eqref{equazioneEps}
\begin{equation}\label{alieni3}
\mathrm{ad}_{X_{\mathtt{A}_1}}[\omega\cdot\partial_{\varphi}-m J]=-\varepsilon X_{\mathcal{K}_1}-(\omega-\overline{\omega})\cdot\partial_{\varphi} X_{\mathtt{A}_1}-(m-1) [X_{\mathtt{A}_1}, J]\in\gotC_{-1}
\end{equation}
since $\mathtt{A}_1$ is almost diagonal. Hence $(\omega-\overline{\omega})\cdot\partial_{\varphi}X_{\mathtt{A}_1}, [X_{\mathtt{A}_1}, J] \in\gotC_{-1}$ (see the proof of Lemma \ref{delfino}) and by Lemma \ref{IncluecompoCLASSI}-$(iii)$ the remainder $\mathcal{R}_7\in\gotC_{-1}$.
By \eqref{R7}, \eqref{alieni}, \eqref{chegioia}, \eqref{alieni2}, \eqref{alieni3}, \eqref{alieni4}, \eqref{contadini} and the fact that $\lvert \omega-\overline{\omega} \rvert\lesssim \varepsilon^2$ we get the bounds
\eqref{caffelatte}.
\end{proof}

\noindent
{\bf Step two (order $\varepsilon^2$).} At this step we want to normalize $\varepsilon^2 X_{\mathcal{K}_2^{(1)}}$ from \eqref{L4dp}. We have

\begin{equation}\label{Kone1000}
\begin{aligned}
\mathcal{K}^{(2)}&:=\mathcal{K}^{(1)}\circ \Upsilon_2^{-1}= 
\,\mathsf{H}_0+\varepsilon^2\Big( \mathtt{Z}_0+ \mathcal{K}^{(2)}_2\Big)
+\varepsilon^3 \mathcal{K}_3^{(2)}+o(\varepsilon^3)\,, 
\qquad \mathcal{K}^{(2)}_2:=\{ H_{\mathtt{A}_2}, \mathsf{H}_0\}_e
+\mathcal{K}^{(1)}_2\,,
\end{aligned}
\end{equation}
where $\mathcal{K}_2^{(1)}$ is given in \eqref{Kone} (see also \eqref{K12}).
We choose $\mathtt{A}_2$ in order to solve the following equation
\begin{equation}\label{seven}
\overline{\omega}\cdot\partial_{\varphi} H_{\mathtt{A}_2}+\{H_{\mathtt{A}_2}, \mathsf{H}_0\}=\mathcal{K}^{(1)}_2-\Pi_{\mbox{Ker}(\mathsf{H}_0)}\mathcal{K}^{(1)}_2\,.
\end{equation}
Hence we choose $\mathtt{B}_2= \nabla \Pi_{\mbox{Rg}(\mathsf{H}_0)} \mathcal{K}^{(1)}_2(u)$ in \eqref{us}. Note that $X_{\mathcal{K}_2}$ is pseudo differential of order $-1$ and $J \mathtt{A}_1$, $X_{\mathcal{K}_1}$ belong to $\gotC_{-1}$ and so also their Poisson brackets. Hence $J \mathtt{B}_2\in\gotC_{-1}$. By Remark \ref{remarkalmost} we have that $J \mathtt{B}_2$ is also almost diagonal.

%Here we have to deal with the problem, coming from resonances, explained....  
 In order to perform the third step in the linear BNF 
 we need to explicitly compute 
 the corrections $O(\e^2)$ coming from $\Pi_{\mbox{Ker}(\mathsf{H}_0)}\mathcal{K}^{(1)}_2$.
The point is that a priori, it is not clear whether the resonant terms $\Pi_{\mbox{Ker}(\mathsf{H}_0)}\mathcal{K}^{(1)}_2$ are supported  only  on trivial resonances. 
%Of course one could try to mimic the strategy used in the weak BNF 
%and use the constants of motion in order to prove this, 
%however this seems a quite delicate question. 
Our approach is then to show that the normal form we obtain must necessarily coincide with the formal one, which is relatively easy to compute.

\begin{defi}\label{chebellosiproietta}
Recalling the notations used in Section \ref{SezioneWBNF}, we denote by
$\Pi^{d_z\leq k}$, respectively  $\Pi^{d_z=k}$,
 the projector of a homogenous Hamiltonian of degree $n$ 
 on the monomials with degree less or equal than $k$, respectively equal $k$,
 in the normal variable $z$, i.e. 
\begin{equation*}
\Pi^{d_z \leq k} H^{(n)}:=H^{(n, \leq k)}, \quad \Pi^{d_z = k} H^{(n)}:=H^{(n, k)}.
\end{equation*}
We denote by 
 $\Pi_{\mathrm{triv}}$
  the projection onto trivial resonances (of the form \eqref{coppiette}), 
  i.e. monomials of the form $$u_{j}u_{-j}u_i u_{-i}\dots u_{k}u_{-k}.$$
\end{defi}

\noindent
The following proposition allows to easily compute the resonant terms 
$\Pi_{\mbox{Ker}(\mathsf{H}_0)}\mathcal{K}^{(1)}_2$ in \eqref{chiK2}.

\begin{teor}{\bf (Normal form identification).}\label{PartialWeak}
Consider the symplectic change of coordinates 
 $A_{\varepsilon}$ in  \eqref{AepsilonDP}. 
Then
\begin{equation}\label{mare}
%\varepsilon^2\Big(\frac{c(\oo)}{2} \int_{\T} z^2\,dx
%+\Pi_{\mbox{\rm Ker}(\mathsf{H}_0)}\mathcal{K}^{(1)}_2\Big)
\Pi^{d_z=2}\Pi_{\mbox{Ker}(\mathtt{H}_0)} \Big(\mathtt{Z}_0+ \mathcal{K}^{(1)}_2\Big)=
\Big[\Pi_{\mbox{\rm triv}}\Pi^{d_z=2} \Big( \frac{1}{2}\{  \mathfrak{F}^{(3)} , H^{(3)}\} \Big)\Big]\circ A_{{1{|_{ \substack{y=0\\ \theta=\f}}}}}\,,
\end{equation}
where $A_1:=A_{{\varepsilon}_{|_{\varepsilon=1}}}$, $\mathtt{H}_0$ is in \eqref{leHamiltonianeLin} and we set 
(recalling \eqref{adjActHAM}) 
$\mathfrak{F}^{(3)}:=[\mathrm{ad}_{H^{(2)}}]^{-1}H^{(3)}$
with $H^{(3)}$ in \eqref{Rinco}.
\end{teor}

\begin{proof}
The proof is postponed to the Appendix \ref{app:Weak}.
\end{proof}

\noindent
As a consequence of the identification above, we have,
by \eqref{mare}, \eqref{DefF3}, \eqref{Rinco},
\begin{equation}
\Big(\frac{c(\oo)}{2}\sum_{j\in S^c}  \lvert u_j \rvert^2+\Pi_{\mbox{Ker}(\mathsf{H}_0)}\mathcal{K}^{(1)}_2\Big)=\frac{1}{2}\sum_{j\in S^c} \lal_j \lvert u_j \rvert^2
\end{equation}
where, by \eqref{parco1},
\begin{equation}\label{lambda}
\lal_j:=\sum_{j_2\in S}\frac{\og(j_2+j)}{\og(j_2)+\og(j)-\og(j_2+j)}\,\xi_{j_2}=\frac{2}{3} \sum_{j_2\in S^+} \,\frac{(1+j_2^2)(1+j^2)(2+j_2^2+j^2)}{(3+j_2^2-j_2 j+j^2)(3+j_2^2+j_2 j +j^2)}\xi_{j_2}.
\end{equation}
We define
 the diagonal operator (recall \eqref{Cxi})
\begin{equation}\label{diagonalop}
\mathfrak{D}:=\mathfrak{D}(\xi)=\mathrm{diag}\,\,(\mathrm{i}\kappa_j)_{j\in S^c}, \qquad \kappa_j=\og(j)\,\big(\lal_j-c(\oo)\big)\in\mathbb{R}.
\end{equation}
\begin{lem}\label{mole}
We have 
\begin{equation}\label{Bali}
 \kappa_j=\og(j)\,\big(\lal_j-c(\oo)\big)\,,\quad  \lvert j \rvert\,\lvert \kappa_j \rvert\le C \qquad \forall j\in S^c, %\quad \kappa_{j}=-\kappa_{-j}.
\end{equation}
for an appropriate constant $C>0$ depending on the set $S$.
\end{lem}
\begin{proof}
Recalling the definitions \eqref{Cxi} and \eqref{lambda} we have, for $j\in S^c$,
\begin{equation}\label{bio}
\lal_j-c(\oo)=-\frac{2}{3}\sum_{j_0\in S^+} \,\frac{(1+j_0^2)(7+5 j_0^2+j_0^4+3 j^2)}{(3+j_0^2-j_0 j+j^2)(3+j_0^2+j_0 j +j^2)}\xi_{j_0}=\frac{P(j)}{Q(j)}
\end{equation}
It is easy to prove that $\lvert \og(j) \rvert\le 4 \lvert j \rvert$,
$3+j_0^2+j^2\pm j_0 j\geq \frac{3}{4} j^2
$
and $(1+j_0^2)(7+5 j_0^2+j_0^4+3 j^2)\le 14 j_0^6 j^2$. Hence
$
\lvert \kappa(j) \rvert\lesssim \frac{\lvert j \rvert}{j^2} \sum_{j_0\in S^+} j_0^6.
$
\end{proof}

\begin{lem}\label{LBNF2}
The transformed operator is (recall \eqref{L4dp})
\begin{equation}\label{L8}
\mathcal{L}_8 :=\Upsilon_2 \mathcal{L}_7 \Upsilon_2^{-1}  =\Pi_S^{\perp}\Big(\omega\cdot\partial_{\varphi}-m J-\varepsilon^2 \mathfrak{D}(\xi)-\varepsilon^3 X_{\mathcal{K}_3^{(2)}} +\mathcal{R}_8 \Big)
\end{equation}
where $\mathcal{K}_3^{(2)}=\mathcal{K}_3^{(1)}$, $\mathfrak{D}(\xi)$ is the diagonal operator of order $-1$ defined in \eqref{diagonalop},
 $\RR_8\in \gotC_{-1}$ satisfies 
\begin{align}\label{pasqua}
\mathbb{B}^{\g}_{\mathcal{R}_8}(s)  \lesssim_s 
\varepsilon^{4-3 a}+\varepsilon\g^{-1}
\lVert \mathfrak{I}_{\delta}\rVert_{s+\tilde{\sigma}}^{\gamma, \calO_0}, \qquad 
\mathbb{B}_{\Delta_{12} \mathcal{R}_8  }(s_0) \lesssim 
\varepsilon \gamma^{-1}(1+\lVert \mathfrak{I}_{\delta}\rVert_{s_0+\tilde{\sigma}})
\lVert i_1-i_2\rVert_{s_0+\tilde{\sigma}},
\end{align}
 for some $\tilde{\sigma}$ possibly larger than the one in Lemma \ref{LBNF1}.
\end{lem}
\begin{proof}
The proof follows by using the same arguments of the proof of Lemma \ref{LBNF1}.
In particular, expanding the left hand side of \eqref{L8}
using \eqref{charlie} we get
\begin{equation}\label{R8}
\begin{aligned}
\mathcal{R}_8 &:=\mathcal{R}_7+\varepsilon^2\mathrm{ad}_{X_{\mathtt{A}_2}}[-\varepsilon^2 \mathfrak{D}(\xi)-\varepsilon^3 X_{\mathcal{K}_3^{(2)}} +\mathcal{R}_8]+\sum_{k\geq 2}\frac{\varepsilon^{2 k}}{k!}\mathrm{ad}_{X_{\mathtt{A}_2}}[\mathcal{L}_7]\,.
\end{aligned}
\end{equation}
By \eqref{seven} and Theorem \ref{PartialWeak} we have that
\[
\mathrm{ad}_{X_{\mathtt{A}_2}}[\omega\cdot\partial_{\varphi}-m J]+\mathcal{K}^{(1)}_2=-\mathfrak{D}(\xi)-(\omega-\overline{\omega})\cdot\partial_{\varphi}X_{\mathtt{A}_2}-(m-1)[X_{\mathtt{A}_2}, J].
\]
By \eqref{mole} $\mathfrak{D}(\xi)\in\gotC_{-1}$ and by the fact that $\mathtt{A}_2$ is almost diagonal we have that  
$(\omega-\overline{\omega})\cdot\partial_{\varphi}X_{\mathtt{A}_2}, [X_{\mathtt{A}_2}, J]\in\gotC_{-1}$. Then $\mathcal{R}_{8}\in \gotC_{-1}$. The bounds \eqref{pasqua} are obtained by using the estimates \eqref{alieni}, \eqref{chegioia}, \eqref{alieni2}, \eqref{bojack} and %the estimates for $\mathcal{R}_7$ in 
\eqref{caffelatte}.
\end{proof}

\noindent
{\bf Step three (order $\varepsilon^3$).} At this step we eliminate $\varepsilon^3 X_{\mathcal{K}^{(2)}_3}$ from \eqref{L8}.
Recalling that  $\mathcal{K}^{(2)}_3$ is given in Lemma \ref{LBNF2},
we have
\begin{equation}\label{formanormalePlus}
\begin{aligned}
&\mathcal{K}^{(3)}:=\mathcal{K}^{(2)}\circ \Upsilon^{-1}_3= \,\mathsf{H}_0+\varepsilon^2 \mathcal{K}^{(2)}_2+\varepsilon^3 \mathcal{K}^{(3)}_3+o(\varepsilon^3), \\
&\mathcal{K}^{(3)}_3:=\{ H_{\mathtt{A}_3}, \mathsf{H}_0+\varepsilon^2 \mathbb{A}\xi\cdot \eta+\frac{\varepsilon^2}{2}\sum_{j\in S^c} \lal_j(\xi) z_j \,z_{-j}\}_e+\mathcal{K}^{(2)}_3\,.
\end{aligned}
\end{equation}
Note that we consider in the normal form also the $\varepsilon^2$-terms.
We want to solve the equation
\begin{equation}\label{eqHomo3}
\mathcal{D}_{\overline{\omega}
+\varepsilon^2 \mathbb{A}\xi} H_{\mathtt{A}_3}
+\{ H_{\mathtt{A}_3}, \mathsf{H}_0
+ \frac{\varepsilon^2}{2}\sum_{j\in S^c} 
\lal_j(\xi) z_j \,z_{-j}\}+ \mathcal{K}_3^{(2)}=0\,.
\end{equation}
Hence we choose the matrix $\mathtt{B}_3:=\nabla \mathcal{K}^{(1)}_3(u)$ (note that $\mathcal{K}^{(1)}_3=\mathcal{K}^{(2)}_3$). Recalling \eqref{K13} it is easy to see that $J \mathtt{B}_3$ is sum of Lie brackets of elements of $\gotC_{-1}$, hence by Lemma \ref{IncluecompoCLASSI} it belongs to $\gotC_{-1}$. By the fact that $\mathtt{A}_1$ is almost diagonal and by Remark \ref{remarkalmost} we have that $J \mathtt{B}_3$ is almost diagonal.

\begin{lem}\label{LBNF3}
The transformed operator is (recall \eqref{L4dp})
\begin{equation}\label{L9}
\mathcal{L}_9 :=\Upsilon_3 \mathcal{L}_8 \Upsilon_3^{-1}  =
\Pi_S^{\perp}\Big(\omega\cdot\partial_{\varphi}
-m J-\varepsilon^2 \mathfrak{D}(\xi) +\mathcal{R}_9 \Big)
\end{equation}
where $\mathfrak{D}(\xi)$ is the diagonal operator of order $-1$ defined 
in \eqref{diagonalop},
$\RR_9\in \mathfrak{C}_{-1}$ satisfies
\begin{equation}\label{pasqua9}
\mathbb{B}^{\g}_{\mathcal{R}_9}(s)  
\lesssim_{s} \varepsilon^{4- 3 a}+
\varepsilon \g^{-1}
\lVert \mathfrak{I}_{\delta}\rVert_{s+\tilde{\sigma}}^{\gamma, \cO_0}\,, 
\qquad 
\mathbb{B}_{\Delta_{12} \mathcal{R}_9  }(s_0)  \lesssim 
\varepsilon\g^{-1}\lVert i_1-i_2\rVert_{s_0+\tilde{\sigma}}\,,
\end{equation}
 for some $\tilde{\sigma}$ possibly larger than the one in Lemma \ref{LBNF2}.
\end{lem}
\begin{proof}
The proof follows the same arguments used for proving Lemma \ref{LBNF1}. 
By \eqref{charlie} we deduce
\begin{equation}\label{R9}
\begin{aligned}
\mathcal{R}_9 &:=\mathcal{R}_8+\varepsilon^3\mathrm{ad}_{X_{\mathtt{A}_3}}[-\varepsilon^3 X_{\mathcal{K}_3^{(2)}}+\mathcal{R}_8]+\sum_{k\geq 2} \frac{1}{k!} \mathrm{ad}_{X_{\mathtt{A}_3}}^k[\mathcal{L}_8]\,.
\end{aligned}
\end{equation}
We note that
by \eqref{eqHomo3} we have (recall \eqref{FreqAmplMapDP} and \eqref{Cxi})
\[
\mathrm{ad}_{X_{\mathtt{A}_3}}[\omega\cdot\partial_{\varphi}-mJ-\varepsilon^2 \mathfrak{D}(\xi)]=\varepsilon^3 X_{\mathcal{K}_3^{(2)}}+(\omega-\overline{\omega}-\varepsilon^2 \mathbb{A}\xi)\cdot \partial_{\varphi} \mathtt{A}_{3}-(m-1-\varepsilon^2 c(\oo))[X_{\mathtt{A}_3}, J]
\in \gotC_{-1}\,,
\]
since $\mathtt{A}_3$ is almost diagonal .
Hence the bounds \eqref{pasqua9} follows by \eqref{bojack}, \eqref{pasqua} and by using Lemma \ref{commutatoC1}.
\end{proof}

\begin{proof}[\textbf{Proof of Theorem \ref{risultatosez8}}]
We choose $\mu_1=\tilde{\sigma}$ given in Lemma \ref{LBNF3}. 
We consider $p$ and $\gotp_1$ so that 
\begin{equation}\label{ipotesimu}
 s_0+\gotp_1-\hat{\s}\geq p\geq s_0\,,\quad  \s_{9}+\sigma_0+(s_1-s_0)+\su+\rho+1\leq \tilde{\s}\, ,
\qquad \tilde{\s}\leq \gotp_1\,, 
%\quad \gotp_1\geq \max\{\hat{\s}, \s_0+\s_7\},
%\mu'> \tilde{\s}\geq \s_{8}+\sigma_0+(s_1-s_0)+\su+\rho+1,
\end{equation}
where $\tilde{\s}$ is the loss of regularity in Lemma \ref{LBNF3}, $\s_0$ has been introduced in Section \ref{regularization}, see estimates \eqref{unasuaziaricca}-
\eqref{dude2}, 
$\su>0$ and  $s_1$  are  given respectively in Lemma \ref{differenzaFlussi} and in Proposition $3.6$ in \cite{FGP1}.
\\
We define the map (recall \eqref{Fi1dp}, \eqref{fi})
\[
\Upsilon:=\Upsilon_3 \circ \Upsilon_2 \circ \Upsilon_1\circ \mathcal{B}\,. 
\]
By Proposition \ref{corollarioStraight} the map $\mathcal{B}$ is defined for $\omega\in \calO^{2\g}_{\infty}$ (see \eqref{condizMel1}), and so also $\Upsilon$.
By \eqref{Drugo}, \eqref{Drugo2}, \eqref{Drugo45}, \eqref{upo1}, \eqref{bojack}, \eqref{saintetienne}, \eqref{uranioi} (with $i=5$) and Lemma \ref{differenzaFlussi} we have \eqref{grecia}.\\
The result follows by setting $\mathcal{L}:=\mathcal{L}_9$ (see \eqref{L9}), 
$\mathcal{P}_0:=\mathcal{R}_9$,
and $m$ is the constant given by Proposition \ref{corollarioStraight}.
Indeed 
 %and by noting that \eqref{pasqua9} implies \eqref{pasqua200}, 
 \eqref{clinica} implies \eqref{clinica100} and \eqref{clinica1000}.
Moreover, by Lemma \ref{LBNF3}, we have that 
$\cP_0\in \mathfrak{C}_{-1}$ and satisfies \eqref{pasqua9}.
 Lemma \ref{LemmaAggancio} implies that $\cP_0$  is $-1$-modulo tame  together with $\tb_0$ derivatives in the variable $\varphi$ (recall the Definition \ref{Mdritta2} and the fact that $\gamma^{3/2}<\g$). The bounds  \eqref{pasqua200}, \eqref{nomidimerda2bis} follow from \eqref{chavez1}- \eqref{chavez20} , the definition of $\mathbb{B}^{\g}_{\cP_0}(s)$ (see \eqref{Mdritta2},\eqref{Mdrittaconlai2}) and \eqref{pasqua9}. By Lemma \ref{mole} we deduce \eqref{diagonalopFinale}.
\end{proof}

\subsection{KAM reducibilty and Inversion of  the linearized operator}\label{SezioneDiagonalization}
In this subsection we prove the claim \eqref{InversionAssumptionDP}
by diagonalizing  
the operator $\mathcal{L}$ in \eqref{operatorefinale}.
We first write
\begin{equation}\label{natalino}
\mathcal{L}:=\omega\cdot\del_{\f}-\mathbf{M}_0, \quad
\mathbf{M}_0 := \cD_0+\cP_0\,,\quad \cD_0 := {\rm diag}(\mathrm{i} \, d_j^{(0)})_{j\in S^c}\,,\quad d_j^{(0)}:= 
d_j^{(0)}(\omega)
= m(\omega)\og(j)+\varepsilon^2 \kappa_j(\omega).	
\end{equation}
Notice that (by the smallness condition \eqref{PiccolezzaperKamredDP})
 Proposition $4.1$ in \cite{FGP1} applies to the operator $\calL$ in \eqref{operatorefinale}. Hence by following almost word by word the 
proof of Theorem $1.7$ in \cite{FGP1} one has the following.

\begin{teor}{\textbf{(Reducibility)}}\label{ReducibilityDP}
Fix $\tilde{\g}\in [\gamma^{3/2}/4, 4\gamma^{3/2}]$.
Assume that $\omega \mapsto i_{\delta}(\omega)$ is a Lipschitz function defined on  $\calO_0\subseteq \Omega_{\varepsilon}$ (recall   \eqref{AssumptionDP}), satisfying \eqref{IpotesiPiccolezzaIdeltaDP} with $\gotp_1 \geq\mu_1$ where $\mu_1:=\mu_1(\nu)$ is given in Proposition \ref{risultatosez8}. %and $\tb_0:=6 \tau+5$. 
There exist $\delta_0\in (0, 1)$, $N_0>0$, $C_0>0$, such that, if (recall that by \eqref{donnapia} $\g= \e^{2+a}$)
\begin{equation}\label{PiccolezzaperKamredDP}
N_0^{C_0} \varepsilon^{4-3 a} \gamma^{-3/2}=N_0^{C_0} \varepsilon^{1-(9/2) a}\le \delta_0, 
\end{equation}
then the following holds.
\begin{itemize}
\item[(i)] \textbf{(Eigenvalues)}. For all $\omega\in \Omega_{\varepsilon}$ there exists a sequence \footnote{Whenever it is not strictly necessary, we shall drop the dependence on $i_{\delta}$.}
\begin{align}\label{FinalEigenvaluesDP}
&d_j^{\infty}(\omega):=d_j^{\infty}(\omega, i_{\delta}(\omega)):=m(\omega, i_{\delta}(\omega))\,\og(j)+\varepsilon^2 \kappa_j (\omega) + r_j^{\infty}(\omega, i_{\delta}(\omega)), \quad j\in S^c,
\end{align}
with $m$ and $\kappa_j$ in \eqref{clinica} and \eqref{diagonalop} respectively. Furthermore, for all $j\in S^c$
\begin{equation}\label{stimeautovalfinaliDP}
 \sup_j \langle j\rangle |r_j^{(\infty)}|^{\gamma^{3/2}} \lesssim \e^{4-3 a}, \qquad r_j^{\infty}=-r^{\infty}_{-j}.
\end{equation}
 All the eigenvalues $\mathrm{i} d_j^{\infty}$ are purely imaginary. 
%We define, for convenience, $d_0^{\infty}(\omega):=0$.
\item[(ii)] \textbf{(Conjugacy)}. For all $\omega$ in the set
\begin{equation}\label{OmegoneInfinitoDP}
\Omega^{2 {\tilde{\gamma}}}_{\infty}:=\Omega^{2 \tilde{\gamma}}_{\infty}(i_{\delta}):=\left\{ \omega\in\calO^{2\gamma}_{\infty} : \lvert \omega\cdot \ell+d_j^{\infty}(\omega)-d_k^{\infty}(\omega)\rvert\geq \frac{2 {\tilde{\gamma}}}{\langle \ell \rangle^{\tau}}, \,\,\forall \ell\in\mathbb{Z}^{\nu}, \,\,\forall j, k\in S^c,\, j\neq k  \right\}
\end{equation}
there is a real, bounded, invertible, linear operator $\Phi_{\infty}(\omega)\colon H^s_{S^{\perp}}(\T^{\nu+1})\to H_{S^{\perp}}^s(\T^{\nu+1})$, with bounded inverse $\Phi_{\infty}^{-1}(\omega)$, that conjugates $\mathcal{L}$ in \eqref{operatorefinale} to constant coefficients, namely 
\begin{equation}\label{Linfinito}
\begin{aligned}
&\mathcal{L}_{\infty}(\omega):=\Phi_{\infty} (\omega) \circ \mathcal{L} \circ \Phi_{\infty}^{-1}(\omega)=\omega\cdot \partial_{\varphi}-\mathcal{D}_{\infty}(\omega),\\
& \mathcal{D}_{\infty}(\omega):=\mathrm{diag}_{j\in S^c} \{ \mathrm{i} d_j^{\infty}(\omega) \}.
\end{aligned}
\end{equation}
The transformations $\Phi_{\infty}, \Phi_{\infty}^{-1}$ are tame and they satisfy for $s_0\le s\le \mathcal{S}$ (recall $\mu_1$ in Theorem \ref{risultatosez8})
\begin{equation}\label{grano}
\lVert (\Phi^{\pm 1}_{\infty}-\mathrm{I}) h \rVert^{\gamma^{3/2}, \Omega_{\infty}^{2 \tilde\gamma}}_{s}\lesssim_s \big(\varepsilon^{4-3a} \gamma^{-3/2}+\varepsilon \gamma^{-5/2} \lVert \mathfrak{I}_{\delta} \rVert_{s+\mu_1}^{\gamma, \calO_0}\big) \lVert h \rVert^{\gamma^{3/2}, \Omega_{\infty}^{2 \tilde\gamma}}_{s_0}+\varepsilon^{4-3a} \gamma^{-3/2} \lVert h \rVert^{\gamma^{3/2}, \Omega_{\infty}^{\tilde\gamma}}_s.
\end{equation}
Moreover $\Phi_{\infty}, \Phi_{\infty}^{-1}$ are symplectic, and $\mathcal{L}_{\infty}$ is a Hamiltonian operator.
\item[(iii)] \textbf{(Dependence on $i_{\delta}(\omega)$)}. Let $i_1(\omega)$ and $i_2(\omega)$ be two Lipschitz maps satisfying \eqref{IpotesiPiccolezzaIdeltaDP} with $\mathfrak{I}_{\delta}\rightsquigarrow i_k(\varphi)-(\varphi, 0, 0)$, $k=1, 2$, and such that
\begin{equation}
\lVert i_1-i_2 \rVert_{s_0+\mu_1}\lesssim \rho N^{-(\tau+1)}
\end{equation}
for $N$ sufficiently large and $0<\rho<\gamma^{3/2}/4$. Fix $\g_1\in [\gamma^{3/2}/2, 2 \gamma^{3/2}]$ and $\g_2:=\g_1-\rho$. Let $r_j^{(\infty)}(\oo, i_k(\oo))$ be the sequence in \eqref{FinalEigenvaluesDP} with $\tilde{\g}\rightsquigarrow \g_k$ for $k=1, 2$. 
Then for all $\omega\in \Omega_{\infty}^{\g_1}(i_1)$ we have, for some $\kappa>(3/2) \tau$, 
\begin{equation}
\label{r12}
\g^{-1}\lvert \Delta_{12} m \rvert+ \sup_j \langle j\rangle |\Delta_{12}r_j^{(\infty)}| \le \e \g^{-1} \| i_1-i_2 \|_{s_0+\mu_1}+\varepsilon^{4- 3 a} N^{-\kappa}.
\end{equation}
\end{itemize}
\end{teor}

\begin{proof}
The proof of items $(i)$-$(ii)$ follow by Theorem $1.9$ in \cite{FGP1}.
The only point left to prove is item $(iii)$. We apply Theorems $1.4$, $1.5$ in \cite{FGP1}. We have that
\begin{equation}\label{giovanna1}
\g^{-1}\lvert  m^{(N)}(i_2)-m(i_1)  \rvert+\langle j \rangle \lvert r_j^{(N)}(i_2)-r_j^{(\infty)}(i_1) \rvert\lesssim \varepsilon\g^{-1} \lVert i_1-i_2 \rVert_{s_0+\mu_1}+\varepsilon^{4- 3 a} N^{-\kappa},
\end{equation}
where $\kappa>\tau$. Here $m^{(N)}(i_2)$, $r_j^{(N)}(i_2)$ are defined in $(1.39)$ of \cite{FGP1} and they are an approximation of $m(i_2)$ and $r_j^{(\infty)}(i_2)$ satisfying
\begin{equation}
\g^{-1}\lvert  m^{(N)}(i_2)-m(i_2)  \rvert+ \langle j \rangle \lvert r_j^{(N)}(i_2)-r_j^{(\infty)}(i_2) \rvert\lesssim \varepsilon^{4- 3 a} N^{-\kappa}\,.\label{giovanna2}
\end{equation}
The bounds \eqref{giovanna1}, \eqref{giovanna2} imply the \eqref{r12}.
\end{proof}

\subsection{Proof of the inversion assumption \eqref{InversionAssumptionDP}}

We are in position to give estimates on the inverse of the operator $\calL_{\omega}$ in \eqref{BoraMaledetta}.
Let us  now define (recall \eqref{gammaDP})
\begin{equation}\label{primediMelnikov}
\calF_{\infty}^{2\gamma}(i_{\delta}):=\{ \omega\in \calO_0 : \lvert \omega\cdot \ell-\,d_j^{\infty}(\omega) \rvert\geq \frac{2\gamma}{\langle \ell \rangle^{\tau}} , \quad \forall \ell\in\mathbb{Z}^{\nu} , \forall j\in S^c\}
\end{equation}
We deduce the inversion assumption \eqref{InversionAssumptionDP} by the following result.

\begin{prop}\label{InversionLomegaDP}
Assume the hypotesis of Theorem \ref{ReducibilityDP}, \eqref{IpotesiPiccolezzaIdeltaDP} with $\gotp_1\geq \mu_1+2\tau+1$, where $\mu_1$ is given in Theorem \ref{risultatosez8}.
%, and \eqref{AssumptionDP} with $\gotp_0\geq \gotp_1+\gotp_1$, where $\gotp_1$ is given in Lemma \ref{isotropictorus}.\\ 
Then for all $\omega\in \Omega_{\infty}:=\Omega_{\infty}^{2 \gamma^{3/2}}(i_{\delta}) \cap \calF_{\infty}^{2\gamma}(i_{\delta})$ (see \eqref{OmegoneInfinitoDP}), for any function $g\in H^{s+2\tau+1}_{S^{\perp}}(\T^{\nu+1})$ the equation $\mathcal{L}_{\omega} h=g$ has a solution $h=\mathcal{L}_{\omega}^{-1} g\in H_{S^{\perp}}^s(\T^{\nu+1})$, satisfying
\begin{equation}\label{TameEstimateLomegaDP}
\begin{aligned}
\lVert \mathcal{L}_{\omega}^{-1} g \rVert_s^{\gamma, \Omega_{\infty}} &\lesssim_s \gamma^{-1} (\lVert g \rVert_{s+2\tau+1}^{\gamma, \Omega_{\infty}}+\varepsilon \gamma^{-5/2}\lVert \mathfrak{I}_{\delta} \rVert_{s+\gotp_1}^{\gamma, \Omega_{\infty}}\lVert g \rVert_{s_0}^{\gamma, \Omega_{\infty}}).
%\\
%&\lesssim_s \gamma^{-1} (\lVert g \rVert_{s+2\tau+1}^{\gamma, \Omega_{\infty}}+\varepsilon\gamma^{-5/2}\{ \lVert \mathfrak{I}_0 \rVert_{s+\gotp_1+\gotp_1}^{\gamma, \Omega_{\infty}}+\gamma^{-1} \lVert \mathfrak{I}_0 \rVert_{s_0+\gotp_1}^{\gamma, \Omega_{\infty}}\lVert Z \rVert_{s+\gotp_1+\gotp_1}^{\gamma, \Omega_{\infty}} \}\lVert g \rVert_{s_0}^{\gamma, \Omega_{\infty}}).
\end{aligned}
\end{equation}
\end{prop}
\begin{proof}
We conjugated the operator $\mathcal{L}_{\omega}$ in \eqref{BoraMaledetta} to a diagonal operator $\mathcal{L}_{\infty}=\chi \mathcal{L}_{\omega} \chi^{-1}$, see \eqref{Linfinito}, with (recall \eqref{grano} and Theorem \ref{risultatosez8})
$\chi:=\Phi_{\infty}\circ \Upsilon$.
Moreover, by \eqref{grecia} and \eqref{grano} %and Lemma \ref{isotropictorus} 
we have the following estimates
\begin{equation*}\label{nebbia}
\lVert \chi^{\pm 1} h \rVert_s^{\gamma, \Omega_{\infty}}\lesssim_s \lVert h \rVert^{\gamma, \Omega_{\infty}}_s+\varepsilon \gamma^{-5/2}\lVert \mathfrak{I}_{\delta} \rVert_{s+\mu_1}^{\gamma, \calO_0}   \lVert h \rVert^{\gamma, \Omega_{\infty}}_{s_0}.
\end{equation*}
We have
\begin{equation}
\mathcal{L}_{\infty}^{-1} g=\sum_{j\neq 0} \frac{g_{\ell j}}{\mathrm{i}\big(\omega\cdot \ell-d_j^{\infty}(\omega)\big)}\,e^{\mathrm{i}(\ell \cdot \varphi+j x)}
\end{equation}
and then
$
\lVert \mathcal{L}_{\infty}^{-1} g \rVert^{\gamma, \Omega_{\infty}}_s\le \gamma^{-1} \lVert g \rVert_{s+2\tau+1}^{\gamma, \Omega_{\infty}}.
$
Thus we get the estimate \eqref{TameEstimateLomegaDP}.
\end{proof}

\section{The Nash-Moser nonlinear iteration}\label{sezione9DP}
In this section we prove Theorem \ref{IlTeoremaDP}. 
It will be a consequence of the Nash-Moser 
theorem \ref{NashMoserDP}.\\
Consider the finite-dimensional subspaces
\[
E_n:=\{ \mathfrak{I}(\varphi)=(\Theta, y, z)(\varphi) :  
\Theta=\Pi_n\Theta, y=\Pi_n y, z=\Pi_n z \}
\]
where 
\begin{equation}\label{Nn}
N_n:=N_0^{\chi^n}\,, \quad n=0, 1, 2, \dots\,, \quad \chi:=3/2\,, \quad N_0>0\,
\end{equation}
and $\Pi_n$ are the projectors 
(which, with a small abuse of notation, 
we denote with the same symbol)
\begin{equation}\label{TagliDP}
\begin{aligned}
&\Pi_n \Theta(\varphi):=\sum_{\lvert \ell \rvert<N_n} \Theta_{\ell}\, 
e^{\mathrm{i} \ell\cdot \varphi}\,, \,\,
\Pi_n y(\varphi):=\sum_{\lvert \ell \rvert<N_n} y_{\ell}\,
e^{\mathrm{i} \ell\cdot \varphi}\,,
\,\,\mbox{where}\,\,
\Theta(\varphi)=\sum_{\ell\in\mathbb{Z}^{\nu}} 
\Theta_{\ell}\,e^{\mathrm{i} \ell\cdot \varphi}\,, \,\,
y(\varphi)=\sum_{\ell\in \mathbb{Z}^{\nu}} y_{\ell} \,
e^{\mathrm{i} \ell\cdot\varphi}\,,\\[2mm]
&\Pi_n z(\varphi, x):=\sum_{\lvert (\ell, j) \rvert<N_n} 
z_{\ell j}\,e^{\mathrm{i}(\ell\cdot \varphi+j x)}\,, 
\,\,\mbox{where}\quad z(\varphi, x)=
\sum_{\ell\in\mathbb{Z}^{\nu}, j\in S^c} z_{\ell j}\,e^{\mathrm{i} (\ell\cdot \varphi+j x)}\,.
\end{aligned}
\end{equation}
We define $\Pi_n^{\perp}=\mathrm{I}-\Pi_n$. 
The classical smoothing properties hold, namely, 
for all $\alpha, s\geq 0$,
\begin{equation}\label{SmoothingDP}
\lVert \Pi_n \mathfrak{I} \rVert_{s+\alpha}^{\g, \calO}\le N_n^{\alpha} \lVert \mathfrak{I}_{\delta} \rVert_s^{\g, \calO}, \quad \forall \,\mathfrak{I}(\omega)\in H^s\,, 
\quad \lVert \Pi_n^{\perp} \mathfrak{I} \rVert_s^{\g, \calO}
\le N_n^{-\alpha} \lVert \mathfrak{I} \rVert_{s+\alpha}^{\g, \calO}\,, 
\quad \forall\, \mathfrak{I}(\omega)\in H^{s+\alpha}\,.
\end{equation}

\noindent
Recall \eqref{gammaDP}, \eqref{donnapia} for the definition of $b$ we set $a:=2b-2$.
We define the following constants 
\begin{equation}\label{parametriNMdp}
\begin{aligned}
&\alpha_0:=3\mu+3\,, \qquad \qquad  \qquad\alpha:=3\alpha_0+1\,, 
\qquad \qquad \qquad \alpha_1:=(\alpha-3 \mu)/2\,,\\
&k:=3(\alpha_0+\rho^{-1})+1\,, 
\qquad \beta_1:=6\alpha_0+3 \rho^{-1} +3\,, 
\qquad \frac{1}{2}\left(  \frac{1-(9/2)a}{C_1 (1+a)}\right)<\rho<\frac{1-(9/2)a}{C_1 (1+a)}\,
\end{aligned}
\end{equation}
where $\mu:=\mu( \nu)>0$ is the 
``loss of regularity'' given by the 
Theorem \ref{TeoApproxInvDP} and $C_1$ is fixed below. 
%\begin{remark}\label{hambre}
%We remark that $\mu > \gotp_1$, where $\gotp_1$ is given in 
%Proposition \ref{InversionLomegaDP} and $\gotp_1\geq \mu_1$ 
%where $\mu_1$ is given in Theorem \ref{risultatosez8}.
%\end{remark}
%

\begin{teor}{\textbf{(Nash-Moser)}}\label{NashMoserDP}
Assume that $f\in C^{\infty}$ (see \eqref{HamiltonianDensity}). Let $\tau:=2\nu+6$. 
Then there exist $C_1>\max\{ \alpha_0+\alpha, C_0 \}$ 
(where $C_0:=C_0( \nu)$ is the one in 
Theorem \ref{ReducibilityDP}), 
$\delta_0:=\delta_0( \nu)>0$ such that, if
\begin{equation}\label{SmallnessConditionNMdp}
N_0^{C_1} \varepsilon^{b_*+1} \gamma^{-7/2}<\delta_0\,, 
\quad \gamma:=\varepsilon^{2+a}=\varepsilon^{2 b}\,, 
\quad N_0:=(\varepsilon\gamma^{-1})^{\rho}\,, 
\quad b_*=9- 2 b\,,
\end{equation}
then there exists $C_{*}=C_{*}(S)>0$ such that 
for all $n\geq 0$ the following holds:
\begin{itemize}
\item[$(\mathcal{P}1)_n$] there exists a function 
$(\mathfrak{I}_n, \zeta_n)\colon \mathcal{G}_n 
\subseteq \Omega_{\varepsilon} \to E_{n-1}\times \mathbb{R}^{\nu}, 
\omega\mapsto (\mathfrak{I}_n(\omega), \zeta_n(\omega)), 
(\mathfrak{I}_0, \zeta_0):=(0, 0), E_{-1}:=\{0\}$,
where the set $\mathcal{G}_0$ is defined 
in \eqref{0di0Melnikov} and the sets 
$\mathcal{G}_n$ for $n\geq 1$ 
are defined inductively by:
\begin{align}
\mathcal{G}_{n+1}:=
\bigcap_{i=0}^2 \Lambda_{n+1}^{(i)}\,,\;\;{\rm with}\;\;\;
&\Lambda^{(0)}_{n+1}:=\left\{ \omega\in \mathcal{G}_n : 
\lvert\omega\cdot \ell+m(i_n) j \rvert\geq 
\frac{2\,\gamma_n}{\langle \ell \rangle^{\tau}}, 
\,\,\forall j\in S^c , \ell\in\mathbb{Z}^{\nu} \right\}\,, \nonumber\\ \label{GnDP}
&\Lambda^{(1)}_{n+1}:=\left\{ \omega\in \mathcal{G}_n : 
\lvert\omega\cdot \ell+d_j^{\infty}(i_n) \rvert\geq 
\frac{2\,\gamma_n}{\langle \ell \rangle^{\tau}}, \,\,
\forall j\in S^c , \ell\in\mathbb{Z}^{\nu} \right\}\,,\\
&\Lambda^{(2)}_{n+1}:=\left\{ \omega\in \mathcal{G}_n : 
\lvert \omega\cdot \ell+d_j^{\infty}(i_n)-d_k^{\infty}(i_n) \rvert\geq 
\frac{2\,\gamma^{3/2}_n\,}{\langle \ell \rangle^{\tau}}, \,\,
\forall j, k\in S^c, j\neq k , \ell\in\mathbb{Z}^{\nu} \right\}\,, \nonumber
\end{align}
where $\gamma_n:=\gamma (1+2^{-n})$, 
$\gamma^{*}_n:=\gamma^{3/2}(1+2^{-n})$ and 
$d_j^{\infty}(\omega):=d_j^{\infty}(\omega, i_n(\omega))$ 
are defined in \eqref{FinalEigenvaluesDP} 
(and $d_0^{\infty}(\omega)=0$).
Moreover 
$\lvert \zeta_n \rvert^{\gamma, \mathcal{G}_n}
\lesssim \lVert \mathcal{F}(U_n) \rVert_{s_0}^{\gamma, \mathcal{G}_n}$ 
and
\begin{equation}\label{ConvergenzaDP}
\lVert \mathfrak{I}_n \rVert_{s_0+\mu}^{\gamma, \mathcal{G}_n}
\le C_* \varepsilon^{b_*} \gamma^{-1}\,, 
\quad \lVert \mathcal{F}(U_n) \rVert_{s_0+\mu+3}^{\gamma, \mathcal{G}_n}
\le C_* \varepsilon^{b_*}\,,
\end{equation}
where $U_n:=(i_n, \zeta_n)$ with 
$i_n(\varphi)=(\varphi, 0, 0)+\mathfrak{I}_n(\varphi)$.
The differences 
$\hat{\mathfrak{I}}_n:=\mathfrak{I}_n-\mathfrak{I}_{n-1}$ 
(where we set $\hat{\mathfrak{I}}_0:=0$) 
is defined on $\mathcal{G}_n$, and satisfy
\begin{equation}\label{FrakHatDP}
\lVert \hat{\mathfrak{I}}_1 \rVert_{s_0+\mu}^{\gamma, \mathcal{G}_1}
\le C_* \varepsilon^{b_*} \gamma^{-1}\,, 
\quad \lVert \hat{\mathfrak{I}}_n \rVert_{s_0+\mu}^{\gamma, \mathcal{G}_{n}}
\le C_* \varepsilon^{b_*} \gamma^{-1} N_{n-1}^{-\alpha}\,, 
\quad \forall n\geq 2\,.
\end{equation}
\item[$(\mathcal{P} 2)_n$] 
$\lVert \mathcal{F}(U_n) \rVert_{s_0}^{\gamma, \mathcal{G}_n}
\le C_* \varepsilon^{b_*} N_{n-1}^{-\alpha}$ 
where we set $N_{-1}:=1$.
\item[$(\mathcal{P}3)_n$]{(High Norms)}. 
$\lVert \mathfrak{I}_n \rVert_{s_0+\beta_1}^{\gamma, \mathcal{G}_n}
\le C_* \varepsilon^{b_*} \gamma^{-1} N_{n-1}^k$ 
and $\lVert \mathcal{F}(U_n) \rVert_{s_0+\beta_1}^{\gamma, \mathcal{G}_n}
\le C_* \varepsilon^{b*} N_{n-1}^k$.
\item[$(\mathcal{P} 4)_n$]{(Measure)}. 
The measure of the ``Cantor-like'' sets $\mathcal{G}_n$ 
satisfies
\begin{equation}\label{MisureDP}
\lvert \Omega_{\varepsilon}\setminus \mathcal{G}_0 \rvert
\le C_* \varepsilon^{2 (\nu-1)} \gamma\,, 
\quad \lvert \mathcal{G}_n \setminus \mathcal{G}_{n+1} \rvert
\le C_* \varepsilon^{2 (\nu-1)} \gamma N_{n-1}^{-1}\,.
\end{equation}  
\end{itemize}
\end{teor}

\begin{proof}
To simplify the notations we omit the index 
$\g, \mathcal{G}_n$ on the norm $\lVert \cdot \rVert_s$.\\
\textit{Proof of $(\mathcal{P}_1)_0, (\mathcal{P}_2)_0, (\mathcal{P}_3)_0$}. 
Recalling \eqref{NonlinearFunctionalDP}, we have, 
by the second estimate in \eqref{EstimatesVecFieldDP}, 
\[
\lVert \mathcal{F}(U_0)\rVert_s
=\lVert \mathcal{F}((\varphi, 0, 0), 0) \rVert_s
=\lVert X_P(\varphi, 0, 0) \rVert_s\lesssim_s \varepsilon^{9-2 b}\,.
\]
Hence the smallness conditions in 
$(\mathcal{P}_1)_0, (\mathcal{P}_2)_0, (\mathcal{P}_3)_0$ 
hold taking $C_*$ large enough.
\item \textit{Assume that $(\mathcal{P}_1)_n, (\mathcal{P}_2)_n, (\mathcal{P}_3)_n$ 
hold for some $n\geq 0$, 
and prove $(\mathcal{P}_1)_{n+1}, (\mathcal{P}_2)_{n+1}, (\mathcal{P}_3)_{n+1}$}. 
By \eqref{parametriNMdp} and \eqref{SmallnessConditionNMdp}
\[
N_0^{C_1} \varepsilon^{b_*+1} \gamma^{-7/2}
=\varepsilon^{1-(9/2) a-\rho\,C_1 (1+a)}<\delta_0
\]
for $\varepsilon$ small enough. 
If we take $C_1$ bigger than $C_0$ in Theorem \ref{ReducibilityDP} then \eqref{PiccolezzaperKamredDP} holds. 	
In \eqref{IpotesiPiccolezzaIdeltaDP} we consider $\mathfrak{p}_0=\gotp_1+\tilde{\mathfrak{p}}$, where $\gotp_1:=\mu_1+2\tau+1$ and $\tilde{\mathfrak{p}}$ appears in \eqref{chegioia2}. Since $\mu\gg \mathfrak{p}_0$ \eqref{ConvergenzaDP} implies \eqref{AssumptionDP} 
and so \eqref{IpotesiPiccolezzaIdeltaDP}, 
and Proposition \ref{InversionLomegaDP} applies. 
Hence the operator 
$\mathcal{L}_{\omega}:=\mathcal{L}_{\omega}(\omega, i_n (\omega))$  
in \eqref{BoraMaledetta} is defined on 
$\calO_0= \calG_n$  and is invertible for all 
$\omega\in \mathcal{G}_{n+1}$ since 
$\mathcal{G}_{n+1}\subseteq \Omega^{2\gamma^{*}_n}_{\infty}(i_n) 
\cap \mathcal{F}_{\infty}^{2\g_n}(i_n)$ 
and the %last estimate in 
\eqref{TameEstimateLomegaDP} holds. 
This means that the assumption \eqref{InversionAssumptionDP} 
of Theorem  \ref{TeoApproxInvDP} is verified with 
$\Omega_{\infty}=\mathcal{G}_{n+1}$. 
By Theorem \ref{TeoApproxInvDP} 
there exists an approximate inverse 
$\textbf{T}_n(\omega):=\textbf{T}_0(\omega, i_n(\omega))$ 
of the linearized operator 
$L_n(\omega):=\td \mathcal{F}(\omega, i_n(\omega),\zeta_n)
\equiv \td \mathcal{F}(\omega, i_n(\omega),0)$, 
satisfying \eqref{TameEstimateApproxInvDP}. 
By \eqref{SmallnessConditionNMdp}, \eqref{ConvergenzaDP}
\begin{align}
&\lVert \textbf{T}_n g\rVert_s\lesssim_s 
\gamma^{-1} (\lVert g \rVert_{s+\mu}
+\varepsilon\gamma^{-5/2} \{ \lVert \mathfrak{I}_n \rVert_{s+\mu}
+\gamma^{-1} \lVert \mathfrak{I}_n \rVert_{s_0+\mu} 
\lVert \mathcal{F}(U_n) \rVert_{s+\mu} \} \lVert g \rVert_{s_0+\mu})\,,\label{9.10dp}\\
&\lVert \textbf{T}_n g \rVert_{s_0}\lesssim \gamma^{-1} \lVert g \rVert_{s_0+\mu}
\end{align}
and, by \eqref{6.41DP}, using also \eqref{SmallnessConditionNMdp}, \eqref{ConvergenzaDP}, \eqref{SmoothingDP},
\begin{align}
\lVert (L_n\circ \textbf{T}_n-\mathrm{I}) g \rVert_s 
\lesssim_s & \varepsilon^{2 b-1}\gamma^{-2} 
(\lVert \mathcal{F}(U_n) \rVert_{s_0+\mu}\lVert g \rVert_{s+\mu}
+\lVert \mathcal{F}(U_n) \rVert_{s+\mu}\lVert g \rVert_{s_0+\mu}\notag\\
& + \varepsilon\gamma^{-5/2}\lVert \mathfrak{I}_n \rVert_{s+\mu}
\lVert \mathcal{F}(U_n) \rVert_{s_0+\mu}\lVert g \rVert_{s_0+\mu})\\
\lVert (L_n\circ \textbf{T}_n-\mathrm{I}) g \rVert_{s_0} 
%\lesssim 
%& \varepsilon^{2 b-1}\gamma^{-2} 
%\lVert \mathcal{F}(U_n) \rVert_{s_0+\mu} 
%\lVert g \rVert_{s_0+\mu}\notag\\\notag
\lesssim & \varepsilon^{2 b-1}\gamma^{-2} 
(\lVert \Pi_n \mathcal{F}(U_n) \rVert_{s_0+\mu}
+\lVert \Pi_n^{\perp} \mathcal{F}(U_n) \rVert_{s_0+\mu})
\lVert g \rVert_{s_0+\mu}\\
\lesssim &\varepsilon^{2 b-1}\gamma^{-2} N_n^{\mu}   \Big(
\lVert \mathcal{F}(U_n) \rVert_{s_0}+N_n^{-\beta_1}
\lVert \mathcal{F}(U_n) \rVert_{s_0+\beta_1} \Big) \lVert g \rVert_{s_0+\mu}\,.
\end{align}
%%The index $\beta_1$ in \eqref{parametriNMdp} 
%is an ultraviolet cut, and it has to be define in order to 
%obtain the convergence of the iteration scheme.
Now, for all $\omega\in\mathcal{G}_{n+1}$, we can define, for $n\geq 0$,
\begin{equation}\label{HnDefDP}
U_{n+1}:=U_n+H_{n+1}, \quad H_{n+1}
:=(\hat{\mathfrak{I}}_{n+1}, \hat{\zeta}_{n+1})
:=-\tilde{\Pi}_n \textbf{T}_n \Pi_n \mathcal{F}(U_n)\in 
E_n\times\mathbb{R}^{\nu}\,,
\end{equation}
where 
$\tilde{\Pi}_n(\mathfrak{I}, \zeta):=(\Pi_n \mathfrak{I}, \zeta)$ 
with $\Pi_n$ defined in \eqref{TagliDP}. By construction we have
\begin{align*}
&\mathcal{F}(U_{n+1})=\mathcal{F}(U_n)+L_n H_{n+1}+Q_n\,,\\
Q_n:=&Q(U_n, H_{n+1})\,, 
\quad Q(U_n, H):=\mathcal{F}(U_n+H)-\mathcal{F}(U_n)-L_n H\,, 
\quad H\in E_n\times\mathbb{R}^{\nu}\,.
\end{align*}
%\[
%\mathcal{F}(U_{n+1})=\mathcal{F}(U_n)+L_n H_{n+1}+Q_n\,,
%\]
%where
%\begin{equation}
%Q_n:=Q(U_n, H_{n+1}), \quad Q(U_n, H):=\mathcal{F}(U_n+H)-\mathcal{F}(U_n)-L_n H, \quad H\in E_n\times\mathbb{R}^{\nu}.
%\end{equation}
Then, by the definition of $H_{n+1}$ in \eqref{HnDefDP}, 
using $[L_n, \Pi_n]$ and writing 
$\tilde{\Pi}_n^{\perp}(\mathfrak{I}, \zeta):=(\Pi_n^{\perp} \mathfrak{I}, 0)$ 
we have
\begin{equation*}
\begin{aligned}
\mathcal{F}(U_{n+1})&=
\mathcal{F}(U_n)-L_n \tilde{\Pi}_n \textbf{T}_n \Pi_n \mathcal{F}(U_n)+Q_n
=\mathcal{F}(U_n)-L_n \textbf{T}_n \Pi_n \mathcal{F}(U_n)
+L_n \tilde{\Pi}_n^{\perp} \textbf{T}_n \Pi_n \mathcal{F}(U_n)+Q_n\\
&=\mathcal{F}(U_n)-\Pi_n L_n \textbf{T}_n \Pi_n \mathcal{F}(U_n)
+(L_n \tilde{\Pi}_n^{\perp}-\Pi_n^{\perp}L_n)
\textbf{T}_n \Pi_n \mathcal{F}(U_n)+Q_n
%\\&
=\Pi_n^{\perp} \mathcal{F}(U_n)+R_n+Q_n+Q'_n
\end{aligned}
\end{equation*}
where
\begin{equation}
R_n:=(L_n \tilde{\Pi}_n^{\perp}-\Pi_n^{\perp}L_n)
\textbf{T}_n \Pi_n \mathcal{F}(U_n)\,, 
\quad Q'_n:=-\Pi_n (L_n \textbf{T}_n-\mathrm{I})
\Pi_n \mathcal{F}(U_n)\,.
\end{equation}
\begin{lem}\label{comekdv}
Define 
\begin{equation}\label{DefWnBnDP}
w_n:=\varepsilon \gamma^{-2} \lVert \mathcal{F}(U_n) \rVert_{s_0}\,, 
\qquad B_n:=\varepsilon\gamma^{-1} \lVert \mathfrak{I}_n \rVert_{s_0+\beta_1}
+\varepsilon\gamma^{-2} \lVert \mathcal{F}(U_n) \rVert_{s_0+\beta_1}\,.
\end{equation}
Then there exists $K:=K(s_0, \beta_1)>0$ 
such that, for all $n \geq 0$, 
setting $\alpha_0:=3\mu+3$
%\comment{questo alpha zero e' giusto?}
\begin{equation}\label{disuguaglianzeDP}
w_{n+1}\le K N_n^{\alpha_0+\rho^{-1}-\beta_1} B_n
+K N_{n}^{\alpha_0} w_n^2\,, 
\qquad B_{n+1}\le K N_n^{\alpha_0+\rho^{-1}} B_n\,.
\end{equation}
\end{lem}
\noindent
The proof of Lemma \ref{comekdv} follows almost word by word the proof of Lemma $9.2$ in \cite{KdVAut}.\\
\noindent
\textit{Proof of $(\mathcal{P}_3)_{n+1}$}. 
By \eqref{disuguaglianzeDP} and $(\mathcal{P}_3)_n$
\begin{equation}\label{BoundDP}
B_{n+1}\le K N_n^{\alpha_0+\rho^{-1}} B_n
\le 2 C_* K \varepsilon^{b_*+1} \gamma^{-2} 
N_n^{\alpha_0+\rho^{-1}} N_{n-1}^k\le 
C_* \varepsilon^{b_*+1} \gamma^{-2} N_n^k\,,
\end{equation}
provided $2 K N_n^{\alpha_0+\rho^{-1}-k}N_{n-1}^k\le 1, \forall n\geq 0$. 
Choosing $k$ as in \eqref{parametriNMdp} 
and $N_0$ large enough, i.e. for $\varepsilon$ small enough. 
By \eqref{DefWnBnDP} and the bound \eqref{BoundDP} 
$(\mathcal{P}_3)_{n+1}$ holds.\\
\textit{Proof of $(\mathcal{P}_2)_{n+1}$}. 
Using \eqref{DefWnBnDP}, \eqref{disuguaglianzeDP} 
and $(\mathcal{P}_2)_n, (\mathcal{P}_3)_n$, we get
\[
w_{n+1}\le K N_n^{\alpha_0+\rho^{-1}-\beta_1} B_n
+K N_n^{\alpha_0} w_n^2\le 
K N_n^{\alpha_0+\rho^{-1}-\beta_1} 
2 C_* \varepsilon^{b_*+1} \gamma^{-2} N_{n-1}^k+K N_n^{\alpha_0} (C_*\varepsilon^{b_*+1}\gamma^{-2} N_{n-1}^{-\alpha})^2
\]
and $w_{n+1}\le C_* \varepsilon^{b_*+1} \gamma^{-2} N_n^{-\alpha}$ 
provided that
\begin{equation}\label{9.36dp}
4 K N_n^{\alpha_0+\rho^{-1}-\beta_1+\alpha} N_{n-1}^k\le 1\,, 
\quad 2 K C_* \varepsilon^{b_*+1} \gamma^{-2}N_n^{\alpha_0+\alpha} 
N_{n-1}^{-2 \alpha}\le 1\,, \,\, \forall n\geq 0\,.
\end{equation}
The inequalities in \eqref{9.36dp} hold by 
\eqref{SmallnessConditionNMdp}, 
taking $\alpha$ as in \eqref{parametriNMdp}, 
$C_1>\alpha_0+\alpha$ and $\delta_0$ 
in \eqref{SmallnessConditionNMdp} small enough. 
By \eqref{DefWnBnDP}, the inequality 
$w_{n+1}\le 
C_* \varepsilon^{b_*+1} \gamma^{-2} N_n^{-\alpha}$ 
implies $(\mathcal{P}_2)_{n+1}$.
\item \textit{Proof of $(\mathcal{P}_1)_{n+1}$}. 
The bound \eqref{FrakHatDP} for 
$\hat{\mathfrak{I}}_1$ 
follows by \eqref{HnDefDP}, 
\eqref{9.10dp} 
(for $s=s_0+\mu$) and 
$
\lVert \mathcal{F}(U_0) \rVert_{s_0+2 \mu}
%=\lVert \mathcal{F}((\varphi, 0, 0), 0) \rVert_{s_0+2\mu}
\lesssim_{s_0+2\mu} \varepsilon^{b_*}\,.
$
The bound \eqref{FrakHatDP} for $\hat{\mathfrak{I}}_{n+1}$ 
follows by \eqref{TagliDP}, $(\mathcal{P}_2)_n$ 
and \eqref{parametriNMdp}. 
It remains to prove that \eqref{ConvergenzaDP} 
holds at the step $n+1$. We have
\begin{equation}
\lVert \mathfrak{I}_{n+1} \rVert_{s_0+\mu}
\le \sum_{k=1}^{n+1} \lVert \hat{\mathfrak{I}}_k \rVert_{s_0+\mu}
\le C_* \varepsilon^{b_*} \gamma^{-1}
 \sum_{k \geq 1} N_{k-1}^{-\alpha_1}\lesssim 
 C_* \varepsilon^{b_*} \gamma^{-1}
\end{equation}
taking $\alpha_1$ as in \eqref{parametriNMdp} 
and $N_0$ large enough, i.e. $\varepsilon$ small enough. 
Moreover, using \eqref{TagliDP}, $(\mathcal{P}_2)_{n+1}, 
(\mathcal{P}_3)_{n+1}$, \eqref{parametriNMdp} we get
\begin{align*}
\lVert \mathcal{F}(U_{n+1}) \rVert_{s_0+\mu+1} 
&\le N_n^{\mu+1} \lVert \mathcal{F}(U_{n+1}) \rVert_{s_0}
+N_n^{\mu+1-\beta_1}\lVert \mathcal{F}(U_{n+1}) \rVert_{s_0+\beta_1}\\
&\le C_* \varepsilon^{b_*} N_n^{\mu+1-\alpha} 
+C_* \varepsilon^{b_*}N_n^{\mu+1-\beta_1+k}\lesssim C_* \varepsilon^{b_*}\,,
\end{align*}
which is the second inequality in \eqref{ConvergenzaDP} at the step $n+1$. The bound $\lvert \zeta_{n+1} \rvert^{\gamma}\lesssim \lVert \mathcal{F}(U_{n+1}) \rVert_{s_0}^{\gamma}$ is a consequence of Lemma $6.1$ in \cite{Giuliani}.
%\eqref{Lemma6.1DP}.

\noindent
To conclude the proof of Theorem \eqref{NashMoserDP} 
it remains to show the bounds \eqref{MisureDP}. 
This is done in the next section.
\end{proof}

\subsection{Measure estimates}\label{sezioneStimeMisura}

Let us define for $0<\eta\le\overline{\jmath}_1^4\,\sqrt{\varepsilon}$, $\sigma\geq 1$ and $n\in \mathbb{N}$
\begin{align}
R_{\ell j k}(\eta, \sigma):=R_{\ell j k}(i_{n},\eta, \sigma)
&:=\{ \omega\in\mathcal{G}_n : \lvert \omega\cdot \ell
+d_j^{\infty}-d_k^{\infty} \rvert\le 
{2\eta}\langle \ell \rangle^{-\sigma}\}\,,\label{singoloEq}\\
Q_{\ell j }(\eta, \sigma):=Q_{\ell j }(i_n,\eta, \sigma)
&:=\{ \omega\in\mathcal{G}_n : \lvert \omega\cdot \ell+m j \rvert\le 
{2\eta}\langle \ell \rangle^{-\sigma}\}\,,\label{singoloeq3}\\
P_{\ell j }(\eta, \sigma):=P_{\ell j }(i_n,\eta, \sigma)
&:=\{ \omega\in \mathcal{G}_n : \lvert \omega\cdot \ell+d_j^{\infty} \rvert\le 
{2\eta}\langle \ell \rangle^{-\sigma}\}\,.\label{singoloEq2}
\end{align}
Recalling \eqref{GnDP} we can write, setting 
$\eta\rightsquigarrow \gamma_n$ for the sets  
$Q_{\ell j }(\eta,\s)$ and $P_{\ell j }(\eta,\s)$,
 $\eta\rightsquigarrow \gamma^{*}_n$ for the set $R_{\ell jk}(\eta,\s)$, 
 and $\sigma\rightsquigarrow \tau$, 
\begin{equation}\label{UnionDP}
\mathcal{G}_n \setminus \mathcal{G}_{n+1}=
\bigcup_{\ell\in\mathbb{Z}^{\nu}, j, k\in S^c} 
\Big( R_{\ell j k}(i_n,\gamma_{n}^{*},\tau) \cup 
Q_{\ell j }(i_n,\gamma_{n},\tau) \cup P_{\ell j}(i_n,\gamma_{n},\tau)\Big)\,.
\end{equation}

Since, by \eqref{0di0Melnikov} and $\gamma>\gamma^{3/2}$ 
(see \eqref{SmallnessConditionNMdp}), 
$R_{\ell j k}(i_n)=\emptyset$ for $j=k$, 
in the sequel we assume that $j\neq k$.
We start with a preliminary lemma, 
which gives a first relation between $\ell,j,k$ which must be 
satisfied in order to have non empty resonant sets.
\begin{lem}\label{BrexitDP}
Let $n \geq 0$.
There is a constant $C>0$  dependent 
of the tangential set and 
independent of $\ell, j, k, n, i_n, \omega$
such that the following holds:
\begin{itemize}
\item if $R_{\ell j k}(i_n,\eta,\s)\neq \emptyset$ \  
then 
\ $\lvert \ell \rvert\geq C \lvert \og(j)-\og(k) \rvert
\geq \frac{C}{2}\,\lvert j-k \rvert$;
\item if $Q_{\ell j}(i_n,\eta,\s)\neq \emptyset$ \  
then 
\ $\lvert \ell \rvert\geq C \lvert j \rvert$;
\item if $P_{\ell j}(i_n,\eta,\s)\neq \emptyset$ \  
then \ $\lvert \ell \rvert\geq C \lvert j \rvert$.
\end{itemize}
\end{lem}

\begin{proof}
If $R_{\ell j k}(i_n)\neq \emptyset$, 
then there exists $\omega$ such that 
\begin{equation*}
\lvert d_j^{\infty}(\omega, i_n(\omega))-d_k^{\infty}(\omega, i_n(\omega)) \rvert
< {2 \eta}{\langle \ell \rangle^{-\sigma}}
+2 \lvert \overline{\omega}\cdot \ell \rvert\,.
\end{equation*}
Moreover, using \eqref{FinalEigenvaluesDP}, 
\eqref{stimeautovalfinaliDP}, \eqref{clinica100}, 
\eqref{diagonalopFinale}, we get
$\lvert d_j^{\infty}(\omega, i_n(\omega))-d_k^{\infty}(\omega, i_n(\omega)) \rvert 
\geq \frac{1}{3} \lvert \og(j)-\og(k) \rvert.
$
Thus, for $\varepsilon$ small enough 
\[
2 \lvert \overline{\omega} \rvert \lvert \ell \rvert
\geq 2 \lvert \overline{\omega}\cdot \ell \rvert
\geq\left( \frac{1}{3}-\frac{2 \eta}{\langle \ell \rangle^{\sigma}
\lvert \og(j)-\og(k) \rvert} \right)\lvert \og(j)
-\og(k) \rvert\geq \frac{1}{4}\lvert \og(j)-\og(k) \rvert
\]
and this proves the first claim on $R_{\ell j k}$.
If  $Q_{\ell j}\neq \emptyset$ then we have 
$|m j|< {2 \eta}{\langle \ell \rangle^{-\sigma}}
+2 \lvert \overline{\omega}\cdot \ell \rvert$. 
Hence, for $\e$ small enough, we have
\[
\lvert j \rvert\le \frac{\lvert \omega\cdot \ell \rvert}{\lvert m \rvert}
\le \frac{1}{\tilde{C}} \lvert \ell \rvert, 
\qquad \tilde{C}:=\frac{\lvert m \rvert}{4 \lvert \overline{\omega} \rvert}.
\]
Following the same arguments and 
by using that $\lvert d_j \rvert\geq C \lvert j \rvert$ 
for some constant $C>0$ we get the last statement.
\end{proof}

\subsubsection{Measure of a resonant set}\label{singleSET}
The aim of this subsection is to prove the following lemma.
\begin{lem}\label{singolo}
There is $\mathtt{r}_0>0$ such that, 
for any $0<\mathtt{r}\leq\mathtt{r}_0$, 
and any choice of $S^{+}\in \mathcal{V}(\mathtt{r})$
we have that 
\begin{equation}\label{stimaBadERRE}
\lvert R_{\ell j k}(\eta, \sigma) \rvert\le 
K \varepsilon^{2(\nu-1)} \eta \langle \ell \rangle^{-\sigma}\,,
\end{equation}
for some $K=K(S)$.
The same holds for $Q_{\ell j}(\eta,\s)$ and $P_{\ell j}(\eta,s)$.
\end{lem}

\noindent
The proof of the lemma above involves many arguments and we split it into several steps. \\
In several bounds we will evidence the dependence of the constants on the tangential set $S$ in order to highlight that the smallness of the amplitudes $\xi$ depends on the choice of the tangential sites.

Let us first consider the set $R_{\ell j k}$,
which is the most difficult case.
We study the sub-levels of the function
$\omega\mapsto\phi_R(\omega)$ defined by (recall \eqref{FreqAmplMapDP},\eqref{FinalEigenvaluesDP})
\begin{equation}\label{phi(omega)DP}
\begin{aligned}
&\phi_R(\omega):=\mathrm{i} \omega\cdot \ell+d_j^{\infty}(\omega)-d_k^{\infty}(\omega)
=\mathrm{i} \omega\cdot \ell+\mathrm{i} m(\omega) (\og(j)-\og(k))+
\mathrm{i} \varepsilon^2(\kappa_j-\kappa_k)(\omega)
+(r_j^{\infty}-r_k^{\infty})(\omega).
\end{aligned}
\end{equation}
We recall that (see \eqref{clinica100}, \eqref{clinica})
\begin{equation}\label{ferro} 
\begin{aligned}
&m=1+\varepsilon^2 c(\oo)+\mathtt{r}_{m}(\omega), 
\quad c(\oo)=\vec{v}\cdot \xi(\oo)\,,  
\quad 
\vec{v}:=(2/3)(1+\overline{\jmath}_k^2)_{k=1}^{\nu}\in\mathbb{R}^{\nu}, 
\quad  \kappa_j(\oo)=\vec{w}_j \cdot \xi(\oo)\,,
\end{aligned}
\end{equation}
where $\kappa_{j}$ is defined in \eqref{diagonalopFinale} (see also \eqref{dXi}) and 
\begin{equation}\label{rm}
\mathtt{r}_{m}:=\varepsilon^4 d(\oo)+O(\varepsilon^{10}\gamma^{-2})\,, 
\quad \lvert\mathtt{r}_{m}\rvert^{\gamma}\lesssim 
\overline{\jmath}_1  \varepsilon^4\,, 
\quad \lvert\nabla_{\omega} \mathtt{r}_{m}\rvert\lesssim 
\overline{\jmath}_1\varepsilon^2+O(\varepsilon^{10} \gamma^{-3})\,. 
\end{equation}
We first study some properties 
of the function $\phi_{R}(\omega)$ in \eqref{phi(omega)DP}.

\vspace{0.9em}
\noindent
{\bf The small divisor $\phi_{R}(\omega)$ as 
an affine function of $\omega$.}
It will be useful to consider $\phi_R(\omega)$
 in \eqref{phi(omega)DP} as a small perturbation 
 of an affine function in $\omega$
\begin{equation}\label{phi(omega)DPOmega}
\begin{aligned}
&\phi_R(\omega)
=a_{j k}+b_{\ell j k}\cdot \omega+q_{j k}(\omega)\,, 
\qquad \ell\in\mathbb{Z}^{\nu},\, j, k\in S^c\,.
\end{aligned}
\end{equation}
where
\begin{align}
a_{j k}&:=\mathrm{i} \Big( (\og(j)-\og(k)) 
[1- \vec v\cdot \mathbb{A}^{-1}\overline{\omega}]
+(\vec{w}_k-\vec{w}_j) \cdot\mathbb{A}^{-1}\overline{\omega}\Big)\,, 
\label{AjkDP}\\
b_{l j k}&:=\mathrm{i} \Big(\ell+ (\og(j)-\og(k)) \mathbb{A}^{-T} \vec v 
+\mathbb{A}^{-T}(\vec{w}_j-\vec{w}_k)\Big)\,, \label{BljkDP} 
\end{align}
and the remainder $q_{j k}(\omega)$ satisfies
\begin{equation}\label{LipQljkDP}
\begin{aligned}
\lvert q_{j k}(\omega) \rvert^{sup}&\lesssim 
\overline{\jmath}_1 \varepsilon^{4}\lvert j-k\rvert+\varepsilon^{4-3a}\,,\\[2mm]
\lvert q_{j k}(\omega) \rvert^{lip}&\le 
\lvert \mathtt{r}_{m}(\omega) \rvert^{lip}\lvert\og(j)-\og(k)\rvert
+\lvert r_j^{\infty}-r_k^{\infty}\rvert^{lip}\lesssim 
\overline{\jmath}_1 \varepsilon^{2}\lvert j-k\rvert+\varepsilon^{1- 4a}\,.
\end{aligned}
\end{equation}
%By Lemma \ref{BrexitDP} it is sufficient to study the measure of the resonant sets 
%$R_{\ell j k}(i_n)$ defined in \eqref{singoloEq} for $(\ell, j, k)\neq (0, j, j)$. 

%\vspace{0.8em}
\noindent
\begin{lem}\label{mammachestime}
Denoting  $\vec p_j = \og(j)\vec v+ \vec w_j$, 
we have the following bounds: %uniformly in  $i=1,\dots, \nu$:
\begin{align*}
&\lvert \vec p_j \rvert \lesssim \bar\jmath_1^4|j|\,, 
\qquad \lvert \vec{w}_j \rvert\lesssim \bar\jmath_1^6|j|^{-1}\,,\qquad
%\label{mamm1}
%\\
%&
\lvert  \vec{p}_j-  \vec{p}_k \rvert \lesssim \bar\jmath_1^4 \lvert j-k \rvert\,,
%\label{mamm2}
%\\
%&
\qquad \lvert  \vec{w}_j-  \vec{w}_k \rvert \lesssim  {\bar\jmath_1^8} 
\lvert j-k \rvert (|j|^{-2}+|jk|^{-1})\,.
%\label{mamm3}
\end{align*}
\end{lem}
\begin{proof}
The first bound 
%\eqref{mamm1} 
follows by the fact that
$3+\overline{\jmath}_i^2+j^2 \pm j_2 
\overline{\jmath}_i\geq 3+\frac{j^2+\overline{\jmath}_i^2}{2}
\geq \frac{j^2}{2}$ and $(1+j^2)(1+\overline{\jmath}_i^2)(2+j^2
+\overline{\jmath}_i^2)\lesssim \bar\jmath_1^4 \,j^4$.
The others follow similarly.
\end{proof}

\noindent
Fix $\alpha\in(0,1/2)$ and let
\begin{equation}\label{ecoeco}
\gamma_0=\varepsilon^{\alpha}, \quad  0<\beta<{(2-\alpha)}{(\sigma+1)^{-1}}.
%\frac{2-\alpha}{\sigma+1}.
\end{equation}
We have the following estimates for sets $R_{\ell j k}$ with $|\ell|$ ``large''.
\begin{lem}\label{ellegrande}
Let $\lvert \ell \rvert> \varepsilon^{-\beta}$. 
Then $R_{\ell jk}(\eta,\sigma)$ satisfies \eqref{stimaBadERRE}.
%$\lvert R_{\ell j k}(\eta, \sigma)\rvert\le 
%K(S) \varepsilon^{2(\nu-1)}\eta \langle \ell \rangle^{-\sigma}$ for some $K(S)>0$.
\end{lem}
\begin{proof} 
In this proof we shall denote by $C(S)$ a 
running positive constant depending on the set $S$.
Suppose that $\lvert j-k \rvert\le c_0 \lvert \ell \rvert$ 
with $c_0=c_0(S)$ small. 
By Lemmata \ref{mammachestime} and \ref{Twist1}
we have 
$\lvert  \mathbb{A}^{-T}(\vec p_j -\vec p_k)\rvert
\lesssim \overline{\jmath}_1^4 \lvert j-k \rvert < {\lvert \ell \rvert}/{2}
$
%By \eqref{mamm2} and Lemma \ref{Twist1}
%\begin{align*}
%\lvert  \mathbb{A}^{-T}(\vec p_j -\vec p_k)\rvert
%\lesssim \overline{\jmath}_1^4 \lvert j-k \rvert < \frac{\lvert \ell \rvert}{2}
%\end{align*}
for $c_0$ sufficiently small. 
%($c_0=O(  \overline{\jmath}_1^{-4})$).  
This means that $|b_{\ell jk}| \gtrsim |\ell|/2$.
Now suppose that $\lvert j-k \rvert> c_0 \lvert \ell \rvert$.
Then
\begin{align*}
\rvert a_{j k}\lvert &\geq \lvert \og(j)-\og(k) \rvert 
\Big( \lvert 1-\mathbb{A}^{-1}\overline{\omega}
\cdot \vec{v}\rvert-\frac{\mathbb{A}^{-T}
(\vec{w}_j-\vec{w}_k)}{\lvert \og(j)-\og(k) \rvert} \rvert  \Big)\\
& \stackrel{\eqref{corto100},\eqref{spero}}{\geq}  \lvert j-k \rvert 
\Big(\frac12- \frac{C(S)\lvert j-k \rvert}{c_0\,\lvert \ell \rvert\,\lvert j k \rvert}  \Big)
\geq \lvert j-k \rvert \Big(\frac12- 2\,c_0^{-1}\,C(S) \varepsilon^{\beta}  \Big)
\geq  \lvert j-k \rvert/4
\end{align*}
%for some $C(S)>0$ and 
for $\varepsilon$ small enough.
By \eqref{singoloEq}, \eqref{LipQljkDP}
\[
2\lvert b_{\ell j k} \rvert\lvert \overline{\omega} \rvert 
\ge \lvert b_{\ell j k} \cdot \omega \rvert
\geq \lvert a_{j k} \rvert -\lvert \phi_{\ell j k}(\omega) \rvert-\lvert q_{j k}(\omega) \rvert
\geq \Big(\frac14-\frac{2\eta}{c_0 \langle \ell \rangle^{\sigma+1}}
-C(S) \varepsilon^4-\frac{\varepsilon^{4-3 a}}{\lvert j-k \rvert} \Big)
\lvert j-k \rvert\geq \frac{1}{8} \lvert j -k\rvert,
\]
for $\varepsilon$ small enough and $\sigma\geq 1$. 
Again we have shown that $|b_{\ell jk}|>\delta |\ell|$ 
with $\delta:= c_0/ 2\lvert \overline{\omega} \rvert $.
Split $\omega=s \hat{b}+b^\perp$ where 
$\hat{b}:=b/\lvert b \rvert$ and $b^\perp\cdot b=0$. 
Let $\Psi_R(s):=\phi_R(s \hat{b}+b^\perp)$. 
For $\varepsilon$ small enough, by \eqref{LipQljkDP}, we get
\[
\lvert \Psi_R(s_1)-\Psi_R(p) \rvert
\geq (\lvert b \rvert-\lvert q_{j k} \rvert^{lip}) \lvert s_1-p \rvert
\geq \left(\delta-C(S)\e^2-\frac{\varepsilon^{1-4 a}}{\lvert j-k \rvert}\right)\,
\lvert j -k\rvert\, \lvert s_1-p \rvert
\geq \frac{\delta_1}{2} \lvert j -k\rvert \,\lvert s_1-p \rvert.
\]
The lemma follows by Fubini's theorem.
\end{proof}

We  now prove that if the main term (in size) of $\phi_{R}(\bar\omega)$
%$ \omega\cdot \ell +d_j^{\infty}-d_k^{\infty}$ 
is big enough and $\lvert \ell \rvert$ is bounded by some constant then the bad set  $R_{\ell j k}(\eta, \sigma)=\emptyset$.
We remark that
\[
\phi_{R}(\bar\omega) - q_{jk}(\bar\omega) = a_{jk} +b_{\ell jk}\cdot\bar\omega = \overline{\omega}\cdot \ell+\og(j)-\og(k)\,.
\]

\begin{lem}\label{singolo2}
	If $\lvert \ell \rvert\le\varepsilon^{-\beta}$  and
	$\lvert \overline{\omega}\cdot \ell+\og(j)-\og(k) \rvert\geq \gamma_0 \langle \ell \rangle^{-\s}$ (see \eqref{ecoeco})
	then $R_{\ell j k}(\eta, \sigma)=\emptyset$.
\end{lem}

\begin{proof}
	By definition
	\[
	\lvert \omega\cdot \ell +d_j^{\infty}-d_k^{\infty} \rvert \ge \gamma_0 \langle \ell \rangle^{-\s} -|b_{\ell jk}||\omega-\bar\omega|-2 |q_{jk}|^{sup}
	\]
	%\comment{Cos'e' $q_{jk}^{\infty}$? forse $q_{j k}$}
	By Lemma \ref{BrexitDP} (recall \eqref{phi(omega)DP}) we have $\lvert j-k \rvert\lesssim C \lvert \ell \rvert$ and so
	\begin{equation}
	\begin{aligned}
	|b_{\ell jk}||\omega-\bar\omega|+2 |q_{jk}|^{sup}
	&\leq C(S)\e^2 (|\ell|+ |j-k|) \lesssim C(S)\e^2 |\ell| 
	\lesssim  C(S)\e^{2-\beta} \le  \e^{\alpha+\s\beta}/2 
	\le \frac{\gamma_0}{2\langle \ell \rangle^{\sigma}} 
	\end{aligned}
	\end{equation}
	%\comment{non capisco la penultima}
	for some $C(S)>0$, 
	for $\varepsilon$ small enough, by \eqref{ecoeco}. 
\end{proof}

\begin{lem}\label{ellepiccoloris}
	Let $\lvert \ell \rvert\le \varepsilon^{-\beta}$ and
	$\lvert \overline{\omega}\cdot \ell+\og(j)-\og(k) \rvert\le {\gamma_0}{\langle \ell \rangle^{-\sigma}}$.
	Then $R_{\ell jk}(\eta,\sigma)$ satisfies \eqref{stimaBadERRE}.
	%Then $\lvert R_{\ell j k}(\eta, \sigma)\rvert\le 
	%K(S) \varepsilon^{2(\nu-1)}\eta \langle \ell \rangle^{-\sigma}$ for some $K(S)>0$.
\end{lem}

\begin{proof}
	Let us call $p:=\overline{\omega}\cdot \ell+\og(j)-\og(k)$ and note that $\lvert p \rvert\lesssim \,\varepsilon^{\alpha}$ (recall \eqref{ecoeco}).
	We also remark that $\ell\ne 0$ since  for $j\ne k$  one has $ |\og(j)-\og(k)|>1/2$. We substitute $p$ in the definition of $b_{\ell jk}$ (see \eqref{BljkDP})
	\[
	|b_{\ell jk}|= \Big|\ell+ (-\bar\omega\cdot \ell + p ) \mathbb{A}^{-T} \vec v +\mathbb{A}^{-T}(\vec{w}_j-\vec{w}_k)\Big| \gtrsim 
	|\ell-  \mathbb{A}^{-T} \vec v \bar\omega^T \ell +\mathbb{A}^{-T}(\vec{w}_j-\vec{w}_k) |+  \varepsilon^{\alpha}.
	\]
%\comment{qualche segno e' sbagliato}
	Then, using \eqref{corto} in Lemma \eqref{ecologia} we have $|b_{\ell jk}|\geq  |\ell|\delta/2$ for $\varepsilon$ small enough. 
	The thesis follows reasoning as
	in Lemma \eqref{ellegrande}.
\end{proof}

\begin{proof}[{\bf Proof of Lemma \ref{singolo}}]
	For the sets $R_{\ell j k}$ the lemma follows by Lemmata \ref{ellegrande}, \ref{singolo2},
	\ref{ellepiccoloris}.
	% \ref{singolo3}, \ref{singolo4}.
	The proof for the sets $Q_{l j}$ and $P_{\ell j}$ follows using the same arguments used for $R_{\ell j k}$. Lemmata \ref{ellegrande}, \ref{singolo2} are identical, with the only difference that the non-resonance condition now reads respectively 
	$
	|\bar\omega\cdot\ell + j|\ge \gamma_0\langle\ell\rangle^{-\s}\,, |\bar\omega\cdot\ell +\og(j) |\ge \gamma_0\langle\ell\rangle^{-\s}
	$
	in the case of $Q_{\ell j}$ and $P_{\ell j}$. 
	Regarding Lemma  \ref{ellepiccoloris}, it follows from \eqref{corto100} in the case of $Q_{\ell j}$ and from \eqref{cortissimo} in the case of $P_{\ell j}$.
\end{proof}

\subsubsection{Summability}\label{siSomma}

\begin{lem}\label{InclusionideiBadSetsDP}
For $n\geq 1, \lvert \ell \rvert\le N_{n-1}$, one has  $R_{\ell j k}(i_n),Q_{\ell j}(i_n),P_{\ell j}(i_n)=\emptyset$.
\end{lem}
\begin{proof}
We first note that, by Lemma \ref{BrexitDP}, if $|\og(j)-\og(k)|> C_1^{-1}|\ell|$
(for some $C_1=C_1(S)$)
 then $R_{\ell j k}(i_n)= R_{\ell j k}(i_{n-1})=\emptyset$, so that our claim is trivial.  Otherwise,  if 
\[
|\og(j)-\og(k)|\le  C_1^{-1}|\ell|\le C_1^{-1} N_{n-1}\,
\]
By \eqref{r12} (with $i_\d^{(1)}\rightsquigarrow i_n$ and $i_\d^{(2)}\rightsquigarrow i_{n-1}$, $N\rightsquigarrow N_{n-1}$) 
and \eqref{FrakHatDP}  we have for all $j, k\in S^c$ 
\begin{equation}\label{marathon1}
\lvert (d_j^{\infty}-d_k^{\infty})(i_n)-(d_j^{\infty}-d_k^{\infty})(i_{n-1})\rvert\le \varepsilon^{4-3 a}
N_{n-1}^{-\mathtt{a}}\qquad\qquad \forall\omega\in \mathcal{G}_n,
\end{equation}
where $\mathtt{a}:=\min\{\kappa,\alpha\}$ (recall $\alpha$ in \eqref{parametriNMdp} and $\kappa$ in \eqref{ReducibilityDP}).
Now for all $j\neq k$, $\lvert \ell \rvert\le N_{n-1}$, $\omega\in\mathcal{G}_n$ by \eqref{marathon1} 
\begin{equation}
\begin{aligned}
&\lvert \omega\cdot \ell +d_j^{\infty}(i_n)-d_k^{\infty}(i_{n})\rvert\geq \lvert \omega\cdot \ell +d_j^{\infty}(i_{n-1})-d_k^{\infty}(i_{n-1})\rvert-\lvert (d_j^{\infty}-d_k^{\infty})(i_n)-(d_j^{\infty}-d_k^{\infty})(i_{n-1})\rvert\\
&\geq 2\gamma^{*}_{n-1}\langle \ell \rangle^{-\tau}-\varepsilon^{4-3 a}N_n^{-\mathtt{a}}\geq 2\gamma^{*}_{n}\langle \ell \rangle^{-\tau}
\end{aligned}
\end{equation}
since $\varepsilon^{4-3 a}\gamma^{-3/2}N_n^{\tau-(2/3)\mathtt{a}}2^{n+1}\le 1$. Since by definition $R_{\ell j k}(i_n)\subseteq \mathcal G_n$ then 
$R_{\ell j k}(i_n)=\emptyset$ .

\vspace{0.5em}

\noindent
Now we prove that $Q_{\ell j}(i_{n-1}) \subseteq Q_{\ell j}(i_n)$. We have
\begin{equation}
\begin{aligned}
\lvert m(i_n)-m(i_{n-1}) \rvert\lvert j \rvert\stackrel{(\ref{clinica})}{\le} & C \varepsilon^3 \lVert i_{n}-i_{n-1} \rVert_{s_0+2}\lvert j \rvert \stackrel{(\ref{FrakHatDP})}{\le} C \varepsilon^{b_*+3}\gamma^{-1} N_{n-1}^{-\alpha} \lvert j \rvert\\
\le\,\, &\, C \varepsilon^{b_*+3}\gamma^{-1} N_{n-1}^{-\alpha} \lvert \ell \rvert
\end{aligned}
\end{equation}
and then
\begin{equation}
\begin{aligned}
\lvert \omega\cdot \ell +m (i_{n}) j \rvert &\geq \lvert \omega\cdot \ell +m (i_{n-1}) j \rvert-\lvert m(i_n)-m(i_{n-1}) \rvert\lvert j \rvert\\
&\geq 2 \gamma_{n-1} \langle \ell \rangle^{-\tau}- \varepsilon^{b_*+3}\gamma^{-1} N_{n-1}^{-\alpha+1} \geq 2 \gamma_n \langle \ell \rangle^{-\tau}
\end{aligned}
\end{equation}
since $\lvert \ell \rvert\le N_{n-1}$.

\vspace{0.5em}
\noindent
As before, by \eqref{r12}, for all $j, k\in S^c$
\begin{equation}\label{marathon100}
\lvert d_j^{\infty}(i_n)-d_j^{\infty}(i_{n-1})\rvert\le \varepsilon^{4-3 a}N_{n-1}^{-\mathtt{a}}\qquad\qquad \forall\omega\in \mathcal{G}_n.
\end{equation}
 For all $j\neq k$, $\lvert \ell \rvert\le N_{n-1}$, $\omega\in\mathcal{G}_n$ by \eqref{marathon1} 
\begin{equation}
\begin{aligned}
&\lvert \omega\cdot \ell +d_j^{\infty}(i_n)\rvert\geq \lvert \omega\cdot \ell +d_j^{\infty}(i_{n-1})\rvert-\lvert d_j^{\infty}(i_n)-d_j^{\infty}(i_{n-1})\rvert\\
&\geq 2\gamma_{n-1} \langle \ell \rangle^{-\tau}-\varepsilon^{4-3 a}N_{n-1}^{-\mathtt{a}}\geq 2\gamma_{n} \langle \ell \rangle^{-\tau}
\end{aligned}
\end{equation}
since $\varepsilon^{4-3 a}\gamma^{-1}N_n^{\tau-(2/3)\mathtt{a}}2^{n+1}\le 1$.
\end{proof}
\noindent
We have proved that
\begin{equation}\label{differenzeInsiemiMisuraDP}
\mathcal{G}_n \setminus \mathcal{G}_{n+1}\subseteq \bigcup_{\substack{j, k\in S^c\\ \lvert \ell \rvert> N_{n-1}}}\Big( R_{\ell j k}(i_n)\cup Q_{\ell j}(i_n)\cup P_{\ell j}(i_n)\Big),  \quad \forall n\geq 1.
\end{equation}

\begin{lem}\label{delpiero}
There exists $\mathtt{C}>0$ such that
if $\lvert j \rvert, \lvert k \rvert\geq \mathtt{C} \langle \ell \rangle^{\nu+2}\gamma^{-(1/2)}$ then (recall that $\tau=2\nu+6>\nu+2$)
\begin{equation}
R_{\ell j k}(\gamma^{3/2}, \tau)\subseteq Q_{\ell, j-k}(\gamma, \nu+2).
\end{equation}
\end{lem}
\begin{proof}
We have that
\begin{equation}
\begin{aligned}
\lvert \omega\cdot \ell +d_j^{\infty}-d_k^{\infty} \rvert &\geq \lvert \omega\cdot \ell +m(j-k) \rvert-\lvert m \rvert \lvert \og(j)-j+k-\og(k) \rvert-\varepsilon^2 \lvert w_j-w_k \rvert-\lvert r_j^{\infty}\rvert-\lvert r_k^{\infty}\rvert	\\
&\geq \frac{2\gamma}{\langle \ell \rangle^{\nu+2}}-2 \lvert j-k \rvert \frac{C}{\lvert j \rvert \lvert k \rvert}-\frac{\tilde{C}\varepsilon^2}{\min\{ \lvert j \rvert, \lvert k \rvert \}}\\
&\geq \frac{2\gamma}{\langle \ell \rangle^{\nu+2}}- \frac{C\gamma}{\mathtt{C}\langle \ell \rangle^{2(\nu+2)-1}}-\frac{\tilde{C}\varepsilon^2 \sqrt{\gamma}}{\mathtt{C} \langle \ell \rangle^{\nu+2}}\geq \frac{\gamma}{\langle \ell \rangle^{\nu+2}} \Big(2-\frac{C}{2\mathtt{C}\langle \ell \rangle^{\nu+1}}-\frac{\tilde{C}\varepsilon^2 }{2\sqrt{\gamma}\mathtt{C} }  \Big)\\
&\geq \frac{\gamma}{\langle \ell \rangle^{\nu+2}} \geq \frac{\gamma^{3/2}}{\langle \ell \rangle^{\tau}}
\end{aligned}
\end{equation}
for $\mathtt{C}$ big enough and since $\varepsilon^2 (\sqrt{\gamma})^{-1}\ll 1$.
\end{proof}

We are in position to prove \eqref{MisureDP}. We have, by \eqref{differenzeInsiemiMisuraDP},
%Lemma \ref{BadSetsDP}
\begin{equation*}
\left\lvert \bigcup_{\ell\in\mathbb{Z}^{\nu}, j, k\in S^c} R_{\ell j k}(i_n)  \right\rvert\le \sum_{\lvert\ell\rvert >N_{n-1}, \lvert j \rvert, \lvert k \rvert\geq \mathtt{C}\langle \ell \rangle^{\nu+2}\gamma^{-(1/2)}}|R_{\ell j k}(i_n)|+\sum_{\lvert\ell\rvert >N_{n-1}, \lvert j \rvert, \lvert k \rvert\leq2 \mathtt{C} \langle \ell \rangle^{\nu+2}\gamma^{-(1/2)}} |R_{\ell j k}(i_n)|.
\end{equation*}
On one hand we have that, using Lemmata \ref{singolo} and  \ref{delpiero}, 
\begin{equation*}
\begin{aligned}
\sum_{\lvert\ell\rvert >N_{n-1}, \lvert j \rvert, \lvert k \rvert\geq \mathtt{C} \langle \ell \rangle^{\nu+2}\gamma^{-(1/2)}}|R_{\ell j k}(i_n)|&\lesssim K\sum_{j-k=h, \lvert h \rvert
\leq C \lvert \ell \rvert} \varepsilon^{2(\nu-1)}\gamma \langle \ell \rangle^{-\nu-2}\lesssim  K\varepsilon^{2(\nu-1)}\gamma\sum_{\lvert \ell \rvert\geq N_{n-1}}  \langle \ell \rangle^{-(\nu+1)}\\
&\lesssim K \e^{2(\nu-1)}\g N_{n-1}^{-1}.
\end{aligned}
\end{equation*}
On the other hand
\begin{equation*}
\begin{aligned}
\sum_{\lvert\ell\rvert >N_{n-1}, \lvert j \rvert, \lvert k \rvert\leq 2\mathtt{C} \langle \ell \rangle^{\tau_1}\gamma^{-(1/2)}, \lvert j-k \rvert 
\le C\lvert \ell \rvert} |R_{\ell j k} (i_n)| 
&
\lesssim K\gamma^{(3/2)} \varepsilon^{2(\nu-1)}  
\sum_{\lvert \ell \rvert\geq N_{n-1}} \frac{\lvert \ell \rvert \langle \ell \rangle^{\nu+2}}{\sqrt{\gamma}\langle \ell \rangle^{\tau}} 
\\
&\lesssim  K\gamma \varepsilon^{2(\nu-1)} 
\sum_{\lvert \ell \rvert\geq N_{n-1}}  \langle \ell \rangle^{-(\tau-\nu-3)}\\
&\lesssim K \gamma \varepsilon^{2(\nu-1)} N_{n-1}^{-1}.
\end{aligned}
\end{equation*}
%We conclude by fixing $\tau_2=\tau_1=\nu+2$. Recalling \eqref{parametriKAM} we check that actually $\tau>\tau_1+\nu+2$.
The discussion above  implies estimates \eqref{MisureDP}.
%\end{proof}

%\textbf{Conclusion of the Proof of Theorem \ref{IlTeoremaDP}}. 
\subsection{Conclusion of the Proof of Theorem \ref{IlTeoremaDP}}
Theorem \ref{NashMoserDP} implies that the sequence $(\mathfrak{I}_n, \zeta_n)$ is well defined for $\omega\in \mathcal{G}_{\infty}:=\cap_{n\geq 0} \mathcal{G}_n$, $\mathfrak{I}_n$ is a Cauchy sequence in $\lVert \cdot \rVert_{s_0+\mu}^{\g,\calG_{\infty}}$ (see \eqref{FrakHatDP}) and $\lvert \zeta_n \rvert^{\gamma}\to 0$. Therefore $\mathfrak{I}_n$ converges to a limit $\mathfrak{I}_{\infty}$ in norm $\lVert \cdot \rVert_{s_0+\mu}^{\gamma,\calG_{\infty}}$ and, by $(\mathcal{P} 2)_n$, for all $\omega\in\mathcal{G}_{\infty}, i_{\infty}(\varphi):=(\varphi, 0, 0)+\mathfrak{I}_{\infty}(\varphi)$ is a solution of
\begin{equation}\label{eccolasol}
\mathcal{F}(i_{\infty}, 0)=0 \qquad \mbox{with} \qquad \lVert \mathfrak{I}_{\infty} \rVert_{s_0+\mu}^{\gamma,\mathcal{G}_{\infty}}\lesssim\,\varepsilon^{9- 2 b} \gamma^{-1}
\end{equation}
by \eqref{ConvergenzaDP}, \eqref{SmallnessConditionNMdp}. 
Therefore $\varphi \mapsto i_{\infty}(\varphi)$ is an invariant torus for the Hamiltonian vector field $X_{H_{\varepsilon}}$ (recall \eqref{HepsilonDP}). By \eqref{MisureDP},
\[
\lvert \Omega_{\varepsilon}\setminus \mathcal{G}_{\infty} \rvert \le \lvert \Omega_{\varepsilon} \setminus \mathcal{G}_0 \rvert+\sum_{n\geq 0} \lvert \mathcal{G}_n \setminus \mathcal{G}_{n+1} \rvert\le 2\,C_* \varepsilon^{2 (\nu-1)} \gamma+C_* \varepsilon^{2(\nu-1)}\gamma \sum_{n\geq 1} N_{n-1}^{-1}
\lesssim C_{*}\varepsilon^{2(\nu-1)} \gamma.
\]
The set $\Omega_{\varepsilon}$ in \eqref{OmegaEpsilonDP} has measure $\lvert \Omega_{\varepsilon} \rvert=O(\varepsilon^{2 \nu})$. Hence $\lvert \Omega_{\varepsilon}\setminus \mathcal{G}_{\infty} \rvert/\lvert \Omega_{\varepsilon} \rvert\to 0$ as $\varepsilon\to 0$ because $\gamma=o(\varepsilon^2)$, and therefore the measure of $\mathcal{C}_{\varepsilon}:=\mathcal{G}_{\infty}$ satisfies \eqref{frazionemisureDP}.

%\comment{la prova della stabilita' io la vorrei scritta alla Rob, come in riducibilita}

It remains to show the linear stability of the embedding $i_{\infty}(\f)$.
By the discussion of Section \ref{sezione6DP} (see also \cite{KdVAut} for further details) and Section \ref{regularization},
since $i_{\infty}(\f)$ is isotropic
and solves the equation \eqref{eccolasol}, it is possible to find a change of coordinates
$G_{\infty}$ (of the form \eqref{GdeltaDP}),
so that in the linearized system of the Hamiltonian $H_{\e}\circ G_{\infty}(\varphi, \eta, w)$ the equation for the actions is given by $\dot{\eta}=0$. Moreover, by Section \ref{regularization} the linear equation for the normal variables $w$ is conjugated, by setting $w=\Upsilon \circ \Phi_{\infty}(z)$ to the diagonal system $\dot{z}_j-\ii d_j^{\infty}(\omega)z_j=f_{j}(\omega t), j\in S^{c}$, where $f(\omega t)$ is a forcing term.\\
Since $d_j^{\infty}\in\mathbb{R}$ a standard argument shows that the Sobolev norms of $w$ do not increase in time. For further details see \cite{KdVAut}, \cite{GiulianiPhD}, \cite{BBHM}.

%
% if $K:=H_{\e}\circ G_{\infty}$
% (see \eqref{HamiltonianaRiscalataDP}),
% then the equation for the actions $\eta$ of the linearized system of $X_{K}(\varphi, \eta, w)$ is 

%\begin{equation}\label{linsistFin}
%\left\{\begin{aligned}
%&\dot{\f}\qquad \qquad\quad\,\,\,=K_{20}(\omega t)\eta+K_{11}^{T}(\omega t)w\\
%&\dot{\eta} \qquad\qquad\quad\;\,\,=0\\
%&\dot{w}-JK_{02}(\omega t) w=J K_{11}(\omega t)\eta,
%\end{aligned}\right.
%\end{equation}
%where $K_{11}=\mathtt{D}_{\eta,w}K$, $K_{20}=\mathtt{D}_{\eta\eta}K$
%and $K_{20}$ as in \eqref{K020}. In subsection \ref{SezioneDiagonalization}
%we showed the reducibility of the linear system
%$\dot{w}-JK_{02}(\omega t)w$ (see Theorem \ref{ReducibilityDP}),
%i.e. the third equation in \eqref{linsistFin} becomes 
%\begin{equation}\label{EQV}
%\dot{v}_j-\ii d_j^{\infty}(\omega)v_j=f_{j}(\omega t), \quad j\in S^{c}, 
%\end{equation}
%with $d_j^{\infty}(\omega)\in \mathbb{R}$ in \eqref{FinalEigenvaluesDP}, $f(\f,x)=\sum_{j\in S^{c}}f_{j}(\f)e^{\ii jx}\in H^{s}\cap H_{S}^{\perp}$. 
%The solution of \eqref{EQV} are
%\[
%v_{j}(t)=c_j e^{\ii d_j^{\infty}(\omega) t }+\tilde{v}_j,\quad 
%\tilde{v}_j=\sum_{l\in \mathbb{Z}^{\nu}}\frac{f_{lj}e^{\ii \omega\cdot l t}}{\ii\omega\cdot l+\ii d_j^{\infty}(\omega) }
%\]
%The \eqref{linsistFin} is stable
%since the actions $\eta(t)=\eta_0$, $\forall\, t\in\mathbb{R}$
%and the Sobolev norm of the solution of \eqref{EQV} with initial
%condition $v(0)=\sum_{j\in S^{c}}v_j(0)e^{\ii jx}$, does not increase in time.

\appendix

\section{Non-degeneracy conditions}\label{PrelimEst}
\begin{proof}[{\bf Proof of Lemma \ref{Twist1}}]
Recalling \eqref{TwistMatrixDP}   
we introduce a matrix $\mathbb K$ so that 
$\mathbb{A}=: (2/9)\,\,
\rm{diag}\Big(\og(\overline{\jmath}_i)(1+\overline{\jmath}_i^2)\Big)
\mathbb K $. 
Now we show that the entries of $\mathbb K$ 
are bounded by some constant independent 
of the $\overline{\jmath}_i$. 
After some direct computations we have that
\[
\mathbb{K}_{j j}=\frac{1+j^2}{j^2}\,, \quad j\in S^+\,,\;\qquad
\mathbb{K}_{j k}=
3 \frac{(1+k^2)(2+k^2+j^2)}{(3+k^2+j^2+k j)(3+k^2+j^2-k j)}\,, 
\quad j, k\in S^+,\,\, j\neq k\,.
\]
Obviously $\lvert   \mathbb{K}_{j j}  \rvert\le 2$; 
regarding the off-diagonal terms, we note that 
$0\le 2 k j \le k^2+j^2$, hence 
$\lvert \mathbb{K}_{j k} \rvert\le 12$ if $j\neq k$.

\noindent	
We consider the variables $x, p_2, \dots, p_{\nu}$ defined as
\begin{equation}\label{montero}
\overline\jmath_1 =: 1/x\,,\quad 
\overline\jmath_i =: p_i/x\,,\quad 0<p_i\le 1\,,
\end{equation}
so that $P(x,p_i)=\det(\mathbb K)$ is a rational function.
It is easily seen that $\mathbb K$ computed 
at $p_i=1$ for all $i$,  
coincides with the matrix 
\[
(1+x^2)\,\big(\mathrm{I}+2\,g(x)\,(U-\mathrm{I})\big)\,, 
\qquad g(x):=(3 x^{2}+1)^{-1}\,,
\] 
where $U$ is the matrix with components 
$U_{ij}=1$ for any $ i,j=1, \dots, \nu$.
Its determinant is 
\begin{equation}\label{arrosticini}
\left(\frac{1+x^2}{1+3 x^2}\right)^{\nu}
(3 x^2-1)^{\nu-1}(3 x^2+2 \nu-1)\,.
\end{equation}
We note that the absolute value of \eqref{arrosticini} 
is $\geq 1$ at $x=0$. We conclude that there exists 
$0<\mathtt r_0<1$ such  that
\[ 
\mbox{if}\quad 0\le x<\mathtt{r}_0\,,\,\,\, 
|p_i-1|\le \mathtt{r}_0 \quad\mbox{ then}
\quad \lvert P(x, p_i) \rvert\ge 1/2\,.
\]
This implies the thesis.
%that \eqref{GenericAssumption2} holds
%for any choice of tangential sites in
% the set given by \eqref{caso2}. 
\end{proof}

\begin{lem}\label{ecologia} There exists $0<\mathtt{r}_0<1$ 
such that,  
for  any $S^+\in \mathcal{V}(\mathtt{r})$ 
with $0<\mathtt{r}\leq \mathtt{r}_0 $ 
(see Definition \ref{Def:cono}), 
the following holds true:
\begin{align}
&\bullet\quad 
 |\sum_{i=1}^{\nu} 
 \frac{\overline{\jmath}_i}{1+\overline{\jmath}_i^2}\,\ell_i|
 > \frac{\mathtt r}{2}\neq 0 \,,
 \quad \forall \ell\in \Z^\nu\,,
 \quad|\ell|=1,2,3,5\,; \label{GenericAssumptionbis}\\
&\bullet \quad \left|{\rm det}\Big( {\rm I}
-\mathbb{A}^{-T} \vec{v} (\overline{\omega})^T\Big)
\right| \geq 1\,; 
%\;\;\;\;{\rm where}\;\;\;\; 
%\vec{v}:=(2/3)(1+\overline{\jmath}_k^2)_{k=1}^{\nu}\,
%;
\label{corto100}\\
&\bullet \quad 
|\ell - \Big( {\rm I}-\mathbb{A}^{-T} 
\vec{v} (\overline{\omega})^T\Big)^{-1} 
\mathbb{A}^{-T}(\vec{w}_j-\vec{w}_k)|
\ge \delta|\ell| \,, 
\qquad \ell\in \Z^\nu, \; j, k\in S^c\label{corto}\\
&\bullet \quad 
|\ell - \Big( {\rm I}-\mathbb{A}^{-T} \vec{v} 
(\overline{\omega})^T\Big)^{-1} \mathbb{A}^{-T}\vec{w}_j|
\ge \delta|\ell|\,, \qquad \ell\in \Z^\nu\,, 
\; j \in S^c\label{cortissimo}
\end{align}		
Here $\vec v$ and $\vec{w}_j$ are defined in \eqref{ferro} 
(see also \eqref{diagonalopFinale}), 
$\bar\omega$ in \eqref{LinearFreqDP} and  
$\delta$ is some appropriately small pure constant.
\end{lem}

\begin{proof}[{\it Proof of \eqref{GenericAssumptionbis}}]
The case $|\ell|=1$ is trivial. 
For $|\ell|=2$ we use the fact 
that the $\bar\jmath_i$ are all distinct. 
For $|\ell|=3,5$ we pass to the variables \eqref{montero} and we get
\[
\bar\jmath_1 |\sum_{i=1}^{\nu} 
\frac{\overline{\jmath}_i}{1+\overline{\jmath}_i^2}\,
\ell_i| = 
|\sum_{i=1}^{\nu} \frac{p_i}{x^2+p_i^2}\,\ell_i| =: L(x,p)\,.
\]
We notice that 
$L(0,1)=|\sum_i \ell_i|\ge 1$ 
(since $\sum_i \ell_i$ has the same parity of $|\ell|$) 
so, by continuity, there exists 
$0<\mathtt r_0<1$ such  that $L(x,p)>1/2$  
for all $	0\le x<\mathtt r_0,\,\,\, |p_i-1|\le \mathtt r_0 $. 
This implies the result.
\end{proof}

\begin{proof}[{\it Proof of \eqref{corto100}}]
We first note that (recall \eqref{ferro})
\begin{equation}\label{spero}
\det(\mathrm{I} -\mathbb{A}^{-T}\vec{v}\bar{\oo}^{T})
=1-\mathbb{A}^{-T}\vec{v}\cdot \bar{\oo}\,.
\end{equation}
Consider the change of variables \eqref{montero}.
One can note that the matrix 
$\mathbb{A}$ in \eqref{TwistMatrixDP} at $p_i=1$, $i=2, \dots, \nu$,
is given by  
\begin{equation}\label{lamatriceA}
\begin{aligned}
\mathbb{A}&:=d(x)\big[\mathrm{I}+e(x)U\big]\,,
\qquad \mathbb{A}^{-1}:=\frac{1}{d(x)}\big[\mathrm{I}-f(x)U\big]\,,\\
d(x)&:=\frac{2(4x^2+1)(3x^4+2x^2+1)}{9 x^3(3x^2+1)}\,,
\qquad e(x):=\frac{2x^2}{3x^4+2x^{2}+1}\,, 
\quad f(x):=\frac{e(x)}{1+\nu e(x)}\,.
\end{aligned}
\end{equation}
By \eqref{lamatriceA} at $p_i=1$ we have
\begin{equation}\label{lamatriceB}
\mathbb{A}^{-1}\vec{v}\cdot \ol{\oo}
=\frac{2(4x^2+ 1) }{3 x}(\vec{1}\cdot \mathbb{A}^{-1}\vec{1})
=3\nu\,\,\frac{(4x^2+1) 
(3x^2+1)}{(4x^2+1)(3x^4+2(1+\nu)x^2+1)}\,.
\end{equation}
We note that, for $x=0$, one has 
$\lvert \det(\mathrm{I} 
-\mathbb{A}^{-1}\vec{v}\bar{\oo}^{T}) \rvert=3\nu-1\geq 2$.
Thus there is
$0<\mathtt r_0<1$ so that for all 
$0\le x<\mathtt r_0$, $|p_i-1|\le \mathtt r_0$ one has
$|\det (\mathrm{I} -\mathbb{A}^{-1}\vec{v}\bar{\oo}^{T})|\geq 1$.
\end{proof}
	
\begin{proof}[{\it Proof of \eqref{corto},\eqref{cortissimo}}]
We systematically use the  variables \eqref{montero}.
% and define 
%\[
%t=j \overline{\jmath}_1^{-1}\qquad
% s = \frac{k}{\overline{\jmath}_1}.
%	\]
We define $\Omega=\diag({\bar\omega_i})$, 
$V=\diag(\vec v_i)$ and  write 
$\mathbb{A}=\Omega \mathbb{H}\,V$ 
where
\[
\mathbb H = \frac{x^2+1}3 
[ I + \frac{2}{(3x^2+1)} (U-I)] + O(|p-1|)
=-\frac{1}{3} (I-2 U)+O(\lvert p-1 \rvert, \lvert x \rvert)\,.
\] 
Then $\vec{w}_j$ in \eqref{ferro}  
can be written as 
$\vec{w}_j=\Omega\,V\,b_j$ with
\[
b_j \cdot \mathtt{e}_i
= -\frac{\og(j)(7+5 \overline{\jmath}_i^2
+\overline{\jmath}_i^4+3 j^2)}{\og(\bar\jmath_i)((3+j^2)^2+(6+j^2)
\overline{\jmath}_i^2
+\overline{\jmath}_i^4)},
\]
where $\mathtt{e}_i$ is the $i$-th vector of the canonical basis of $\mathbb{R}^{\nu}$.
In the new coordinates $(x,p)$ in \eqref{montero}, and setting
$t=j (\overline{\jmath}_1)^{-1}$,
 this reads as
\begin{align*}
&b_j \cdot \mathtt{e}_i=:b(t, x, p)\cdot \mathtt{e}_i = f(t, x) g(t, x, p_i)+\,t\,g(t, x, p_i)\,, \\
&f(t,x):= \frac{3\,t\,x^2}{x^2+t^2}\,, 
\qquad g(t, x, p):=
\frac{(x^2+p_i^2)(7x^4+5 x^2p_i^2+p_i^4+3x^2 t^2)}
{p_i\,(4x^2+p_i^2) ((3x^2+t^2)^2+(6x^2+t^2)p_i^2+p_i^4)}\,.
\end{align*}
We claim that
\begin{equation}\label{formula1}
b(t, x, p)
=-\frac{t}{t^4+t^2+1} \vec{1}+ O(|p-1|, \lvert x \rvert)\,.
\end{equation}
where the term $ O(\lvert p-1\rvert, \lvert x \rvert)$ is uniform in $t$.

\noindent
By direct computations we have
\[
\lvert g(t, x, p)-g(t, 0, 1)\rvert
\lesssim C\big(\lvert x \rvert+\lvert p-1 \rvert\big)(1+t^2)^{-1}
% \frac{C(\lvert x \rvert+\lvert p-1 \rvert)}{1+t^2}
\,,\quad 
\lvert g(t, x, p) \rvert\le C(1+t^{2})^{-1}\,,
\]
with $C$ independent of $t$, 
and for $x,p$ in a neighborhood of $(0,1)$. 
Moreover
$\sup_{t} \lvert f(t, x) \rvert\le 3|x|^{2}/2$.
Thus, for $x,p$ sufficiently close to $(0,1)$, we 
have that 
\begin{equation*}
\lvert b(t, x, p)-b(t, 0, 1)\rvert 
 \le \lvert t \rvert\lvert g(t, x, p)-g(t, 0, 1)\rvert
 +\lvert f(t, x) \rvert \lvert g(t, x, p) \rvert\lesssim O(|p-1|,|x|)\,,
\end{equation*}
which implies the claim \eqref{formula1}.
Hence, setting $s=k (\overline{\jmath}_1)^{-1}$,
we have
\[
\ell -\Omega^{-1} (\mathbb{H}^T-U)^{-1} 
\Omega (b(t, x , p)-b(s, x, p))=
\ell- 3(1+\nu)^{-1} h(t, s)\,\vec{1}+O(\lvert p-1 \rvert, \lvert x \rvert)\,,
\]
where
\[
h(t, s):= t\, (t^{4}+t^{2}+1)^{-1}- s\, (s^{4}+s^{2}+1)^{-1}\,.
%\Big(  \frac{t}{t^4+t^2+1}-\frac{s}{s^4+s^2+1}\Big)\,.
\]
We note that $\lvert t\, (t^4+t^{2}+1)^{-1} \rvert\leq 0.41$, hence each single component 
satisfies 	
\[
| 3(1+\nu)^{-1} h(t, s)+O(\lvert p-1 \rvert, \lvert x \rvert)\cdot \mathtt e_i|
\le 3 (1+\nu)^{-1} (1-\delta)  \leq(1-\delta)
\]
provided that $|p-1|,|x| \le \mathtt r_0$ small enough.
Hence, if $\nu\geq 2$, one has
	$
	|\ell_i- \big(\frac{3 h(t, s)}{1+\nu}\,\vec{1}+O(\lvert p-1 \rvert, \lvert x \rvert)\big)\cdot \mathtt e_i |\ge  \delta|\ell_i|\,.
	$
	This concludes the proof of \eqref{corto}. The proof of \eqref{cortissimo} is the same just setting
$h(t)= t\, (t^4+t^2+1)^{-1}$. 
\end{proof}

 \begin{proof}[{\bf Proof of Lemma \ref{measG0}}]\label{quadro}
It is well known that, thanks to the choice of $\tau$ in \eqref{gammaDP},
$\lvert \Omega_{\varepsilon}\setminus\mathcal{G}_0^{(0)} \rvert
\le C \varepsilon^{2(\nu-1)}\gamma$ for some $C=C(S)>0$. 
Thus we focus on the estimate for the measure of $\mathcal{G}_0^{(1)}$. 
For indexes $\ell\in \mathbb{Z}\setminus\{0\}$, $j,k\in S^{c}$
satisfying
\begin{equation}\label{momo} 
|\ell|\leq 3\,,\qquad {\rm and}\qquad
\sum_{i=1}^{\nu} \overline{\jmath}_i\,\ell_i+j=k\,,
\end{equation}
we define the sets
$
T_{\ell j k}:=\{ \omega \in\Omega_{\varepsilon} : 
\lvert \overline{\omega}\cdot \ell 
+\varepsilon^2 \mathbb{A}\xi\cdot \ell 
+ \og(j)-\og(k)+\varepsilon^2 (\og(j)\lal_{j}-\og(k)\lal_k) 
\rvert\le C \gamma \}\,.
$
Recalling \eqref{divisoriLBNF3} we have that
$\Omega_{\varepsilon}\setminus\mathcal{G}^{(1)}_{0}=\bigcup T_{\ell j k}$
where the union is restricted to $\ell,j,k$ satisfying \eqref{momo}.

%We have that
%\[
%\Omega_{\varepsilon}\setminus \mathcal{G}^{(1)}_0=\bigcup_{\substack{ j, k \in S^c, \ell\in\mathbb{Z}^{\nu}\setminus\{0\},\lvert \ell \rvert\le 3,\\ \sum_{i=1}^{\nu} \overline{\jmath}_i\,\ell_i+j=k     }} T_{\ell j k}
%\]
%where
%\begin{equation}
%T_{\ell j k}:=\{ \omega \in\Omega_{\varepsilon} : \lvert \overline{\omega}\cdot \ell +\varepsilon^2 \mathbb{A}\xi\cdot \ell + \og(j)-\og(k)+\varepsilon^2 (\og(j)\lal_{j}-\og(k)\lal_k) \rvert\le C \gamma \}
%\end{equation}
%for $j, k\in S^c$ such that \begin{equation}\label{momo}
%\sum_{i=1}^{\nu} \overline{\jmath}_i\,\ell_i+j=k
%\end{equation}

\smallskip

\noindent
Let us first study the $\e$-independent part of our small divisor. By \eqref{dispersionLaw} and \eqref{momo} we have
\begin{equation}\label{kia}
\ol{\omega}\cdot \ell+\og(j)-\og(k)=3\sum_{i=1}^{\nu} \frac{\overline{\jmath}_i}{1+\overline{\jmath}_i^2}\,\ell_i+3\frac{j}{1+j^2}-3\frac{(\sum_{i=1}^{\nu} \overline{\jmath}_i\,\ell_i+j)}{1+(\sum_{i=1}^{\nu} \overline{\jmath}_i\,\ell_i+j)^2}.
\end{equation}	
By Lemma \ref{ecologia} (see \eqref{GenericAssumptionbis}) 
if $|j|\geq \hat C(S)$, 
which is easily computed in terms of $\mathtt r$,
\eqref{kia} is bounded from below by 
$\mathtt r/4$.   By \eqref{momo},  $\lvert \og(j)-\og(k) \rvert\le C(S)$. 
Since $0<\lvert \ell \rvert\le 3$ and substituting $\lal_j=c(\oo)+\kappa_j/\og(j)$ (recall Lemma \ref{mole} ),
where $c(\oo),\kappa_j$ are  defined in  \eqref{clinica100}, \eqref{diagonalopFinale})
% \eqref{diagonalop}, \eqref{lambda}, \eqref{bio})
\[
\begin{aligned}
&|\bar{\oo}\cdot\ell+\e^{2}\mathbb{A}\x\cdot\ell+(1+\varepsilon^2 c(\oo)) (\og(j)-\og(k))+\varepsilon^2 (\kappa_j-\kappa_k)|  \ge \mathtt r/8
\end{aligned}
\]
for $\e$ small enough (depending on $\mathtt r$) and by using the fact that $\kappa_j-\kappa_k$ is uniformly bounded in $j, k$. This implies that $T_{\ell jk}=\emptyset$ for $\varepsilon$ small enough.
\\
We are left to deal with the case $|j|\leq \hat C(S)$. We write \eqref{kia} as $P(\ol{\jmath}_i,j)/Q(\ol{\jmath}_i, j)$ where $P,Q$ are polynomials with integer coefficients and $Q$ has no real zeros. We remark that
$1<Q< C(S)$ due to the condition $|j|\le \hat C(S)$. 
 If $P\neq 0$ then $|P|\geq1$ and again \eqref{kia} is larger than some $K(S)$. 
 We conclude that $T_{\ell jk}=\emptyset$ by reasoning as in 
 the case $j$ large. Now we study the case in which $P=0$.
Fixed $\ell$ in \eqref{kia}, then $P$ has degree four in $j$ and so the condition $P=0$ fixes at most four choices of  $j$ that we call $\hat{\jmath}_1, \hat{\jmath}_2, \hat{\jmath}_3, \hat{\jmath}_4$.
For 
$P=0$ (which is \eqref{kia}$=0$) we have
\begin{equation}
\begin{aligned}
\bar{\oo}\cdot\ell+\e^{2}\mathbb{A}\x\cdot\ell+\lambda(j)(1+\e^{2}\lal_j)-\lambda(k)(1+\e^{2}\lal_k)
%&=\e^2( {\mathbb A}\xi \cdot \ell +(\og(j)-\og(k)) 
%\vec{v}\cdot \xi +( w_{j}-w_{k})\cdot \xi)\\
%&
=\e^2( {\mathbb A}\xi \cdot \ell -(\bar\omega\cdot \ell) \vec{v}\cdot \xi 
+( w_{j}-w_{k})\cdot \xi)\,,
\end{aligned}
\end{equation}
where $\vec{v}$ is in \eqref{ferro} and 
$\kappa_j=w_j\cdot \xi$ with $\kappa_j$ in \eqref{diagonalopFinale}.
These are a finite number (depending only on $\nu$) of linear functions of $\xi$. 
We compute the derivative in $\xi$ which is  
\begin{equation}\label{metraggio}
({\mathbb A}^T + \vec{v}\bar{\omega}^T )\ell 
+ ( w_{\hat{\jmath}}-w_{k})\,,
\end{equation}
where $\hat{\jmath}\in \{ \hat{\jmath}_1, \hat{\jmath}_2, \hat{\jmath}_3, \hat{\jmath}_4 \}$.
Now \eqref{corto}  implies that the quantity \eqref{metraggio}
is bounded from below by a constant depending on $S$.
 This lower bound and  Fubini Theorem imply that $|T_{\ell jk}|\leq C(S)\e^{2(\nu-1)}\g$
 for some $C(S)>0$ depending on $S$.
 By the discussion above we have
 \[
 |\Omega_{\e}\setminus \calG_{0}^{(1)}|\leq 
 \sum_{|\ell |\leq 3, |j|,|k|\leq K(S)}|T_{\ell jk}|
 \leq C(S)\e^{2(\nu-1)}\g,
 \]
 where $K(S)>0$  and $C(S)>0$ are constant depending on the set $S$.
This implies the thesis.
\end{proof}

%
%\section{Pseudo differential calculus and the classes of remainders
%%${\gotL_{\rho, p}}$
%}

%\section{Properties of the smoothing remainders}\label{restismooth}

\section{Normal form identification}\label{restismooth100}
%
%
%\noindent
%\textbf{Proof of Proposition \ref{PartialWeak}}
\begin{proof}[Proof of Theorem \ref{PartialWeak}]\label{app:Weak}
%Recall Remark \ref{boccadasse}. We show how the terms $R(v^2 z^2)$ of the original Hamiltonian has changed after the change of variables performed in Sections.... and we prove that these one coincides with the resonant terms of degree four and quadratic in the normal variables of the Partial Birkhoff normal form procedure. We remark again that the latter is not a well defined procedure, but it has the advantage of involving easier algebraic computations.
%
%\smallskip 

The core of Theorem \ref{PartialWeak}
is to show that the terms in the l.h.s. of \eqref{mare},
which are obtained through a rather complicated sets of bounded 
changes of coordinates, coincide
with the ones obtained by a purely formal full Birkhoff Normal Form procedure.
In \cite{FGPa}
it has been shown that, at purely formal level,  the latter is well-defined and not resonant, i.e.
the resonant Hamiltonian is supported only on trivial resonances as in \eqref{coppiette}.
We procede as follows.

\noindent{\bf Step 1}.
The first step is to show that resonant terms at order $\e^{2}$ 
of the Full Birkhoff normal form
coincide with the ones obtained by using the \emph{weak} BNF procedure in Section 
\ref{SezioneWBNF},  passing to
action-angle variables and finally using a \emph{formal linear} BNF.

%\noindent{\bf Step 2}. The second step is to show that the resonant terms (up to 
%order $\e^{2}$) formal  linear BNF
%is not affected by changes of coordinates.

\noindent{\bf Step 2}. In order to conclude we note that the bounded 
maps we 
applied in subsections  \ref{preliminare} and \ref{LinearBNF}
%are of the form of the maps used in Step 2. Moreover
%they 
are, as functions of $\e$, $C^{3}$ with values in $\mathcal{L}(H^{s}, H^{s-3})$ .
Therefore the Taylor expansion of the  Hamiltonian associated 
to the operators in \eqref{L4dp} coincides with the Lie series of the generator up to order $\e^2$ (see also \eqref{K13}). 
Then we show that the Lie series coincides (up to order $\e^2$) to the one obtained in step 1.
Even though we taylor our proof to the particular set of changes of variables that we use in section \ref{regularization}, the argument is quite general and is essentially that the \emph{linear} BNF up to order $\e^2$ is coordinate independent.

\subsection*{Formal equivalence between  \emph{weak} $+$ \emph{linear} 
and \emph{full} BNF procedure}
One step of formal BNF means that we apply the formal change of variables generated by
\begin{equation}\label{fine2}
\mathfrak{F}^{(3)}:=[\mathrm{ad}_{H^{(2)}}]^{-1}H^{(3)}
\end{equation}
which  removes completely $H^{(3)}$ and conjugates $H$ (using \eqref{formLIE}) to
$H\circ{\Phi}_{(\mathfrak{F}^{(3)})^{-1}}=H^{(2)}+\frac12 \{\mathfrak{F}^{(3)}, H^{(3)}\} +$ h.o.t. 
\\
The scecond step of formal BNF removes all the non-resonant terms of degree four thus we get
\[
H^{(2)}+\frac12\Pi_{{\mbox{Ker}}(H^{(2)})} \{\mathfrak{F}^{(3)},H^{(3)}\} +\mbox{h.o.t.}
\]
In \cite{FGPa} it has been proved  that 
$\big(\Pi_{{\mbox{Ker}}(H^{(2)})}
-\Pi_{\mbox{triv}}\big) \{\mathfrak{F}^{(3)}, H^{(3)}\}= 0$
 (see  the notation of Definition \ref{chebellosiproietta}) and hence
\begin{equation}\label{megliotrivial}
\Pi^{d_z=2}\Pi_{{\mbox{Ker}}(H^{(2)})} \{\mathfrak{F}^{(3)},H^{(3)}\}
= \Pi^{d_z=2}\Pi_{\mbox{triv}}\{\mathfrak{F}^{(3)},H^{(3)}\}\,.
\end{equation}
We claim that
\begin{equation}\label{occhio1}
\begin{aligned}
&\Pi_{\mbox{triv}}\{\mathfrak{F}^{(3,1)},H^{(3,3)}\}=\Pi_{\mbox{triv}}\{\mathfrak{F}^{(3,3)},H^{(3,1)}\}=
\Pi_{\mbox{triv}}\{\mathfrak{F}^{(3,2)},H^{(3,0)}\}=
\Pi_{\mbox{triv}}
\{\mathfrak{F}^{(3,0)},H^{(3,2)}\}=0\,.
%\\
%&\Pi_{\mbox{triv}}
%\{\mathfrak{F}^{(3,2)},H^{(3,0)}\}=
%\Pi_{\mbox{triv}}
%\{\mathfrak{F}^{(3,0)},H^{(3,2)}\}=0\,.
\end{aligned}
\end{equation}
This implies that 
\begin{equation}\label{rambo5}
\begin{aligned}
\Pi_{{\mbox{Ker}}(H^{(2)})}
&\Pi^{d_z=2}\{\mathfrak{F}^{(3)},H^{(3)}\}=
\Pi_{\mbox{triv}}\Pi^{d_z=2}\Big(
\{\mathfrak{F}^{(3,\leq1)},H^{(3,\leq1)}\}+\{\mathfrak{F}^{(3,\leq1)},H^{(3,>1)}\}
\Big)\\
&+\Pi_{\mbox{triv}}\Pi^{d_z=2}\Big(
\{\mathfrak{F}^{(3,2)},H^{(3,0)}\}+\{\mathfrak{F}^{(3,2)},H^{(3,2)},\}
+\{\mathfrak{F}^{(3,3)},H^{(3,1)}\}
\Big)\\
&=\Pi_{\mbox{triv}}\Pi^{d_z=2}\Big(
\{\mathfrak{F}^{(3,\leq1)},H^{(3,\leq1)}\}+\{\mathfrak{F}^{(3,2)},H^{(3,2)}\}
\Big)\,.
\end{aligned}
\end{equation}
\begin{proof}[Proof of the claim \eqref{occhio1}]
Let us consider the term $\{\mathfrak{F}^{(3,1)},H^{(3,3)}\}$, where
\[
\mathfrak{F}^{(3,1)}=\sum_{\substack{ j_1,j_2\in S,j_3\in S^c\\ 
j_1+j_2+j_3=0}}C_{j_1j_2j_3}u_{j_1}u_{j_2}u_{j_3}\,,
\qquad
H^{(3,3)}=\sum_{\substack {j_1,j_2,j_3 \in S^c \\
j_1+j_2+j_3=0}} \widetilde{C}_{j_1j_2j_3}u_{j_1}u_{j_2}u_{j_3}\,,
\]
for some coefficients ${C}_{j_1j_2j_3}, \widetilde{C}_{j_1j_2j_3}\in\mathbb{C}$. 
Using \eqref{PoissonBracketDP} one gets
\begin{equation}\label{trivtriv}
\{\mathfrak{F}^{(3,1)},H^{(3,3)}\}=\sum_{\substack{j_1,j_2\in S, \, k_1,k_2\in S^{c} \\ j_1+j_2+k_1+k_2=0}} 
P_{j_1j_2k_1k_2}u_{j_1}u_{j_2}u_{k_1}u_{k_2},\qquad P_{j_1j_2k_1k_2}\in \mathbb{C}\,.
\end{equation}
A monomial $u_{j_1}u_{j_2}u_{k_1}u_{k_2}$ is supported on trivial resonances \eqref{coppiette}
 only if 
 $j_1=-j_2$ and $k_1=-k_2$, since $j_i\in S$, $k_i\in S^{c}$, $i=1,2$. 
 On the other hand, by the momentum condition, we have that
$j_1+j_2=-(k_1+k_2)$, which is not possible since $0\notin S\cup S^{c}$
(see \eqref{TangentialSitesDP}). Hence the \eqref{occhio1} holds. 
The others equalities in \eqref{occhio1} in  follows in the same way.
\end{proof}

Let us now perform the same Birkhoff procedure 
by first cancelling the terms of degree 
$\le 1$ in $z$ (weak BNF) 
and then the terms of degree two (linear BNF) .

%We denote by $\Pi^{n\le k}$ the projector of a homogenous Hamiltonian on the monomials of degree less or equal than $k$ in the variable $u=v+z$.

\smallskip

By the discussion in Section \ref{SezioneWBNF}, recall the notations of Proposition \ref{LemmaBelloWeak}, 
we have that, after two steps of weak Birkhoff normal form, 
the Hamiltonian of degree less or equal than $4$ is
\begin{equation}\label{afterweak}
\Pi^{d\leq 4} (H\circ \Phi^{-1}_2)=H^{(2)}+ Z^{(4,0)}_2+H_{2}^{( 3, \geq 2)}
+H_{2}^{(4, \geq 2)}.
\end{equation}
Here $Z^{(4,0)}_2$ is defined in \eqref{posta2}, 
\begin{equation}\label{NewHam3000}
H_2^{(3, 2)}=H_1^{(3, 2)}=H^{(3,2)}, 
\qquad  H_2^{(4,\geq2)}=H_1^{(4,\geq2)},
\end{equation}
where $H_1^{(3)}, H_1^{(4)}$ are defined in \eqref{NewHam3}.
The monomials of degree greater than $4$ will be not involved in this computation, so we omit them.
It is important to notice that, by direct inspection, $\mathfrak{F}^{(3,\le 1)}= {F}^{(3,\le 1)}$ defined in formula \eqref{DefF3}.
By Lemma \ref{LemmaBelloWeak}, we know that the same change of variables puts one of the constants of motion (lets say $K_3$ and drop the subindex $3$) into normal form, 
\[
\Pi^{d\leq 4} (K\circ \Phi^{-1}_2)=K^{(2)}+ W^{(4,0)}_2+K_{2}^{( 3, \geq 2)}
+K_{2}^{(4, \geq 2)}.
\]
The step of formal linear BNF entails applying the formal change of variables generated by 
\begin{equation}\label{macchine14}
{F}^{(3,2)}:=[\mathrm{ad}_{H^{(2)}}]^{-1}H^{(3,2)}\,.
\end{equation}
Again, by direct inspection, one can note that 
$F^{(3,2)}\equiv\mathfrak{F}^{(3,2)}:=\Pi^{d_{x}=2}\mathfrak{F}^{(3)}$,
where $\mathfrak{F}^{(3)}$ is given in \eqref{fine2}.
Thus obtaining the Hamiltonian
\begin{equation}\label{afterformlin}
H^{(2)}+ Z^{(4,0)}_2+ H_{2}^{( 3,  3)} +\frac12 \{\mathfrak{F}^{(3,2)},H_2^{(3,\ge 2)}\} + H_2^{(4,\geq2)} + h.o.t.
\end{equation}
Since \eqref{macchine14} puts in normal form also $K\circ \Phi^{-1}_2$ following the same reasoning as in \cite{FGPa} 
and Lemma \ref{LemmaBelloWeak}, we get that
\begin{equation}\label{resLINBNF}
\begin{aligned}
&\Pi^{d_z=2}\Pi_{{\mbox{Ker}}(H^{(2)})}
( \frac12 \{\mathfrak{F}^{(3,2)}, H_2^{(3,\ge 2)}\}+ H_2^{(4,\geq2)})
= \Pi^{d_z=2}\Pi_{\mbox{triv}} (\frac12 \{\mathfrak{F}^{(3,2)},H_2^{(3,\ge 2)}\} 
+ H_2^{(4,\geq2)})\\
&=\Pi^{d_z=2}\Pi_{\mbox{triv}} (\frac12 \{\mathfrak{F}^{(3,2)},H_2^{(3, 2)}\}
+ \frac{1}{2} \{\mathfrak{F}^{(3,\leq1)}, H^{(3, \leq 1)}\}+\{  \mathfrak{F}^{(3,1)}, H^{(3, 3)}\})\\
&\stackrel{\eqref{occhio1}}{=} 
\frac12 \Pi^{d_z=2}\Pi_{\mbox{triv}} 
( \{\mathfrak{F}^{(3,2)}, H_2^{(3, 2)}\}
+ \frac{1}{2} \{\mathfrak{F}^{(3,\leq1)}, H^{(3, \leq 1)}\})
\\ &\stackrel{\eqref{rambo5},\eqref{megliotrivial}}{=} 
\frac12 \Pi^{d_z=2}\Pi_{{\mbox{triv}}} \{\mathfrak{F}^{(3)}, H^{(3)}\}\,.
\end{aligned}
\end{equation}
We now want to pass to the action-angle variables introduced in \eqref{AepsilonDP}.
Since the rescaling with the parameter $\varepsilon$ is 
covariant under the change of variables that we use, 
we consider instead of $A_{\varepsilon}$ 
the symplectic change of variables $A_1:=A_{\varepsilon_{|_{\varepsilon=1}}}$.
By recalling that
\[
\mathsf{H}_0:=H^{(2)}\circ A_{{1{|_{ \substack{y=0\\ \theta=\f}}}}}\,, \quad \mathsf{H}_1:=  H_2^{(3, 2)}
\circ{A}_{{1{|_{ \substack{y=0\\ \theta=\f}}}}} \,, \qquad 
\Pi^{d_z=2}\big(\mathsf{H}_2+\mathsf{H}_{\mathcal{R}_2}\big):= H_2^{(4, 2)}\circ{A}_{{1{|_{ \substack{y=0\\ \theta=\f}}}}}
\]
and setting 
\begin{equation}\label{macchina}
\begin{aligned}
%&\mathsf{H}_0:=H^{(2)}\circ A_1:=\overline{\oo}\cdot \eta
%+\sum_{j\in S^{c}}z_{j}z_{-j}, 
%\qquad \widetilde{H}^{(4,0)}_{res}:=Z_2^{(4,0)}\circ A_1
%=\mathbb{A}(\x+\Omega \eta)\cdot \eta\,,\\
\widetilde{F}^{(3, 2)}:=\mathfrak{F}^{(3, 2)}
\circ{A}_{{1{|_{ \substack{y=0\\ \theta=\f}}}}}\,,
%\qquad
%\widetilde{H}_2^{(3, 2)}:= \mathsf{H}_1= H_2^{(3, 2)}
%\circ{A}_{{1{|_{ \substack{y=0\\ \theta=\f}}}}}\,,
%\qquad 
%\widetilde{H}_2^{(4, 2)} 
%:=H_2^{(4, 2)}\circ{A}_{{1{|_{ \substack{y=0\\ \theta=\f}}}}}\,,
\end{aligned}
\end{equation}
we have that \eqref{resLINBNF} reads
\begin{equation}\label{macchine177}
\begin{aligned}
&\Pi^{d_z=2}\Pi_{{\mbox{Ker}}(\mathsf{H}_0)}
( \frac12 \{\widetilde{F}^{(3,2)}, \mathsf{H}_1\}+\mathsf{H}_2+\mathsf{H}_{\mathcal{R}_2 })=
\frac12\left[ \Pi^{d_z=2}\Pi_{{\mbox{triv}}} 
\{\mathfrak{F}^{(3)}, H^{(3)}\}\right]\circ
A_{{1{|_{ \substack{y=0\\ \theta=\f}}}}}\,.
\end{aligned}
\end{equation}

\subsection*{The  \emph{rigorous} procedure of subsections \ref{preliminare},
\ref{LinearBNF} and the \emph{linear} BNF}
Since the r.h.s. of \eqref{macchine177} is the r.h.s. of \eqref{mare} 
it remains to show 
%that the l.h.s. in \eqref{mare} coincides with the l.h.s. of 
% \eqref{macchine177}, i.e. 
that the $\e^{2}$-terms of the 
Hamiltonian associated to the operator $\mathcal{L}_{7}$ in \eqref{L4dp}
coincides with  the l.h.s. 
in \eqref{macchine177}.
We remark that the operator $\mathcal{L}_7$ has been constructed through a rigorous
procedure
providing also ``tame'' estimates of the remainder of higher order in $\e$.
As already explained the maps $\mathcal{B}$, $\Upsilon_i$, $i=1,2$,
are, as functions of $\e$, $C^{3}$ with values in $\mathcal{L}(H^{s}, H^{s-3})$ .
Therefore the Taylor expansion of the  Hamiltonian associated 
to the operators in \eqref{L4dp} coincides 
with the Lie series of the generator up to order $\e^2$ (see also \eqref{K13}).

Let us then Taylor expand the Hamiltonian  $\mathcal{K}^{(2)}=\mathsf{H}\circ\mathcal{B}^{-1}\circ\Upsilon_1^{-1}\circ{\Upsilon_2}^{-1}$
(see \eqref{fi}, \eqref{Divac}, \eqref{Kone} and \eqref{Kone1000}) up to order $\e^{2}$.
It is sufficient to consider
 $\mathcal{B}_1$ the flow generated by the Hamiltonian in \eqref{Stau} and  
 $\Upsilon_1$
 the flow of the Hamiltonian $H_{\mathtt{A_1}}$ 
 (which has  the form \eqref{HAMa1} and
satisfies \eqref{equazioneEps}). 
We have that
  \begin{align}
&\Pi_{\mbox{Ker}(\mathtt{H}_0)} \Big(\mathtt{Z}_0+ \mathcal{K}^{(1)}_2\Big)\stackrel{\eqref{Kone1000},\eqref{seven}}{=}\Pi_{\mbox{Ker}(\mathtt{H}_0)} 
 \mathsf{H}\circ \B_1^{-1}\circ\Upsilon_1^{-1}+ O(\e^3)\label{Divac10}\\
%& \stackrel{\eqref{Kone1000},\eqref{seven}}{=}
%\Pi_{\mbox{Ker}(\mathtt{H}_0)} \mathtt{H}_0+\e^{2}
% \Big(\Pi_{\mbox{Ker}(\mathtt{H}_0)} \mathtt{Z}_0+ \Pi_{\mbox{Ker}(\mathtt{H}_0)} \mathcal{K}^{(1)}_2\Big)
% +O(\e^{3})\\
  & \stackrel{\eqref{georgeHill10},\eqref{Kone}}{=}\mathtt{H}_0+\e^{2}
 \Pi_{\mbox{Ker}(\mathtt{H}_0)}\Big(\mathsf{H}_2^{(2)}+
\frac{1}{2}\{ H_{\mathtt{A}_1}, \{H_{\mathtt{A}_1},\mathsf{H}_0\}_e\}_e
+\{H_{\mathtt{A}_1}, \mathcal{K}_1\}_e\Big)+O(\e^{3})\nonumber\\
& \stackrel{\eqref{espansioLIE2},\eqref{georgeHill10}, \eqref{espansioLIE}}{=}
\mathtt{H}_0+\e^{2}
 \Pi_{\mbox{Ker}(\mathtt{H}_0)}\Big(\mathsf{H}_2^{(1)}+
\frac{1}{2}\{ H_{\mathtt{A}_1}, \{H_{\mathtt{A}_1},\mathsf{H}_0\}_e\}_e
+\{H_{\mathtt{A}_1}, \mathsf{H}_1\}_e+
\{H_{\mathtt{A}_1},\{ S_1,\mathsf{H}_0\}_e\}_e\Big)+O(\e^{3})\,.\nonumber
\\
&\stackrel{\eqref{espansioLIE22}}{=}
\mathtt{H}_0+\e^{2}
\Pi_{\mbox{Ker}(\mathtt{H}_0)}\Big(\frac{1}{2}\{ S_1,\{ S_1,  \mathsf{H}_0\}_e\}_{e}
+ \{S_1,\mathsf{H}_1\}_{e}+\frac{1}{2}\{S_2,\mathsf{H}_0\}_{e}
+\mathsf{H}_2+\mathsf{H}_{\mathcal{R}_2}\\
&\qquad\quad\qquad\quad\qquad\quad +
\frac{1}{2}\{ H_{\mathtt{A}_1}, \{H_{\mathtt{A}_1},\mathsf{H}_0\}_e\}_e
+\{H_{\mathtt{A}_1}, \mathsf{H}_1\}_e+
\{H_{\mathtt{A}_1},\{ S_1,\mathsf{H}_0\}_e\}_e\Big)+O(\e^{3})\,.\nonumber
& \end{align}
 Using the Jacobi identity we have
 \begin{equation}\label{Divac11}
 \begin{aligned}
 \frac{1}{2}\{ S_1, \{S_1,\mathsf{H}_0\}_e\}_e&
 + \frac{1}{2}\{ H_{\mathtt{A}_1}, \{H_{\mathtt{A}_1},\mathsf{H}_0\}_e\}_e+
\{H_{\mathtt{A}_1},\{ S_1,\mathsf{H}_0\}_e\}_e=\\
&=\frac{1}{2}\{ S_1+H_{\mathtt{A}_1}, \{S_1+H_{\mathtt{A}_1},\mathsf{H}_0\}_e\}_e
+\frac{1}{2}\{\mathsf{H}_0,\{S_1,S_1+H_{\mathtt{A}_1}\}\}_e\,.
 \end{aligned}
 \end{equation}
 Moreover, setting $\mathsf{F}_1:=S_1+H_{\mathtt{A}_1}$, we note that
 \begin{equation}\label{Divac13}
 \{\mathsf{F}_1,\mathsf{H}_0\}_e= \{H_{\mathtt{A}_1},\mathsf{H}_0\}_e
 +\{S_1,\mathsf{H}_0\}_e
 \stackrel{\eqref{espansioLIE}, \eqref{georgeHill10},\eqref{equazioneEps}}{=}-\mathsf{H}_1\,.
 \end{equation}
 Since $\Pi_{\mbox{Ker}(\mathtt{H}_0)}\{S_2,\mathsf{H}_0\}_{e}\equiv0$, by \eqref{Divac10}, 
\eqref{Divac11},
and \eqref{Divac13}
 we get
 \begin{equation}\label{chiK2}
\Pi^{d_{w}=2}\Pi_{\mbox{Ker}(\mathtt{H}_0)} \Big(\mathtt{Z}_0+ \mathcal{K}^{(1)}_2\Big)
=\Pi^{d_{w}=2}
 \Pi_{\mbox{Ker}(\mathtt{H}_0)}\Big(
 \frac{1}{2}\{\mathsf{F}_1, \mathsf{H}_1\}_{e}
+\mathsf{H}_2+\mathsf{H}_{\mathcal{R}_2}
 \Big)\,.
 \end{equation}

%Clearly we have that $\mathsf{H}_0:=H^{(2)}\circ A_1$.
%Moreover, 
%\eqref{pioveNantes}, \eqref{macchina}, \eqref{macchine14}, \eqref{Divac13} 
%and the fact that
Since $\mbox{Ker}(H^{(2)})$
is trivial
on cubic monomials, we deduce that 
\begin{equation}\label{macchine178}
\mathsf{F}_1 = -({\rm ad}^{-1}(\mathsf{H}_0)(\mathsf{H}_1 ) = \widetilde F^{(3,2)}
\end{equation}
and this concludes the proof.
%Using \eqref{macchine178}
%we deduce that the l.h.s. in \eqref{macchine177}
%is equal to the l.h.s. of \eqref{chiK2}. This implies the thesis.
%
%
%
%
%In order to verify that $\eqref{chiK2}=\eqref{macchine177}$
%it is sufficient to set 
%\[
%S^{(3,2)}=S_1, \quad G^{(3,2)}=H_{\mathtt{A}_1}
%\]
%where $S^{(3,2)}$, $G^{(3,2)}$ are in \eqref{S32}, \eqref{macchina6}
%and $S_1$ is in \eqref{esse1}, $H_{\mathtt{A}_1}$ in \eqref{equazioneEps}.
%This is possible since $S_1,H_{\mathtt{A}_1}$ have the same
%structure of the Hamiltonians $S^{(3,2)},G^{(3,2)}$  considered in the abstract procedure.
\end{proof}

\section{Technical Lemmata}\label{restismooth}

\subsection{Flows of hyperbolic PDEs}\label{tecnica}
In this subsection we study some properties of the flow of \eqref{rosso}. 
We start by studying the time-one flow map of the pseudo differential PDE 
\begin{equation}\label{diffeotot}
\del_{\tau}\Psi^{\tau}(u)=(J\circ b) \Psi^{\tau}(u)\,,\qquad 
\Psi^{0}u=u\,,
\end{equation}
and how this differs from $\Phi^{\tau}$.
Proposition $3.1$ in \cite{FGP1} guarantees 
that the flow of \eqref{diffeotot} 
is the composition of the diffeomorphism 
of the torus
\begin{equation}\label{ignobel}
\begin{aligned}
&\mathcal{A}^{\tau}h(\varphi, x):=(1+\tau \beta_x(\varphi, x)) 
h(\varphi, x+\tau\beta(\varphi, x))\,, 
\qquad \quad \,\, \varphi\in \T^{\nu}\,,\,  x\in \T\,,\\
&(\mathcal{A}^{\tau})^{-1}h(\varphi, y)
:=(1+ \tilde{\beta}_y(\tau,\varphi, y))\,h(\varphi, y+\tilde{\beta}(\tau, \varphi, y))\,, 
\quad \varphi\in \T^{\nu}\,,\, y\in \T\,,
\end{aligned}
\end{equation}
where $\tilde{\be}(\tau;x,\x)$ is such that
\[
x\mapsto y=x+\tau\be(\f,x) \; \Leftrightarrow \; 
y\mapsto x=y+\tilde{\be}(\tau,\f,x)\,, 
\;\; \tau\in[0,1]\,,
\]
with a pseudo differential operator of order $-1$ up 
to smoothing remainders belonging to the class 
$\gotL_{\rho,p}$ (see Definition 2.8 in \cite{FGP1}) 
for any $\rho\in\mathbb{N}$, $\rho\geq 3$, $p\geq s_0$.
For completeness we restate the definition  of 
$\gotL_{\rho,p}$ in  Appendix  \ref{ellerhosez} (see Definition \ref{ellerho}).
The class $\gotL_{\rho,p}$ has the property 
of being closed w.r.t.  changes of variables as $\Phi^\tau$ (see Lemma $B.10$ in \cite{FGP1}).
%\ref{restismooth}, 
Moreover, in Lemma \ref{IncluecompoCLASSI}, 
we show that $\gotL_{\rho,p}$ is contained 
into  $\gotC_{-1}$ which is another class 
of ``tame'' operators. Such class is introduced in 
 Definition \ref{Cuno} and is included in $-1$-tame operators, 
 see Lemma \ref{LemmaAggancio}.

We refer to Proposition $3.1$, Corollary $3.2$ 
and Appendix $A$ in \cite{FGP1} 
for  properties and estimates of 
$\Psi^{\tau}$ and $\mathcal{A}^{\tau}$
in \eqref{ignobel}.

\begin{lem}\label{differenzaFlussi}
Fix $\rho\geq 3$ and $p\geq s_0$. 
There exist $\delta\ll 1$, $\su:=\su(\rho,p)$  such that if
\begin{equation}\label{flow1}
\lVert \beta\rVert_{s_0+\s_1} \le \delta
\end{equation} 
then the following holds. 
Let $\Psi^{\tau}$ be the flow of the system 
\eqref{diffeotot}, then the flow of 
\eqref{rosso} $\Phi^\tau$ is well defined 
for $|\tau|\le 1$ and one has
$\Phi^1= \Pi_S^\perp\Psi^1\Pi_S^\perp\circ(\rm{I} +\mathcal{R})$
where $\mathcal{R}$ is an operator 
with the form \eqref{FiniteDimFormDP}.
Moreover $\mathcal{R}$ belongs to $\gotL_{\rho,p}(\calO)$ 
and satisfies
\begin{equation}\label{giallo2}
\mathbb{M}^{\gamma}_{\mathcal{R}}(s, \tb)
\lesssim_s  \lVert \beta \rVert^{\gamma, \calO}_{s+\su}\,, 
\qquad
\mathbb{M}_{\Delta_{12}\mathcal{R} }(p, \tb)\lesssim _{p}
\|\Delta_{12}\beta \|_{p+\su}\,.
\end{equation}
\end{lem}

\begin{proof}
Proposition $3.1$ in \cite{FGP1} provides 
$\delta, \sigma_1$ such that if the 
smallness condition \eqref{flow1} 
is satisfied with such parameters, 
then the flow $\Psi^{\tau}$ is well-defined 
for $\lvert \tau \rvert\le 1$.
Let us define $\Upsilon^{\tau}$ as the flow of the following Cauchy problem 
\begin{equation}\label{rosso2}
\partial_{\tau} \Upsilon^{\tau} u=
-\big(\Psi^{\tau}_* \mathtt{Z} \big) \Upsilon^{\tau} u\,,\qquad 
\Upsilon^0 u=u
\end{equation}
with 
\[
\begin{aligned}
\mathtt{Z}u&:=(J\circ b) \Pi_S[ u]+\Pi_S(J\circ b) \Pi_S^{\perp}[ u]=
\sum_{j\in S} \big(g_j(\tau), u \big)_{L^2}\,
\chi_j(\tau)+\sum_{j\in S} \big(\tilde{g}_j(\tau), u \big)_{L^2}\,
\tilde{\chi}_j(\tau)\,,\\
g_j&=\tilde{\chi}_{j}:=e^{\mathrm{i} j x}\,, 
\quad \chi_j=J (b(\tau) \,e^{\mathrm{i} j x})\,, 
\quad \tilde{g}_j:=\og(j)\Pi_S^{\perp}[b(\tau)\,e^{\mathrm{i} j x}]\,.
\end{aligned}
\]
Equation \eqref{rosso2} is well posed on 
$H^{s}$ since its vector field is finite rank. 
By the following computations
\begin{equation*}
\begin{aligned}
(\partial_{\tau} \Psi^{\tau}) (\Upsilon^{\tau} u)
+\Psi^{\tau} (\partial_{\tau} \Upsilon^{\tau} u) 
&=(J\circ b) \Pi_S^{\perp}[\Psi^{\tau} (\Upsilon^{\tau} u)]
-\Pi_S[(J\circ b) \Pi_S^{\perp}[\Psi^{\tau} (\Upsilon^{\tau} u)]\\
&=(J\circ b) \Psi^{\tau} (\Upsilon^{\tau} u)
-\Big((J\circ b) \Pi_S[\Psi^{\tau} (\Upsilon^{\tau} u)]
+\Pi_S[(J\circ b) \Pi_S^{\perp}[\Psi^{\tau} (\Upsilon^{\tau} u)]\Big)\,
\end{aligned}
\end{equation*}
we have that $\Phi^{\tau}=\Pi_S^{\perp}\circ\Psi^{\tau}\circ \Upsilon^{\tau}$ 
is well-defined on $H^s$ and solves \eqref{rosso}.

\noindent
Now we show that $\Upsilon^{\tau}-\mathrm{I}$ is
of the form \eqref{FiniteDimFormDP}.
By Taylor expansion at $\tau=0$ we get
\begin{equation}\label{verde}
\Upsilon^{\tau} u-u=
-\tau \big(\Psi^{\tau}_* \mathtt{Z} (\Upsilon^{\tau} u) \big)_{|_{\tau=0}}
+\int_0^{\tau}(1-t) \big(\partial_{t t} \Upsilon^{t}(u)\big)\,dt\,.
\end{equation}
Note that 
\[
\Psi^{\tau}_* \mathtt{Z} (\Upsilon^{\tau} u)=\sum_{j\in S} 
\big((\Phi^{\tau})^* g_j(\tau),  u \big)_{L^2}\,(\Psi^{\tau})^{-1}\chi_j(\tau)
+\sum_{j\in S} 
\big((\Phi^{\tau})^*\tilde{g}_j(\tau),  u \big)_{L^2}\,(\Psi^{\tau})^{-1}\tilde{\chi}_j(\tau)
\]
has already the form \eqref{FiniteDimFormDP} and 
$\big(\Psi^{\tau}_* \mathtt{Z} 
(\Upsilon^{\tau} u)\big)_{|_{\tau=0}}=\mathtt{Z} u$.
We denoted by $(\Psi^{\tau})^*$, $(\Phi^{\tau})^*$ the flows of the adjoint PDEs
\begin{equation*}
\partial_{\tau} (\Psi^{\tau})^{*} u=-b J( (\Psi^{\tau})^{*} u)\,,\qquad
\partial_{\tau} (\Phi^{\tau})^{*} u
=-\Pi_S^{\perp}\,b\, J\, \Pi_S^{\perp}[ (\Phi^{\tau})^{*} u]\,.
\end{equation*}
We have
\[
\int_0^{\tau} (1-t)\partial_{t t} \Upsilon^t(u)\,ds=
\sum_{j\in S} \int_0^{\tau}(1-t) \big(\mathtt{g}_j(t), u \big)_{L^2}\,
\mathtt{f}_j(t)\,dt+\sum_{j\in S}\int_0^{\tau} (1-t)
\big(\tilde{\mathtt{g}}_j(t), u \big)_{L^2}\,\tilde{\mathtt{f}}_j(t)\,dt
\]
where
\begin{align*}
\mathtt{g}_j&:=
\big(-(\Psi_*^s \mathtt{Z})^* (\Phi^s)^*-(\Upsilon^s)^* b\,J ((\Psi^s)^*)\big)(g_j)\,, 
\qquad \mathtt{f}_j:=-(J\circ b) ((\Psi^{s})^{-1})\chi_j
+(\Psi^s)^{-1}J(b'(s) e^{\mathrm{i} j x})\,,\\
\tilde{\mathtt{g}}_j&:=
\big(-(\Psi_*^s \mathtt{Z})^* (\Phi^s)^*-(\Upsilon^s)^* b\,
J ((\Psi^s)^*)\big)(\tilde{g}_j)
+(\Phi^s)^*\og(j)\Pi_S^{\perp}[b'(s) e^{\mathrm{i} j x}]\,, 
\qquad \tilde{\mathtt{f}}_j:=-(J\circ b) ((\Psi^{s})^{-1})\tilde{\chi}_j\,.
\end{align*}
Thus by \eqref{verde}, for $\tau=1$ we get
\begin{align}
&\Upsilon^1 u-u=\mathcal{R} u\,, 
\qquad 
\mathcal{R} u:=\mathcal{R}_1 u+\mathcal{R}_2 u\,,\label{viola1}\\
&\mathcal{R}_1 u:=-\mathtt{Z} u\,, 
\qquad \mathcal{R}_2 u:=\int_0^{1} (1-t)\partial_{t t} \Upsilon^t(u)\,dt\,,\label{viola2}
\end{align}
where $\mathcal{R}_1$ has the finite dimensional 
form \eqref{FiniteDimFormDP} and 
$\mathcal{R}_2$ has the form \eqref{FiniteDimFormDP}. 
Hence by Lemma \ref{lrofiniterank} we have
\begin{equation}\label{verde1}
\mathbb{M}^{\gamma}_{\mathcal{R}_1}(s, \tb) 
\lesssim_s 
\sum_{\lvert j \rvert\le C} \,\,\sup_{\tau\in [0,1]} 
(\lVert \mathtt{f}_j(\tau)\rVert_s^{\gamma, \calO_0}
\lVert \mathtt{g}_j(\tau) \rVert_{s_0}^{\gamma, \calO_0}
+\lVert \mathtt{f}_j(\tau)\rVert_{s_0}^{\gamma, \calO_0}
\lVert \mathtt{g}_j(\tau)\rVert_s^{\gamma, \calO_0})\,,
\end{equation}
\begin{equation}\label{verde2}
\mathbb{M}^{\gamma}_{\mathcal{R}_2}(s, \tb) 
\lesssim_s \sum_{\lvert j \rvert\le C} \,\,\sup_{\tau\in [0,1]} 
(\lVert \tilde{\mathtt{f}}_j(\tau)\rVert_s^{\gamma, \calO_0}
\lVert \tilde{\mathtt{g}}_j(\tau) \rVert_{s_0}^{\gamma, \calO_0}
+\lVert \tilde{\mathtt{f}}_j(\tau)\rVert_{s_0}^{\gamma, \calO_0}
\lVert \tilde{\mathtt{g}}_j(\tau)\rVert_s^{\gamma, \calO_0})\,.
\end{equation}
By using the estimates in Proposition $3.1$ and Corollary $3.2$ in \cite{FGP1}
%\ref{DPdiffeo}  
we have
\begin{equation}\label{giallo1}
\lVert \mathtt{f}_j \rVert^{\gamma, \calO_0}_s, 
\lVert \tilde{ \mathtt{g}_j} \rVert^{\gamma, \calO_0}_s 
\lesssim_s \lVert b \rVert^{\gamma, \calO_0}_{s+s_0+2}
+\lVert \partial_{\tau} b(\tau) \rVert^{\gamma, \calO_0}_{s+1}\,,\qquad
\lVert \tilde{\mathtt{f}_j}\rVert^{\gamma, \calO_0}_s\,, 
\lVert \mathtt{g}_j \rVert_s^{\gamma, \calO_0}
\lesssim_s \lVert b \rVert^{\gamma, \calO_0}_{s+s_0+2}\,.
\end{equation}
In the same way, the bounds for the variation 
on the $i$-variable (the second in \eqref{giallo2}) 
follows by the estimates on the derivatives 
of the coefficients $\mathtt{g}_j, \tilde{\mathtt{g}_j}, \chi_j, \tilde{\chi_j}$ 
whose depend on the variation $\Delta_{1 2}$ 
of the flows $\Phi^{\tau}$, $\Psi^{\tau}$ 
and their adjoints. %, for instance recall \eqref{flow25}.
We have proved that $\Upsilon^1u=u+\mathcal{R} u$ and hence 
$\Phi^1 u=\Psi^{1}\circ(\rm{I}+\mathcal{R})u$. 
By \eqref{verde1}, \eqref{verde2} and \eqref{giallo1} we obtain \eqref{giallo2}.
\end{proof}

The system \eqref{rosso} is a Hyperbolic PDE, thus we shall use a version of Egorov Theorem 
to study how pseudo differential operators change under the flow $\Phi^{\tau}$.
This is the content of Theorem $3.4$ in \cite{FGP1}
%\ref{EgorovQuantitativo} 
which provides precise estimates for the transformed pseudo differential operators.

%\subsubsection{A conjugation Lemma for a class of pseudo differential operators}

\noindent
The following proposition is the counter-part of 
Proposition $3.5$ in \cite{FGP1}.
It describes the structure of an operator like $\mathcal{L}_{\omega}$ 
conjugated by a flow of a system like \eqref{rosso}.

\begin{prop}[\textbf{Conjugation}]\label{ConjugationLemma}
Let $\mathcal{O}$ be a subset of $\R^\nu$. 
Fix $\rho\geq 3$, $\alpha\in \mathbb{N}$, 
$p\geq s_0$ 
and consider a linear operator
\begin{equation}\label{onizuka6}
\mathcal{L}:=\Pi_S^{\perp}\Big( \omega\cdot\partial_{\varphi}-
J \circ (m+a(\varphi, x)) +\mathcal{Q} \Big)
\end{equation}
where $m=m(\omega)$ is a real constant,
$a=a(\omega,i(\omega))\in C^\infty(\T^{\nu+1})$ is real valued, 
both are Lipschitz in $\omega\in \calO$ 
and $a $ is Lipschitz in the variable $i$.
Moreover $\calQ=\op(\mathtt{q}(\varphi, x, \xi))+\widehat{\calQ}$ 
with  $\widehat{\mathcal{Q}}\in\gotL_{\rho,p}(\calO)$  
and $\tq=\tq(\omega,i(\omega))\in S^{-1}$ satisfying
\begin{align}
\lvert \mathtt{q} \rvert_{-1, s, \alpha}^{\gamma, \calO}
&\lesssim_{s, \alpha} \tk_1+ \tk_2 \lVert f \rVert^{\g,\calO}_{s+\sigma}\,,\label{docq}\\
\lvert \Delta_{12} \mathtt{q}  \rvert_{-1, p, \alpha} 
&\lesssim_{p, \alpha} \tk_3(\lVert \Delta_{12} f  \rVert_{p+\sigma}
+\lVert \Delta_{12} f   \rVert_{p+\sigma}\lVert f \rVert_{p+\sigma})\,.\label{docqi}
\end{align}
Here $\tk_1, \tk_2, \tk_3,\sigma>0$ 
are constants depending on 
$\mathtt{q}$ while 
$f=f(\omega,i(\omega))\in C^\infty(\T^{\nu+1})$, 
is Lipschitz in $\omega$ and in the variable $i$ .
There exist $\tilde{\s}=\tilde{\s}(\rho)>0$ 
and $\delta_*:=\delta_*(\rho)$ such that, if
\begin{equation}\label{filini}
\lVert \bt \rVert^{\gamma, \calO}_{s_0+\tilde{\s}}
+\lVert a \rVert^{\gamma, \calO}_{s_0+\tilde{\s}}
+\tk_2 \lVert f\rVert^{\gamma, \calO}_{s_0+\tilde{\s}}
+\tk_1+\mathbb{M}^{\g}_{\widehat{\mathcal{Q}}}(s_0, \tb) 
\le \delta_*\,,
\end{equation}
the following holds for $p+\s\leq s_0+\tilde{\s}$.
Consider $\Phi:=\Phi^{1}$ the flow at time 
one of the system \eqref{rosso}, 
where $b$ is defined in \eqref{pseudo}. 
Then we have
\begin{equation}\label{ellepiu}
\mathcal{L}_+:=\Phi \mathcal{L}\Phi^{-1}
=\Pi_S^{\perp} \Big(\omega\cdot\partial_{\varphi}
-J \circ (m+ a_+(\varphi, x))+\mathcal{Q}_+\Big)
\end{equation}
where
\begin{equation}\label{Round}
m+a_+(\varphi, x):=
-(\omega\cdot\partial_{\varphi} \tilde{\beta})(\varphi, x+\beta(\varphi, x))
+(m+a(\varphi, x+\beta(\varphi, x)))
(1+\tilde{\beta}_x(\varphi, x+\beta(\varphi, x)))
\end{equation}
with $\tilde{\beta}$ the function such that 
$x+\tilde{\beta}(\varphi, x)$ is the inverse 
of the diffeomorphism of the torus $x\mapsto x+\beta(\varphi, x)$.
The operator 
$\mathcal{Q}_+:=\op(\mathtt{q}_+(\varphi, x, \xi))
+\widehat{\mathcal{Q}}_+$, with
\begin{equation}\label{lapiuimportante}
\begin{aligned}
\lvert \mathtt{q}_+ \rvert^{\g, \calO}_{-1, s, \alpha} 
&\lesssim_{s, \alpha, \rho} \tk_1+\tk_2 \lVert f\rVert^{\g, \calO}_{s+\tilde{\s}}
+\lVert \beta \rVert^{\g, \calO}_{s+\tilde{\s}}+\lVert a \rVert^{\g, \calO}_{s+\tilde{\s}}\,,\\
\lvert \Delta_{12} \mathtt{q}_+ \rvert_{-1, p, \alpha} 
&\lesssim_{p, \alpha, \rho} 
\mathtt{k}_3(\lVert \Delta_{12} f \rVert_{p+\tilde{\s}}
+\lVert \Delta_{12} f \rVert_{p+\tilde{\s}}\lVert f \rVert_{p+\tilde{\s}})
+\lVert \Delta_{12}\beta \rVert_{p+\tilde{\s}}
+\lVert \Delta_{12} a \rVert_{p+\tilde{\s}}
\end{aligned}
\end{equation}
and
$\widehat{\mathcal{Q}}_+\in\gotL_{\rho,p}(\calO)$
with, for $s_0\le s \le \mathcal{S}$,
\begin{equation}\label{JeTame}
\mathbb{M}^{\g}_{\widehat{\mathcal{Q}}_+}(s, \tb)
\lesssim_{s, \rho} \mathbb{M}^{\g}_{\widehat{\calQ}}(s, \tb)
+ \|\be\|^{\g,\calO}_{s+\tilde{\s}} +\tk_1
+\tk_2 \|f\|^{\g,\calO}_{s+\tilde{\s}}+\|a\|^{\g,\calO}_{s+\tilde{\s}}\,,
\end{equation}
for any $0\le \tb\le \rho-2$ and
\begin{equation}\label{ServelloniMazzantiVienDalMare}
\begin{aligned}
\mathbb{M}_{\Delta_{12} \widehat{\calQ}_+  }(p, \tb) &\lesssim_{p, \rho} \mathbb{M}_{\Delta_{12} \widehat{\calQ}}(p, \tb)\\
&+\mathtt{k}_3(\lVert \Delta_{12} f \rVert_{p+\tilde{\s}}
+\lVert \Delta_{12} f \rVert_{p+\tilde{\s}}\lVert f \rVert_{p+\tilde{\s}})
+\lVert \Delta_{12}\beta \rVert_{p+\tilde{\s}}+\lVert \Delta_{12} a \rVert_{p+\tilde{\s}}
\end{aligned}
\end{equation}
for any $0\le\tb\le \rho-3$.
\end{prop}

\begin{proof}
The strategy is the following.

\noindent
$\bullet$  We conjugate 
\begin{equation}\label{Llibero}
\mathcal{L}^0:= \omega\cdot\partial_{\varphi}-
J \circ (m+a(\varphi, x)) +\mathcal{Q} 
\end{equation}
by the flow $\Psi^{\tau}$ in \eqref{diffeotot}. 
In order to do this we use 
Proposition $3.5$ in \cite{FGP1}.

\medskip
\noindent
$\bullet$ The operator $\mathcal{L}^0$ 
differs from $\mathcal{L}$ 
by a infinitely regularizing operator 
of the form \eqref{FiniteDimFormDP}. 
By using this fact and Lemma \ref{differenzaFlussi} 
we estimate the difference 
between $\Phi^{\tau} \mathcal{L} (\Phi^{\tau})^{-1}$ 
and $\Psi^{\tau} \mathcal{L}^0 (\Psi^{\tau})^{-1}$. 
	
	\vspace{0.8em}
	
\noindent
Let $\Psi^{\tau}$ be the flow in \eqref{diffeotot}. 
The  \eqref{docq}-\eqref{docqi}  imply that 
$\calL_0$ in \eqref{Llibero} satisfies the hypotheses of 
Proposition $3.5$ in \cite{FGP1}. Hence, 
if $y:=x+\beta(\varphi, x)$, then we have
\begin{align*}
\mathcal{L}^0_+ h&:=\Psi \mathcal{L}^0 \Psi^{-1}
= \omega\cdot\partial_{\varphi} h-J 
\{\big(m+a_+(\varphi, x)\big)\,h\}
+\op(\mathtt{q}_+)+\widehat{\mathcal{Q}}_{\star} h\,,\\
m+a_+(\varphi, x)&:=
-\omega\cdot\partial_{\varphi}\tilde{\beta}(\varphi, y)
+(m+a(\varphi, y))(1+\tilde{\beta}_y(\varphi, y))\,.
\end{align*}
In particular $\mathtt{q}_+$ and 
$\widehat{\mathcal{Q}}_{\star} $ satisfy bounds like 
\eqref{lapiuimportante}-\eqref{ServelloniMazzantiVienDalMare}.
By Lemma \ref{differenzaFlussi} 
we have (recall \eqref{viola1}, \eqref{viola2})
\begin{equation*}
\begin{aligned}
\widehat{\mathcal{Q}}_{\star \star}:= 
\Phi \mathcal{L} \Phi^{-1}-\Pi_S^{\perp} \mathcal{L}^0_+ \Pi_S^{\perp} 
&=\Pi_S^{\perp} \Psi \mathcal{L}^0 \Psi^{-1}\Pi_S
+\Pi_S \Psi \mathcal{L}^0\Psi^{-1}
+\Psi \mathcal{L} \Gamma \Psi^{-1}
+\Psi \mathcal{L} \mathcal{R} \Psi^{-1}\\
&+\Psi \mathcal{R} \mathcal{L} \Gamma \Psi^{-1}
-\Psi \mathcal{L}^0 \Pi_S \Psi^{-1}
-\Psi \Pi_S \mathcal{L} \Pi_S^{\perp} \Psi^{-1}.
\end{aligned}
\end{equation*}
We define the remainder 
$\widehat{\mathcal{Q}}_+:=\widehat{\mathcal{Q}}_{\star}+\widehat{\mathcal{Q}}_{\star \star}$.
To conclude the proof we show 
that $\widehat{\mathcal{Q}}_{\star \star}$ 
satisfies the bounds \eqref{JeTame} 
and \eqref{ServelloniMazzantiVienDalMare}. We note that
\begin{equation}\label{primavera}
\begin{aligned}
&\Pi_S^{\perp} \Psi \mathcal{L}^0 \Psi^{-1}\Pi_S h
=\sum_{j\in S} (h, g_j^{(1)})_{L^2} \chi_j^{(1)}\,, 
\quad g_j^{(1)}:=e^{\mathrm{i} j x}\,, 
\quad \chi_j^{(1)}:=\Psi \mathcal{L}^0 \Psi^{-1}e^{\mathrm{i} j x}\,,\\
&\Pi_S \Psi \mathcal{L}^0\Psi^{-1} h
=\sum_{j\in S} (h, g_j^{(2)})_{L^2} \chi_j^{(2)}\,, 
\quad g_j^{(2)}:=\Psi \mathcal{L}^0\Psi^{-1} e^{\mathrm{i} j x}\,, 
\quad \chi_j^{(2)}=e^{\mathrm{i} j x}\,,\\
&\Psi \mathcal{L}^0 \Pi_S \Psi^{-1}
=\sum_{j\in S} (h, g_j^{(3)})_{L^2} \chi_j^{(3)}\,,
\quad g_j^{(3)}:=(\Psi^{-1})^* e^{\mathrm{i} j x}\,, 
\quad \chi_j^{(3)}:=\Psi \mathcal{L}^0 e^{\mathrm{i} j x}\,,\\
&\Psi \Pi_S \mathcal{L} \Pi_S^{\perp} \Psi^{-1} h
=\sum_{j\in S} (h, g_j^{(4)})_{L^2} \chi_j^{(4)}\,, 
\quad g_j^{(4)}:=(\mathcal{L} \Pi_S^{\perp} \Psi^{-1})^* e^{{\rm i} j x}\,, 
\quad \chi_j^{(4)}:=\Psi e^{{\rm i} j x}\,.
\end{aligned}
\end{equation}
Thus by Lemma 
\ref{lrofiniterank}, Corollary $3.2$ in \cite{FGP1} 
(for the estimates on $\Psi$) 
and \eqref{Llibero} we get the bounds 
\eqref{JeTame} and 
\eqref{ServelloniMazzantiVienDalMare} 
for the operators \eqref{primavera}.
The bounds on $\Psi \mathcal{L} \Gamma \Psi^{-1}, 
\Psi \mathcal{L} \mathcal{R} \Psi^{-1}, 
\Psi \mathcal{R} \mathcal{L} \Gamma \Psi^{-1}$ 
follow by Proposition $3.1$ in \cite{FGP1} and  \eqref{giallo2}.
\end{proof}

\begin{remark}
We point out that the remainder in Proposition 
\ref{ConjugationLemma} is of the order of $\beta$, 
i.e. of the generator of the torus diffeomorphism.  
This will create problems in fulfilling the smallness 
conditions in the KAM reducibility scheme 
of Section \ref{SezioneDiagonalization}, 
where a term is perturbative if it is 
small w.r.t. $\g^{3/2}$ ( and $\g\ll \e^2$ ).
\end{remark}

\subsection{Classes of ``smoothing'' operators}\label{ellerhosez}

In the first step of our reduction procedure 
(see Theorem \ref{risultatosez8})
%\ref{EgorovQuantitativo} 
we need to work with operators which are pseudo differential up to a remainder in the class $\mathfrak{L}_{\rho,p}$
defined as follows. This class of smoothing (in space) operators has been introduced in \cite{FGP1}.

\begin{defi}\label{ellerho}
Fix $s_0\geq (\nu+1)/2$ and $p,\,\mathcal{S}\in \mathbb{N}$ with $s_0\le p< \mathcal{S}$ with possibly $\mathcal{S}=+\infty$.
Fix $\rho\in\mathbb{N}$,  with $\rho\geq 3$ 
and consider any subset $\calO$ of $\R^\nu$. 
We denote by $\gotL_{\rho, p}=\gotL_{\rho, p}(\calO)$ 
the set of the linear operators 
$A=A(\omega)\colon H^s(\T^{\nu+1})\to H^{s}(\T^{\nu+1})$, $\omega\in \calO$ with the following properties:

\vspace{0.5em}
\noindent
$\bullet$ the operator $A$ is Lipschitz in $\omega$,

\vspace{0.5em}
\noindent		
$\bullet$  the operators $\partial_{\varphi}^{\vec{\tb}} A$,  $[\partial_{\varphi}^{\vec{\tb}} A, \partial_x]$, for all 
$\vec{\tb}=(\tb_1,\ldots,\tb_\nu)\in \mathbb{N}^{\nu}$ with 
$0\le |\vec{\tb}| \leq \rho-2$  have the following properties, 
for any $s_0\le s\le \mathcal{S}$, with possibly $\mathcal{S}=+\infty$:
\begin{itemize}
		
\item[(i)] 
for any $m_{1},m_{2}\in \mathbb{R}$, $m_{1},m_{2}\geq0$
and $m_{1}+m_{2}=\rho-|\vec{\tb}|$ 
one has that  the operator
$\langle D_{x}\rangle^{m_{1}} \del_{\f}^{\vec{\mathtt{b}}}A \langle D_{x}\rangle^{m_{2}}$
is Lip-$0$-tame according to Definition \ref{LipTameConstants} 
and we set
\begin{equation}\label{megaTame2}
\gotM^{\gamma}_{\del_{\f}^{\vec{\mathtt{b}}}A}(-\rho+|\vec{\tb}|,s):=
\sup_{\substack{m_{1}+m_{2}=\rho-|\vec{\tb}|\\
m_{1},m_{2}\geq0}}\gotM^{\gamma}_{\langle D_{x}\rangle^{m_{1}} \del_{\f}^{\vec{\mathtt{b}}}A 
\langle D_{x}\rangle^{m_{2}}}(0,s)\,;
\end{equation}
		
\item[(ii)]
for any $m_{1},m_{2}\in \mathbb{R}$, $m_{1},m_{2}\geq0$
and $m_{1}+m_{2}=\rho- |\vec{\tb}|-1$ 
one has that  
$\langle D_{x}\rangle^{m_{1}} [\del_{\f}^{\vec{\mathtt{b}}}A,\del_{x}] \langle D_{x}\rangle^{m_{2}}$
is Lip-$0$-tame according to Definition \ref{LipTameConstants} 
and we set
\begin{equation}\label{megaTame4}
\gotM^{\gamma}_{[\del_{\f}^{\vec{\mathtt{b}}}A,\del_{x}]}(-\rho+|\vec{\tb}|+1,s):=
\sup_{\substack{m_{1}+m_{2}=\rho-|\vec{\tb}|-1\\
m_{1},m_{2}\geq0}}\gotM^{\gamma}_{\langle D_{x}\rangle^{m_{1}} 
[\del_{\f}^{\vec{\mathtt{b}}}A,\del_{x}] \langle D_{x}\rangle^{m_{2}}}(0,s)\,.
\end{equation}
\end{itemize}
We define for $0\le \tb\le \rho-2$
\begin{equation}\label{Mdritta}
\begin{aligned}\mathbb{M}^{\gamma}_{A}(s, \mathtt{b}):=&\max_{0\leq |\vec{\tb}|\leq \tb}\max\left(
\gotM_{\del_{\f}^{\vec{\mathtt{b}}}A}^{\gamma}(-\rho+|\vec{\tb}|,s),
\gotM_{\del_{\f}^{\vec{\mathtt{b}}}[A,\del_{x}]}^{\gamma}(-\rho+|\vec{\tb}|+1,s)\right).
\end{aligned}
\end{equation}
Assume now that  the set $\calO$ 
and the operator $A$ depend on 
$i=i(\oo)$, and are well defined for $\oo\in \calO_0\subseteq\Omega_{\e}$
for all $i$ satisfying \eqref{IpotesiPiccolezzaIdeltaDP}.
%\eqref{ipopiccolezza}.
We consider  $i_{1}=i_{1}(\oo)$, $i_{2}=i_{2}(\oo)$ and 
for $\oo\in \calO(i_1)\cap\calO(i_2)$
we define 
\begin{equation}\label{DELTA12}
\Delta_{12}A:=A(i_1)-A(i_2)\,.
\end{equation}
%
%$\Delta_{12}A:=A(i_1)-A(i_2)$ as in  \eqref{DELTA12}.
%\begin{equation}\label{DELTA1200}
%\Delta_{12}A:=A(i_1)-A(i_2).
%\end{equation}
We require the following:

\vspace{0.5em}
\noindent		
$\bullet$ The operators
$\del_{\f}^{\vec{\mathtt{b}}}\Delta_{12}A $, 
$[\del_{\f}^{\vec{\mathtt{b}}}\Delta_{12}A ,\del_{x}]$, for  $0\le |\vec{\tb}|\leq \rho-3$, 
have the following properties, for any $s_0\le s\le \mathcal{S}$, 
with possibly $\mathcal{S}=+\infty$:

\begin{itemize}
\item[(iii)] 
for any $m_{1},m_{2}\in \mathbb{R}$, $m_{1},m_{2}\geq0$
and $m_{1}+m_{2}=\rho-|\vec{\tb}|-1$ 
one has that  
$\langle D_{x}\rangle^{m_{1}} \del_{\f}^{\vec{\mathtt{b}}} \Delta_{12}A  \langle D_{x}\rangle^{m_{2}}$
is  bounded on $H^{p}$ into itself. More precisely there is a positive constant 
$\gotN_{\del_{\f}^{\vec{\mathtt{b}}}\Delta_{12}A }(-\rho+|\vec{\tb}|+1,p)$ such that, for any $h\in H^{p}$,
we have
\begin{equation}\label{megalipTame2}
\sup_{\substack{m_{1}+m_{2}=\rho-|\vec{\tb}|-1\\
m_{1},m_{2}\geq0}}
\|\langle D_{x}\rangle^{m_{1}}
\del_{\f}^{\vec{\mathtt{b}}}\Delta_{12}A\langle D_{x}\rangle^{m_{2}} h\|_{p}\leq 
\gotN_{\del_{\f}^{\vec{\mathtt{b}}}\Delta_{12}A }(-\rho+|\vec{\tb}|+1,p)\|h\|_{p}\,;
\end{equation}
		
\item[(iv)] 
for any $m_{1},m_{2}\in \mathbb{R}$, $m_{1},m_{2}\geq0$
and $m_{1}+m_{2}=\rho-|\vec{\tb}|-2$ 
one has that  
$\langle D_{x}\rangle^{m_{1}} [\del_{\f}^{\vec{\mathtt{b}}} \Delta_{12}A ,\del_{x}] 
\langle D_{x}\rangle^{m_{2}}$
is  bounded on $H^{p}$ into itself. More precisely there is a positive constant 
$\gotN_{[\del_{\f}^{\vec{\mathtt{b}}}\Delta_{12}A,\del_{x}] }(-\rho+|\vec{\tb}|+2,p)$ such that for any 
$h\in H^{p}$ one has
\begin{equation}\label{megalipTame4}
\sup_{\substack{m_{1}+m_{2}=\rho-|\vec{\tb}|-2\\
m_{1},m_{2}\geq0}}
\|\langle D_{x}\rangle^{m_{1}}[
\del_{\f}^{\vec{\mathtt{b}}}\Delta_{12}A,\del_{x}]\langle D_{x}\rangle^{m_{2}} h\|_{p}\leq 
\gotN_{[\del_{\f}^{\vec{\mathtt{b}}}\Delta_{12}A, \del_{x}] }(-\rho+|\vec{\tb}|+2,p)\|h\|_{p}\,.
\end{equation}
\end{itemize}
		
We define for $0\le \tb\le \rho-3$
\begin{equation}\label{Mdrittaconlai}
\begin{aligned}\mathbb{M}_{\Delta_{12}A }(p, \mathtt{b}):=&\max_{0\le |\vec{\tb}|\leq \tb}\max \left(
\gotN_{\del_{\f}^{\vec{\mathtt{b}}}\Delta_{12}A }(-\rho+|\vec{\tb}|+1,p),
\gotN_{\del_{\f}^{\vec{\mathtt{b}}}[\Delta_{12}A ,\del_{x}]}(-\rho+|\vec{\tb}|+2,p)\right)\,.
\end{aligned}
\end{equation}

By construction one has that 
$\mathbb{M}^{\gamma}_{A}(s, \mathtt{b}_1)\leq 
\mathbb{M}^{\gamma}_{A}(s, \mathtt{b}_2)$
 if $\tb_1\le \tb_2\le \rho-2$ and 
$\mathbb{M}_{\Delta_{12}A }(p, \mathtt{b}_1)\leq 
\mathbb{M}_{\Delta_{12}A }(p, \mathtt{b}_2)$
if $\mathtt{b}_1\leq \mathtt{b}_2\leq \rho-3$.
\end{defi}

We shall also deal with ``tame'' operators in the following class.

\begin{defi}\label{Cuno}
Fix %$s_0\geq (\nu+1)/2$ and 
$\mathcal{S}\in \mathbb{N}$ with possibly $\mathcal{S}=+\infty$.
Fix $\tb=s_0+6\tau+6$ and consider $\calO\subseteq \R^\nu$. 
We denote by $\gotC_{-1}:=\gotC_{-1}(\calO)$ 
the set of the linear operators 
$A=A(\omega)\colon H^s(\T^{\nu+1})\to H^{s}(\T^{\nu+1})$, 
$\omega\in \calO$ 	
which satisfy the following for any $s_0\le s\le \mathcal{S}$: 
	
\noindent
$\bullet$ \;\; $\langle D_x \rangle^{1/2} A \langle D_x \rangle^{1/2}$, 
$\langle D_x \rangle^{1/2}\partial_{\varphi_m}^{\tb} A \langle D_x \rangle^{1/2}$,  
$\langle D_x \rangle^{1/2}[\partial_{\varphi_m}^{\tb} A, \partial_x]\langle D_x \rangle^{1/2}$, 
for $m=1, \dots, \nu$, $0\le \tb_1\le \tb$  
are Lip-$0$-tame operators 
(see Definition \ref{LipTameConstants}) and we define
\begin{equation}
\begin{aligned}
\mathfrak{M}^{\g}_{\partial_{\varphi_m}^{\tb_1}A}(-1, s)&:=
\mathfrak{M}^{\g}_{\langle D_x \rangle^{1/2}
\partial_{\varphi_m}^{\tb_1}A \langle D_x \rangle^{1/2}}(0, s),\\
\mathfrak{M}^{\g}_{\partial_{\varphi_m}^{\tb_1}[A, \partial_x]}(-1, s) 
&:=\mathfrak{M}^{\g}_{\langle D_x \rangle^{1/2}\partial_{\varphi_m}^{\tb_1}
[A, \partial_x]\langle D_x \rangle^{1/2}}(0, s) \,,
\end{aligned}
\end{equation}
%and
\begin{equation}\label{Mdritta2}
\mathbb{B}^{\g}_A(s):=\max_{\substack{0\le \tb_1\le \tb\\ m=1,\dots, \nu}} 
\max\Big(\mathfrak{M}^{\g}_{\partial_{\varphi_m}^{\tb_1}A}(-1, s), 
\mathfrak{M}^{\g}_{\partial_{\varphi_m}^{\tb_1}[A, \partial_x]}(-1, s)  \Big)\,.
\end{equation}	
Assume now that  the set $\calO$ 
and the operator $A$ depend on 
$i=i(\oo)$, and are well defined for $\oo\in \calO\subseteq\R^\nu$
for all $i$ satisfying \eqref{IpotesiPiccolezzaIdeltaDP}.
We consider  $i_{1}=i_{1}(\oo)$, $i_{2}=i_{2}(\oo)$ and 
for $\oo\in \calO(i_1)\cap\calO(i_2)$ 
let $\Delta_{12}A$ as in \eqref{DELTA12}.
%we define
%\begin{equation}\label{DELTA12}
%\Delta_{12}A:=A(i_1)-A(i_2)\,.
%\end{equation}
We require the following:
	
\noindent		
$\bullet$ \;\; $\langle D_x \rangle^{1/2} \partial_{\varphi_m}^{\tb_1} 
\Delta_{12} A \langle D_x \rangle^{1/2}$, 
$\langle D_x \rangle^{1/2}[\partial_{\varphi_m}^{\tb_1} 
\Delta_{12} A, \partial_x]\langle D_x \rangle^{1/2}$ 
for $m=1, \dots, \nu$, $0\le \tb_1\le \tb$ are 
bounded operators on $H^{s_0}$ into itself. More precisely 
there are positive constants 
$\mathfrak{N}_{\partial_{\varphi_m}^{\tb_1}\Delta_{12} A}(-1, s_0)$, and 
$\mathfrak{N}_{[\partial_{\varphi_m}^{\tb_1}\Delta_{12} A,\del_{x}]}(-1, s_0)$
such that, for any $h\in H^{s_0}$, we have
\begin{equation}\label{StimeDeltai}
\begin{aligned}
\|\langle D_x \rangle^{1/2} \partial_{\varphi_m}^{\tb_1} 
\Delta_{12} A \langle D_x \rangle^{1/2}h\|_{s_0}
&\leq 
\mathfrak{N}_{\partial_{\varphi_m}^{\tb_1}\Delta_{12} A}(-1, s_0)\|h\|_{s_0}\,,\\
\|\langle D_x \rangle^{1/2} [\partial_{\varphi_m}^{\tb_1} \Delta_{12} A,\del_{x}] 
\langle D_x \rangle^{1/2}h\|_{s_0}
&\leq 
\mathfrak{N}_{[\partial_{\varphi_m}^{\tb_1}\Delta_{12} A,\del_{x}]}(-1, s_0)\|h\|_{s_0}\,.
\end{aligned}
\end{equation}
We define
\begin{equation}\label{Mdrittaconlai2}
\begin{aligned}
\mathbb{B}_{\Delta_{12} A}(s_0):=
\max_{\substack{0\le \tb_1\le \tb\\ m=1,\dots, \nu}} 
\max\Big(\mathfrak{N}_{\partial_{\varphi_m}^{\tb_1}\Delta_{12}A}(-1, s_0), 
\mathfrak{N}_{\partial_{\varphi_m}^{\tb_1}[\Delta_{12}A, \partial_x]}(-1, s_0) \Big)\,.
\end{aligned}
\end{equation}
\end{defi}

\smallskip

The next Lemma shows that the finite rank operators of the form 
\eqref{FiniteDimFormDP} are in $\gotL_{\rho, p}$.
\begin{lem}\label{lrofiniterank}
Fix $\rho\geq 3$. Let $\mathcal{R}$ be an operator of the form \eqref{FiniteDimFormDP},
where the functions $g_j(\tau), \chi_j(\tau)$ belong to $H^s$ for $\tau\in [0, 1]$ and depend in a Lipschitz way on the parameter $\omega\in \calO \subset \mathbb{R}^{\nu}$. Then 
there exists $\su=\su(\rho)>0$ such that 
$\mathcal{R}$ belongs to $\gotL_{\rho, p}$ and
\begin{equation}
\begin{aligned}
&\mathbb{M}^{\g}_{\mathcal{R}}(s, \tb)\lesssim_s 
\sum_{\lvert j \rvert\le C} \,\,\sup_{\tau\in [0,1]} 
(\lVert \chi_j(\tau)\rVert_{s+\su}^{\gamma, \calO}
\lVert g_j \rVert_{s_0+\su}^{\gamma, \calO}
+\lVert \chi_j(\tau)\rVert_{s_0+\su}^{\gamma, \calO}
\lVert g_j(\tau)\rVert_{s+\su}^{\gamma, \calO})\,,
\end{aligned}
\end{equation}
\begin{equation}
\begin{aligned}
\mathbb{M}_{\Delta_{12} \mathcal{R}  }(p, \tb) \le_{p} 
\sum_{\lvert j \rvert\le C} &
\,\,\sup_{\tau\in [0,1]} \Big( \lVert \Delta_{12} \chi_j(\tau)  \rVert_{p+\su}
\lVert g_j \rVert_{p+\su}+
\lVert  \chi_j(\tau) \rVert_{p+\su}\lVert \Delta_{12} g_j   \rVert_{p+\su} \Big)\,.
\end{aligned}
\end{equation}
\end{lem}

\begin{proof}
The Lemma follows by  reasoning  as
in  the proof of Lemma B.2 in \cite{FGP1}
and using the explicit formula 
\eqref{FiniteDimFormDP}.
\end{proof}

We conclude this section by showing the connection between the class 
$\mathfrak{L}_{\rho,p}$ and the class $\mathfrak{C}_{-1}$ 
in Definition \ref{Cuno}.

\begin{lem}\label{IncluecompoCLASSI}
Consider $\mathtt{b}\in \mathbb{N}$ and $\rho\in \mathbb{N}$ 
with $\rho\geq\mathtt{b}+3$.
The following holds.
\begin{itemize}
\item[(i)] If $A\in \gotL_{\rho,p}$ (see Definition \ref{ellerho}) then 
$A\in \gotC_{-1}$ (see Definition \ref{Cuno})
with 
\begin{equation}\label{pasta10}
\mathbb{B}^{\gamma}_{A}(s,\mathtt{b})
\leq_{\rho,s}\mathbb{M}^{\gamma}_{A}(s,\rho-2)\,, 
\qquad 
\mathbb{B}_{\Delta_{12}A }(p,\mathtt{b})
\leq_{\rho,p}\mathbb{M}_{\Delta_{12}A }(p,\rho-3)\,.
\end{equation}

\item[(ii)] 
Consider a symbol $a=a(\oo, i(\omega))$ in $S^{m}$ with $m\leq-1$
depending on $\oo\in \calO_0\subset \mathbb{R}^{\nu}$ in a Lipschitz way 
and on $i$ in a  Lipschitz way and let $A:=\op(a(x,\x))$.
Then one has that $A\in \gotC_{-1}$
with
\begin{equation}\label{pasta11}
\mathbb{B}^{\gamma}_{A}(s,\mathtt{b})
\leq_{s}|a|^{\g,\calO_0}_{m,s+\tb+2,\al}\,.
\end{equation}

\item[(iii)] Let $A,B\in \gotC_{-1}$. Then $A\circ B\in \gotC_{-1}$
with 
\begin{align}
\mathbb{B}^{\gamma}_{A\circ B}(s,\mathtt{b})&\leq_{s}
\mathbb{B}^{\gamma}_{A}(s,\mathtt{b})\mathbb{B}^{\gamma}_{A}(s_0,\mathtt{b})+
\mathbb{B}^{\gamma}_{A}(s_0,\mathtt{b})\mathbb{B}^{\gamma}_{A}(s,\mathtt{b})\label{pasta12}\\
\mathbb{B}_{\Delta_{12} (A\circ B)  }(p, \tb)&\le_{p, \rho} 
\mathbb{B}_{\Delta_{12} A  }(p,\mathtt{b})\mathbb{B}_B(p,\mathtt{b})\,,
+\mathbb{B}_{\Delta_{12} B  }(p,\mathtt{b})\mathbb{B}_A(p,\mathtt{b})\,.\label{pasta13}
\end{align}
\end{itemize}
\end{lem}

\begin{proof}
Let us check item $(i)$.
The fact that  $\langle D_{x}\rangle^{1/2}A\langle D_x \rangle^{1/2}$  
is Lip-$0$-tame  follows by \eqref{megaTame2} 
since $\rho\geq 1$. 
Indeed $\langle D_{x}\rangle^{-\rho+1}$ 
is bounded in $x$ and for any $h\in H^{s}$
\[
\begin{aligned}
\|\langle D_{x}\rangle^{\frac{1}{2}}A\langle D_x \rangle^{\frac{1}{2}}h\|_{s}^{\g,\calO_0}
&\leq 
\|\langle D_{x}\rangle^{-\rho+1}\big(\langle D_{x}\rangle^{\rho-\frac{1}{2}}A
\langle D_x \rangle^{\frac{1}{2}}\big)  h\|_{s}^{\g,\calO_0} \leq_{s}\gotM^{\gamma}_{A}(-\rho,s)\|h\|^{\gamma,\calO_0}_{s_0}+
\gotM^{\gamma}_{A}(-\rho,s_0)\|h\|^{\gamma,\calO_0}_{s}. 
\end{aligned}
\]
By  studying the tameness constant of the operators 
$\partial_{\varphi}^{\vec{\tb}} A, [A, \partial_x], [\partial_{\varphi}^{\vec{\tb}} A, \partial_x]
\Delta_{12}A ,\partial_{\f}^{\vec{\tb}}\Delta_{12}A , [\Delta_{12}A ,\partial_{x}]$,
$[\partial_{\f}^{\vec{\tb}}\Delta_{12}A ,\partial_{x}]$
for $\vec{\tb}\in \mathbb{N}^{\nu}$, $|\vec{\tb}|=\tb$, following the same reasoning as above,
one gets the \eqref{pasta10}.

\noindent
In order to prove  item $(ii)$ one can follow almost word  by word the proof of Lemma $A.4$.  
%\comment{frase lasciata a meta'}
\noindent
Let us check \eqref{pasta12}.
Let $\tb\in \mathbb{N}$ and consider $0\leq \tb_1\leq \mathtt{b}$, $m=1, \dots, \nu$. 
One has 
\begin{equation}\label{pasta15}
\langle D_{x}\rangle^{\frac{1}{2}}\partial_{\f_m}^{\tb_1}(A\circ B)
\langle D_{x}\rangle^{\frac{1}{2}}
=\sum_{\mathtt{c}_1+\mathtt{c}_2=\tb_1}C(\mathtt{c}_1, \mathtt{c}_2) \langle D_{x}\rangle^{\frac{1}{2}}(\partial_{\f_m}^{\mathtt{c}_1}A)
(\partial_{\f_m}^{\mathtt{c}_2}B)\langle D_{x}\rangle^{\frac{1}{2}}\,.
\end{equation}
We show that each summand in \eqref{pasta15} is a Lip-$0$-tame operator.
We have for $h\in H^{s}$
\begin{equation}\label{pasta16}
\begin{aligned}
&\|\langle D_{x}\rangle^{\frac{1}{2}}
(\partial_{\f_m}^{\mathtt{c}_1}A)
(\partial_{\f_m}^{\mathtt{c}_2}B)\langle D_{x}\rangle^{\frac{1}{2}}h
\|_{s}^{\gamma,\calO_0}
\leq 
\|\langle D_{x}\rangle^{\frac{1}{2}}
(\partial_{\f_m}^{\mathtt{c}_1}A)\langle D_{x}\rangle^{\frac{1}{2}}
\langle D_{x}\rangle^{-1}
\langle D_{x}\rangle^{\frac{1}{2}}
(\partial_{\f_m}^{\mathtt{c}_2}B)\langle D_{x}\rangle^{\frac{1}{2}}h
\|_{s}^{\gamma,\calO_0}\\
&\qquad \quad\leq_{s}
(\mathbb{B}^{\gamma}_{A}(s,\mathtt{b})
\mathbb{B}^{\gamma}_{B}(s_0,\mathtt{b})
+\mathbb{B}^{\gamma}_{A}(s_0,\mathtt{b})\mathbb{B}^{\gamma}_{B}(s,\mathtt{b}))\|h\|_{s_0}^{\gamma,\calO_0}
+\mathbb{B}^{\gamma}_{A}(s_0,\mathtt{b})
\mathbb{B}^{\gamma}_{B}(s_0,\mathtt{b})\|h\|_{s}^{\gamma,\calO_0}\,.
\end{aligned}
\end{equation}
In \eqref{pasta16} 
we used the fact that 
$\langle D_{x}\rangle^{\frac{1}{2}}
(\partial_{\f_m}^{\mathtt{c}_1}A)\langle D_{x}\rangle^{\frac{1}{2}}$
and $\langle D_{x}\rangle^{\frac{1}{2}}
(\partial_{\f_m}^{\mathtt{c}_2}B)\langle D_{x}\rangle^{\frac{1}{2}}$ 
are $0$-tame by hypothesis (see Definition \eqref{Cuno}).
This proves \eqref{pasta12} for the operators 
$A\circ B$ and $\partial_{\f_m}^{\tb_1}(A\circ B)$ 
for any $0\le \tb_1\le \tb$, $m=1, \dots, \nu$. 
One concludes the proof of \eqref{pasta12} and \eqref{pasta13} 
followings the same ideas used above. 
For further details we refer to the proof of Lemma $B.1$ in \cite{FGP1}.
%\ref{chiusuracompoclasseL} 
\end{proof}

\begin{lem}\label{commutatoC1}
Let $X,Y\in \gotC_{-1}$ then 
$\mathrm{ad}^k_{X}[Y]\in \gotC_{-1}$ for any 
$k\geq 1$ (recall \eqref{def:AD}). 
Moreover for any $k\geq 1$ we have
\begin{equation}\label{alieni}
\mathbb{B}^{\g}_{\mathrm{ad}^k_{X}[Y]}(s)
\lesssim_s \mathbb{B}_{X}^{\g}(s)
\big(\mathbb{B}_{X}^{\g}(s_0)\big)^{k-1}\mathbb{B}^{\g}_Y(s_0)
+\big(\mathbb{B}_{X}^{\g}(s_0)\big)^k \mathbb{B}^{\g}_Y(s)\,.
\end{equation}
Moreover if $X, Y$ depend on some parameter $i$ we have
\begin{equation}\label{chegioia}
\mathbb{B}_{\Delta_{12}\mathrm{ad}^k_{X}[Y]}(s_0)\lesssim \sum_{j=1}^{k-1} \mathbb{B}_{X(i_2)}^{j-1}(s_0)\,\mathbb{B}_{\Delta_{12} X}(s_0)\,\mathbb{B}_{X(i_1)}^{k-j}(s_0)\,\mathbb{B}_{Y(i_1)}(s_0)+\mathbb{B}_{X(i_2)}^k \mathbb{B}_{\Delta_{12} Y}(s_0).
\end{equation}
\end{lem}

\begin{proof}
It follows by using the formula
\[
\Delta_{12} \mathrm{ad}_X^k[Y]=\sum_{j=1}^{k-1} \mathrm{ad}_{X(i_2)}^{j-1}\,\mathrm{ad}_{\Delta_{12} X}\,\mathrm{ad}^{k-j}_{X(i_1)}[Y(i_1)]+\mathrm{ad}_{X(i_2)}^{k}[\Delta_{12} Y]
\]
and applying iteratively the estimates of Lemma \ref{IncluecompoCLASSI}.
\end{proof}

\begin{lem}\label{LemmaAggancio}
 Let $A\in \mathfrak{C}_{-1}$ then $A$ is a Lip-$-1$-modulo tame operator
 according to Definition \ref{menounomodulotame}. 
 Moreover
\begin{align}
&\mathfrak{M}^{\sharp, \gamma^{3/2}}_A( s)\le 
\max_{m=1, \dots, \nu} 
\mathfrak{M}^{\gamma^{3/2}}_{\partial_{\varphi_m}^{s_0}  [A, \partial_x]}(-1,s)\,,
\qquad \quad\mathfrak{M}^{\sharp, \gamma^{3/2}}_A( s, \tb_0)\le 
\max_{m=1,\dots, \nu}  
\mathfrak{M}^{\gamma^{3/2}}_{\partial_{\varphi_m}^{s_0+\tb_0}  [A, \partial_x]}(-1,s)\,,\label{chavez1}\\
&\|\lD^{1/2}\underline{\Delta_{12} A}\lD^{1/2}\|_{\mathcal L(H^{s_0})}, \,\, 
\|\lD^{1/2}\underline{\Delta_{12} \langle 
\partial_{\f}\rangle^{{\mathtt b_0}} A  }
\lD^{1/2}\|_{\mathcal L(H^{s_0})} \le  
\mathbb{B}_{\Delta_{12} A}(s_0, \mathtt{b}_0)\,.\label{chavez20}
\end{align}
\end{lem}

\begin{proof}
It follows 
arguing as  in the proof of Lemma $A.4$ 
in Appendix $A$ of \cite{FGP1} using the definition of $\mathfrak{C}_{-1}$
in Definition \ref{Cuno}.
\end{proof}

%\comment{Da qui e' incomprensibile, va spiegato meglio}

%In the following lemma we give bounds on the coefficients 
%of the matrices $\mathtt{B}_i$ appearing 
%at the numerator of \eqref{us} and \eqref{them}.

\subsection{Linear Birkhoff normal form}

The next lemma regards almost diagonal vector fields which belong to $\gotC_{-1}$.

\begin{lem}\label{Numeratore}
Let $X_A:=J A\in \gotC_{-1}$ with $A$ 
almost diagonal in the sense of Definition \ref{almostDIAG} .
Then $\lvert A_{j, l}^{j', l'} \rvert\le C\langle j,j'\rangle^{-1}$
for some constant $C>0$.
\end{lem}

\begin{proof}
The operator $B:=\langle D_x \rangle^{1/2} J A \langle D_x \rangle^{1/2}$ 
belongs to $\mathcal{L}(H^s)$ for any $s$. 
Then $\lvert B_{j, l}^{j', l'} \rvert\le C$ for all 
$j, j'\in \mathbb{Z}$, $l, l'\in\mathbb{Z}^{\nu}$. 
The thesis follows by the fact that
\[
\lvert B_{j, l}^{j', l'} \rvert
=\lvert A_{j, l}^{j', l'} \rvert \langle j \rangle^{1/2}\langle j' \rangle^{1/2}\lvert 
\og(j)\rvert\gtrsim \lvert A_{j, l}^{j', l'} \rvert \langle j \rangle^{1/2}
\langle j' \rangle^{1/2}\lvert j \rvert
\]
and $\lvert j \rvert\gtrsim\langle j \rangle^{1/2}\langle j' \rangle^{1/2}$, since 
$A$ is almost diagonal.
\end{proof}

 %In the following lemma we prove the lower bounds for the denominators $\delta_{\ell j j'}$.

%\subsubsection{Flows of smoothing ``tame'' operators}
We now study the flows generated by operators belonging to the  class $\mathfrak{C}_{-1}$
given in Definition \ref{Cuno}.

\noindent
In subsection \ref{LinearBNF} 
we look for symplectic changes of variable 
$\Upsilon_i\colon H^s_{S^{\perp}}(\T^{\nu+1})\to H^s_{S^{\perp}}(\T^{\nu+1})$, 
$i=1,2, 3$, that are the time-$1$ flow of quadratic Hamiltonians
\begin{equation}\label{HAMa1}
H_{\mathtt{A}_i}(u):=\varepsilon\sum_{j, j'\in S^c} 
(\mathtt{A}_i)_j^{j'}(\varphi)\,u_{j'}\,\overline{u}_{j}\,,
\end{equation}
where $\mathtt{A}_i(\varphi)$ is a self-adjoint operator 
$\forall \varphi\in\T^{\nu}$ and thus 
\begin{equation}\label{Fi1dp}
\Upsilon_i:=\exp(\varepsilon\, J \mathtt{A}_i)
=\mathrm{I}_{H_S^{\perp}}+\varepsilon\, J \mathtt{A}_i
+\varepsilon^2 \frac{(J \mathtt{A}_i)^2}{2}
+\varepsilon^3 R_i\,, 
\quad R_i:=\sum_{k\geq 3} 
\frac{\varepsilon^{k-3}}{k!}\,(J \mathtt{A}_i)^k\,.
\end{equation}

\noindent
Given linear operators $\mathtt{B}_i$, $i=1,2,3$ define 
the matrices $\mathtt{A}_i$, $i=1,2,3$ as
\begin{align}
(\mathtt{A}_i)_{j, l}^{j', l'}&=(\mathtt{A}_i)_{j}^{j'}(l-l')
=-\frac{(\mathtt{B}_i)_{j}^{j'}(l-l')}{\delta_{l j j'}}\,, 
\qquad \delta_{l j j'}\neq 0\,,\quad \lvert j-j'\rvert\le 2\,i\,\overline{\jmath}_1\,,
\,\,\,\,\lvert l-l'\rvert\le i, \quad i=1, 2\,,\label{us}\\
(\mathtt{A}_3)_{j, l}^{j', l'}&=(\mathtt{A}_3)_{j}^{j'}(l-l')
=-\frac{(\mathtt{B}_3)_{j}^{j'}(l-l')}{\delta^*_{l j j'}}\,, 
\qquad \delta^*_{l j j'}\neq 0\,,\quad \lvert j-j'\rvert\le 6\,\overline{\jmath}_1\,,
\,\,\,\,\lvert l-l'\rvert\le 3\,,\label{them}
\end{align}
where, recall \eqref{dispersionLaw}, \eqref{TwistMatrixDP},\eqref{lambdaJ0},
\begin{equation}\label{deltinoino}
\delta_{\ell j j'}:=\overline{\omega}\cdot \ell+\og(j)-\og(j'), \quad
\delta_{\ell j j'}^*:=\delta_{\ell j j'}
+\varepsilon^2 ( \mathbb{A}\xi\cdot \ell +\og(j')\lal_{j'}-\og(j)\lal_j).
\end{equation}

\begin{lem}\label{Denominatore}
Let $j, j'\in S^c$, $j\neq j'$. If $\sum_{i=1}^{\nu} \overline{\jmath}_i\,\ell_i+j-j'=0$, $0<\lvert \ell \rvert\le 2$, $\delta_{\ell j j'}\neq 0$,
where $\delta_{\ell j j'}$ are given in \eqref{deltinoino},
 then there exists a constant $C$ depending on the set $S$ such that
 $ \lvert \delta_{\ell j j'} \rvert\geq C$.
\end{lem}

\begin{proof}
 If $\lvert \ell \rvert=1$ we have by the preservation of momentum
\[
\delta_{\ell, j, j'}=\og(j-j')-\og(j)+\og(j')=\frac{3 j j'(j-j')[3+j j'+(j-j')^2]}{(1+j^2)(1+j'^2)(1+(j-j')^2)}\,.
\]
It is easy to verify that (recall that $\lvert j-j'\rvert\le 2 \overline{\jmath}_1$)
\[
\lvert j\,\, j' \rvert\lvert j-j'\rvert\lvert 3+j j'+(j-j')^2 \rvert\geq \lvert j\,\, j' \rvert^2,
\quad (1+j^2)(1+j'^2)(1+(j-j')^2)\le K  \lvert j\,\, j' \rvert^2\overline{\jmath}_1^2\,.
\]
Which implies the thesis for $\lvert \ell \rvert=1$.
Now suppose $\lvert \ell \rvert=2$ and consider $j_1, j_2\in S$. We can write $\delta_{\ell j j'}=\og(j_1)+\og(j_2)+\og(j)-\og(j')$ and by the conservation of momentum
\begin{equation}\label{eurodespar}
\og(j_1)+\og(j_2)+\og(j)-\og(j')=(j_1+j_2)(j_1+j)(j_2+j) \,P(j_1, j_2, j)
\end{equation}
and $P$ is the rational function
\begin{equation}\label{esselunga}
P(x, y, z):=\frac{3+x^2+y^2+z^2+xy+xz+ y z+x y z(x+y+z)}{(1+x^2)(1+y^2)(1+z^2)(1+(x+y+z)^2)}\,.
\end{equation}
If $\lvert j \rvert > N$, 
where $N=N(S)$ is a large constant to be fixed and which depends on the set $S$, then
%, recalling \eqref{eurodespar} and \eqref{esselunga}
\[
\lvert (j_1+j_2) (j_1+j) (j_2+j)\rvert\geq  C_1\,j^2 
\]
for some constant $C_1:=C_1(N)>0$ (possibly small), provided that $N$ is large enough. Moreover 
\[
\lvert (1+j_1 j_2) j^2+(j_1^2 j_2+j_1 j_2^2) j+3+j_1^2+j_2^2 \rvert\geq C_2\, j^2\,,
\quad
\lvert (1+j_1^2)(1+j_2^2)(1+j^2)(1+(j_1+j_2+j)^2)\rvert\le C_3\, j^4
\]
for some small constant $C_2:=C_2(N)>0$, provided that $N$ is large enough, and some big constant $C_3:=C_3(S)>0$. Thus
$
\lvert \og(j_1)+\og(j_2)+\og(j)-\og(j') \rvert\geq C_1 C_2 C_{3}^{-1}>0.
$\\
Now consider $\lvert j \rvert\le N$. Then we have
\[
\lvert (j_1+j_2) (j_1+j) (j_2+j)\rvert\lvert (1+j_1 j_2) j^2+(j_1^2 j_2+j_1 j_2^2) j+3+j_1^2+j_2^2 \rvert\geq M_1
\]
\[
\lvert (1+j_1^2)(1+j_2^2)(1+j^2)(1+(j_1+j_2+j)^2\rvert\le M_2
\]
for some constant $M_1, M_2>0$ depending on $S$. Set $C_4:=M_1/ M_2$.\\
Therefore $\lvert \og(j_1)+\og(j_2)+\og(j)-\og(j')\rvert\geq C_4>0$. At the end we choose $C\geq\max\{ C_4, C_1 C_2/ C_3 \}$.
\end{proof}

\begin{lem}\label{delfino}
Assume that 
$\mathtt{B}_{i}$, $i=1,2,3$ 
are such that $J\mathtt{B}_{i}\in \mathfrak{C}_{-1} $
and that they are \emph{almost diagonal} 
(see Definition \ref{almostDIAG}).
Then, for any $\omega\in \mathcal{G}^{(1)}_0$ 
(see \eqref{divisoriLBNF3}), the following holds true:
	
\noindent
$(i)$ The linear vector fields 
$X_{\mathtt{A}_i}:=J \mathtt{A}_i$, 
with $\mathtt{A}_i$ defined as in \eqref{us},\eqref{them}, 
belongs to the class $\gotC_{-1}$, in particular it satisfies the following: 
\begin{equation}\label{bojack}
\mathbb{B}_{\varepsilon^i X_{\mathtt{A}_i}}^{\g}(s)
\le C(s) \varepsilon^i \quad i=1,2\,, 
\qquad  \mathbb{B}_{\varepsilon^i X_{\mathtt{A}_3}}^{\g}(s)
\le C(s) \varepsilon^3\g^{-1}\,,\qquad \forall s\geq s_0\,.
\end{equation}
Note that $X_{\mathtt{A}_i}$ does not depend on $i(\oo)$.
	
\noindent
$(ii)$ The transformation 
$\Upsilon_i\colon H^s(\T^{\nu+1})\to H^s(\T^{\nu+1})$, $i=1,2,3$ 
defined in \eqref{Fi1dp} is invertible and satisfies, 
for any $u=u(\omega)\in H^s$ Lipschitz in $\omega\in\calO_{\infty}^{2\g}$,
\begin{align}
\lVert (\Upsilon_i^{\pm 1}-\mathrm{I}) u \rVert_s^{\g, \calO_{\infty}^{2\g}}
&\lesssim \varepsilon^i C(s_0) 
\lVert u \rVert^{\g, \calO_{\infty}^{2\g}}_s
+ \varepsilon^i C(s) \lVert u \rVert^{\g, \calO_{\infty}^{2\g}}_{s_0}\,, 
\quad i=1,2\,,\label{upo1}\\
\lVert (\Upsilon_3^{\pm 1}-\mathrm{I}) u \rVert_s^{\g, \calO_{\infty}^{2\g}}
&\lesssim \varepsilon^3\gamma^{-1} C(s_0) 
\lVert u \rVert^{\g, \calO_{\infty}^{2\g}}_s
+ \varepsilon^3\gamma^{-1} C(s) \lVert u \rVert^{\g, \calO_{\infty}^{2\g}}_{s_0}\,. \label{upo3}
\end{align}
\end{lem}

\begin{proof} 
First of all notice that, 
by Lemmata \ref{Numeratore}, 
\ref{Denominatore} 
and the fact that $\omega\in \mathcal{G}_0^{(1)}$ 
(see \eqref{divisoriLBNF3}) 
we will have
\begin{equation}\label{rw}
\lvert (\mathtt{A}_i)_{j, l}^{j', l'} \rvert
\le \frac{C}{\langle j, j'\rangle} \quad i=1,2, 
\qquad \lvert (\mathtt{A}_3)_{j, l}^{j', l'} \rvert
\le \frac{C\gamma}{\langle j, j'\rangle}\,, 
\qquad \forall j, j'\in S^c,\,\, l, l'\in\mathbb{Z}^{\nu}
\end{equation}
for some constant $C>0$ depending on the set $S$.
	
\noindent
\emph{Proof of item (i).}
First we note that that 
$B:=\langle D_x \rangle^{1/2} J \mathtt{A}_i \langle D_x \rangle^{1/2}$ 
maps $H^s$ to itself for all $s\geq 0$. 
Indeed it is sufficient to exploit the fact that 
the matrix entries $B_j^{j'}(l-l')$ are 
uniformly bounded by a constant and 
$B$ is almost diagonal.
The matrix elements of 
$\partial_{\varphi_m}^{\tb} X_{\mathtt{A}_1}$, 
$[X_{\mathtt{A}_1}, \partial_x]$, 
$[\partial_{\varphi_m}^{\tb} X_{\mathtt{A}_1}, \partial_x]$ 
are respectively
\[
\langle \ell_m-\ell'_m \rangle^{\tb} \og(j) (\mathtt{A}_1)_{j}^{j'}(\ell-\ell')\,, \;\;
(j-j') \og(j) (\mathtt{A}_1)_{j}^{j'}(\ell-\ell')\,,\;\; 
\langle \ell_m-\ell'_m \rangle^{\tb}(j-j') \og(j) (\mathtt{A}_1)_{j}^{j'}(\ell-\ell')\,.
\]
Note that by the definition of $\mathtt{A}_1$ in \eqref{us}
\[
\langle \ell_m-\ell'_m \rangle^{\tb}, \lvert j-j'\rvert\le C
\]
for some constant $C$ depending on the set $S$. 
Thus arguing as above one can easily 
prove that $\partial_{\varphi_m}^{\tb} X_{\mathtt{A}_1}$, 
$[X_{\mathtt{A}_1}, \partial_x]$, 
$[\partial_{\varphi_m}^{\tb} X_{\mathtt{A}_1}, \partial_x]$ 
are $-1$-Lip-tame operators. 
This concludes the proof of the \eqref{bojack}
	
\noindent
\emph{Proof of item (ii).}
By \eqref{Fi1dp} we have
\begin{equation}\label{charlie}
(\Upsilon_i-\mathrm{I})u
=\sum_{k\geq 1} \frac{\varepsilon^{i k} X_{\mathtt{A}_i}^k u}{k!}\,, 
\qquad (\Upsilon_1^{-1}-\mathrm{I})u
=\sum_{k\geq 1} (-1)^k\frac{\varepsilon^{i k} X_{\mathtt{A}_i}^k u}{k!}\,.
\end{equation}
By using iteratively the property $(iii)$ of Lemma \ref{IncluecompoCLASSI} and 
item $(i)$ we have that
\[
\lVert X^k_{\mathtt{A}_i} u \rVert_s^{\g, \calO_{\infty}^{2\g}} 
\le \mathbb{B}^{\g}_{X_{\mathtt{A}_i}}(s) 
(\mathbb{B}^{\g}_{X_{\mathtt{A}_i}}(s_0) )^{k-1}
\lVert u \rVert_{s_0}^{\g, \calO_{\infty}^{2\g}}
+(\mathbb{B}^{\g}_{X_{\mathtt{A}_i}}(s_0) )^{k}
\lVert u \rVert_s^{\g, \calO_{\infty}^{2\g}}\,.
\]
By using this relation to estimate 
the Lip-Sobolev norm of \eqref{charlie} 
and by noting that $\varepsilon^{n} C(s_0)^n$ 
is a summable sequence, for $\varepsilon$ small enough, 
we prove the thesis.
\end{proof}

\end{document}